\pdfoutput=1 
\documentclass[12pt]{amsart}
\usepackage[utf8]{inputenc}
\usepackage[english]{babel}
\usepackage{amsmath,amssymb, mathtools, MnSymbol}
\usepackage{tikz}
\usetikzlibrary{decorations.pathreplacing, calc,intersections,fit, matrix, shapes.geometric, shapes.misc, patterns,cd}
\usepackage{tkz-euclide}
\usetkzobj{all}
\usepackage{IEEEtrantools}
\usepackage{mleftright}
\usepackage{wrapfig}
\usepackage{csquotes}
\usepackage{multicol}
\usepackage{cutwin}
\usepackage{xfrac}
\usepackage[all]{xy}
\usepackage{mathrsfs}
\usepackage{chngcntr}
\usepackage{tabularx}
\usepackage{longtable}
\usepackage{xr}

\usepackage[backend=biber,style=alphabetic, maxnames=5]{biblatex}
\renewbibmacro{in:}{\ifentrytype{article}{}{\printtext{\bibstring{in}\intitlepunct}}}
\usepackage{enumitem}
\setlist[enumerate]{nosep, label=(\arabic*)}

\usepackage{hyperref}
\hypersetup{
    pdftitle={Non-Hyperoctahedral Categories of Two-Colored Partitions, Part~II},
    pdfauthor={Alexander Mang, Moritz Weber},
    pdfsubject={Non-Hyperoctahedral Categories of Two-Colored Partitions},
    pdfkeywords={quantum groups, unitary easy quantum groups, unitary group,  tensor category, two-colored partitions, partitions of sets, categories of partitions}
}
\usepackage{xpatch}
\xpatchcmd{\proof}{\itshape}{\normalfont\proofnamefont}{}{}
\xpatchcmd{\paragraph}{\normalfont}{{\normalfont\itshape}}{}{}

\makeatletter
\@namedef{r@tocindent4}{0pt}
\@namedef{r@tocindent5}{0pt} 
\makeatother
\newcommand\Tstrut{\rule{0pt}{2.6ex}}
\makeatletter
\let\save@mathaccent\mathaccent
\newcommand*\if@single[3]{%
  \setbox0\hbox{${\mathaccent"0362{#1}}^H$}%
  \setbox2\hbox{${\mathaccent"0362{\kern0pt#1}}^H$}%
  \ifdim\ht0=\ht2 #3\else #2\fi
  }
\newcommand*\rel@kern[1]{\kern#1\dimexpr\macc@kerna}
\newcommand*\widebar[1]{\@ifnextchar^{{\wide@bar{#1}{0}}}{\wide@bar{#1}{1}}}
\newcommand*\wide@bar[2]{\if@single{#1}{\wide@bar@{#1}{#2}{1}}{\wide@bar@{#1}{#2}{2}}}
\newcommand*\wide@bar@[3]{%
  \begingroup
  \def\mathaccent##1##2{%
    \let\mathaccent\save@mathaccent
    \if#32 \let\macc@nucleus\first@char \fi
    \setbox\z@\hbox{$\macc@style{\macc@nucleus}_{}$}%
    \setbox\tw@\hbox{$\macc@style{\macc@nucleus}{}_{}$}%
    \dimen@\wd\tw@
    \advance\dimen@-\wd\z@
    \divide\dimen@ 3
    \@tempdima\wd\tw@
    \advance\@tempdima-\scriptspace
    \divide\@tempdima 10
    \advance\dimen@-\@tempdima
    \ifdim\dimen@>\z@ \dimen@0pt\fi
    \rel@kern{0.6}\kern-\dimen@
    \if#31
      \overline{\rel@kern{-0.6}\kern\dimen@\macc@nucleus\rel@kern{0.4}\kern\dimen@}%
      \advance\dimen@0.4\dimexpr\macc@kerna
      \let\final@kern#2%
      \ifdim\dimen@<\z@ \let\final@kern1\fi
      \if\final@kern1 \kern-\dimen@\fi
    \else
      \overline{\rel@kern{-0.6}\kern\dimen@#1}%
    \fi
  }%
  \macc@depth\@ne
  \let\math@bgroup\@empty \let\math@egroup\macc@set@skewchar
  \mathsurround\z@ \frozen@everymath{\mathgroup\macc@group\relax}%
  \macc@set@skewchar\relax
  \let\mathaccentV\macc@nested@a
  \if#31
    \macc@nested@a\relax111{#1}%
  \else
    \def\gobble@till@marker##1\endmarker{}%
    \futurelet\first@char\gobble@till@marker#1\endmarker
    \ifcat\noexpand\first@char A\else
      \def\first@char{}%
    \fi
    \macc@nested@a\relax111{\first@char}%
  \fi
  \endgroup
}
\makeatother
\newenvironment{mycenter}[1][\topsep]
  {\setlength{\topsep}{#1}\par\kern\topsep\centering}% \begin{mycenter}[<len>]
  {\par\kern\topsep}% \end{mycenter}
\newcommand{\proofnamefont}{\scshape}

\newtheoremstyle{break}% name
  {}%         Space above, empty = `usual value'
  {}%         Space below
  {\itshape}% Body font
  {\parindent}%         Indent amount (empty = no indent, \parindent = para indent)
  {\scshape}% Thm head font
  {.}%        Punctuation after thm head
  {\newline}% Space after thm head: \newline = linebreak
  {}%         Thm head spec
\theoremstyle{break}

\theoremstyle{plain}

\newtheorem{theorem}{Theorem}[section]
\newtheorem{proposition}[theorem]{Proposition}
\newtheorem{lemma}[theorem]{Lemma}

\newtheorem*{umain}{Main Theorem}
\newtheorem*{umainseries}{Main Theorem of the Series}

\theoremstyle{definition}
\newtheorem{definition}[theorem]{Definition}

\newtheorem{notation}[theorem]{Notation}
\newtheorem{axioms}[theorem]{Axioms}

\DeclareMathOperator{\inte}{int}
\DeclareMathOperator{\alte}{anti}
\DeclareMathOperator{\tre}{tr}
\DeclareMathOperator{\bde}{bd}
\DeclareMathOperator{\tye}{tp}

\DeclareMathOperator{\lcm}{lcm}

\newcommand{\mr}{\mathrm}
\newcommand{\mc}{\mathcal}
\newcommand{\mx}{\mathscr}
\newcommand{\eqpd}{\coloneqq}
\newcommand{\upp}[1]{\prescript{\filledsquare}{}{#1}}
\newcommand{\lop}[1]{\prescript{}{\filledsquare}{#1}}
\newcommand{\bup}[1]{\prescript{\square}{}{#1}}
\newcommand{\blo}[1]{\prescript{}{\square}{#1}}
\newcommand{\nnint}{{\mathbb{N}_0}}
\newcommand{\pint}{\mathbb{N}}
\newcommand{\integers}{\mathbb{Z}}
\newcommand{\Cp}{\mc{P}^{\circ\bullet}}
\newcommand{\Cpp}{\mc{P}^{\circ\bullet}_{2}}
\newcommand{\Cppnb}{\mc{P}^{\circ\bullet}_{2,\mathrm{nb}}}
\newcommand{\resbr}{\mx{B}_{\mathrm{qres}}}
\newcommand{\minbr}{\mx{B}_{\mathrm{qmin}}}
\newcommand{\pminbr}{\mx{B}_{\mathrm{pmin}}}
\newcommand{\tminbr}{\mx{B}_{\mathrm{min}}}
\newcommand{\dubr}{\prescript{\mathrm{co}}{}{\mx{B}}_{\mathrm{qdu}}}
\newcommand{\brt}[3]{\mathrm{Br}\left({#1}\mid {#3}\mid {#2}\right)}
\newcommand{\bbrt}[3]{\prescript{}{2}{\mathrm{Br}}\left({#1}\mid {#3}\mid {#2}\right)}
\newcommand{\bcbrt}[3]{\prescript{\mathrm{co}}{2}{\mathrm{Br}}\left({#1}\mid {#3}\mid {#2}\right)}
\newcommand{\qcbrt}[3]{\prescript{\mathrm{co}}{4}{\mathrm{Br}}\left({#1}\mid {#3}\mid {#2}\right)}
\newcommand{\brck}{\mathrm{Br}}
\newcommand{\trmw}{\mathrm{WIn}}
\newcommand{\trms}{\mathrm{SIn}}
\newcommand{\trmu}{\mathrm{UIn}}
\newcommand{\Arg}{\mathrm{Arg}}
\newcommand{\colors}{\{\circ,\bullet\}}
\newcommand{\toco}{\Sigma}
\newcommand{\brkt}[3]{\mathrm{Br}({#1}\mid {#3} \mid {#2})}
\newcommand{\idpt}[1]{\mathrm{Id}({#1})}
\newcommand{\spt}[1]{\mathrm{Si}({#1})}
\newcommand{\ltm}{\mathrm{LT}}
\newcommand{\rtm}{\mathrm{RT}}
\newcommand{\brts}{\mx{B}}
\newcommand{\bbrts}{\prescript{}{2}{\mx{B}}}
\newcommand{\mbrts}{\prescript{}{3+}{\mx{B}}}
\newcommand{\cbrts}{\prescript{\mathrm{co}}{}{\mx{B}}}
\newcommand{\bcbrts}{\prescript{\mathrm{co}}{2}{\mx{B}}}
\newcommand{\qcbrts}{\prescript{\mathrm{co}}{4}{\mx{B}}}
\newcommand{\clb}{\Lambda}
\newcommand{\cwb}{\Theta}
\newcommand{\cdf}{\Xi}
\newcommand{\tfs}{\mx{F}}
\newcommand{\nthbr}{\mathrm{O}}
\newcommand{\ntbr}{\mathrm{H}}
\newcommand{\ucbr}{\mathrm{G}}
\newcommand{\bir}{\mathrm{F}}
\newcommand{\cbid}{\mathrm{E}}
\newcommand{\cbir}{\mathrm{D}}
\newcommand{\szerox}{\mathcal{R}_0}
\newcommand{\embd}{\mathrm{Tr}}
\newcommand{\nhoc}{\mathsf{PCat}^{\circ\bullet}_{\mathrm{NHO}}}
\newcommand{\cotcp}{\mathsf{PCat}^{\circ\bullet}}
\newcommand{\dwi}[1]{\lsem{#1}\rsem}

\newcommand{\tightoverset}[2]{%
  \mathop{#2}\limits^{\vbox to -.5ex{\kern-0.75ex\hbox{$#1$}\vss}}}
\newcommand{\notastightoverset}[2]{%
  \mathop{#2}\limits^{\vbox to -.0ex{\kern-0.75ex\hbox{$#1$}\vss}}}

\newcommand{\PartBracketWWB}{%
  \begin{tikzpicture}[scale=0.25,baseline=0.04cm]
    \def\xdist{0.666}
    \def\ydist{1.25}
    \def\ypar{0.5}
    \node [scale=0.33, circle, draw=black, fill=white] (a1) at ({0*\xdist},{0*\ydist}) {};
    \node [scale=0.33, circle, draw=black, fill=white] (a2) at ({1*\xdist},{0*\ydist}) {};
    \node [scale=0.33, circle, draw=black, fill=black] (a3) at ({2*\xdist},{0*\ydist}) {};
    \node [scale=0.33, circle, draw=black, fill=white] (b1) at ({0*\xdist},{1*\ydist}) {};
    \node [scale=0.33, circle, draw=black, fill=white] (b2) at ({1*\xdist},{1*\ydist}) {};
    \node [scale=0.33, circle, draw=black, fill=black] (b3) at ({2*\xdist},{1*\ydist}) {};
    \draw (a1) -- ++ (0,{\ypar}) -| (a3);
    \draw (b1) -- ++ (0,{-\ypar}) -| (b3);
    \draw (a2) to (b2);
  \end{tikzpicture}
}

\newcommand{\PartBracketWBB}{%
  \begin{tikzpicture}[scale=0.25,baseline=0.04cm]
    \def\xdist{0.666}
    \def\ydist{1.25}
    \def\ypar{0.5}
    \node [scale=0.33, circle, draw=black, fill=white] (a1) at ({0*\xdist},{0*\ydist}) {};
    \node [scale=0.33, circle, draw=black, fill=black] (a2) at ({1*\xdist},{0*\ydist}) {};
    \node [scale=0.33, circle, draw=black, fill=black] (a3) at ({2*\xdist},{0*\ydist}) {};
    \node [scale=0.33, circle, draw=black, fill=white] (b1) at ({0*\xdist},{1*\ydist}) {};
    \node [scale=0.33, circle, draw=black, fill=black] (b2) at ({1*\xdist},{1*\ydist}) {};
    \node [scale=0.33, circle, draw=black, fill=black] (b3) at ({2*\xdist},{1*\ydist}) {};
    \draw (a1) -- ++ (0,{\ypar}) -| (a3);
    \draw (b1) -- ++ (0,{-\ypar}) -| (b3);
    \draw (a2) to (b2);
  \end{tikzpicture}
}

\newcommand{\PartBracketBBW}{%
  \begin{tikzpicture}[scale=0.25,baseline=0.04cm]
    \def\xdist{0.666}
    \def\ydist{1.25}
    \def\ypar{0.5}
    \node [scale=0.33, circle, draw=black, fill=black] (a1) at ({0*\xdist},{0*\ydist}) {};
    \node [scale=0.33, circle, draw=black, fill=black] (a2) at ({1*\xdist},{0*\ydist}) {};
    \node [scale=0.33, circle, draw=black, fill=white] (a3) at ({2*\xdist},{0*\ydist}) {};
    \node [scale=0.33, circle, draw=black, fill=black] (b1) at ({0*\xdist},{1*\ydist}) {};
    \node [scale=0.33, circle, draw=black, fill=black] (b2) at ({1*\xdist},{1*\ydist}) {};
    \node [scale=0.33, circle, draw=black, fill=white] (b3) at ({2*\xdist},{1*\ydist}) {};
    \draw (a1) -- ++ (0,{\ypar}) -| (a3);
    \draw (b1) -- ++ (0,{-\ypar}) -| (b3);
    \draw (a2) to (b2);
  \end{tikzpicture}
}

\newcommand{\PartBracketBWW}{%
  \begin{tikzpicture}[scale=0.25,baseline=0.04cm]
    \def\xdist{0.666}
    \def\ydist{1.25}
    \def\ypar{0.5}
    \node [scale=0.33, circle, draw=black, fill=black] (a1) at ({0*\xdist},{0*\ydist}) {};
    \node [scale=0.33, circle, draw=black, fill=white] (a2) at ({1*\xdist},{0*\ydist}) {};
    \node [scale=0.33, circle, draw=black, fill=white] (a3) at ({2*\xdist},{0*\ydist}) {};
    \node [scale=0.33, circle, draw=black, fill=black] (b1) at ({0*\xdist},{1*\ydist}) {};
    \node [scale=0.33, circle, draw=black, fill=white] (b2) at ({1*\xdist},{1*\ydist}) {};
    \node [scale=0.33, circle, draw=black, fill=white] (b3) at ({2*\xdist},{1*\ydist}) {};
    \draw (a1) -- ++ (0,{\ypar}) -| (a3);
    \draw (b1) -- ++ (0,{-\ypar}) -| (b3);
    \draw (a2) to (b2);
  \end{tikzpicture}
}

\newcommand{\PartBracketBBWW}{%
  \begin{tikzpicture}[scale=0.25,baseline=0.04cm]
    \def\xdist{0.666}
    \def\ydist{1.25}
    \def\ypar{0.5}
    \node [scale=0.33, circle, draw=black, fill=black] (a1) at ({0*\xdist},{0*\ydist}) {};
    \node [scale=0.33, circle, draw=black, fill=black] (a2) at ({1*\xdist},{0*\ydist}) {};
    \node [scale=0.33, circle, draw=black, fill=white] (a3) at ({2*\xdist},{0*\ydist}) {};
    \node [scale=0.33, circle, draw=black, fill=white] (a4) at ({3*\xdist},{0*\ydist}) {};
    \node [scale=0.33, circle, draw=black, fill=black] (b1) at ({0*\xdist},{1*\ydist}) {};
    \node [scale=0.33, circle, draw=black, fill=black] (b2) at ({1*\xdist},{1*\ydist}) {};
    \node [scale=0.33, circle, draw=black, fill=white] (b3) at ({2*\xdist},{1*\ydist}) {};
    \node [scale=0.33, circle, draw=black, fill=white] (b4) at ({3*\xdist},{1*\ydist}) {};
    \draw (a1) -- ++ (0,{\ypar}) -| (a4);
    \draw (b1) -- ++ (0,{-\ypar}) -| (b4);
    \draw (a2) to (b2);
    \draw (a3) to (b3);    
  \end{tikzpicture}
}

\newcommand{\PartBracketWWBB}{%
  \begin{tikzpicture}[scale=0.25,baseline=0.04cm]
    \def\xdist{0.666}
    \def\ydist{1.25}
    \def\ypar{0.5}
    \node [scale=0.33, circle, draw=black, fill=white] (a1) at ({0*\xdist},{0*\ydist}) {};
    \node [scale=0.33, circle, draw=black, fill=white] (a2) at ({1*\xdist},{0*\ydist}) {};
    \node [scale=0.33, circle, draw=black, fill=black] (a3) at ({2*\xdist},{0*\ydist}) {};
    \node [scale=0.33, circle, draw=black, fill=black] (a4) at ({3*\xdist},{0*\ydist}) {};
    \node [scale=0.33, circle, draw=black, fill=white] (b1) at ({0*\xdist},{1*\ydist}) {};
    \node [scale=0.33, circle, draw=black, fill=white] (b2) at ({1*\xdist},{1*\ydist}) {};
    \node [scale=0.33, circle, draw=black, fill=black] (b3) at ({2*\xdist},{1*\ydist}) {};
    \node [scale=0.33, circle, draw=black, fill=black] (b4) at ({3*\xdist},{1*\ydist}) {};
    \draw (a1) -- ++ (0,{\ypar}) -| (a4);
    \draw (b1) -- ++ (0,{-\ypar}) -| (b4);
    \draw (a2) to (b2);
    \draw (a3) to (b3);    
  \end{tikzpicture}
}

\newcommand{\PartBracketWWWB}{%
  \begin{tikzpicture}[scale=0.25,baseline=0.04cm]
    \def\xdist{0.666}
    \def\ydist{1.25}
    \def\ypar{0.5}
    \node [scale=0.33, circle, draw=black, fill=white] (a1) at ({0*\xdist},{0*\ydist}) {};
    \node [scale=0.33, circle, draw=black, fill=white] (a2) at ({1*\xdist},{0*\ydist}) {};
    \node [scale=0.33, circle, draw=black, fill=white] (a3) at ({2*\xdist},{0*\ydist}) {};
    \node [scale=0.33, circle, draw=black, fill=black] (a4) at ({3*\xdist},{0*\ydist}) {};
    \node [scale=0.33, circle, draw=black, fill=white] (b1) at ({0*\xdist},{1*\ydist}) {};
    \node [scale=0.33, circle, draw=black, fill=white] (b2) at ({1*\xdist},{1*\ydist}) {};
    \node [scale=0.33, circle, draw=black, fill=white] (b3) at ({2*\xdist},{1*\ydist}) {};
    \node [scale=0.33, circle, draw=black, fill=black] (b4) at ({3*\xdist},{1*\ydist}) {};
    \draw (a1) -- ++ (0,{\ypar}) -| (a4);
    \draw (b1) -- ++ (0,{-\ypar}) -| (b4);
    \draw (a2) to (b2);
    \draw (a3) to (b3);    
  \end{tikzpicture}
}

\newcommand{\PartBracketWBBB}{%
  \begin{tikzpicture}[scale=0.25,baseline=0.04cm]
    \def\xdist{0.666}
    \def\ydist{1.25}
    \def\ypar{0.5}
    \node [scale=0.33, circle, draw=black, fill=white] (a1) at ({0*\xdist},{0*\ydist}) {};
    \node [scale=0.33, circle, draw=black, fill=black] (a2) at ({1*\xdist},{0*\ydist}) {};
    \node [scale=0.33, circle, draw=black, fill=black] (a3) at ({2*\xdist},{0*\ydist}) {};
    \node [scale=0.33, circle, draw=black, fill=black] (a4) at ({3*\xdist},{0*\ydist}) {};
    \node [scale=0.33, circle, draw=black, fill=white] (b1) at ({0*\xdist},{1*\ydist}) {};
    \node [scale=0.33, circle, draw=black, fill=black] (b2) at ({1*\xdist},{1*\ydist}) {};
    \node [scale=0.33, circle, draw=black, fill=black] (b3) at ({2*\xdist},{1*\ydist}) {};
    \node [scale=0.33, circle, draw=black, fill=black] (b4) at ({3*\xdist},{1*\ydist}) {};
    \draw (a1) -- ++ (0,{\ypar}) -| (a4);
    \draw (b1) -- ++ (0,{-\ypar}) -| (b4);
    \draw (a2) to (b2);
    \draw (a3) to (b3);    
  \end{tikzpicture}
}

\newcommand{\PartBracketBBBW}{%
  \begin{tikzpicture}[scale=0.25,baseline=0.04cm]
    \def\xdist{0.666}
    \def\ydist{1.25}
    \def\ypar{0.5}
    \node [scale=0.33, circle, draw=black, fill=black] (a1) at ({0*\xdist},{0*\ydist}) {};
    \node [scale=0.33, circle, draw=black, fill=black] (a2) at ({1*\xdist},{0*\ydist}) {};
    \node [scale=0.33, circle, draw=black, fill=black] (a3) at ({2*\xdist},{0*\ydist}) {};
    \node [scale=0.33, circle, draw=black, fill=white] (a4) at ({3*\xdist},{0*\ydist}) {};
    \node [scale=0.33, circle, draw=black, fill=black] (b1) at ({0*\xdist},{1*\ydist}) {};
    \node [scale=0.33, circle, draw=black, fill=black] (b2) at ({1*\xdist},{1*\ydist}) {};
    \node [scale=0.33, circle, draw=black, fill=black] (b3) at ({2*\xdist},{1*\ydist}) {};
    \node [scale=0.33, circle, draw=black, fill=white] (b4) at ({3*\xdist},{1*\ydist}) {};
    \draw (a1) -- ++ (0,{\ypar}) -| (a4);
    \draw (b1) -- ++ (0,{-\ypar}) -| (b4);
    \draw (a2) to (b2);
    \draw (a3) to (b3);    
  \end{tikzpicture}
}

\newcommand{\PartBracketBWWW}{%
  \begin{tikzpicture}[scale=0.25,baseline=0.04cm]
    \def\xdist{0.666}
    \def\ydist{1.25}
    \def\ypar{0.5}
    \node [scale=0.33, circle, draw=black, fill=black] (a1) at ({0*\xdist},{0*\ydist}) {};
    \node [scale=0.33, circle, draw=black, fill=white] (a2) at ({1*\xdist},{0*\ydist}) {};
    \node [scale=0.33, circle, draw=black, fill=white] (a3) at ({2*\xdist},{0*\ydist}) {};
    \node [scale=0.33, circle, draw=black, fill=white] (a4) at ({3*\xdist},{0*\ydist}) {};
    \node [scale=0.33, circle, draw=black, fill=black] (b1) at ({0*\xdist},{1*\ydist}) {};
    \node [scale=0.33, circle, draw=black, fill=white] (b2) at ({1*\xdist},{1*\ydist}) {};
    \node [scale=0.33, circle, draw=black, fill=white] (b3) at ({2*\xdist},{1*\ydist}) {};
    \node [scale=0.33, circle, draw=black, fill=white] (b4) at ({3*\xdist},{1*\ydist}) {};
    \draw (a1) -- ++ (0,{\ypar}) -| (a4);
    \draw (b1) -- ++ (0,{-\ypar}) -| (b4);
    \draw (a2) to (b2);
    \draw (a3) to (b3);    
  \end{tikzpicture}
}

\newcommand{\PartBracketBBwW}{%
  \begin{tikzpicture}[scale=0.25,baseline=0.04cm]
    \def\xdist{0.666}
    \def\ydist{1.25}
    \def\ypar{0.5}
    \node [scale=0.33, circle, draw=black, fill=black] (a1) at ({0*\xdist},{0*\ydist}) {};
    \node [scale=0.33, circle, draw=black, fill=black] (a2) at ({1*\xdist},{0*\ydist}) {};
    \node [scale=0.33, circle, draw=black, fill=white] (a3) at ({4.2*\xdist},{0*\ydist}) {};
    \node [scale=0.33, circle, draw=black, fill=black] (b1) at ({0*\xdist},{1*\ydist}) {};
    \node [scale=0.33, circle, draw=black, fill=black] (b2) at ({1*\xdist},{1*\ydist}) {};
    \node [scale=0.33, circle, draw=black, fill=white] (b3) at ({4.2*\xdist},{1*\ydist}) {};
    \node [scale=0.55] (l1) at ({2.5*\xdist},{1*\ydist}) {$\otimes w$ };
    \draw (a1) -- ++ (0,{\ypar}) -| (a3);
    \draw (b1) -- ++ (0,{-\ypar}) -| (b3);
    \draw (a2) to (b2);
  \end{tikzpicture}
}

\newcommand{\PartBracketBWwW}{%
  \begin{tikzpicture}[scale=0.25,baseline=0.04cm]
    \def\xdist{0.666}
    \def\ydist{1.25}
    \def\ypar{0.5}
    \node [scale=0.33, circle, draw=black, fill=black] (a1) at ({0*\xdist},{0*\ydist}) {};
    \node [scale=0.33, circle, draw=black, fill=white] (a2) at ({1*\xdist},{0*\ydist}) {};
    \node [scale=0.33, circle, draw=black, fill=white] (a3) at ({4.2*\xdist},{0*\ydist}) {};
    \node [scale=0.33, circle, draw=black, fill=black] (b1) at ({0*\xdist},{1*\ydist}) {};
    \node [scale=0.33, circle, draw=black, fill=white] (b2) at ({1*\xdist},{1*\ydist}) {};
    \node [scale=0.33, circle, draw=black, fill=white] (b3) at ({4.2*\xdist},{1*\ydist}) {};
    \node [scale=0.55] (l1) at ({2.5*\xdist},{1*\ydist}) {$\otimes w$ };
    \draw (a1) -- ++ (0,{\ypar}) -| (a3);
    \draw (b1) -- ++ (0,{-\ypar}) -| (b3);
    \draw (a2) to (b2);
  \end{tikzpicture}
}

\newcommand{\PartBracketWBwB}{%
  \begin{tikzpicture}[scale=0.25,baseline=0.04cm]
    \def\xdist{0.666}
    \def\ydist{1.25}
    \def\ypar{0.5}
    \node [scale=0.33, circle, draw=black, fill=white] (a1) at ({0*\xdist},{0*\ydist}) {};
    \node [scale=0.33, circle, draw=black, fill=black] (a2) at ({1*\xdist},{0*\ydist}) {};
    \node [scale=0.33, circle, draw=black, fill=black] (a3) at ({4.2*\xdist},{0*\ydist}) {};
    \node [scale=0.33, circle, draw=black, fill=white] (b1) at ({0*\xdist},{1*\ydist}) {};
    \node [scale=0.33, circle, draw=black, fill=black] (b2) at ({1*\xdist},{1*\ydist}) {};
    \node [scale=0.33, circle, draw=black, fill=black] (b3) at ({4.2*\xdist},{1*\ydist}) {};
    \node [scale=0.55] (l1) at ({2.5*\xdist},{1*\ydist}) {$\otimes w$ };
    \draw (a1) -- ++ (0,{\ypar}) -| (a3);
    \draw (b1) -- ++ (0,{-\ypar}) -| (b3);
    \draw (a2) to (b2);
  \end{tikzpicture}
}

\newcommand{\PartBracketWWwB}{%
  \begin{tikzpicture}[scale=0.25,baseline=0.04cm]
    \def\xdist{0.666}
    \def\ydist{1.25}
    \def\ypar{0.5}
    \node [scale=0.33, circle, draw=black, fill=white] (a1) at ({0*\xdist},{0*\ydist}) {};
    \node [scale=0.33, circle, draw=black, fill=white] (a2) at ({1*\xdist},{0*\ydist}) {};
    \node [scale=0.33, circle, draw=black, fill=black] (a3) at ({4.2*\xdist},{0*\ydist}) {};
    \node [scale=0.33, circle, draw=black, fill=white] (b1) at ({0*\xdist},{1*\ydist}) {};
    \node [scale=0.33, circle, draw=black, fill=white] (b2) at ({1*\xdist},{1*\ydist}) {};
    \node [scale=0.33, circle, draw=black, fill=black] (b3) at ({4.2*\xdist},{1*\ydist}) {};
    \node [scale=0.55] (l1) at ({2.5*\xdist},{1*\ydist}) {$\otimes w$ };
    \draw (a1) -- ++ (0,{\ypar}) -| (a3);
    \draw (b1) -- ++ (0,{-\ypar}) -| (b3);
    \draw (a2) to (b2);
  \end{tikzpicture}
}

\newcommand{\PartBracketBWBW}{%
  \begin{tikzpicture}[scale=0.25,baseline=0.04cm]
    \def\xdist{0.666}
    \def\ydist{1.25}
    \def\ypar{0.5}
    \node [scale=0.33, circle, draw=black, fill=black] (a1) at ({0*\xdist},{0*\ydist}) {};
    \node [scale=0.33, circle, draw=black, fill=white] (a2) at ({1*\xdist},{0*\ydist}) {};
    \node [scale=0.33, circle, draw=black, fill=black] (a3) at ({2*\xdist},{0*\ydist}) {};
    \node [scale=0.33, circle, draw=black, fill=white] (a4) at ({3*\xdist},{0*\ydist}) {};
    \node [scale=0.33, circle, draw=black, fill=black] (b1) at ({0*\xdist},{1*\ydist}) {};
    \node [scale=0.33, circle, draw=black, fill=white] (b2) at ({1*\xdist},{1*\ydist}) {};
    \node [scale=0.33, circle, draw=black, fill=black] (b3) at ({2*\xdist},{1*\ydist}) {};
    \node [scale=0.33, circle, draw=black, fill=white] (b4) at ({3*\xdist},{1*\ydist}) {};
    \draw (a1) -- ++ (0,{\ypar}) -| (a4);
    \draw (b1) -- ++ (0,{-\ypar}) -| (b4);
    \draw (a2) to (b2);
    \draw (a3) to (b3);    
  \end{tikzpicture}
}

\newcommand{\PartBracketWBWB}{%
  \begin{tikzpicture}[scale=0.25,baseline=0.04cm]
    \def\xdist{0.666}
    \def\ydist{1.25}
    \def\ypar{0.5}
    \node [scale=0.33, circle, draw=black, fill=white] (a1) at ({0*\xdist},{0*\ydist}) {};
    \node [scale=0.33, circle, draw=black, fill=black] (a2) at ({1*\xdist},{0*\ydist}) {};
    \node [scale=0.33, circle, draw=black, fill=white] (a3) at ({2*\xdist},{0*\ydist}) {};
    \node [scale=0.33, circle, draw=black, fill=black] (a4) at ({3*\xdist},{0*\ydist}) {};
    \node [scale=0.33, circle, draw=black, fill=white] (b1) at ({0*\xdist},{1*\ydist}) {};
    \node [scale=0.33, circle, draw=black, fill=black] (b2) at ({1*\xdist},{1*\ydist}) {};
    \node [scale=0.33, circle, draw=black, fill=white] (b3) at ({2*\xdist},{1*\ydist}) {};
    \node [scale=0.33, circle, draw=black, fill=black] (b4) at ({3*\xdist},{1*\ydist}) {};
    \draw (a1) -- ++ (0,{\ypar}) -| (a4);
    \draw (b1) -- ++ (0,{-\ypar}) -| (b4);
    \draw (a2) to (b2);
    \draw (a3) to (b3);    
  \end{tikzpicture}
}

\newcommand{\PartBracketBBvWvW}{%
  \begin{tikzpicture}[scale=0.25,baseline=0.04cm]
    \def\xdist{0.666}
    \def\ydist{1.25}
    \def\ypar{0.5}
    \node [scale=0.33, circle, draw=black, fill=black] (a1) at ({0*\xdist},{0*\ydist}) {};
    \node [scale=0.33, circle, draw=black, fill=black] (a2) at ({1*\xdist},{0*\ydist}) {};
    \node [scale=0.33, circle, draw=black, fill=white] (a3) at ({3.75*\xdist},{0*\ydist}) {};
    \node [scale=0.33, circle, draw=black, fill=white] (a4) at ({6.5*\xdist},{0*\ydist}) {};
    \node [scale=0.33, circle, draw=black, fill=black] (b1) at ({0*\xdist},{1*\ydist}) {};
    \node [scale=0.33, circle, draw=black, fill=black] (b2) at ({1*\xdist},{1*\ydist}) {};
    \node [scale=0.33, circle, draw=black, fill=white] (b3) at ({3.75*\xdist},{1*\ydist}) {};
    \node [scale=0.33, circle, draw=black, fill=white] (b4) at ({6.5*\xdist},{1*\ydist}) {};
    \node [scale=0.55] (l1) at ({2.25*\xdist},{1*\ydist}) {$\otimes v$ };
    \node [scale=0.55] (l2) at ({5*\xdist},{1*\ydist}) {$\otimes v$ };        
    \draw (a1) -- ++ (0,{\ypar}) -| (a4);
    \draw (b1) -- ++ (0,{-\ypar}) -| (b4);
    \draw (a2) to (b2);
    \draw (a3) to (b3);    
  \end{tikzpicture}
}

\newcommand{\PartBracketWWvBvB}{%
  \begin{tikzpicture}[scale=0.25,baseline=0.04cm]
    \def\xdist{0.666}
    \def\ydist{1.25}
    \def\ypar{0.5}
    \node [scale=0.33, circle, draw=black, fill=white] (a1) at ({0*\xdist},{0*\ydist}) {};
    \node [scale=0.33, circle, draw=black, fill=white] (a2) at ({1*\xdist},{0*\ydist}) {};
    \node [scale=0.33, circle, draw=black, fill=black] (a3) at ({3.75*\xdist},{0*\ydist}) {};
    \node [scale=0.33, circle, draw=black, fill=black] (a4) at ({6.5*\xdist},{0*\ydist}) {};
    \node [scale=0.33, circle, draw=black, fill=white] (b1) at ({0*\xdist},{1*\ydist}) {};
    \node [scale=0.33, circle, draw=black, fill=white] (b2) at ({1*\xdist},{1*\ydist}) {};
    \node [scale=0.33, circle, draw=black, fill=black] (b3) at ({3.75*\xdist},{1*\ydist}) {};
    \node [scale=0.33, circle, draw=black, fill=black] (b4) at ({6.5*\xdist},{1*\ydist}) {};
    \node [scale=0.55] (l1) at ({2.25*\xdist},{1*\ydist}) {$\otimes v$ };
    \node [scale=0.55] (l2) at ({5*\xdist},{1*\ydist}) {$\otimes v$ };        
    \draw (a1) -- ++ (0,{\ypar}) -| (a4);
    \draw (b1) -- ++ (0,{-\ypar}) -| (b4);
    \draw (a2) to (b2);
    \draw (a3) to (b3);    
  \end{tikzpicture}
}

\newcommand{\PartBracketBBvdWvdW}{%
  \begin{tikzpicture}[scale=0.25,baseline=0.04cm]
    \def\xdist{0.666}
    \def\ydist{1.25}
    \def\ypar{0.5}
    \node [scale=0.33, circle, draw=black, fill=black] (a1) at ({0*\xdist},{0*\ydist}) {};
    \node [scale=0.33, circle, draw=black, fill=black] (a2) at ({1*\xdist},{0*\ydist}) {};
    \node [scale=0.33, circle, draw=black, fill=white] (a3) at ({4*\xdist},{0*\ydist}) {};
    \node [scale=0.33, circle, draw=black, fill=white] (a4) at ({7*\xdist},{0*\ydist}) {};
    \node [scale=0.33, circle, draw=black, fill=black] (b1) at ({0*\xdist},{1*\ydist}) {};
    \node [scale=0.33, circle, draw=black, fill=black] (b2) at ({1*\xdist},{1*\ydist}) {};
    \node [scale=0.33, circle, draw=black, fill=white] (b3) at ({4*\xdist},{1*\ydist}) {};
    \node [scale=0.33, circle, draw=black, fill=white] (b4) at ({7*\xdist},{1*\ydist}) {};
    \node [scale=0.55] (l1) at ({2.5*\xdist},{1*\ydist}) {$\otimes v'$ };
    \node [scale=0.55] (l2) at ({5.5*\xdist},{1*\ydist}) {$\otimes v'$ };        
    \draw (a1) -- ++ (0,{\ypar}) -| (a4);
    \draw (b1) -- ++ (0,{-\ypar}) -| (b4);
    \draw (a2) to (b2);
    \draw (a3) to (b3);    
  \end{tikzpicture}
}

\newcommand{\PartBracketWWvdBvdB}{%
  \begin{tikzpicture}[scale=0.25,baseline=0.04cm]
    \def\xdist{0.666}
    \def\ydist{1.25}
    \def\ypar{0.5}
    \node [scale=0.33, circle, draw=black, fill=white] (a1) at ({0*\xdist},{0*\ydist}) {};
    \node [scale=0.33, circle, draw=black, fill=white] (a2) at ({1*\xdist},{0*\ydist}) {};
    \node [scale=0.33, circle, draw=black, fill=black] (a3) at ({4*\xdist},{0*\ydist}) {};
    \node [scale=0.33, circle, draw=black, fill=black] (a4) at ({7*\xdist},{0*\ydist}) {};
    \node [scale=0.33, circle, draw=black, fill=white] (b1) at ({0*\xdist},{1*\ydist}) {};
    \node [scale=0.33, circle, draw=black, fill=white] (b2) at ({1*\xdist},{1*\ydist}) {};
    \node [scale=0.33, circle, draw=black, fill=black] (b3) at ({4*\xdist},{1*\ydist}) {};
    \node [scale=0.33, circle, draw=black, fill=black] (b4) at ({7*\xdist},{1*\ydist}) {};
    \node [scale=0.55] (l1) at ({2.5*\xdist},{1*\ydist}) {$\otimes v'$ };
    \node [scale=0.55] (l2) at ({5.5*\xdist},{1*\ydist}) {$\otimes v'$ };        
    \draw (a1) -- ++ (0,{\ypar}) -| (a4);
    \draw (b1) -- ++ (0,{-\ypar}) -| (b4);
    \draw (a2) to (b2);
    \draw (a3) to (b3);    
  \end{tikzpicture}
}

\newcommand{\PartHalfLibWBW}{%
  \begin{tikzpicture}[scale=0.25,baseline=0.04cm]
    \def\xdist{0.75}
    \def\ydist{1.25}
    \def\ypar{0.5}
    \node [scale=0.33, circle, draw=black, fill=white] (a1) at ({0*\xdist},{0*\ydist}) {};
    \node [scale=0.33, circle, draw=black, fill=black] (a2) at ({1*\xdist},{0*\ydist}) {};
    \node [scale=0.33, circle, draw=black, fill=white] (a3) at ({2*\xdist},{0*\ydist}) {};
    \node [scale=0.33, circle, draw=black, fill=white] (b1) at ({0*\xdist},{1*\ydist}) {};
    \node [scale=0.33, circle, draw=black, fill=black] (b2) at ({1*\xdist},{1*\ydist}) {};
    \node [scale=0.33, circle, draw=black, fill=white] (b3) at ({2*\xdist},{1*\ydist}) {};
    \draw (a1) to (b3);    
    \draw (a2) to (b2);
    \draw (a3) to (b1);    
  \end{tikzpicture}
}

\newcommand{\PartHalfLibBWB}{%
  \begin{tikzpicture}[scale=0.25,baseline=0.04cm]
    \def\xdist{0.75}
    \def\ydist{1.25}
    \def\ypar{0.5}
    \node [scale=0.33, circle, draw=black, fill=black] (a1) at ({0*\xdist},{0*\ydist}) {};
    \node [scale=0.33, circle, draw=black, fill=white] (a2) at ({1*\xdist},{0*\ydist}) {};
    \node [scale=0.33, circle, draw=black, fill=black] (a3) at ({2*\xdist},{0*\ydist}) {};
    \node [scale=0.33, circle, draw=black, fill=black] (b1) at ({0*\xdist},{1*\ydist}) {};
    \node [scale=0.33, circle, draw=black, fill=white] (b2) at ({1*\xdist},{1*\ydist}) {};
    \node [scale=0.33, circle, draw=black, fill=black] (b3) at ({2*\xdist},{1*\ydist}) {};
    \draw (a1) to (b3);    
    \draw (a2) to (b2);
    \draw (a3) to (b1);    
  \end{tikzpicture}
  }

  \newcommand{\PartHalfLibBBB}{%
  \begin{tikzpicture}[scale=0.25,baseline=0.04cm]
    \def\xdist{0.75}
    \def\ydist{1.25}
    \def\ypar{0.5}
    \node [scale=0.33, circle, draw=black, fill=black] (a1) at ({0*\xdist},{0*\ydist}) {};
    \node [scale=0.33, circle, draw=black, fill=black] (a2) at ({1*\xdist},{0*\ydist}) {};
    \node [scale=0.33, circle, draw=black, fill=black] (a3) at ({2*\xdist},{0*\ydist}) {};
    \node [scale=0.33, circle, draw=black, fill=black] (b1) at ({0*\xdist},{1*\ydist}) {};
    \node [scale=0.33, circle, draw=black, fill=black] (b2) at ({1*\xdist},{1*\ydist}) {};
    \node [scale=0.33, circle, draw=black, fill=black] (b3) at ({2*\xdist},{1*\ydist}) {};
    \draw (a1) to (b3);    
    \draw (a2) to (b2);
    \draw (a3) to (b1);    
  \end{tikzpicture}
  }

  \newcommand{\PartHalfLibWWW}{%
  \begin{tikzpicture}[scale=0.25,baseline=0.04cm]
    \def\xdist{0.75}
    \def\ydist{1.25}
    \def\ypar{0.5}
    \node [scale=0.33, circle, draw=black, fill=white] (a1) at ({0*\xdist},{0*\ydist}) {};
    \node [scale=0.33, circle, draw=black, fill=white] (a2) at ({1*\xdist},{0*\ydist}) {};
    \node [scale=0.33, circle, draw=black, fill=white] (a3) at ({2*\xdist},{0*\ydist}) {};
    \node [scale=0.33, circle, draw=black, fill=white] (b1) at ({0*\xdist},{1*\ydist}) {};
    \node [scale=0.33, circle, draw=black, fill=white] (b2) at ({1*\xdist},{1*\ydist}) {};
    \node [scale=0.33, circle, draw=black, fill=white] (b3) at ({2*\xdist},{1*\ydist}) {};
    \draw (a1) to (b3);    
    \draw (a2) to (b2);
    \draw (a3) to (b1);    
  \end{tikzpicture}
}

\newcommand{\PartCrossWW}{%
  \begin{tikzpicture}[scale=0.25,baseline=0.015cm]
    \def\xdist{1}
    \def\ydist{1}
    \node [scale=0.33, circle, draw=black, fill=white] (a1) at ({0*\xdist},{0*\ydist}) {};
    \node [scale=0.33, circle, draw=black, fill=white] (a2) at ({1*\xdist},{0*\ydist}) {};
    \node [scale=0.33, circle, draw=black, fill=white] (b1) at ({0*\xdist},{1*\ydist}) {};
    \node [scale=0.33, circle, draw=black, fill=white] (b2) at ({1*\xdist},{1*\ydist}) {};
    \draw (a1) to (b2);    
    \draw (a2) to (b1);
  \end{tikzpicture}
  }

  \newcommand{\PartIdenB}{%
  \begin{tikzpicture}[scale=0.25,baseline=0.015cm]
    \def\xdist{1}
    \def\ydist{1}
    \node [scale=0.33, circle, draw=black, fill=black] (a1) at ({0*\xdist},{0*\ydist}) {};
    \node [scale=0.33, circle, draw=black, fill=black] (b1) at ({0*\xdist},{1*\ydist}) {};
    \draw (a1) to (b1);    
  \end{tikzpicture}
  }

  \newcommand{\PartIdenW}{%
  \begin{tikzpicture}[scale=0.25,baseline=0.015cm]
    \def\xdist{1}
    \def\ydist{1}
    \node [scale=0.33, circle, draw=black, fill=white] (a1) at ({0*\xdist},{0*\ydist}) {};
    \node [scale=0.33, circle, draw=black, fill=white] (b1) at ({0*\xdist},{1*\ydist}) {};
    \draw (a1) to (b1);    
  \end{tikzpicture}
  }

\newcommand{\PartIdenUpBW}{%
  \begin{tikzpicture}[scale=0.25,baseline=-0.015cm]
    \def\xdist{1}
    \def\ydist{1}
    \def\ypar{1}
    \node [scale=0.33, circle, draw=black, fill=black] (a1) at ({0*\xdist},{1*\ydist}) {};
    \node [scale=0.33, circle, draw=black, fill=white] (a2) at ({1*\xdist},{1*\ydist}) {};
    \draw (a1) -- ++ (0,{-\ypar}) -| (a2);    
  \end{tikzpicture}
  }
  
\newcommand{\PartIdenLoBW}{%
  \begin{tikzpicture}[scale=0.25,baseline=-0.015cm]
    \def\xdist{1}
    \def\ydist{1}
    \def\ypar{1}
    \node [scale=0.33, circle, draw=black, fill=black] (a1) at ({0*\xdist},{0*\ydist}) {};
    \node [scale=0.33, circle, draw=black, fill=white] (a2) at ({1*\xdist},{0*\ydist}) {};
    \draw (a1) -- ++ (0,{\ypar}) -| (a2);    
  \end{tikzpicture}
  }
\newcommand{\PartIdenLoWB}{%
  \begin{tikzpicture}[scale=0.25,baseline=-0.015cm]
    \def\xdist{1}
    \def\ydist{1}
    \def\ypar{1}
    \node [scale=0.33, circle, draw=black, fill=white] (a1) at ({0*\xdist},{0*\ydist}) {};
    \node [scale=0.33, circle, draw=black, fill=black] (a2) at ({1*\xdist},{0*\ydist}) {};
    \draw (a1) -- ++ (0,{\ypar}) -| (a2);    
  \end{tikzpicture}
  }
  \newcommand{\PartIdenLoBB}{%
  \begin{tikzpicture}[scale=0.25,baseline=-0.015cm]
    \def\xdist{1}
    \def\ydist{1}
    \def\ypar{1}
    \node [scale=0.33, circle, draw=black, fill=black] (a1) at ({0*\xdist},{0*\ydist}) {};
    \node [scale=0.33, circle, draw=black, fill=black] (a2) at ({1*\xdist},{0*\ydist}) {};
    \draw (a1) -- ++ (0,{\ypar}) -| (a2);    
  \end{tikzpicture}
}
  \newcommand{\PartIdenLoWW}{%
  \begin{tikzpicture}[scale=0.25,baseline=-0.015cm]
    \def\xdist{1}
    \def\ydist{1}
    \def\ypar{1}
    \node [scale=0.33, circle, draw=black, fill=white] (a1) at ({0*\xdist},{0*\ydist}) {};
    \node [scale=0.33, circle, draw=black, fill=white] (a2) at ({1*\xdist},{0*\ydist}) {};
    \draw (a1) -- ++ (0,{\ypar}) -| (a2);    
  \end{tikzpicture}
}

  \newcommand{\PartGlobColWWBBTensor}{%
  \begin{tikzpicture}[scale=0.25,baseline=-0.015cm]
    \def\xdist{1}
    \def\ydist{1}
    \def\ypar{1}
    \node [scale=0.33, circle, draw=black, fill=white] (a1) at ({0*\xdist},{0*\ydist}) {};
    \node [scale=0.33, circle, draw=black, fill=white] (a2) at ({1*\xdist},{0*\ydist}) {};
    \node [scale=0.33, circle, draw=black, fill=black] (a3) at ({2.25*\xdist},{0*\ydist}) {};
    \node [scale=0.33, circle, draw=black, fill=black] (a4) at ({3.25*\xdist},{0*\ydist}) {};
    \draw (a1) -- ++ (0,{\ypar}) -| (a2);
    \draw (a3) -- ++ (0,{\ypar}) -| (a4);
    \node [ scale=0.666] at ($(a2)+({0.625*\xdist},{0.325*\ydist})$) {$\otimes$};
  \end{tikzpicture}
}

  \newcommand{\PartSpecBWBW}{%
    \begin{tikzpicture}[scale=0.25,baseline=-0.015cm]
    \def\xdist{0.666}
    \def\ydist{1.25}
    \def\ypar{0.5}
    \node [scale=0.33, circle, draw=black, fill=black] (a1) at ({0*\xdist},{0*\ydist}) {};
    \node [scale=0.33, circle, draw=black, fill=white] (a2) at ({1*\xdist},{0*\ydist}) {};
    \node [scale=0.33, circle, draw=black, fill=black] (a3) at ({2*\xdist},{0*\ydist}) {};
    \node [scale=0.33, circle, draw=black, fill=white] (a4) at ({3*\xdist},{0*\ydist}) {};
    \draw (a1) -- ++ (0,{0.66*\ydist}) -| (a3);
    \draw (a2) -- ++ (0,{1*\ydist}) -| (a4);
    \end{tikzpicture}
  }
  \newcommand{\PartSpecWBWB}{%
    \begin{tikzpicture}[scale=0.25,baseline=-0.015cm]
    \def\xdist{0.666}
    \def\ydist{1.25}
    \def\ypar{0.5}
    \node [scale=0.33, circle, draw=black, fill=white] (a1) at ({0*\xdist},{0*\ydist}) {};
    \node [scale=0.33, circle, draw=black, fill=black] (a2) at ({1*\xdist},{0*\ydist}) {};
    \node [scale=0.33, circle, draw=black, fill=white] (a3) at ({2*\xdist},{0*\ydist}) {};
    \node [scale=0.33, circle, draw=black, fill=black] (a4) at ({3*\xdist},{0*\ydist}) {};
    \draw (a1) -- ++ (0,{0.66*\ydist}) -| (a3);
    \draw (a2) -- ++ (0,{1*\ydist}) -| (a4);
    \end{tikzpicture}
  }

\newcommand{\PartBracketBWB}{%
  \begin{tikzpicture}[scale=0.25,baseline=0.04cm]
    \def\xdist{0.666}
    \def\ydist{1.25}
    \def\ypar{0.5}
    \node [scale=0.33, circle, draw=black, fill=black] (a1) at ({0*\xdist},{0*\ydist}) {};
    \node [scale=0.33, circle, draw=black, fill=white] (a2) at ({1*\xdist},{0*\ydist}) {};
    \node [scale=0.33, circle, draw=black, fill=black] (a3) at ({2*\xdist},{0*\ydist}) {};
    \node [scale=0.33, circle, draw=black, fill=black] (b1) at ({0*\xdist},{1*\ydist}) {};
    \node [scale=0.33, circle, draw=black, fill=white] (b2) at ({1*\xdist},{1*\ydist}) {};
    \node [scale=0.33, circle, draw=black, fill=black] (b3) at ({2*\xdist},{1*\ydist}) {};
    \draw (a1) -- ++ (0,{\ypar}) -| (a3);
    \draw (b1) -- ++ (0,{-\ypar}) -| (b3);
    \draw (a2) to (b2);
  \end{tikzpicture}
}
\newcommand{\PartBracketWBW}{%
  \begin{tikzpicture}[scale=0.25,baseline=0.04cm]
    \def\xdist{0.666}
    \def\ydist{1.25}
    \def\ypar{0.5}
    \node [scale=0.33, circle, draw=black, fill=white] (a1) at ({0*\xdist},{0*\ydist}) {};
    \node [scale=0.33, circle, draw=black, fill=black] (a2) at ({1*\xdist},{0*\ydist}) {};
    \node [scale=0.33, circle, draw=black, fill=white] (a3) at ({2*\xdist},{0*\ydist}) {};
    \node [scale=0.33, circle, draw=black, fill=white] (b1) at ({0*\xdist},{1*\ydist}) {};
    \node [scale=0.33, circle, draw=black, fill=black] (b2) at ({1*\xdist},{1*\ydist}) {};
    \node [scale=0.33, circle, draw=black, fill=white] (b3) at ({2*\xdist},{1*\ydist}) {};
    \draw (a1) -- ++ (0,{\ypar}) -| (a3);
    \draw (b1) -- ++ (0,{-\ypar}) -| (b3);
    \draw (a2) to (b2);
  \end{tikzpicture}
}

\newcommand{\PartBracketBBlB}{%
  \begin{tikzpicture}[scale=0.25,baseline=0.04cm]
    \def\xdist{0.666}
    \def\ydist{1.25}
    \def\ypar{0.5}
    \node [scale=0.33, circle, draw=black, fill=black] (a1) at ({0*\xdist},{0*\ydist}) {};
    \node [scale=0.33, circle, draw=black, fill=black] (a2) at ({1*\xdist},{0*\ydist}) {};
    \node [scale=0.33, circle, draw=black, fill=black] (a3) at ({5.2*\xdist},{0*\ydist}) {};
    \node [scale=0.33, circle, draw=black, fill=black] (b1) at ({0*\xdist},{1*\ydist}) {};
    \node [scale=0.33, circle, draw=black, fill=black] (b2) at ({1*\xdist},{1*\ydist}) {};
    \node [scale=0.33, circle, draw=black, fill=black] (b3) at ({5.2*\xdist},{1*\ydist}) {};
    \node [scale=0.55] (l1) at ({3.1*\xdist},{1.0*\ydist}) {$\otimes l\hspace{-1.75pt}-\hspace{-3pt}1$ };
    \draw (a1) -- ++ (0,{\ypar}) -| (a3);
    \draw (b1) -- ++ (0,{-\ypar}) -| (b3);
    \draw (a2) to (b2);
  \end{tikzpicture}
}

\newcommand{\PartBracketWWlW}{%
  \begin{tikzpicture}[scale=0.25,baseline=0.04cm]
    \def\xdist{0.666}
    \def\ydist{1.25}
    \def\ypar{0.5}
    \node [scale=0.33, circle, draw=black, fill=white] (a1) at ({0*\xdist},{0*\ydist}) {};
    \node [scale=0.33, circle, draw=black, fill=white] (a2) at ({1*\xdist},{0*\ydist}) {};
    \node [scale=0.33, circle, draw=black, fill=white] (a3) at ({5.2*\xdist},{0*\ydist}) {};
    \node [scale=0.33, circle, draw=black, fill=white] (b1) at ({0*\xdist},{1*\ydist}) {};
    \node [scale=0.33, circle, draw=black, fill=white] (b2) at ({1*\xdist},{1*\ydist}) {};
    \node [scale=0.33, circle, draw=black, fill=white] (b3) at ({5.2*\xdist},{1*\ydist}) {};
    \node [scale=0.55] (l1) at ({3.1*\xdist},{1.0*\ydist}) {$\otimes l\hspace{-1.75pt}-\hspace{-3pt}1$ };
    \draw (a1) -- ++ (0,{\ypar}) -| (a3);
    \draw (b1) -- ++ (0,{-\ypar}) -| (b3);
    \draw (a2) to (b2);
  \end{tikzpicture}
}

\newcommand{\PartIdenBW}{%
  \begin{tikzpicture}[scale=0.25,baseline=0.015cm]
    \def\xdist{1}
    \def\ydist{1}
    \node [scale=0.33, circle, draw=black, fill=black] (a1) at ({0*\xdist},{0*\ydist}) {};
    \node [scale=0.33, circle, draw=black, fill=white] (b1) at ({0*\xdist},{1*\ydist}) {};
    \draw (a1) to (b1);    
  \end{tikzpicture}
  }

\newcommand{\PartIdenWB}{%
  \begin{tikzpicture}[scale=0.25,baseline=0.015cm]
    \def\xdist{1}
    \def\ydist{1}
    \node [scale=0.33, circle, draw=black, fill=white] (a1) at ({0*\xdist},{0*\ydist}) {};
    \node [scale=0.33, circle, draw=black, fill=black] (b1) at ({0*\xdist},{1*\ydist}) {};
    \draw (a1) to (b1);    
  \end{tikzpicture}
  }

\newcommand{\PartBracketWW}{%
  \begin{tikzpicture}[scale=0.25,baseline=0.04cm]
    \def\xdist{0.666}
    \def\ydist{1.25}
    \def\ypar{0.5}
    \node [scale=0.33, circle, draw=black, fill=white] (a1) at ({0*\xdist},{0*\ydist}) {};
    \node [scale=0.33, circle, draw=black, fill=white] (a2) at ({1*\xdist},{0*\ydist}) {};
    \node [scale=0.33, circle, draw=black, fill=white] (b1) at ({0*\xdist},{1*\ydist}) {};
    \node [scale=0.33, circle, draw=black, fill=white] (b2) at ({1*\xdist},{1*\ydist}) {};
    \draw (a1) -- ++ (0,{\ypar}) -| (a2);
    \draw (b1) -- ++ (0,{-\ypar}) -| (b2);
  \end{tikzpicture}
}

\newcommand{\PartBracketBB}{%
  \begin{tikzpicture}[scale=0.25,baseline=0.04cm]
    \def\xdist{0.666}
    \def\ydist{1.25}
    \def\ypar{0.5}
    \node [scale=0.33, circle, draw=black, fill=black] (a1) at ({0*\xdist},{0*\ydist}) {};
    \node [scale=0.33, circle, draw=black, fill=black] (a2) at ({1*\xdist},{0*\ydist}) {};
    \node [scale=0.33, circle, draw=black, fill=black] (b1) at ({0*\xdist},{1*\ydist}) {};
    \node [scale=0.33, circle, draw=black, fill=black] (b2) at ({1*\xdist},{1*\ydist}) {};
    \draw (a1) -- ++ (0,{\ypar}) -| (a2);
    \draw (b1) -- ++ (0,{-\ypar}) -| (b2);
  \end{tikzpicture}
}

\newcommand{\PartBracketWB}{%
  \begin{tikzpicture}[scale=0.25,baseline=0.04cm]
    \def\xdist{0.666}
    \def\ydist{1.25}
    \def\ypar{0.5}
    \node [scale=0.33, circle, draw=black, fill=white] (a1) at ({0*\xdist},{0*\ydist}) {};
    \node [scale=0.33, circle, draw=black, fill=black] (a2) at ({1*\xdist},{0*\ydist}) {};
    \node [scale=0.33, circle, draw=black, fill=white] (b1) at ({0*\xdist},{1*\ydist}) {};
    \node [scale=0.33, circle, draw=black, fill=black] (b2) at ({1*\xdist},{1*\ydist}) {};
    \draw (a1) -- ++ (0,{\ypar}) -| (a2);
    \draw (b1) -- ++ (0,{-\ypar}) -| (b2);
  \end{tikzpicture}
}

\newcommand{\PartBracketBW}{%
  \begin{tikzpicture}[scale=0.25,baseline=0.04cm]
    \def\xdist{0.666}
    \def\ydist{1.25}
    \def\ypar{0.5}
    \node [scale=0.33, circle, draw=black, fill=black] (a1) at ({0*\xdist},{0*\ydist}) {};
    \node [scale=0.33, circle, draw=black, fill=white] (a2) at ({1*\xdist},{0*\ydist}) {};
    \node [scale=0.33, circle, draw=black, fill=black] (b1) at ({0*\xdist},{1*\ydist}) {};
    \node [scale=0.33, circle, draw=black, fill=white] (b2) at ({1*\xdist},{1*\ydist}) {};
    \draw (a1) -- ++ (0,{\ypar}) -| (a2);
    \draw (b1) -- ++ (0,{-\ypar}) -| (b2);
  \end{tikzpicture}
}

\newcommand{\PartFourWWBB}{%
  \begin{tikzpicture}[scale=0.25,baseline=-0.015cm]
    \def\xdist{0.666}
    \def\ydist{1}
    \def\ypar{1}
    \node [scale=0.33, circle, draw=black, fill=white] (a1) at ({0*\xdist},{0*\ydist}) {};
    \node [scale=0.33, circle, draw=black, fill=white] (a2) at ({1*\xdist},{0*\ydist}) {};
    \node [scale=0.33, circle, draw=black, fill=black] (a3) at ({2*\xdist},{0*\ydist}) {};
    \node [scale=0.33, circle, draw=black, fill=black] (a4) at ({3*\xdist},{0*\ydist}) {};
    \draw (a1) -- ++ (0,{\ypar}) -| (a4);
    \draw (a2) -- ++ (0,{\ypar});
    \draw (a3) -- ++ (0,{\ypar});    
  \end{tikzpicture}
}

\newcommand{\PartThreeWBW}{%
  \begin{tikzpicture}[scale=0.25,baseline=-0.015cm]
    \def\xdist{0.666}
    \def\ydist{1}
    \def\ypar{1}
    \node [scale=0.33, circle, draw=black, fill=white] (a1) at ({0*\xdist},{0*\ydist}) {};
    \node [scale=0.33, circle, draw=black, fill=black] (a2) at ({1*\xdist},{0*\ydist}) {};
    \node [scale=0.33, circle, draw=black, fill=white] (a3) at ({2*\xdist},{0*\ydist}) {};
    \draw (a1) -- ++ (0,{\ypar}) -| (a3);
    \draw (a2) -- ++ (0,{\ypar});
  \end{tikzpicture}
}

\newcommand{\PartFourBBWW}{%
  \begin{tikzpicture}[scale=0.25,baseline=-0.015cm]
    \def\xdist{0.666}
    \def\ydist{1}
    \def\ypar{1}
    \node [scale=0.33, circle, draw=black, fill=black] (a1) at ({0*\xdist},{0*\ydist}) {};
    \node [scale=0.33, circle, draw=black, fill=black] (a2) at ({1*\xdist},{0*\ydist}) {};
    \node [scale=0.33, circle, draw=black, fill=white] (a3) at ({2*\xdist},{0*\ydist}) {};
    \node [scale=0.33, circle, draw=black, fill=white] (a4) at ({3*\xdist},{0*\ydist}) {};
    \draw (a1) -- ++ (0,{\ypar}) -| (a4);
    \draw (a2) -- ++ (0,{\ypar});
    \draw (a3) -- ++ (0,{\ypar});    
  \end{tikzpicture}
}

\newcommand{\PartFourWBWB}{%
  \begin{tikzpicture}[scale=0.25,baseline=-0.015cm]
    \def\xdist{0.666}
    \def\ydist{1}
    \def\ypar{1}
    \node [scale=0.33, circle, draw=black, fill=white] (a1) at ({0*\xdist},{0*\ydist}) {};
    \node [scale=0.33, circle, draw=black, fill=black] (a2) at ({1*\xdist},{0*\ydist}) {};
    \node [scale=0.33, circle, draw=black, fill=white] (a3) at ({2*\xdist},{0*\ydist}) {};
    \node [scale=0.33, circle, draw=black, fill=black] (a4) at ({3*\xdist},{0*\ydist}) {};
    \draw (a1) -- ++ (0,{\ypar}) -| (a4);
    \draw (a2) -- ++ (0,{\ypar});
    \draw (a3) -- ++ (0,{\ypar});    
  \end{tikzpicture}
}

\newcommand{\PartFourBWBW}{%
  \begin{tikzpicture}[scale=0.25,baseline=-0.015cm]
    \def\xdist{0.666}
    \def\ydist{1}
    \def\ypar{1}
    \node [scale=0.33, circle, draw=black, fill=black] (a1) at ({0*\xdist},{0*\ydist}) {};
    \node [scale=0.33, circle, draw=black, fill=white] (a2) at ({1*\xdist},{0*\ydist}) {};
    \node [scale=0.33, circle, draw=black, fill=black] (a3) at ({2*\xdist},{0*\ydist}) {};
    \node [scale=0.33, circle, draw=black, fill=white] (a4) at ({3*\xdist},{0*\ydist}) {};
    \draw (a1) -- ++ (0,{\ypar}) -| (a4);
    \draw (a2) -- ++ (0,{\ypar});
    \draw (a3) -- ++ (0,{\ypar});    
  \end{tikzpicture}
}

\newcommand{\PartSinglesWB}{%
  \begin{tikzpicture}[scale=0.25,baseline=-0.015cm]
    \def\xdist{1}
    \def\ydist{1}
    \def\ypar{1}
    \node [scale=0.33, circle, draw=black, fill=white] (a1) at ({0*\xdist},{0*\ydist}) {};
    \node [scale=0.33, circle, draw=black, fill=black] (a2) at ({1*\xdist},{0*\ydist}) {};
    \draw[->] (a1) -- ++ (0,{\ypar});
    \draw[->] (a2) -- ++ (0,{\ypar});
  \end{tikzpicture}
}
\newcommand{\PartSinglesWBTensor}{%
  \begin{tikzpicture}[scale=0.25,baseline=-0.015cm]
    \def\xdist{1}
    \def\ydist{1}
    \def\ypar{1}
    \node [scale=0.33, circle, draw=black, fill=white] (a1) at ({0*\xdist},{0*\ydist}) {};
    \node [scale=0.33, circle, draw=black, fill=black] (a2) at ({1.5*\xdist},{0*\ydist}) {};
    \draw[->] (a1) -- ++ (0,{\ypar});
    \draw[->] (a2) -- ++ (0,{\ypar});
    \node [scale=0.66] at ($(a1)+({0.75*\xdist},{0.325*\ydist})$) {$\otimes$};
  \end{tikzpicture}
}

\newcommand{\PartSinglesBW}{%
  \begin{tikzpicture}[scale=0.25,baseline=-0.015cm]
    \def\xdist{1}
    \def\ydist{1}
    \def\ypar{1}
    \node [scale=0.33, circle, draw=black, fill=black] (a1) at ({0*\xdist},{0*\ydist}) {};
    \node [scale=0.33, circle, draw=black, fill=white] (a2) at ({1*\xdist},{0*\ydist}) {};
    \draw[->] (a1) -- ++ (0,{\ypar});
    \draw[->] (a2) -- ++ (0,{\ypar});
  \end{tikzpicture}
}
\newcommand{\PartSinglesBWTensor}{%
  \begin{tikzpicture}[scale=0.25,baseline=-0.015cm]
    \def\xdist{1}
    \def\ydist{1}
    \def\ypar{1}
    \node [scale=0.33, circle, draw=black, fill=black] (a1) at ({0*\xdist},{0*\ydist}) {};
    \node [scale=0.33, circle, draw=black, fill=white] (a2) at ({1.5*\xdist},{0*\ydist}) {};
    \draw[->] (a1) -- ++ (0,{\ypar});
    \draw[->] (a2) -- ++ (0,{\ypar});
    \node [scale=0.66] at ($(a1)+({0.75*\xdist},{0.325*\ydist})$) {$\otimes$};
  \end{tikzpicture}
}

\newcommand{\PartSinglesProjW}{%
  \begin{tikzpicture}[scale=0.25,baseline=0.04cm]
    \def\xdist{0.666}
    \def\ydist{1.25}
    \def\ypar{0.65}
    \node [scale=0.33, circle, draw=black, fill=white] (a1) at ({0*\xdist},{0*\ydist}) {};
    \node [scale=0.33, circle, draw=black, fill=white] (b1) at ({0*\xdist},{1*\ydist}) {};
    \draw[->] (a1) --++({0.5*\ypar},{\ypar});
    \draw[->] (b1) --++({-0.5*\ypar},{-\ypar});
  \end{tikzpicture}
}

\newcommand{\PartSinglesProjB}{%
  \begin{tikzpicture}[scale=0.25,baseline=0.04cm]
    \def\xdist{0.666}
    \def\ydist{1.25}
    \def\ypar{0.65}
    \node [scale=0.33, circle, draw=black, fill=black] (a1) at ({0*\xdist},{0*\ydist}) {};
    \node [scale=0.33, circle, draw=black, fill=black] (b1) at ({0*\xdist},{1*\ydist}) {};
    \draw[->] (a1) --++({0.5*\ypar},{\ypar});
    \draw[->] (b1) --++({-0.5*\ypar},{-\ypar});
  \end{tikzpicture}
}

\newcommand{\PartFourProjWB}{%
  \begin{tikzpicture}[scale=0.25,baseline=0.04cm]
    \def\xdist{1}
    \def\ydist{1.25}
    \def\ypar{0.5}
    \node [scale=0.33, circle, draw=black, fill=white] (a1) at ({0*\xdist},{0*\ydist}) {};
    \node [scale=0.33, circle, draw=black, fill=black] (a2) at ({1*\xdist},{0*\ydist}) {};
    \node [scale=0.33, circle, draw=black, fill=white] (b1) at ({0*\xdist},{1*\ydist}) {};
    \node [scale=0.33, circle, draw=black, fill=black] (b2) at ({1*\xdist},{1*\ydist}) {};
    \draw (a1) -- (b1);
    \draw (a2) -- (b2);
    \path (a1) -- ++ (0,{0.5*\ydist}) [draw=black]-- ++ ({1*\xdist},0);
  \end{tikzpicture}
}

\newcommand{\PartFourProjBW}{%
  \begin{tikzpicture}[scale=0.25,baseline=0.04cm]
    \def\xdist{1}
    \def\ydist{1.25}
    \def\ypar{0.5}
    \node [scale=0.33, circle, draw=black, fill=black] (a1) at ({0*\xdist},{0*\ydist}) {};
    \node [scale=0.33, circle, draw=black, fill=white] (a2) at ({1*\xdist},{0*\ydist}) {};
    \node [scale=0.33, circle, draw=black, fill=black] (b1) at ({0*\xdist},{1*\ydist}) {};
    \node [scale=0.33, circle, draw=black, fill=white] (b2) at ({1*\xdist},{1*\ydist}) {};
    \draw (a1) -- (b1);
    \draw (a2) -- (b2);
    \path (a1) -- ++ (0,{0.5*\ydist}) [draw=black]-- ++ ({1*\xdist},0);
  \end{tikzpicture}
}

\newcommand{\PartFourProjWW}{%
  \begin{tikzpicture}[scale=0.25,baseline=0.04cm]
    \def\xdist{1}
    \def\ydist{1.25}
    \def\ypar{0.5}
    \node [scale=0.33, circle, draw=black, fill=white] (a1) at ({0*\xdist},{0*\ydist}) {};
    \node [scale=0.33, circle, draw=black, fill=white] (a2) at ({1*\xdist},{0*\ydist}) {};
    \node [scale=0.33, circle, draw=black, fill=white] (b1) at ({0*\xdist},{1*\ydist}) {};
    \node [scale=0.33, circle, draw=black, fill=white] (b2) at ({1*\xdist},{1*\ydist}) {};
    \draw (a1) -- (b1);
    \draw (a2) -- (b2);
    \path (a1) -- ++ (0,{0.5*\ydist}) [draw=black]-- ++ ({1*\xdist},0);
  \end{tikzpicture}
}

\newcommand{\PartFourProjBB}{%
  \begin{tikzpicture}[scale=0.25,baseline=0.04cm]
    \def\xdist{1}
    \def\ydist{1.25}
    \def\ypar{0.5}
    \node [scale=0.33, circle, draw=black, fill=black] (a1) at ({0*\xdist},{0*\ydist}) {};
    \node [scale=0.33, circle, draw=black, fill=black] (a2) at ({1*\xdist},{0*\ydist}) {};
    \node [scale=0.33, circle, draw=black, fill=black] (b1) at ({0*\xdist},{1*\ydist}) {};
    \node [scale=0.33, circle, draw=black, fill=black] (b2) at ({1*\xdist},{1*\ydist}) {};
    \draw (a1) -- (b1);
    \draw (a2) -- (b2);
    \path (a1) -- ++ (0,{0.5*\ydist}) [draw=black]-- ++ ({1*\xdist},0);
  \end{tikzpicture}
}

\newcommand{\PartCoBracketWBWB}{%
  \begin{tikzpicture}[scale=0.25,baseline=0.04cm]
    \def\xdist{0.666} 
    \def\ydist{1.25}
    \def\ypar{0.5}
    \node [scale=0.33, circle, draw=black, fill=white] (a1) at ({0*\xdist},{0*\ydist}) {};
    \node [scale=0.33, circle, draw=black, fill=black] (a2) at ({1*\xdist},{0*\ydist}) {};
    \node [scale=0.33, circle, draw=black, fill=white] (a3) at ({2*\xdist},{0*\ydist}) {};
    \node [scale=0.33, circle, draw=black, fill=black] (a4) at ({3*\xdist},{0*\ydist}) {};
    \node [scale=0.33, circle, draw=black, fill=white] (b1) at ({0*\xdist},{1*\ydist}) {};
    \node [scale=0.33, circle, draw=black, fill=black] (b2) at ({1*\xdist},{1*\ydist}) {};
    \node [scale=0.33, circle, draw=black, fill=white] (b3) at ({2*\xdist},{1*\ydist}) {};
    \node [scale=0.33, circle, draw=black, fill=black] (b4) at ({3*\xdist},{1*\ydist}) {};
    \node [scale=0.2, draw=black, fill=black] at ({0*\xdist},{0.5*\ydist}) {};
    \node [scale=0.2, draw=black, fill=black] at ({3*\xdist},{0.5*\ydist}) {};        
    \draw (a1) to (b1);
    \draw (a2) to (b2);
    \draw (a3) to (b3);
    \draw (a4) to (b4);    
    \path (a1) -- ++ (0,{0.5*\ydist}) [draw=black]-- ++ ({3*\xdist},0);
  \end{tikzpicture}
}

\newcommand{\PartCoBracketBWBW}{%
  \begin{tikzpicture}[scale=0.25,baseline=0.04cm]
    \def\xdist{0.666} 
    \def\ydist{1.25}
    \def\ypar{0.5}
    \node [scale=0.33, circle, draw=black, fill=black] (a1) at ({0*\xdist},{0*\ydist}) {};
    \node [scale=0.33, circle, draw=black, fill=white] (a2) at ({1*\xdist},{0*\ydist}) {};
    \node [scale=0.33, circle, draw=black, fill=black] (a3) at ({2*\xdist},{0*\ydist}) {};
    \node [scale=0.33, circle, draw=black, fill=white] (a4) at ({3*\xdist},{0*\ydist}) {};
    \node [scale=0.33, circle, draw=black, fill=black] (b1) at ({0*\xdist},{1*\ydist}) {};
    \node [scale=0.33, circle, draw=black, fill=white] (b2) at ({1*\xdist},{1*\ydist}) {};
    \node [scale=0.33, circle, draw=black, fill=black] (b3) at ({2*\xdist},{1*\ydist}) {};
    \node [scale=0.33, circle, draw=black, fill=white] (b4) at ({3*\xdist},{1*\ydist}) {};
    \node [scale=0.2, draw=black, fill=black] at ({0*\xdist},{0.5*\ydist}) {};
    \node [scale=0.2, draw=black, fill=black] at ({3*\xdist},{0.5*\ydist}) {};        
    \draw (a1) to (b1);
    \draw (a2) to (b2);
    \draw (a3) to (b3);
    \draw (a4) to (b4);    
    \path (a1) -- ++ (0,{0.5*\ydist}) [draw=black]-- ++ ({3*\xdist},0);
  \end{tikzpicture}
}

\newcommand{\PartCoBracketWBW}{%
  \begin{tikzpicture}[scale=0.25,baseline=0.04cm]
    \def\xdist{0.666} 
    \def\ydist{1.25}
    \def\ypar{0.5}
    \node [scale=0.33, circle, draw=black, fill=white] (a1) at ({0*\xdist},{0*\ydist}) {};
    \node [scale=0.33, circle, draw=black, fill=black] (a2) at ({1*\xdist},{0*\ydist}) {};
    \node [scale=0.33, circle, draw=black, fill=white] (a3) at ({2*\xdist},{0*\ydist}) {};
    \node [scale=0.33, circle, draw=black, fill=white] (b1) at ({0*\xdist},{1*\ydist}) {};
    \node [scale=0.33, circle, draw=black, fill=black] (b2) at ({1*\xdist},{1*\ydist}) {};
    \node [scale=0.33, circle, draw=black, fill=white] (b3) at ({2*\xdist},{1*\ydist}) {};
    \node [scale=0.2, draw=black, fill=black] at ({0*\xdist},{0.5*\ydist}) {};
    \node [scale=0.2, draw=black, fill=black] at ({2*\xdist},{0.5*\ydist}) {};        
    \draw (a1) to (b1);
    \draw (a2) to (b2);
    \draw (a3) to (b3);
    \path (a1) -- ++ (0,{0.5*\ydist}) [draw=black]-- ++ ({2*\xdist},0);
  \end{tikzpicture}
}

\newcommand{\PartCoBracketBWB}{%
  \begin{tikzpicture}[scale=0.25,baseline=0.04cm]
    \def\xdist{0.666} 
    \def\ydist{1.25}
    \def\ypar{0.5}
    \node [scale=0.33, circle, draw=black, fill=black] (a1) at ({0*\xdist},{0*\ydist}) {};
    \node [scale=0.33, circle, draw=black, fill=white] (a2) at ({1*\xdist},{0*\ydist}) {};
    \node [scale=0.33, circle, draw=black, fill=black] (a3) at ({2*\xdist},{0*\ydist}) {};
    \node [scale=0.33, circle, draw=black, fill=black] (b1) at ({0*\xdist},{1*\ydist}) {};
    \node [scale=0.33, circle, draw=black, fill=white] (b2) at ({1*\xdist},{1*\ydist}) {};
    \node [scale=0.33, circle, draw=black, fill=black] (b3) at ({2*\xdist},{1*\ydist}) {};
    \node [scale=0.2, draw=black, fill=black] at ({0*\xdist},{0.5*\ydist}) {};
    \node [scale=0.2, draw=black, fill=black] at ({2*\xdist},{0.5*\ydist}) {};        
    \draw (a1) to (b1);
    \draw (a2) to (b2);
    \draw (a3) to (b3);
    \path (a1) -- ++ (0,{0.5*\ydist}) [draw=black]-- ++ ({2*\xdist},0);
  \end{tikzpicture}
}

\newcommand{\PartSingleW}{
    \begin{tikzpicture}[scale=0.25,baseline=-0.015cm]
    \def\xdist{1}
    \def\ydist{1}
    \def\ypar{1}
    \node [scale=0.33, circle, draw=black, fill=white] (a1) at ({0*\xdist},{0*\ydist}) {};
    \draw[->] (a1) -- ++ (0,{\ypar});
  \end{tikzpicture}
}

\newcommand{\PartSingleB}{
    \begin{tikzpicture}[scale=0.25,baseline=-0.015cm]
    \def\xdist{1}
    \def\ydist{1}
    \def\ypar{1}
    \node [scale=0.33, circle, draw=black, fill=black] (a1) at ({0*\xdist},{0*\ydist}) {};
    \draw[->] (a1) -- ++ (0,{\ypar});
  \end{tikzpicture}
}

  \newcommand{\PartCoSingleWBWB}{%
    \begin{tikzpicture}[scale=0.25,baseline=-0.015cm]
    \def\xdist{0.666}
    \def\ydist{1.25}
    \def\ypar{1}
    \node [scale=0.33, circle, draw=black, fill=white] (a1) at ({0*\xdist},{0*\ydist}) {};
    \node [scale=0.33, circle, draw=black, fill=black] (a2) at ({1*\xdist},{0*\ydist}) {};
    \node [scale=0.33, circle, draw=black, fill=white] (a3) at ({2*\xdist},{0*\ydist}) {};
    \node [scale=0.33, circle, draw=black, fill=black] (a4) at ({3*\xdist},{0*\ydist}) {};
    \draw (a1) -- ++ (0,{1*\ydist}) -| (a3);
    \draw[->] (a2) -- ++ (0,{1*\ypar});
    \draw[->] (a4) -- ++ (0,{1*\ypar});    
    \end{tikzpicture}
  }
\newcommand{\PartCoSingleBWBW}{%
    \begin{tikzpicture}[scale=0.25,baseline=-0.015cm]
    \def\xdist{0.666}
    \def\ydist{1.25}
    \def\ypar{1}
    \node [scale=0.33, circle, draw=black, fill=black] (a1) at ({0*\xdist},{0*\ydist}) {};
    \node [scale=0.33, circle, draw=black, fill=white] (a2) at ({1*\xdist},{0*\ydist}) {};
    \node [scale=0.33, circle, draw=black, fill=black] (a3) at ({2*\xdist},{0*\ydist}) {};
    \node [scale=0.33, circle, draw=black, fill=white] (a4) at ({3*\xdist},{0*\ydist}) {};
    \draw (a1) -- ++ (0,{1*\ydist}) -| (a3);
    \draw[->] (a2) -- ++ (0,{1*\ypar});
    \draw[->] (a4) -- ++ (0,{1*\ypar});    
    \end{tikzpicture}
  }

\newcommand{\PartSpecThreeCoSingleWBWB}{%
    \begin{tikzpicture}[scale=0.25,baseline=-0.015cm]
    \def\xdist{0.666}
    \def\ydist{1}
    \def\ypar{1}
    \node [scale=0.33, circle, draw=black, fill=white] (a1) at ({0*\xdist},{0*\ydist}) {};
    \node [scale=0.33, circle, draw=black, fill=black] (a2) at ({1*\xdist},{0*\ydist}) {};
    \node [scale=0.33, circle, draw=black, fill=white] (a3) at ({2*\xdist},{0*\ydist}) {};
    \node [scale=0.33, circle, draw=black, fill=black] (a4) at ({3*\xdist},{0*\ydist}) {};
    \draw (a1) -- ++ (0,{1*\ydist}) -| (a3);
    \draw ($(a3)+(0,{1*\ydist})$) -| (a4);    
    \draw[->] (a2) -- ++ (0,{0.75*\ypar});
    \end{tikzpicture}
  }
  
  \newcommand{\PartCoSingleRevWWBB}{%
    \begin{tikzpicture}[scale=0.25,baseline=-0.015cm]
    \def\xdist{0.666}
    \def\ydist{1.25}
    \def\ypar{1}
    \node [scale=0.33, circle, draw=black, fill=white] (a1) at ({0*\xdist},{0*\ydist}) {};
    \node [scale=0.33, circle, draw=black, fill=white] (a2) at ({1*\xdist},{0*\ydist}) {};
    \node [scale=0.33, circle, draw=black, fill=black] (a3) at ({2*\xdist},{0*\ydist}) {};
    \node [scale=0.33, circle, draw=black, fill=black] (a4) at ({3*\xdist},{0*\ydist}) {};
    \draw (a2) -- ++ (0,{1*\ydist}) -| (a4);
    \draw[->] (a1) -- ++ (0,{1*\ypar});
    \draw[->] (a3) -- ++ (0,{1*\ypar});    
    \end{tikzpicture}
  }
    \newcommand{\PartCoSingleRevWBWB}{%
    \begin{tikzpicture}[scale=0.25,baseline=-0.015cm]
    \def\xdist{0.666}
    \def\ydist{1.25}
    \def\ypar{1}
    \node [scale=0.33, circle, draw=black, fill=white] (a1) at ({0*\xdist},{0*\ydist}) {};
    \node [scale=0.33, circle, draw=black, fill=black] (a2) at ({1*\xdist},{0*\ydist}) {};
    \node [scale=0.33, circle, draw=black, fill=white] (a3) at ({2*\xdist},{0*\ydist}) {};
    \node [scale=0.33, circle, draw=black, fill=black] (a4) at ({3*\xdist},{0*\ydist}) {};
    \draw (a2) -- ++ (0,{1*\ydist}) -| (a4);
    \draw[->] (a1) -- ++ (0,{1*\ypar});
    \draw[->] (a3) -- ++ (0,{1*\ypar});    
    \end{tikzpicture}
  }
    \newcommand{\PartCoSingleProjBWB}{%
      \begin{tikzpicture}[scale=0.25,baseline=0.04cm]
    \def\xdist{0.666} 
    \def\ydist{1.25}
    \def\ypar{1}
    \node [scale=0.33, circle, draw=black, fill=black] (a1) at ({0*\xdist},{0*\ydist}) {};
    \node [scale=0.33, circle, draw=black, fill=white] (a2) at ({1*\xdist},{0*\ydist}) {};
    \node [scale=0.33, circle, draw=black, fill=black] (a3) at ({2*\xdist},{0*\ydist}) {};
    \node [scale=0.33, circle, draw=black, fill=black] (b1) at ({0*\xdist},{1*\ydist}) {};
    \node [scale=0.33, circle, draw=black, fill=white] (b2) at ({1*\xdist},{1*\ydist}) {};
    \node [scale=0.33, circle, draw=black, fill=black] (b3) at ({2*\xdist},{1*\ydist}) {};
    \node [scale=0.2, draw=black, fill=black] at ({0*\xdist},{0.5*\ydist}) {};
    \node [scale=0.2, draw=black, fill=black] at ({2*\xdist},{0.5*\ydist}) {};        
    \draw (a1) to (b1);
    \draw (a3) to (b3);
    \path (a1) -- ++ (0,{0.5*\ydist}) [draw=black]-- ++ ({2*\xdist},0);
    \end{tikzpicture}
  }

      \newcommand{\PartCoSingleProjWBW}{%
      \begin{tikzpicture}[scale=0.25,baseline=0.04cm]
    \def\xdist{0.666} 
    \def\ydist{1.25}
    \def\ypar{1}
    \node [scale=0.33, circle, draw=black, fill=white] (a1) at ({0*\xdist},{0*\ydist}) {};
    \node [scale=0.33, circle, draw=black, fill=black] (a2) at ({1*\xdist},{0*\ydist}) {};
    \node [scale=0.33, circle, draw=black, fill=white] (a3) at ({2*\xdist},{0*\ydist}) {};
    \node [scale=0.33, circle, draw=black, fill=white] (b1) at ({0*\xdist},{1*\ydist}) {};
    \node [scale=0.33, circle, draw=black, fill=black] (b2) at ({1*\xdist},{1*\ydist}) {};
    \node [scale=0.33, circle, draw=black, fill=white] (b3) at ({2*\xdist},{1*\ydist}) {};
    \node [scale=0.2, draw=black, fill=black] at ({0*\xdist},{0.5*\ydist}) {};
    \node [scale=0.2, draw=black, fill=black] at ({2*\xdist},{0.5*\ydist}) {};        
    \draw (a1) to (b1);
    \draw (a3) to (b3);
    \path (a1) -- ++ (0,{0.5*\ydist}) [draw=black]-- ++ ({2*\xdist},0);
    \end{tikzpicture}
  }

   \newcommand{\PartBrackSingleWBW}{%
      \begin{tikzpicture}[scale=0.25,baseline=0.04cm]
    \def\xdist{0.666} 
    \def\ydist{1.25}
    \def\ypar{0.5}
    \node [scale=0.33, circle, draw=black, fill=white] (a1) at ({0*\xdist},{0*\ydist}) {};
    \node [scale=0.33, circle, draw=black, fill=black] (a2) at ({1*\xdist},{0*\ydist}) {};
    \node [scale=0.33, circle, draw=black, fill=white] (a3) at ({2*\xdist},{0*\ydist}) {};
    \node [scale=0.33, circle, draw=black, fill=white] (b1) at ({0*\xdist},{1*\ydist}) {};
    \node [scale=0.33, circle, draw=black, fill=black] (b2) at ({1*\xdist},{1*\ydist}) {};
    \node [scale=0.33, circle, draw=black, fill=white] (b3) at ({2*\xdist},{1*\ydist}) {};
    \draw (a1) --++(0,{\ypar}) -| (a3);
    \draw (b1) --++(0,{-\ypar}) -| (b3);    
    \end{tikzpicture}
  }
     \newcommand{\PartBrackSingleBWB}{%
      \begin{tikzpicture}[scale=0.25,baseline=0.04cm]
    \def\xdist{0.666} 
    \def\ydist{1.25}
    \def\ypar{0.5}
    \node [scale=0.33, circle, draw=black, fill=black] (a1) at ({0*\xdist},{0*\ydist}) {};
    \node [scale=0.33, circle, draw=black, fill=white] (a2) at ({1*\xdist},{0*\ydist}) {};
    \node [scale=0.33, circle, draw=black, fill=black] (a3) at ({2*\xdist},{0*\ydist}) {};
    \node [scale=0.33, circle, draw=black, fill=black] (b1) at ({0*\xdist},{1*\ydist}) {};
    \node [scale=0.33, circle, draw=black, fill=white] (b2) at ({1*\xdist},{1*\ydist}) {};
    \node [scale=0.33, circle, draw=black, fill=black] (b3) at ({2*\xdist},{1*\ydist}) {};
    \draw (a1) --++(0,{\ypar}) -| (a3);
    \draw (b1) --++(0,{-\ypar}) -| (b3);    
    \end{tikzpicture}
  }

\newcommand{\PartGlobColProjW}{%
  \begin{tikzpicture}[scale=0.25,baseline=0.04cm]
    \def\xdist{1}
    \def\ydist{1.25}
    \def\ypar{0.5}
    \node [scale=0.33, circle, draw=black, fill=white] (a1) at ({0*\xdist},{0*\ydist}) {};
    \node [scale=0.33, circle, draw=black, fill=white] (a2) at ({1*\xdist},{0*\ydist}) {};
    \node [scale=0.33, circle, draw=black, fill=white] (b1) at ({0*\xdist},{1*\ydist}) {};
    \node [scale=0.33, circle, draw=black, fill=white] (b2) at ({1*\xdist},{1*\ydist}) {};
    \draw (a1) --++(0,{\ypar})-| (a2);
    \draw (b1) --++(0,{-\ypar})-| (b2);
  \end{tikzpicture}
}

\newcommand{\PartGlobColProjB}{%
  \begin{tikzpicture}[scale=0.25,baseline=0.04cm]
    \def\xdist{1}
    \def\ydist{1.25}
    \def\ypar{0.5}
    \node [scale=0.33, circle, draw=black, fill=black] (a1) at ({0*\xdist},{0*\ydist}) {};
    \node [scale=0.33, circle, draw=black, fill=black] (a2) at ({1*\xdist},{0*\ydist}) {};
    \node [scale=0.33, circle, draw=black, fill=black] (b1) at ({0*\xdist},{1*\ydist}) {};
    \node [scale=0.33, circle, draw=black, fill=black] (b2) at ({1*\xdist},{1*\ydist}) {};
    \draw (a1) --++(0,{\ypar})-| (a2);
    \draw (b1) --++(0,{-\ypar})-| (b2);
  \end{tikzpicture}
}

\newcommand{\PartFatWBWB}{%
  \begin{tikzpicture}[scale=0.25,baseline=0.04cm]
    \def\xdist{0.666} 
    \def\ydist{1.25}
    \def\ypar{0.5}
    \node [scale=0.33, circle, draw=black, fill=white] (a1) at ({0*\xdist},{0*\ydist}) {};
    \node [scale=0.33, circle, draw=black, fill=black] (a2) at ({1*\xdist},{0*\ydist}) {};
    \node [scale=0.33, circle, draw=black, fill=white] (a3) at ({2*\xdist},{0*\ydist}) {};
    \node [scale=0.33, circle, draw=black, fill=black] (a4) at ({3*\xdist},{0*\ydist}) {};
    \node [scale=0.33, circle, draw=black, fill=white] (b1) at ({0*\xdist},{1*\ydist}) {};
    \node [scale=0.33, circle, draw=black, fill=black] (b2) at ({1*\xdist},{1*\ydist}) {};
    \node [scale=0.33, circle, draw=black, fill=white] (b3) at ({2*\xdist},{1*\ydist}) {};
    \node [scale=0.33, circle, draw=black, fill=black] (b4) at ({3*\xdist},{1*\ydist}) {};
    \node [scale=0.15, draw=black, fill=black] at ({0*\xdist},{0.425*\ydist}) {};
    \node [scale=0.15, draw=black, fill=black] at ({3*\xdist},{0.425*\ydist}) {};
    \node [scale=0.15, draw=black, fill=black] at ({1*\xdist},{0.575*\ydist}) {};
    \node [scale=0.15, draw=black, fill=black] at ({2*\xdist},{0.575*\ydist}) {};            
    \draw (a1) to (b1);
    \draw (a2) to (b2);
    \draw (a3) to (b3);
    \draw (a4) to (b4);    
    \path (a1) -- ++ (0,{0.425*\ydist}) [draw=black]-- ++ ({3*\xdist},0);
    \path (a2) -- ++ (0,{0.575*\ydist}) [draw=black]-- ++ ({1*\xdist},0);    
  \end{tikzpicture}
}

\newcommand{\PartPairPoWBW}{%
  \begin{tikzpicture}[scale=0.25,baseline=0.04cm]
    \def\xdist{0.666} 
    \def\ydist{1.25}
    \def\ypar{0.5}
    \node [scale=0.33, circle, draw=black, fill=white] (a1) at ({1*\xdist},{0*\ydist}) {};
    \node [scale=0.33, circle, draw=black, fill=black] (a2) at ({2*\xdist},{0*\ydist}) {};
    \node [scale=0.33, circle, draw=black, fill=white] (a3) at ({3*\xdist},{0*\ydist}) {};
    \node [scale=0.33, circle, draw=black, fill=white] (b1) at ({0*\xdist},{1*\ydist}) {};
    \node [scale=0.33, circle, draw=black, fill=black] (b2) at ({1*\xdist},{1*\ydist}) {};
    \node [scale=0.33, circle, draw=black, fill=white] (b3) at ({2*\xdist},{1*\ydist}) {};
    \draw (b1) --++ (0,{-0.5*\ydist}) -| (a3);
    \draw (b2) --++ (0,{-0.5*\ydist}) -| (a2);
    \node [scale=0.15, draw=black, fill=black] at ({0*\xdist},{0.5*\ydist}) {};
    \node [scale=0.15, draw=black, fill=black] at ({1*\xdist},{0.5*\ydist}) {};
    \node [scale=0.15, draw=black, fill=black] at ({2*\xdist},{0.5*\ydist}) {};
    \node [scale=0.15, draw=black, fill=black] at ({3*\xdist},{0.5*\ydist}) {};    
    \draw (a1) to (b3);
  \end{tikzpicture}
}

  \newcommand{\PartSpecWWBB}{%
    \begin{tikzpicture}[scale=0.25,baseline=-0.015cm]
    \def\xdist{0.666}
    \def\ydist{1.25}
    \def\ypar{0.5}
    \node [scale=0.33, circle, draw=black, fill=white] (a1) at ({0*\xdist},{0*\ydist}) {};
    \node [scale=0.33, circle, draw=black, fill=white] (a2) at ({1*\xdist},{0*\ydist}) {};
    \node [scale=0.33, circle, draw=black, fill=black] (a3) at ({2*\xdist},{0*\ydist}) {};
    \node [scale=0.33, circle, draw=black, fill=black] (a4) at ({3*\xdist},{0*\ydist}) {};
    \draw (a1) -- ++ (0,{0.66*\ydist}) -| (a3);
    \draw (a2) -- ++ (0,{1*\ydist}) -| (a4);
    \end{tikzpicture}
  }

  \newcommand{\PartSpecWWWBBB}{%
    \begin{tikzpicture}[scale=0.25,baseline=-0.015cm]
    \def\xdist{0.666}
    \def\ydist{1.25}
    \def\ypar{0.5}
    \node [scale=0.33, circle, draw=black, fill=white] (a1) at ({0*\xdist},{0*\ydist}) {};
    \node [scale=0.33, circle, draw=black, fill=white] (a2) at ({1*\xdist},{0*\ydist}) {};
    \node [scale=0.33, circle, draw=black, fill=white] (a3) at ({2*\xdist},{0*\ydist}) {};
    \node [scale=0.33, circle, draw=black, fill=black] (a4) at ({3*\xdist},{0*\ydist}) {};
    \node [scale=0.33, circle, draw=black, fill=black] (a5) at ({4*\xdist},{0*\ydist}) {};
    \node [scale=0.33, circle, draw=black, fill=black] (a6) at ({5*\xdist},{0*\ydist}) {};
    \draw (a1) -- ++ (0,{0.66*\ydist}) -| (a4);
    \draw (a2) -- ++ (0,{0.82*\ydist}) -| (a5);
    \draw (a3) -- ++ (0,{1.00*\ydist}) -| (a6);    
    \end{tikzpicture}
  }

    \newcommand{\TWPartCoSingleWWBBr}{%
    \begin{tikzpicture}[scale=0.25,baseline=-0.015cm]
    \def\xdist{0.666}
    \def\ydist{1.25}
    \def\ypar{1}
    \node [scale=0.33, circle, draw=black, fill=white] (a1) at ({0*\xdist},{0*\ydist}) {};
    \node [scale=0.33, circle, draw=black, fill=white] (a2) at ({3.25*\xdist},{0*\ydist}) {};
    \node [scale=0.33, circle, draw=black, fill=black] (a3) at ({4.375*\xdist},{0*\ydist}) {};
    \node [scale=0.33, circle, draw=black, fill=black] (a4) at ({7.625*\xdist},{0*\ydist}) {};
    \draw (a2) -- ++ (0,{1*\ydist}) -| (a4);
    \draw[->] (a1) -- ++ (0,{1*\ypar});
    \draw[->] (a3) -- ++ (0,{1*\ypar});
    \node[scale=0.66] at ($(a1)+({1.5*\xdist},{0.66*\ypar})$) {$\otimes r$};
    \node[scale=0.66] at ($(a3)+({1.5*\xdist},{0.66*\ypar})$) {$\otimes r$};    
    \end{tikzpicture}
  }
  \newcommand{\TWPartCoSingleWWBBd}{%
    \begin{tikzpicture}[scale=0.25,baseline=-0.015cm]
    \def\xdist{0.666}
    \def\ydist{1.25}
    \def\ypar{1}
    \node [scale=0.33, circle, draw=black, fill=white] (a1) at ({0*\xdist},{0*\ydist}) {};
    \node [scale=0.33, circle, draw=black, fill=white] (a2) at ({3.25*\xdist},{0*\ydist}) {};
    \node [scale=0.33, circle, draw=black, fill=black] (a3) at ({4.375*\xdist},{0*\ydist}) {};
    \node [scale=0.33, circle, draw=black, fill=black] (a4) at ({7.625*\xdist},{0*\ydist}) {};
    \draw (a2) -- ++ (0,{1*\ydist}) -| (a4);
    \draw[->] (a1) -- ++ (0,{1*\ypar});
    \draw[->] (a3) -- ++ (0,{1*\ypar});
    \node[scale=0.66] at ($(a1)+({1.5*\xdist},{0.66*\ypar})$) {$\otimes d$};
    \node[scale=0.66] at ($(a3)+({1.5*\xdist},{0.66*\ypar})$) {$\otimes d$};    
    \end{tikzpicture}
  }
  \newcommand{\TWPartCoSingleWWBBTwor}{%
    \begin{tikzpicture}[scale=0.25,baseline=-0.015cm]
    \def\xdist{0.666}
    \def\ydist{1.25}
    \def\ypar{1}
    \node [scale=0.33, circle, draw=black, fill=white] (a1) at ({0*\xdist},{0*\ydist}) {};
    \node [scale=0.33, circle, draw=black, fill=white] (a2) at ({3.5*\xdist},{0*\ydist}) {};
    \node [scale=0.33, circle, draw=black, fill=black] (a3) at ({4.625*\xdist},{0*\ydist}) {};
    \node [scale=0.33, circle, draw=black, fill=black] (a4) at ({8.125*\xdist},{0*\ydist}) {};
    \draw (a2) -- ++ (0,{1*\ydist}) -| (a4);
    \draw[->] (a1) -- ++ (0,{1*\ypar});
    \draw[->] (a3) -- ++ (0,{1*\ypar});
    \node[right,scale=0.66] at ($(a1)+({0*\xdist},{0.66*\ypar})$) {$\otimes 2r$};
    \node[right,scale=0.66] at ($(a3)+({0*\xdist},{0.66*\ypar})$) {$\otimes 2r$};    
    \end{tikzpicture}
  }  
  \newcommand{\TWPartCoSingleWBBBr}{%
    \begin{tikzpicture}[scale=0.25,baseline=-0.015cm]
    \def\xdist{0.666}
    \def\ydist{1.25}
    \def\ypar{1}
    \node [scale=0.33, circle, draw=black, fill=white] (a1) at ({0*\xdist},{0*\ydist}) {};
    \node [scale=0.33, circle, draw=black, fill=black] (a2) at ({4.625*\xdist},{0*\ydist}) {};
    \node [scale=0.33, circle, draw=black, fill=black] (a3) at ({5.625*\xdist},{0*\ydist}) {};
    \node [scale=0.33, circle, draw=black, fill=black] (a4) at ({10.25*\xdist},{0*\ydist}) {};
    \draw (a2) -- ++ (0,{1*\ydist}) -| (a4);
    \draw[->] (a1) -- ++ (0,{1*\ypar});
    \draw[->] (a3) -- ++ (0,{1*\ypar});
    \node[right,scale=0.66] at ($(a1)+({0*\xdist},{0.66*\ypar})$) {$\otimes r\!+\!1$};
    \node[right,scale=0.66] at ($(a3)+({0*\xdist},{0.66*\ypar})$) {$\otimes r\!-\!1$};    
    \end{tikzpicture}
  }

\newcommand{\UCPartCross}{%
  \begin{tikzpicture}[scale=0.25,baseline=0.015cm]
    \def\xdist{1}
    \def\ydist{1}
    \node [scale=0.25, regular polygon, regular polygon sides=4, draw=black, fill=gray] (a1) at ({0*\xdist},{0*\ydist}) {};
    \node [scale=0.25, regular polygon, regular polygon sides=4, draw=black, fill=gray] (a2) at ({1*\xdist},{0*\ydist}) {};
    \node [scale=0.25, regular polygon, regular polygon sides=4, draw=black, fill=gray] (b1) at ({0*\xdist},{1*\ydist}) {};
    \node [scale=0.25, regular polygon, regular polygon sides=4, draw=black, fill=gray] (b2) at ({1*\xdist},{1*\ydist}) {};
    \draw (a1) to (b2);    
    \draw (a2) to (b1);
  \end{tikzpicture}
  }

\newcommand{\UCPartHalfLib}{%
  \begin{tikzpicture}[scale=0.25,baseline=0.04cm]
    \def\xdist{0.75}
    \def\ydist{1.25}
    \def\ypar{0.5}
    \node [scale=0.25, regular polygon, regular polygon sides=4, draw=black, fill=gray] (a1) at ({0*\xdist},{0*\ydist}) {};
    \node [scale=0.25, regular polygon, regular polygon sides=4, draw=black, fill=gray] (a2) at ({1*\xdist},{0*\ydist}) {};
    \node [scale=0.25, regular polygon, regular polygon sides=4, draw=black, fill=gray] (a3) at ({2*\xdist},{0*\ydist}) {};
    \node [scale=0.25, regular polygon, regular polygon sides=4, draw=black, fill=gray] (b1) at ({0*\xdist},{1*\ydist}) {};
    \node [scale=0.25, regular polygon, regular polygon sides=4, draw=black, fill=gray] (b2) at ({1*\xdist},{1*\ydist}) {};
    \node [scale=0.25, regular polygon, regular polygon sides=4, draw=black, fill=gray] (b3) at ({2*\xdist},{1*\ydist}) {};
    \draw (a1) to (b3);    
    \draw (a2) to (b2);
    \draw (a3) to (b1);    
  \end{tikzpicture}
  }

\newcommand{\UCPartFour}{%
  \begin{tikzpicture}[baseline=0.0cm]
    \def\xdist{0.2cm}
    \def\ydist{0.32cm}
    \def\ypar{0.2cm}
    \coordinate (a1) at ({0*\xdist},{0*\ydist});
    \coordinate (a2) at ({1*\xdist},{0*\ydist});
    \coordinate (a3) at ({2*\xdist},{0*\ydist});
    \coordinate (a4) at ({3*\xdist},{0*\ydist});
    \draw (a1) -- ++ (0,{\ypar}) -| (a4);
    \draw (a2) -- ++ (0,{\ypar});
    \draw (a3) -- ++ (0,{\ypar});    
  \end{tikzpicture}
}

\newcommand{\UCPartSingle}{%
  \begin{tikzpicture}[baseline=0.0cm]
    \def\xdist{0.2cm}
    \def\ydist{0.32cm}
    \def\ypar{0.25cm}
    \coordinate (a1) at ({0*\xdist},{0*\ydist});
    \draw[->] (a1) -- ++ (0,{\ypar});
  \end{tikzpicture}
}

\newcommand{\UCPartSinglesTensor}{%
  \begin{tikzpicture}[baseline=0.0cm]
    \def\xdist{0.2cm}
    \def\ydist{0.32cm}
    \def\ypar{0.25cm}
    \coordinate (a1) at ({0*\xdist},{0*\ydist});
    \coordinate (a2) at ({1.75*\xdist},{0*\ydist});    
    \draw[->] (a1) -- ++ (0,{\ypar});
    \draw[->] (a2) -- ++ (0,{\ypar});
    \node [scale=0.66] at ($(a1)+({0.5*1.75*\xdist},{0.325*\ydist})$) {$\otimes$};
  \end{tikzpicture}
}
\newcommand{\UCPartSingles}{%
  \begin{tikzpicture}[baseline=0.0cm]
    \def\xdist{0.2cm}
    \def\ydist{0.32cm}
    \def\ypar{0.2cm}
    \coordinate (a1) at ({0*\xdist},{0*\ydist});
    \coordinate (a2) at ({1*\xdist},{0*\ydist});
    \draw[->] (a1) -- ++ (0,{\ypar});
    \draw[->] (a2) -- ++ (0,{\ypar});
  \end{tikzpicture}
}

\newcommand{\UCPartCoSingle}{%
    \begin{tikzpicture}[baseline=0.00cm]
    \def\xdist{0.2cm}
    \def\ydist{0.32cm}
    \def\ypar{0.2cm}
    \coordinate (a1) at ({0*\xdist},{0*\ydist});
    \coordinate (a2) at ({1*\xdist},{0*\ydist});
    \coordinate (a3) at ({2*\xdist},{0*\ydist});
    \coordinate (a4) at ({3*\xdist},{0*\ydist});
    \draw (a1) -- ++ (0,{1*\ydist}) -| (a3);
    \draw[->] (a2) -- ++ (0,{1*\ypar});
    \draw[->] (a4) -- ++ (0,{1*\ypar});    
    \end{tikzpicture}
  }
  \newcommand{\UCPartCoSingleRev}{%
    \begin{tikzpicture}[baseline=0.10cm]
    \def\xdist{0.2cm}
    \def\ydist{0.32cm}
    \def\ypar{0.2cm}
    \coordinate (a1) at ({0*\xdist},{0*\ydist});
    \coordinate (a2) at ({1*\xdist},{0*\ydist});
    \coordinate (a3) at ({2*\xdist},{0*\ydist});
    \coordinate (a4) at ({3*\xdist},{0*\ydist});
    \draw (a2) -- ++ (0,{1*\ydist}) -| (a4);
    \draw[->] (a1) -- ++ (0,{1*\ypar});
    \draw[->] (a3) -- ++ (0,{1*\ypar});    
    \end{tikzpicture}
  }
\newcommand{\UCPartBracketThree}{%
  \begin{tikzpicture}[baseline=0.10cm]
    \def\xdist{0.2cm}
    \def\ydist{0.4cm}
    \def\ypar{0.15}
    \coordinate (a1) at ({0*\xdist},{0*\ydist});
    \coordinate (a2) at ({1*\xdist},{0*\ydist});
    \coordinate (a3) at ({2*\xdist},{0*\ydist});
    \coordinate (b1) at ({0*\xdist},{1*\ydist});
    \coordinate (b2) at ({1*\xdist},{1*\ydist});
    \coordinate (b3) at ({2*\xdist},{1*\ydist});
    \draw (a1) -- ++ (0,{\ypar}) -| (a3);
    \draw (b1) -- ++ (0,{-\ypar}) -| (b3);
    \draw (a2) to (b2);
  \end{tikzpicture}
}
\newcommand{\UCPartBracketFour}{%
  \begin{tikzpicture}[baseline=0.10cm]
    \def\xdist{0.2cm}
    \def\ydist{0.4cm}
    \def\ypar{0.15cm}
    \coordinate (a1) at ({0*\xdist},{0*\ydist});
    \coordinate (a2) at ({1*\xdist},{0*\ydist});
    \coordinate (a3) at ({2*\xdist},{0*\ydist});
    \coordinate (a4) at ({3*\xdist},{0*\ydist});    
    \coordinate (b1) at ({0*\xdist},{1*\ydist});
    \coordinate (b2) at ({1*\xdist},{1*\ydist});
    \coordinate (b3) at ({2*\xdist},{1*\ydist});
    \coordinate (b4) at ({3*\xdist},{1*\ydist});
    \draw (a1) -- ++ (0,{\ypar}) -| (a4);
    \draw (b1) -- ++ (0,{-\ypar}) -| (b4);
    \draw (a2) to (b2);
    \draw (a3) to (b3);
  \end{tikzpicture}
}

\begin{document}
\title[Non-Hyperoctahedral Categories, Part~II]{Non-Hyperoctahedral Categories\\of Two-Colored Partitions\\Part~II: All Possible Parameter Values} 
\author{Alexander Mang}
\author{Moritz Weber}
\address{Saarland University, Fachbereich Mathematik, 
	66041 Saarbrücken, Germany}
\email{s9almang@stud.uni-saarland.de, weber@math.uni-sb.de}
\thanks{The first author was supported by an IRTG scholarship of the SFB-TRR~195. The second author was supported by the SFB-TRR~195, and by the DFG project \emph{Quantenautomorphismen von Graphen}. This work was part of the first author's Master's thesis.}

\date{\today}

\begin{abstract}

This article is part of a series with the aim of classifying all non-hy\-per\-oc\-ta\-he\-dral categories of two-colored partitions. Those  constitute by some Tannaka-Krein type result the co-representation categories of a specific class of quantum groups. However, our series of articles is purely combinatorial. In Part~I we introduced a class of parameters which gave rise to many new non-hy\-per\-oc\-ta\-he\-dral categories of partitions. In the present article we show that this class actually contains all possible parameter values of all non-hy\-per\-oc\-ta\-he\-dral categories of partitions. This is an important step towards the classification of all non-hy\-per\-oc\-ta\-he\-dral categories. 
\end{abstract}
\maketitle
\section{Introduction} 
\label{section:introduction}

Co-representations of compact quantum groups  correspond to certain involutive monoidal categories (\cite{Wo87}, \cite{Wo88}, \cite{Wo98}). Banica and Speicher showed how to construct examples of such categories by taking rows of points as objects and partitions of two such rows as morphisms  (\cite{BaSp09}). By additionally painting the points different colors,  Freslon, Tarrago and the second author (\cite{FrWe14}, \cite{TaWe15a}, \cite{TaWe15b}) extended this construction to produce even more categories. For two (mutually inverse) colors, one obtains quantum subgroups of the free unitary quantum group $U_n^+$ of Wang's (\cite{Wa95}). For the precise definitions of two-col\-ored partitions and their categories the reader is referred to \cite[Sec\-tions~2 and~3]{MWNHO1}. See also \cite{TaWe15a} for more details and examples.
\par
In \cite{TaWe15a} Tarrago and the second author initiated a program to classify all  categories of two-col\-ored partitions. Different subclasses have since been indexed (\cite{TaWe15a}, \cite{Gr18}, \cite{MaWe18a}, \cite{MaWe18b}) by various contributors. The present article is the second part of a series aiming to determine and describe all so-called non-hy\-per\-oc\-ta\-he\-dral categories, i.e., all categories $\mc C\subseteq \Cp$ with $\PartSinglesWBTensor\in \mc C$ or $\PartFourWBWB\notin\mc C$.
  \par
In this regard the first article \cite{MWNHO1} and the present one pursue complementary approaches for detecting whether a given set of partitions is a non-hy\-per\-oc\-ta\-he\-dral category: Part~I gave sufficient conditions for being a non-hy\-per\-octa\-he\-dral category, Part~II now provides necessary ones. 
\par
\vspace{0.5em}
Let us take a closer look at the findings of Part~I, \cite{MWNHO1}.
Every two-col\-ored partition can be equipped with two natural structures on its set of points:  a measure-like one, the \emph{color sum}, and  a metric-like one, the \emph{color distance}. Both \cite{MWNHO1} and the present article study tuples of six properties of any given partition:
\begin{enumerate}[label=(\arabic*)]
\item the set of block sizes,
  \item the set of block  color sums,
  \item the color sum of the set of all points, 
  \item the set of color distances between subsequent legs of the same block with identical (normalized) colors,
  \item the set of color distances between subsequent legs of the same block with different (normalized) colors and
  \item the set of color distances between legs belonging to crossing blocks.
  \end{enumerate}
  By forming unions, one can aggregate these data over a given set of partitions. This information extracted from a set $\mc S\subseteq \Cp$ of partitions was called $Z(\mc S)$ in \cite{MWNHO1}.
  \par
  There it was shown that one can give constraints on the above six  properties which are preserved under category operations: A partially ordered set $(\mathsf Q,\leq)$ of parameters was introduced to prove that the sets of the form
  \begin{IEEEeqnarray*}{rCl}
    \mc R_Q\eqpd\{p\in \Cp\mid Z(\{p\})\leq Q\}\quad\text{for }Q\in\mathsf Q
  \end{IEEEeqnarray*}
  form non-hy\-per\-octa\-he\-dral categories.
  \par
  The current article now shows that these constraints encoded in $Z$ and $(\mathsf Q,\leq)$ are natural in the following sense. (See also Sec\-tion~\ref{section:reminder} for the definitions.) 
  \begin{umain}{\normalfont [{Theorem~\ref{theorem:main}}]}
Given any non-hy\-per\-octa\-he\-dral category $\mc C\subseteq \Cp$ of two-col\-ored partitions, we have $Z(\mc C)\in\mathsf Q$.
  \end{umain}
  The importance of this result comes from its role in the overall program of the article series. On the one hand, it will be crucial to proving the main assertions of the ensuing articles. On the other hand, once those have been established, it will combine with them to show the final result of the entire series, roughly:
  \begin{umainseries}[Excerpt]
      $Z$ restricts to a one-to-one correspondence between  the set $\nhoc$ of non-hy\-per\-oc\-ta\-he\-dral categories of two-col\-ored partitions and the parameter set $\mathsf Q$. 
  \end{umainseries}
  The proof will go as follows: By Part~I of the series, $\mc R_Q\subseteq \nhoc$ for every $Q\in \mathsf Q$. Conversely, by the above Main Theorem of Part~II, $Z(\mc C)\in \mathsf Q$ for any $\mc C\in\nhoc$. In the subsequent articles we will define a set $\mc G_{Z(\mc C)}\subseteq \Cp$ and show
      \begin{IEEEeqnarray*}{rCl}
      \mc G_{Z(\mc C)}\subseteq \mc C\subseteq \langle \mc G_{Z(\mc C)}\rangle  \quad\text{and} \quad \mc G_{Z(\mc R_{Z(\mc C)})}\subseteq \mc R_{Z(\mc C)}\subseteq \langle \mc G_{Z(\mc R_{Z(\mc C)})}\rangle.
    \end{IEEEeqnarray*}
    Proving $Z(\mc R_{Z(\mc C)})=Z(\mc C)$ will then let us conclude $\mc C=\langle \mc G_{Z(\mc C)}\rangle=\langle \mc G_{Z(\mc R_{Z(\mc C)})}\rangle=\mc R_{Z(\mc C)}$.

\section{Reminder on Definitions from Part~I}
\label{section:reminder}
For the convenience of the reader we briefly repeat those definitions from \cite[Sec\-tions~3--5]{MWNHO1} which are relevant to the current article. For definitions of partitions and categories of partitions see \cite[Sec\-tions~3.1 and~4.2]{MWNHO1}. Throughout this article we will use the notations and definitions from \cite[Sec\-tions~3--5]{MWNHO1}.
\begin{notation}
  For every set $S$ denote its power set by $\mathfrak P(S)$.
\end{notation}
\begin{definition}{{\normalfont\cite[Definition~5.2]{MWNHO1}}}
The \emph{parameter domain}  $\mathsf L$ is the sixfold Cartesian product of  $\mathfrak{P}(\integers)$.  
\end{definition}

\begin{definition}{{\normalfont\cite[Definition~5.3]{MWNHO1}}}
  \label{definition:Z}
  Using the notation from \cite[Sec\-tions~3--5]{MWNHO1}, we define the \emph{analyzer}  $Z: \, \mathfrak{P}(\Cp)\to \mathsf L$ by
	\begin{align*}
	Z\eqpd (\, F,\, V,\,  \toco,\,  L,\,  K,\,  X\,)
	\end{align*}
	where, for all $\mc S\subseteq \Cp$,
	\begin{enumerate}[label=(\alph*)]
		\item  \(F(\mc S)\eqpd \{\, |B| \mid p\in \mc S,\, B\text{ block of } p\}\) is the set of block sizes,
		\item  \(V(\mc S)\eqpd  \{\,\sigma_p(B)\mid p\in \mc S,\, B\text{ block of }p\}\) is the set of block color sums,
		\item \(\toco(\mc S)\eqpd \{\,\toco(p)\mid p\in \mc S\}\) is the set of total color sums,
		\item $\begin{aligned}[t]
		L(\mc S)\eqpd \{\,\delta_p(\alpha_1,\alpha_2)\mid&\, p\in \mc S, \, B\text{ block of }p,\, \alpha_1,\alpha_2\in B,\, \alpha_1\neq \alpha_2,\\
		&\, ]\alpha_1,\alpha_2[_p\cap B=\emptyset,\, \sigma_p(\{\alpha_1,\alpha_2\})\neq 0\}
		\end{aligned}$ \\
		is the set of color distances between any two subsequent legs of the \emph{same} block having the \emph{same} normalized color,
		\item $\begin{aligned}[t]
		K(\mc S)\eqpd \{\,\delta_p(\alpha_1,\alpha_2)\mid\,& p\in \mc S, \, B\text{ block of }p,\, \alpha_1,\alpha_2\in B,\, \alpha_1\neq \alpha_2,\\
		&\, ]\alpha_1,\alpha_2[_p\cap B=\emptyset,\, \sigma_p(\{\alpha_1,\alpha_2\})= 0\}
		\end{aligned}$\\
		is the set of color distances between any two subsequent legs of the \emph{same} block having \emph{different} normalized colors and
		\item $\begin{aligned}[t]
		X(\mc S)\eqpd \{\,\delta_p(\alpha_1,\alpha_2) \mid \, & p\in \mc S,\, B_1,B_2\text{ blocks of }p, \, B_1\text{ crosses } B_2,\\
		&\,  \alpha_1\in B_1,\,\alpha_2\in B_2\}
		\end{aligned}$ \\
		is the set of color distances between any two legs belonging to two \emph{crossing} blocks.
	\end{enumerate}
\end{definition}     
   \begin{notation}
     \label{integer-notations}
     \begin{enumerate}[label=(\alph*)]
            \item\label{integer-notations-2} For all $x,y\in \integers$ and $A,B\subseteq \integers$ write
         \begin{align*}
           xA+yB\eqpd \{xa+yb\mid a\in A, \,b\in B\}.
         \end{align*}
         Moreover, put $xA-yB\eqpd xA+(-y)B$. Per $A=\{1\}$ expressions like $x+yB$  are defined as well, and per $x=1$ so are such like $A+yB$. 
       \item\label{integer-notations-3} Let  $\pm S\eqpd S\cup(-S)$ for all sets $S\subseteq \integers$.
       \item\label{integer-notations-4} For all $m\in \integers$ and $D\subseteq \integers$ define
         \begin{align*}
           D_m\eqpd (D\cup(m\!-\! D))+m\integers \quad\text{and}\quad D_m'\eqpd (D\cup(m\!-\!D)\cup \{0\})+m\integers.
         \end{align*}
                \item\label{integer-notations-5}   Use the abbreviations $\dwi{0}\eqpd\emptyset$ and $\dwi{k}\eqpd \{1,\ldots,k\}$ for all $k\in \pint$.
     \end{enumerate}
   \end{notation}

   \begin{definition}[{\cite[Definition~5.7]{MWNHO1}}]
     \label{definition:Q}
    Define the \emph{parameter range} $\mathsf Q$ as the subset of $\mathsf L$ comprising all tuples $(f,v,s,l,k,x)$
    % \begin{align*}
    %   (f,v,s,l,k,x)
    % \end{align*}
    listed below, where  $u\in\{0\}\cup \pint$, where $m\in \pint$, where $D\subseteq \{0\}\cup\dwi{\lfloor\frac{m}{2}\rfloor}$,  where $E\subseteq \{0\}\cup \pint$ and where $N$ is a sub\-se\-mi\-group of $(\pint,+)$:
                    \begin{align*}
                      \begin{matrix}
                        f&v&s&l&k& x  \\ \hline \\[-0.85em]
                        \{2\} & \pm\{0, 2\} & 2um\integers & m\integers & m\integers & \integers\\
                      \{2\} & \pm\{0, 2\} & 2um\integers & m\hspace{-2.5pt}+\hspace{-2.5pt}2m\integers & 2m\integers & \integers \\
                        \{2\} & \pm \{0, 2\} & 2um\integers & m\hspace{-2.5pt}+\hspace{-2.5pt}2m\integers & 2m\integers & \integers\backslash m\integers \\
                      \{2\} & \{0\} & \{0\} & \emptyset & m\integers & \integers\\
                      \{2\} & \pm\{0, 2\} & \{0\} & \{0\} & \{0\} &  \integers\backslash N_0 \\
                      \{2\} & \{0\} & \{0\} & \emptyset & \{0\} & \integers\backslash N_0 \\
                      \{2\} & \{0\} & \{0\} & \emptyset & \{0\} &\integers\backslash N_0' \\
                         \{1,2\}&\pm\{0, 1, 2\} & um\integers & m\integers & m\integers & \integers\backslash D_m\\
                                                 \{1,2\}&\pm\{0, 1, 2\} & 2um\integers & m\hspace{-2.5pt}+\hspace{-2.5pt}2m\integers & 2m\integers & \integers\backslash D_m\\
                         \{1,2\}&\pm \{0, 1\} & um\integers & \emptyset & m\integers & \integers\backslash D_m \\
                        \{1,2\}&\pm\{0, 1, 2\} & \{0\} & \{0\} & \{0\} & \integers\backslash E_0 \\
                      \{1,2\}&\pm\{0,  1\} & \{0\} & \emptyset & \{0\} & \integers\backslash E_0\\
                        \pint & \integers & um\integers & m\integers & m\integers & \integers\backslash D_m\\
                        \pint & \integers & \{0\} & \{0\} & \{0\} & \integers\backslash E_0\\
                    \end{matrix}
                    \end{align*}
                  \end{definition}

The goal of this article, as sketched in the introduction, is to prove that $Z$ restricts to a map $\nhoc\to \mathsf Q$ (see Theorem~\ref{theorem:main}). Evidently, $\mathsf Q$ is not a Cartesian product; the  six entries of the tuples cannot vary independently. Rather, only very special tuples of sets are allowed. Hence, if the claim $Z:\nhoc\to \mathsf Q$ is to be true, then it is not enough to study the components of $Z$ individually. We must also investigate the relations between them. In consequence, the argument follows a winding path, taking components into and out of consideration underway as required or convenient. 
\section{Tools: Equivalence and Projection}
\label{section:equivalence}
We introduce an equivalence relation on pairs of partitions and consecutive sets therein by which to compare partitions locally (cf.\ \cite[Definition~6.2]{MaWe18a}).

\begin{definition}
  For all $i\in \{1,2\}$, let $P_{p_i}$ denote the set of all points of $p_i\in \Cp$ and let $S_i\subseteq P_{p_i}$ be consecutive. We call $(p_1,S_1)$ and $(p_2,S_2)$ \emph{equivalent} if $S_1=S_2=\emptyset$ or if the following  is true:
There exist $n\in \mathbb N$ and for each $i\in \{1,2\}$ pairwise distinct points $\gamma_{i,1},\ldots,\gamma_{i,n}$ in $p_i$ such that $(\gamma_{i,1},\ldots, \gamma_{i,n})$ is ordered in $p_i$ and $S_i=\{\gamma_{i,1},\ldots,\gamma_{i,n}\}$ and such that for all $j,j'\in \{1,\ldots,n\}$ (possibly $j=j'$) the following are true:
  \begin{enumerate}[label=(\arabic*)]
    \item The normalized colors of $\gamma_{1,j}$ in $p_1$ and $\gamma_{2,j}$ in $p_2$ agree.
    \item The points $\gamma_{1,j}$ and $\gamma_{1,j'}$ both belong to a block $B_1$ of $p_1$ with $B_1\subseteq S_1$ if and only if $\gamma_{2,j}$ and $\gamma_{2,j'}$ both belong to a block $B_2$ of $p_2$ with $B_2\subseteq S_2$.
          \item The points $\gamma_{1,j}$ and $\gamma_{1,j'}$ both belong to a block $B_1$ of $p_1$ with $B_1\not\subseteq S_1$ if and only if $\gamma_{2,j}$ and $\gamma_{2,j'}$ both belong to a block $B_2$ of $p_2$ with $B_2\not\subseteq S_2$.
          \end{enumerate}
          \begin{mycenter}[1em]
                \begin{tikzpicture}[baseline=2.5cm*0.666]
    \def\scp{0.666}
    \def\linksize{\scp*0.075cm}
    \def\pointsize{\scp*0.25cm}
    \def\dd{\scp*0.5cm}
    \def\dx{\scp*1cm}
    \def\cx{\scp*0.3cm}
    \def\txu{6*\dx}    
    \def\txl{6*\dx}
    \def\dy{\scp*1cm}
    \def\cy{\scp*0.3cm}
    \def\ty{5*\dy}
    \tikzset{whp/.style={circle, inner sep=0pt, text width={\pointsize}, draw=black, fill=white}}
    \tikzset{blp/.style={circle, inner sep=0pt, text width={\pointsize}, draw=black, fill=black}}
    \tikzset{lk/.style={regular polygon, regular polygon sides=4, inner sep=0pt, text width={\linksize}, draw=black, fill=black}}
    \draw[dotted] ({0-\dd},{0}) -- ({\txl+\dd},{0});
    \draw[dotted] ({0-\dd},{\ty}) -- ({\txl+\dd},{\ty});
    \node [whp] (l1) at ({0+0*\dx},{0+0*\ty}) {};
    \node [blp] (l2) at ({0+1*\dx},{0+0*\ty}) {};
    \node [whp] (l3) at ({0+2*\dx},{0+0*\ty}) {};
    \node [whp] (l4) at ({0+3*\dx},{0+0*\ty}) {};
    \node [blp] (l5) at ({0+4*\dx},{0+0*\ty}) {};
    \node [blp] (l6) at ({0+5*\dx},{0+0*\ty}) {};
    \node [whp] (l7) at ({0+6*\dx},{0+0*\ty}) {};        
    \node [blp] (u1) at ({0+0*\dx},{0+1*\ty}) {};
    \node [whp] (u2) at ({0+1*\dx},{0+1*\ty}) {};
    \node [blp] (u3) at ({0+2*\dx},{0+1*\ty}) {};
    \node [whp] (u4) at ({0+3*\dx},{0+1*\ty}) {};
    \node [whp] (u5) at ({0+4*\dx},{0+1*\ty}) {};
    \node [whp] (u6) at ({0+5*\dx},{0+1*\ty}) {};
    \node [lk, yshift={2*\dy}] at (l1) {};
    \node [lk, yshift={2*\dy}] at (l5) {};
    \node [lk, yshift={2*\dy}] at (l6) {};
    \node [lk, yshift={-1*\dy}] at (u1) {};
    \node [lk, yshift={-1*\dy}] at (u3) {};
    \node [lk, yshift={-2*\dy}] at (u2) {};
    \node [lk, yshift={-2*\dy}] at (u4) {};
    \node [lk, yshift={-2*\dy}] at (u5) {};
    \node [lk, yshift={3*\dy}] at (l7) {};                
    \draw [->] (l4) --++ (0,{\dy});
    \draw [->] (u6) --++ (0,{-\dy});    
    \draw (l2) -- (u2);
    \draw (l3) -- (u3);    
    \draw (l1) --++(0,{2*\dy}) -| (l5);
    \draw (l5) --++(0,{2*\dy}) -| (l6);
    \draw (u1) --++(0,{-1*\dy}) -| (u3);
    \draw (u2) --++(0,{-2*\dy}) -| (u4);
    \draw (u4) --++(0,{-2*\dy}) -| (u5);
    \draw (u5) --++(0,{-2*\dy}) -| (l7);
        \draw [densely dashed] ({-\dd},{\cy}) -- ($(l4)+({\cx},{\cy})$) --++ (0,{-2*\cy}) -- ({-\dd},{-\cy});
        \draw [densely dashed] ({-\dd},{\cy+\ty}) -- ($(u3)+({\cx},{\cy})$) --++ (0,{-2*\cy}) -- ({-\dd},{-\cy+\ty});
            \node at ({\txl+\dx},{0.5*\ty}) {$p$};
    \node (lab1) at ({-2*\dx},{0.5*\ty}) {$S$};
    \draw [densely dotted, ->, shorten >=5pt] (lab1.east) -- ({-\dd},0);
    \draw [densely dotted, ->, shorten >=5pt] (lab1.east) -- ({-\dd},{\ty});
  \end{tikzpicture}\quad\quad            
                \begin{tikzpicture}[baseline=2cm*0.666]
    \def\scp{0.666}
    \def\linksize{\scp*0.075cm}
    \def\pointsize{\scp*0.25cm}
    \def\dd{\scp*0.5cm}
    \def\dx{\scp*1cm}
    \def\cx{\scp*0.3cm}
    \def\txu{6*\dx}    
    \def\txl{6*\dx}
    \def\dy{\scp*1cm}
    \def\cy{\scp*0.3cm}
    \def\ty{4*\dy}
    \tikzset{whp/.style={circle, inner sep=0pt, text width={\pointsize}, draw=black, fill=white}}
    \tikzset{blp/.style={circle, inner sep=0pt, text width={\pointsize}, draw=black, fill=black}}
    \tikzset{lk/.style={regular polygon, regular polygon sides=4, inner sep=0pt, text width={\linksize}, draw=black, fill=black}}
    \draw[dotted] ({0-\dd},{0}) -- ({\txl+\dd},{0});
    \draw[dotted] ({0-\dd},{\ty}) -- ({\txl+\dd},{\ty});
    \node [whp] (l1) at ({0+0*\dx},{0+0*\ty}) {};
    \node [whp] (l2) at ({0+1*\dx},{0+0*\ty}) {};
    \node [whp] (l3) at ({0+2*\dx},{0+0*\ty}) {};
    \node [whp] (l4) at ({0+3*\dx},{0+0*\ty}) {};
    \node [whp] (l5) at ({0+4*\dx},{0+0*\ty}) {};
    \node [blp] (l6) at ({0+5*\dx},{0+0*\ty}) {};
    \node [whp] (l7) at ({0+6*\dx},{0+0*\ty}) {};        
    \node [blp] (u1) at ({0+0*\dx},{0+1*\ty}) {};
    \node [blp] (u2) at ({0+1*\dx},{0+1*\ty}) {};
    \node [whp] (u3) at ({0+2*\dx},{0+1*\ty}) {};
    \node [blp] (u4) at ({0+3*\dx},{0+1*\ty}) {};
    \node [lk, yshift={-1*\dy}] at (u3) {};
    \node [lk, yshift={3*\dy}] at (l6) {};
    \node [lk, yshift={2*\dy}] at (l5) {};
    \node [lk, yshift={2*\dy}] at (l7) {};
    \node [lk, yshift={2*\dy}] at (l2) {};    
    \draw [->] (u1) --++ (0,{-\dy});
    \draw (l4) -- (u4);
    \draw (l3) --++(0,{3*\dy});
    \draw (l1) --++(0,{1*\dy}) -| (l2);
    \draw (l5) --++(0,{2*\dy}) -| (l7);
    \draw (l5) --++(0,{2*\dy}) -| (u2);
    \draw (l6) --++(0,{3*\dy}) -| (u3);
    \draw [densely dashed] ({\txl+\dd},{\cy}) -- ($(l5)+({-\cx},{\cy})$) --++ (0,{-2*\cy}) -- ({\txl+\dd},{-\cy});
    \draw [densely dashed] ({\txu+\dd},{\cy+\ty}) -- ($(u1)+({-\cx},{\cy})$) --++ (0,{-2*\cy}) -- ({\txl+\dd},{-\cy+\ty});
    \node at ({-\dx},{0.5*\ty}) {$p'$};
    \node (lab1) at ({\txl+2*\dx},{0.5*\ty}) {$S'$};
    \draw [densely dotted, ->, shorten >=5pt] (lab1.west) -- ({\txl+\dd},0);
    \draw [densely dotted, ->, shorten >=5pt] (lab1.west) -- ({\txl+\dd},{\ty});
  \end{tikzpicture}
\\[1em]
                  \begin{tikzpicture}
    \def\scp{0.666}
    \def\linksize{\scp*0.075cm}
    \def\pointsize{\scp*0.25cm}
    \def\dd{\scp*0.5cm}
    \def\dx{\scp*1cm}
    \def\cx{\scp*0.3cm}
    \def\txu{6*\dx}    
    \def\txl{6*\dx}
    \def\dy{\scp*1cm}
    \def\cy{\scp*0.3cm}
    \def\ty{4*\dy}
    \tikzset{whp/.style={circle, inner sep=0pt, text width={\pointsize}, draw=black, fill=white}}
    \tikzset{blp/.style={circle, inner sep=0pt, text width={\pointsize}, draw=black, fill=black}}
    \tikzset{lk/.style={regular polygon, regular polygon sides=4, inner sep=0pt, text width={\linksize}, draw=black, fill=black}}
    \draw[dotted] ({0-\dd},{0}) -- ({\txl+\dd},{0});
 %   \draw[dotted] ({0-\dd},{\ty}) -- ({\txl+\dd},{\ty});
    %
    \node [whp] (l1) at ({0+0*\dx},{0+0*\ty}) {};
    \node [blp] (l2) at ({0+1*\dx},{0+0*\ty}) {};
    \node [whp] (l3) at ({0+2*\dx},{0+0*\ty}) {};
    \node [whp] (l4) at ({0+3*\dx},{0+0*\ty}) {};
    \node [blp] (l5) at ({0+4*\dx},{0+0*\ty}) {};
    \node [whp] (l6) at ({0+5*\dx},{0+0*\ty}) {};
    \node [whp] (l7) at ({0+6*\dx},{0+0*\ty}) {};        
    \node [lk, yshift={2*\dy}] at (l1) {};
    \node [lk, yshift={2*\dy}] at (l3) {};
    \node [lk, yshift={2*\dy}] at (l6) {};
    \node [lk, yshift={1*\dy}] at (l2) {};
    \node [lk, yshift={1*\dy}] at (l5) {};    
    \draw [->] (l7) --++ (0,{\dy});
    \draw (l2) --++(0,{3*\dy});
    \draw (l4) --++(0,{3*\dy});
    \draw (l5) --++(0,{3*\dy});    
    \draw (l1) --++(0,{2*\dy}) -| (l3);
    \draw (l3) --++(0,{2*\dy}) -| (l6);
    \draw (l2) --++(0,{1*\dy}) -| (l5);    
    \node at ({0.5*\txl},{-\dy}) {class of $(p,S)\cong (p',S')$};
  \end{tikzpicture}
  \end{mycenter}
\end{definition}
If $(p_1,S_1)$ and $(p_2,S_2)$ are equivalent, then $S_1$ and $S_2$ agree in size and normalized coloring up to a rotation $\varrho$ and the induced partitions $\{B_1\cap S_1\,\vert\, B_1\text{ block of }p_1\}$ of $S_1$ and $\{B_2\cap S_2\,\vert\, B_2\text{ block of }p_2\}$ of $S_2$ concur up to $\varrho$. However, this is only a necessary condition. Equivalence further requires that a block $B_1\cap S_1$ of the restriction of $p_1$ stems from a block $B_1$ of $p_1$ which has legs outside $S_1$ if and only if the corresponding statement $B_2\not\subseteq S_2$ is true for the block $B_2$ of $p_2$ which $B_1$ is mapped to under $\varrho$.
\par
We define and construct special representatives of the classes of this equivalence relation. Recall that a partition $p\in \Cp$ is called \emph{projective} if $p$ is self-adjoint, i.e., $p=p^\ast$, and idempotent, i.e., the pair $(p,p)$ is composable and $pp=p$.

\begin{definition}
  For every consecutive set $S$ in $p\in \Cp$ we call the unique projective partition $q$ with lower row $M$ such that $(q,M)$ and $(p,S)$ are equivalent the \emph{projection} $P(p,S)$ of $(p,S)$. 
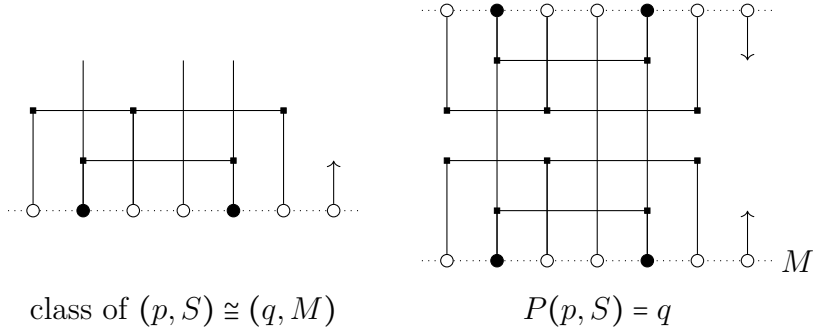
\begin{mycenter}[0.5em]
                  \begin{tikzpicture}[baseline=1.5cm*0.666]
    \def\scp{0.666}
    \def\linksize{\scp*0.075cm}
    \def\pointsize{\scp*0.25cm}
    \def\dd{\scp*0.5cm}
    \def\dx{\scp*1cm}
    \def\cx{\scp*0.3cm}
    \def\txu{6*\dx}    
    \def\txl{6*\dx}
    \def\dy{\scp*1cm}
    \def\cy{\scp*0.3cm}
    \def\ty{4*\dy}
    \tikzset{whp/.style={circle, inner sep=0pt, text width={\pointsize}, draw=black, fill=white}}
    \tikzset{blp/.style={circle, inner sep=0pt, text width={\pointsize}, draw=black, fill=black}}
    \tikzset{lk/.style={regular polygon, regular polygon sides=4, inner sep=0pt, text width={\linksize}, draw=black, fill=black}}
    \draw[dotted] ({0-\dd},{0}) -- ({\txl+\dd},{0});
 %   \draw[dotted] ({0-\dd},{\ty}) -- ({\txl+\dd},{\ty});
    %
    \node [whp] (l1) at ({0+0*\dx},{0+0*\ty}) {};
    \node [blp] (l2) at ({0+1*\dx},{0+0*\ty}) {};
    \node [whp] (l3) at ({0+2*\dx},{0+0*\ty}) {};
    \node [whp] (l4) at ({0+3*\dx},{0+0*\ty}) {};
    \node [blp] (l5) at ({0+4*\dx},{0+0*\ty}) {};
    \node [whp] (l6) at ({0+5*\dx},{0+0*\ty}) {};
    \node [whp] (l7) at ({0+6*\dx},{0+0*\ty}) {};        
    \node [lk, yshift={2*\dy}] at (l1) {};
    \node [lk, yshift={2*\dy}] at (l3) {};
    \node [lk, yshift={2*\dy}] at (l6) {};
    \node [lk, yshift={1*\dy}] at (l2) {};
    \node [lk, yshift={1*\dy}] at (l5) {};    
    \draw [->] (l7) --++ (0,{\dy});
    \draw (l2) --++(0,{3*\dy});
    \draw (l4) --++(0,{3*\dy});
    \draw (l5) --++(0,{3*\dy});    
    \draw (l1) --++(0,{2*\dy}) -| (l3);
    \draw (l3) --++(0,{2*\dy}) -| (l6);
    \draw (l2) --++(0,{1*\dy}) -| (l5);    
    \node at ({0.5*\txl},{-2*\dy}) {class of $(p,S)\cong (q,M)$};
  \end{tikzpicture}\quad\quad
                    \begin{tikzpicture}[baseline=2.5cm*0.666]
    \def\scp{0.666}
    \def\linksize{\scp*0.075cm}
    \def\pointsize{\scp*0.25cm}
    \def\dd{\scp*0.5cm}
    \def\dx{\scp*1cm}
    \def\cx{\scp*0.3cm}
    \def\txu{6*\dx}    
    \def\txl{6*\dx}
    \def\dy{\scp*1cm}
    \def\cy{\scp*0.3cm}
    \def\ty{5*\dy}
    \tikzset{whp/.style={circle, inner sep=0pt, text width={\pointsize}, draw=black, fill=white}}
    \tikzset{blp/.style={circle, inner sep=0pt, text width={\pointsize}, draw=black, fill=black}}
    \tikzset{lk/.style={regular polygon, regular polygon sides=4, inner sep=0pt, text width={\linksize}, draw=black, fill=black}}
    \draw[dotted] ({0-\dd},{0}) -- ({\txl+\dd},{0});
    \draw[dotted] ({0-\dd},{\ty}) -- ({\txl+\dd},{\ty});
    \node [whp] (l1) at ({0+0*\dx},{0+0*\ty}) {};
    \node [blp] (l2) at ({0+1*\dx},{0+0*\ty}) {};
    \node [whp] (l3) at ({0+2*\dx},{0+0*\ty}) {};
    \node [whp] (l4) at ({0+3*\dx},{0+0*\ty}) {};
    \node [blp] (l5) at ({0+4*\dx},{0+0*\ty}) {};
    \node [whp] (l6) at ({0+5*\dx},{0+0*\ty}) {};
    \node [whp] (l7) at ({0+6*\dx},{0+0*\ty}) {};
    \node [whp] (u1) at ({0+0*\dx},{0+1*\ty}) {};
    \node [blp] (u2) at ({0+1*\dx},{0+1*\ty}) {};
    \node [whp] (u3) at ({0+2*\dx},{0+1*\ty}) {};
    \node [whp] (u4) at ({0+3*\dx},{0+1*\ty}) {};
    \node [blp] (u5) at ({0+4*\dx},{0+1*\ty}) {};
    \node [whp] (u6) at ({0+5*\dx},{0+1*\ty}) {};
    \node [whp] (u7) at ({0+6*\dx},{0+1*\ty}) {};            
    \node [lk, yshift={2*\dy}] at (l1) {};
    \node [lk, yshift={2*\dy}] at (l3) {};
    \node [lk, yshift={2*\dy}] at (l6) {};
    \node [lk, yshift={1*\dy}] at (l2) {};
    \node [lk, yshift={1*\dy}] at (l5) {};
    \node [lk, yshift={-2*\dy}] at (u1) {};
    \node [lk, yshift={-2*\dy}] at (u3) {};
    \node [lk, yshift={-2*\dy}] at (u6) {};
    \node [lk, yshift={-1*\dy}] at (u2) {};
    \node [lk, yshift={-1*\dy}] at (u5) {};    
    \draw [->] (l7) --++ (0,{\dy});
    \draw [->] (u7) --++ (0,{-\dy});    
    \draw (l2) -- (u2);
    \draw (l4) -- (u4);
    \draw (l5) -- (u5);
    \draw (l1) --++(0,{2*\dy}) -| (l3);
    \draw (l3) --++(0,{2*\dy}) -| (l6);
    \draw (l2) --++(0,{1*\dy}) -| (l5);
    \draw (u1) --++(0,{-2*\dy}) -| (u3);
    \draw (u3) --++(0,{-2*\dy}) -| (u6);
    \draw (u2) --++(0,{-1*\dy}) -| (u5);        
    \node at ({0.5*\txl},{-\dy}) {$P(p,S)=q$};
    \node at ({\txl+\dx},0) {$M$};
  \end{tikzpicture}
\end{mycenter}
\end{definition}
In truth, of course, for any consecutive set $S$ in $p\in \Cp$ the projection $P(p,S)$ depends only on the equivalence class of $(p,S)$. The following lemma constitutes a generalization of \cite[Lem\-ma~6.4]{MaWe18a}.
\begin{lemma}
  \label{lemma:projection}
  $P(p,S)\in \langle p\rangle$ for any consecutive set $S$ in any $p\in \Cp$.
\end{lemma}
\begin{proof}
As $S=\emptyset$ implies $P(p,S)=\emptyset\in \langle p\rangle$, let $S\neq \emptyset$. By rotation we can assume that $S$ is the lower row of $p$. Then $S$ has the same size and coloring in $p$ as in $q\eqpd pp^*$. We show $q=P(p,S)$. By the nature of composition the blocks of $p$ which are contained in $S$ are blocks of $q$ as well. We  only need to care about the other blocks of $q$. If we identify the upper row of $p$ and the lower row of $p^\ast$, the same partition $s$ is induced there by $p$ and $p^\ast$. Consequently, the meet of the two induced partitions is identical with $s$ as well. That means that every block $D$ of $s$ intersects exactly one block $B$ of $p$ and exactly one block of $p^*$, namely the mirror image of $B$. The block of $q$ resulting from $D$ therefore contains exactly the restriction of $B$ to the lower row and the mirror image of that set on the upper row. That means $q=P(p,S)$, which proves the claim.
\end{proof}
                  
        \section{Step~1: Component $F$ in Isolation}
        \label{section:block-sizes}
We now take our first step towards proving the main result that the analyzer $Z$ from Definition~\ref{definition:Z} restricts to a map $\nhoc\to \mathsf Q$ (see Theorem~\ref{theorem:main}). Namely, we verify (see Proposition~\ref{proposition:result-F}) that, for every non-hy\-per\-octa\-he\-dral category $\mc C\subseteq \Cp$, the set
\begin{align*}
	F(\mc C)\eqpd \{|B| \mid p\in\mc C,\, B\text{ block of }p\}
\end{align*}
of block sizes appearing in $\mc C$ can only be one of the three sets of integers admissible as a first component for tuples in $\mathsf Q$ by Definition~\ref{definition:Q}. 

\begin{lemma}{\normalfont \cite[Lem\-mata~1.3 (b), 2.1 (a)]{TaWe15a}}
  \label{lemma:singletons}
  Let $\mc C\subseteq \Cp$ be a category.
  \begin{enumerate}[label=(\alph*)]
  \item\label{lemma:singletons-1}
          \(\langle \PartSinglesWBTensor\rangle =\langle \PartSinglesBWTensor\rangle=\langle \PartSinglesProjW\rangle=\langle \PartSinglesProjB\rangle\).
  \item\label{lemma:singletons-2} The following statements are equivalent:
    \begin{enumerate}[label=(\arabic*)]
    \item There exists in $\mc C$ a partition with a singleton block.
    \item $\PartSinglesWBTensor\in \mc C$.
    \end{enumerate}
    \item\label{lemma:singletons-3} If $\PartSinglesWBTensor\in \mc C$, then $\mc C$ is closed under disconnecting points from their blocks.
  \end{enumerate}
\end{lemma}
\begin{proof}
  \begin{enumerate}[wide,label=(\alph*)]
    \item All transformations can be achieved by  basic and cyclic rotations.
    \item Projecting to a singleton block produces $\PartSinglesProjW$ or $\PartSinglesProjB$. Hence, Part~\ref{lemma:singletons-1} and Lem\-ma~\ref{lemma:projection}  prove the claim.
      \item Rotate a given partition such that the leg to disconnect from its block is the only lower point. Composing from below with $\PartSinglesProjW$ or $\PartSinglesProjB$, depending on the color of the leg, and reversing the rotation achieves what is claimed. Hence, Part~\ref{lemma:singletons-1} concludes the proof.  \qedhere
  \end{enumerate}
\end{proof}

\begin{lemma}{\normalfont \cite[Lem\-mata~1.3 (d), 2.1 (b)]{TaWe15a}}
  \label{lemma:companies}
  Let $\mc C\subseteq \Cp$ be a category.
  \begin{enumerate}[label=(\alph*)]
  \item  \label{lemma:companies-1}
          \(\langle \PartFourWBWB\rangle =\langle \PartFourBWBW\rangle=\langle \PartFourProjWB\rangle=\langle \PartFourProjBW\rangle\).
  \item  \label{lemma:companies-2} The following statements are equivalent:
    \begin{enumerate}[label=(\arabic*)]
    \item There exists in $\mc C$ a partition with a block with at least three legs.
    \item $\PartFourWBWB\in \mc C$.
    \end{enumerate}
    \item  \label{lemma:companies-3} If $\PartFourWBWB\in \mc C$, then $\mc C$ is closed under connecting the two points in any turn.
  \end{enumerate}
\end{lemma}
\begin{proof}
  \begin{enumerate}[label=(\alph*),wide]
  \item Once again, by basic and cyclic rotations we can transform the partitions into each other. 
  \item Suppose $B$ is a block in $p\in \mc C$ with at least three legs, $\alpha,\beta\in B$, $\alpha\neq \beta$ and $]\alpha,\beta[_p\cap B=\emptyset$.  Let $T$ be the set of the first lower and the first upper point of  $P(p,[\alpha,\beta]_p)$. The partition $P(P(p,[\alpha,\beta]_p),T)$ is either $\PartFourProjWB$ or $\PartFourProjBW$. Thus follows the claim by Part~\ref{lemma:companies-1} and Lem\-ma~\ref{lemma:projection}.
    \item Let $T$ be the turn in $p\in \mc C$ whose points we want to connect. By rotation we can assume that $T$ is the upper row of $p$. By composing $p$ from above with $\PartFourProjWB$ or $\PartFourProjBW$, depending on the sequence of colors in $T$, and reversing the initial rotation we achieve exactly what is claimed. So, Part~\ref{lemma:companies-1} implies the assertion.\qedhere
  \end{enumerate}
\end{proof}
Recall the cases $\mc O$, $\mc B$, $\mc S$ from \cite[Definition~4.1]{MWNHO1}.
\begin{proposition}
  \label{proposition:result-F}
  Let $\mc C\subseteq \Cp$ be a non-hy\-per\-oc\-ta\-hed\-ral category.
  \begin{enumerate}[label=(\alph*)]
    \item \label{proposition:result-F-0} The set $F(\mc C)$ is given by $\{2\}$, $\{1,2\}$ or $\pint$.
  \item\label{proposition:result-F-1} If $\mc C$ is case~$\mc O$, then $F(\mc C)=\{2\}$.
  \item\label{proposition:result-F-2} If $\mc C$ is case~$\mc B$, then $F(\mc C)=\{1,2\}$.
  \item\label{proposition:result-F-3} If $\mc C$ is case~$\mc S$, then $F(\mc C)=\pint$.        
  \end{enumerate}
\end{proposition}
\begin{proof}
  By definition of a category, $\PartIdenLoWB\in \mc C$ and thus $\{2\}\subseteq  F(\mc C)$.
  \begin{enumerate}[label=(\alph*),wide]
    \item The first claim follows from the other three.
  \item Because $\PartSinglesWBTensor\notin \mc C$ and $\PartFourWBWB\notin \mc C$, Lem\-mata~\hyperref[lemma:singletons-2]{\ref*{lemma:singletons}~\ref*{lemma:singletons-2}} and~\hyperref[lemma:companies-2]{\ref*{lemma:companies}~\ref*{lemma:companies-2}} show that every block in every partition of $\mc C$ has exactly two legs, i.e., $F(\mc C)=\{2\}$.
  \item The assumption $\PartFourWBWB\notin \mc C$ implies by Lem\-ma~\hyperref[lemma:companies-2]{\ref*{lemma:companies}~\ref*{lemma:companies-2}} that no partition of $\mc C$ has blocks with more than two legs: $F(\mc C)\subseteq \{1,2\}$. Because $\PartSinglesWBTensor\in \mc C$, it is clear that $\{1\}\subseteq F(\mc C)$. Thus, $F(\mc C)=\{1,2\}$ has been proven. 
    \item  It suffices to show $\pint\subseteq F(\mc C)$. Let $n\in \pint$ be arbitrary. Then, \[p\eqpd (\PartSinglesWBTensor)^{\otimes \left\lceil \frac{n}{2}\right\rceil}\in \mc C.\] Thanks to $\PartFourWBWB\in \mc C$ we can, by Lem\-ma~\hyperref[lemma:companies-3]{\ref*{lemma:companies}~\ref*{lemma:companies-3}},  connect the first $n$ points in $p$ to produce a partition in $\mc C$ containing a block with $n$ points, proving $\{n\}\subseteq F(\mc C)$. \qedhere
  \end{enumerate}
\end{proof}

\section{Step~2: Component $V$ and its Relation to $F$ and $L$}
\label{section:block-color-sums}
The next objective is to narrow down the range of the component $V$ of $Z$ over $\nhoc$. Given a non-hy\-per\-octa\-he\-dral category $\mc C\subseteq \Cp$, we show that the set
\begin{align*}
	V(\mc C)\eqpd \{\sigma_p(B)\mid p\in \mc C,\, B\text{ block of }p\}
\end{align*}
of block color sums occurring in $\mc C$
can only be one of the five sets allowed as second components for tuples of $\mathsf Q$ by Definition~\ref{definition:Q}. Beyond that, we can use Proposition~\ref{proposition:result-F} to show a result about the three parameters $V(\mc C)$, $F(\mc C)$ and
\begin{align*}
	L(\mc C)\eqpd \{\,\delta_p(\alpha_1,\alpha_2)\mid&\, p\in \mc C, \, B\text{ block of }p,\, \alpha_1,\alpha_2\in B,\, \alpha_1\neq \alpha_2,\\
	&\, ]\alpha_1,\alpha_2[_p\cap B=\emptyset,\, \sigma_p(\{\alpha_1,\alpha_2\})\neq 0\},  
\end{align*}
the set of color distances between legs of the same block with identical normalized colors appearing in $\mc C$: Viewed together as $(F,V,L)(\mc C)$, they  satisfy the conditions necessary for $Z(\mc C)$ to be element of $\mathsf Q$ by Definition~\ref{definition:Q}.
    \begin{proposition}
      \label{proposition:result-V}
      Let $\mc C\subseteq \Cp$ be a non-hy\-per\-oc\-ta\-hed\-ral category.
      \begin{enumerate}[label=(\alph*)]
        \item\label{proposition:result-V-0} The set $V(\mc C)$ is given by $\{0\}$, $\pm\{0,2\}$, $\pm \{0,1\}$, $\pm \{0,1,2\}$ or $\integers$.
      \item\label{proposition:result-V-1} If $\mc C$ is case $\mc O$, then
        \begin{align*}
          V(\mc C)=
          \begin{cases}
            \pm\{0,2\}&\text{if }L(\mc C)\neq \emptyset,\\
            \phantom{\pm}\{0\}&\text{otherwise.}
          \end{cases}
        \end{align*}
      \item\label{proposition:result-V-2} If $\mc C$ is case $\mc B$, then
\begin{align*}
          V(\mc C)=
          \begin{cases}
            \pm\{0,1,2\}&\text{if }L(\mc C)\neq \emptyset,\\
            \pm \{0,1\}&\text{otherwise.}
          \end{cases}
        \end{align*}
        \item\label{proposition:result-V-3} If $\mc C$ is case $\mc S$, then $L(\mc C)\neq \emptyset$ and $V(\mc C)=\integers$.
      \end{enumerate}      
    \end{proposition}
    \begin{proof} 
      Two general facts about $V(\mc C)$ in advance: In any case, $0\in V(\mc C)$ since $V(\{\PartIdenLoWB\})=\{0\}$. And \cite[Lem\-ma~6.4]{MWNHO1}, using the fact that $p\in \mc C$ implies $\tilde p\in \mc C$, showed $V(\mc C)=-V(\mc C)$.
      \par
      \begin{enumerate}[label=(\alph*),wide]
      \item  Claim~\ref{proposition:result-V-0} follows from the other three.
      \item      A pair block $B$ in $p\in \mc C$ satisfies $\sigma_p(B)=0$ if and only if that block has no two (necessarily subsequent) legs of the same normalized colors. Otherwise it has color sum $-2$ or $2$.
      \item  And a singleton block always has color sums $-1$ or $1$. The rest follows from the proof of Part~\ref{proposition:result-V-1}.
        \item If $\mc C$ is case~$\mc S$, then $\PartFourWBWB\in \mc C$  and $\PartSinglesWBTensor\in \mc C$. Hence, we can use $\PartSinglesWBTensor$ to disconnect the left black point in $\PartFourWBWB$ by Lem\-ma~\hyperref[lemma:singletons-3]{\ref*{lemma:singletons}~\ref*{lemma:singletons-3}} to obtain $p\eqpd\PartSpecThreeCoSingleWBWB\in \mc C$ with $V(\{p\})=\{-1, 1\}$. Given any $n\in \pint$, we use $\PartFourWBWB$ to connect in $p^{\otimes n}\in \mc C$ all the $n$ many three-leg blocks together (leaving the disconnected singletons alone) in accordance with Lem\-ma~\hyperref[lemma:companies-3]{\ref*{lemma:companies}~\ref*{lemma:companies-3}}. That procedure results in the partition $q\in \mc C$ with $V(\{q\})=\{-1,n\}$. By $V(\mc C)=-V(\mc C)$ it then follows $V(\mc C)=\integers$ as claimed.\qedhere
      \end{enumerate}
    \end{proof}

    \section{\texorpdfstring{Step~3: Component $\toco$ in Isolation}{Step~3: Parameter Sigma in Isolation}}
    \label{section:total-color-sums}
Easily, we can confirm that for all non-hy\-per\-octa\-he\-dral categories $\mc C\subseteq \Cp$ the set
\begin{align*}
	\toco(\mc C)\eqpd \{\toco(p)\mid p\in \mc C\}
\end{align*}
of all total color sums appearing in $\mc C$
is within the range of allowed third entries of tuples in $\mathsf Q$ by Definition~\ref{definition:Q}.
The following proposition contains a generalization of \cite[Lem\-ma~2.6]{TaWe15a} and \cite[Proposition~2.7]{TaWe15a}.
\begin{proposition}
  \label{proposition:result-S}
  For every category $\mc C\subseteq \Cp$ the set $\toco(\mc C)$ is a subgroup of $\integers$.
\end{proposition}
\begin{proof}
   \cite[Lem\-ma~6.5~(c)]{MWNHO1} implies $\toco(\mc C)+\toco(\mc C)\subseteq \toco(\mc C)$. And $-\toco(\mc C)\subseteq \toco(\mc C)$ was shown in \cite[Lem\-ma~6.4]{MWNHO1}. As also $\toco(\PartIdenLoWB)=0$ and $\PartIdenLoWB\in \mc C$  by definition, the set $\toco(\mc C)$ is indeed a subgroup of $\integers$.
\end{proof}

\section{\texorpdfstring{Step~4: General Relations between $\toco$, $L$, $K$ and $X$}{Step~4: General Relations between Sigma, L, K and X}}
\label{section:color-lengths}
The goal remains proving that $Z$ (see Definition~\ref{definition:Z}) maps the set $\nhoc$ of non-hy\-per\-octa\-he\-dral categories to $\mathsf Q$ (see Definition~\ref{definition:Q}). So far, we have tackled this problem, more or less, one component of $Z$ at a time. In that way, what we have managed to show is, mostly, that the values over $\nhoc$ of each of the three maps $F$, $V$ and $\toco$, viewed individually, are confined to the range of parameters allowed by $\mathsf Q$ as corresponding entries of its elements. To complete this picture, we would also like to see that for any non-hy\-per\-octa\-he\-dral category $\mc C\subseteq \Cp$ the three sets $L(\mc C)$,
\begin{align*}
	K(\mc C)\eqpd \{\,\delta_p(\alpha_1,\alpha_2)\mid\,& p\in \mc C, \, B\text{ block of }p,\, \alpha_1,\alpha_2\in B,\, \alpha_1\neq \alpha_2,\\
	&\, ]\alpha_1,\alpha_2[_p\cap B=\emptyset,\, \sigma_p(\{\alpha_1,\alpha_2\})= 0\}\\	
\text{and}\quad	X(\mc C)\eqpd \{\,\delta_p(\alpha_1,\alpha_2) \mid \, & p\in \mc C,\, B_1,B_2\text{ blocks of }p, \, B_1\text{ crosses } B_2,\\
&\,  \alpha_1\in B_1,\,\alpha_2\in B_2\},	
\end{align*}
too, can only be of the kinds allowed as fourth, fifth and sixth components of tuples in $\mathsf Q$, respectively, by Definition~\ref{definition:Q}. However, due to the strong interdependences between these three components of $Z$, it is not even possible to prove this basic claim about the ranges of the individual maps by studying them one at a time. Instead, now, the reasonable thing to do is to consider the tuple $(\toco, L, K,X)$ and make inferences about its range over $\nhoc$. That will give us (see Proposition~\ref{lemma:result-s-l-k-x}) the claim about the individual ranges of $L$, $K$ and $X$ but also many more of the relations between them (and $\toco$), which we need to verify  the main result.

\subsection{Abstract Arithmetic Lemma}
As a first step, it is best to study the relationship between the $\toco$-, $L$-, $K$- and $X$-components of $Z$ in an abstract context, merely talking about arbitrary subsets of $\integers$ subject to certain axioms. Our goal for this sub\-sec\-tion is to prove the \hyperref[lemma:arithmetic]{Arithmetic Lem\-ma (\ref*{lemma:arithmetic})}: Assuming certain axioms (\ref{axioms:arithmetic}), we may deduce a certain parameter range. We will show in Sub\-sec\-tion~\ref{section:verifying-the-axioms} that for non-hy\-per\-oc\-ta\-he\-dral categories $\mc C\subseteq \Cp$ our sets $\toco(\mc C)$, $L(\mc C)$, $K(\mc C)$ and $X(\mc C)$ satisfy these axioms. Recall $\overline\bullet\eqpd\circ$ and $\overline\circ\eqpd\bullet$.
\begin{axioms}
  \label{axioms:arithmetic}
	Let $\sigma$ as well as $\kappa_{c_1,c_2}$ and $\xi_{c_1,c_2}$  for all $c_1,c_2\in\colors$ be subsets of $\integers$.
	Throughout this sub\-sec\-tion, make the following assumptions:
	\begin{enumerate}[label=(\roman*)]
		\item\label{lemma:gradedsets-condition-0}  $\sigma$ is a subgroup of $\integers$.
	\end{enumerate}
For all $(\omega_{c_1,c_2})_{c_1,c_2\in \colors}\in \{(\kappa_{c_1,c_2})_{c_1,c_2\in \colors},(\xi_{c_1,c_2})_{c_1,c_2\in \colors}\}$ and for all $c_1,c_2\in\colors$:
\begin{enumerate}[label=(\roman*), start=2]
	\item\label{lemma:gradedsets-condition-1} $\omega_{c_1,c_2}+\sigma\subseteq \omega_{c_1,c_2}$.
	\item\label{lemma:gradedsets-condition-3} $\omega_{c_1,c_2}\subseteq - \omega_{\overline{c_2},\overline{c_1}}$.
	\item\label{lemma:gradedsets-condition-2} $\omega_{c_1,c_2}\subseteq -\omega_{c_2,c_1}+\sigma$.       
\end{enumerate}
For all $c_1,c_2,c_3\in \colors$:
\begin{enumerate}[label=(\roman*), start=5]
	\item\label{lemma:gradedsets-condition-6} $\xi_{c_1,c_2}\subseteq \xi_{c_1,\overline{c_2}}\cup \left(-\xi_{c_2,\overline{c_1}}+\sigma\right)$.         
	\item\label{lemma:gradedsets-condition-4} $0\in \kappa_{\circ\bullet}\cap \kappa_{\bullet\circ}$.
	\item\label{lemma:gradedsets-condition-5} $\kappa_{c_1,c_2}+\kappa_{\overline{c_2},c_3}\subseteq \kappa_{c_1,c_3}$.
	\item\label{lemma:gradedsets-condition-7} $\kappa_{c_1,c_2}+\xi_{\overline{c_2},c_3}\subseteq \xi_{c_1,c_3}$.
\end{enumerate}
\end{axioms}
Let us first study how much $\kappa_{c_1,c_2}$ and $\xi_{c_1,c_2}$ depend on $c_1,c_2\in\colors$.

\begin{lemma} 
\label{lemma:omegas}
	For any $(\omega_{c_1,c_2})_{c_1,c_2\in \colors}\in \{(\kappa_{c_1,c_2})_{c_1,c_2\in \colors},(\xi_{c_1,c_2})_{c_1,c_2\in \colors}\}$:
	\begin{enumerate}[label=(\alph*)]
		\item\label{lemma:omegas-1} $\omega_{\circ\circ}=\omega_{\bullet\bullet}$ and $\omega_{\circ\circ}=-\omega_{\circ\circ}=\omega_{\circ\circ}+\sigma$.
		 \item\label{lemma:omegas-2} $\omega_{\circ\bullet}=\omega_{\bullet\circ}$ and $\omega_{\circ\bullet}=-\omega_{\circ\bullet}=\omega_{\circ\bullet}+\sigma$.
	\end{enumerate}
\end{lemma}
\begin{proof}
			Because $0\in \sigma$ by Assumption~\ref{lemma:gradedsets-condition-0}, the Assumption~\ref{lemma:gradedsets-condition-1} actually means
			\begin{align}
			 \label{eq:gradedsets-1-alternative}
			 	\omega_{c_1,c_2}=\omega_{c_1,c_2}+\sigma\tag{ii'}
			\end{align}
			 for all $c_1,c_2\in \colors$. And with this new identity we can, for all $c_1,c_2\in \colors$, refine  Assumption~\ref{lemma:gradedsets-condition-2} to 
			 \begin{align}
			 	\label{eq:gradedsets-2-alternative}
			 	\omega_{c_1,c_2}\subseteq -\omega_{c_2,c_1}\tag{iv'}
			 \end{align}
                         as $-\omega_{c_2,c_1}+\sigma=-(\omega_{c_2,c_1}-\sigma)=-(\omega_{c_2,c_1}+ \sigma)=-\omega_{c_2,c_1}$ due to $\sigma=-\sigma$.
	\begin{enumerate}[wide, label=(\alph*)]
		\item Version~\eqref{eq:gradedsets-1-alternative} of Assumption~\ref{lemma:gradedsets-condition-1} yields $\omega_{\circ\circ}=\omega_{\circ\circ}+\sigma$ as claimed. 
 And Assumption~\ref{lemma:gradedsets-condition-2} in the form of \eqref{eq:gradedsets-2-alternative} proves
		\begin{align*}
		\omega_{\circ\circ}\overset{\ref{lemma:gradedsets-condition-2}}{\subseteq }-\omega_{\circ\circ} \overset{\ref{lemma:gradedsets-condition-2}}{\subseteq }\omega_{\circ\circ} \quad \text{and}\quad        \omega_{\bullet\bullet}\overset{\ref{lemma:gradedsets-condition-2}}{\subseteq }-\omega_{\bullet\bullet} \overset{\ref{lemma:gradedsets-condition-2}}{\subseteq }\omega_{\bullet\bullet},
		\end{align*}
		thus verifying $\omega_{\circ\circ}=-\omega_{\circ\circ}$ and $\omega_{\bullet\bullet}=-\omega_{\bullet\bullet}$.
		Now, if we apply Assumption~\ref{lemma:gradedsets-condition-3}  to conclude
		\begin{align*}
		\omega_{\circ\circ}\overset{\ref{lemma:gradedsets-condition-3}}{\subseteq }-\omega_{\bullet\bullet}\overset{\ref{lemma:gradedsets-condition-3}}{\subseteq } \omega_{\circ\circ},
		\end{align*}
		we can infer $\omega_{\circ\circ}=\omega_{\bullet\bullet}$. 
		That proves the remainder of the claims about $\omega_{\circ\circ}$ and $\omega_{\bullet\bullet}$.
		\item Here also, Version~\eqref{eq:gradedsets-1-alternative} of Assumption~\ref{lemma:gradedsets-condition-1} implies $\omega_{\circ\bullet}=\omega_{\circ\bullet}+\sigma$. 
		Now, though, for $\omega_{\circ\bullet}$ and $\omega_{\bullet\circ}$ the roles of Assumptions~\ref{lemma:gradedsets-condition-3} and~\ref{lemma:gradedsets-condition-2} reverse. First, we apply the former to conclude
		\begin{align*}
		\omega_{\circ\bullet}\overset{\ref{lemma:gradedsets-condition-3}}{\subseteq }-\omega_{\circ\bullet}\overset{\ref{lemma:gradedsets-condition-3}}{\subseteq } \omega_{\circ\bullet}\quad\text{and}\quad        \omega_{\bullet\circ}\overset{\ref{lemma:gradedsets-condition-3}}{\subseteq }-\omega_{\bullet\circ}\overset{\ref{lemma:gradedsets-condition-3}}{\subseteq } \omega_{\bullet\circ},
		\end{align*}
		which shows the claims $\omega_{\circ\bullet}=-\omega_{\circ\bullet}$ and $\omega_{\bullet\circ}=-\omega_{\bullet\circ}$. Then, it is the refined version \eqref{eq:gradedsets-2-alternative} of As\-sump\-tion~\ref{lemma:gradedsets-condition-2} that yields
		\begin{align*}
		\omega_{\circ\bullet}\overset{\ref{lemma:gradedsets-condition-2}}{\subseteq }-\omega_{\bullet\circ}\overset{\ref{lemma:gradedsets-condition-2}}{\subseteq }\omega_{\circ\bullet},
		\end{align*}
		implying $\omega_{\circ\bullet}=\omega_{\bullet\circ}$ and thus completing the proof.\qedhere
	\end{enumerate}
\end{proof}
In the case of $(\omega_{c_1,c_2})_{c_1,c_2\in\colors}=(\xi_{c_1,c_2})_{c_1,c_2\in\colors}$ of Lem\-ma~\ref{lemma:omegas} we can go even further and combine the objects of Parts~\ref{lemma:omegas-1} and~\ref{lemma:omegas-2}.

\begin{lemma}
  \label{lemma:xi-well-defined}
	$\xi_{\circ\circ}=\xi_{\circ\bullet}$.
\end{lemma}
\begin{proof}
	Since $\xi_{c_2,\overline{c_1}}=\xi_{c_2,\overline{c_1}}+\sigma$ for all $c_1,c_2\in \colors$ by Version~\eqref{eq:gradedsets-1-alternative} of Axiom~\ref{lemma:gradedsets-condition-1}, our Assumption~\ref{lemma:gradedsets-condition-6} actually spells 
	\begin{align}
	\label{eq:gradedsets-6-alternative}
		\xi_{c_1,c_2}\subseteq \xi_{c_1,\overline{c_2}}\cup (-\xi_{c_2,\overline{c_1}})\tag{v'}
	\end{align}
	 for all $c_1,c_2\in \colors$ as $\sigma=- \sigma$. Using this version of the assumption twice, we conclude
	\begin{align*}
	\xi_{\circ\circ}\overset{\ref{lemma:gradedsets-condition-6}}\subseteq \xi_{\circ\bullet}\cup (-\xi_{\circ\bullet})=\xi_{\circ\bullet}\overset{\ref{lemma:gradedsets-condition-6}}{\subseteq}\xi_{\circ\circ}\cup(-\xi_{\bullet\bullet})=\xi_{\circ\circ},
	\end{align*}
	where we have used the results $\xi_{\circ\bullet}=-\xi_{\circ\bullet}$ and $\xi_{\circ\circ}=-\xi_{\bullet\bullet}$ of Lem\-ma~\ref{lemma:omegas}. It follows that  indeed  $\xi_{\circ\circ}=\xi_{\circ\bullet}$. 
\end{proof}

\begin{definition}
  \label{definition:xi}
  \label{definition:lambda-kappa}
  	Write $\lambda\eqpd \kappa_{\circ\circ}=\kappa_{\bullet\bullet}$ and $\kappa\eqpd \kappa_{\circ\bullet}=\kappa_{\bullet\circ}$ and
 $\xi\eqpd \xi_{\circ\circ}=\xi_{\bullet\bullet}=\xi_{\circ\bullet}=\xi_{\bullet\circ}$.
\end{definition}
Our next step is to show that the pair $(\lambda,\kappa)$ is of a very simple form (Lem\-ma~\ref{lemma:lambda-kappa-summary}).
\begin{definition}
  \label{definition:d-l}
	Define the non-negative integers
	\begin{align*}
	d\eqpd
	\begin{cases}
	\min\left(\kappa\cap \pint\right) &\text{if }\kappa\cap \pint\neq \emptyset,\\
	0&\text{otherwise},
	\end{cases}
	\quad\text{and}\quad
	l\eqpd
	\begin{cases}
	\min\left(\lambda\cap \pint\right) &\text{if }\lambda\cap \pint\neq \emptyset,\\
	0&\text{otherwise}.
	\end{cases}
	\end{align*}
\end{definition}
\begin{lemma}
\label{lemma:lambda-kappa}
\phantomsection
\begin{enumerate}[label=(\alph*)]
\item\label{lemma:lambda-kappa-1} $\kappa=d\integers$.
\item\label{lemma:lambda-kappa-2} If $\lambda\neq \emptyset$, then $l\in\lambda$ and $\lambda-l\supseteq \kappa$.
\item\label{lemma:lambda-kappa-3} $\lambda-l\subseteq \kappa$.
\item\label{lemma:lambda-kappa-4} If $\lambda\neq \emptyset$ and $d\neq 0$, then $l\leq d$.
\item\label{lemma:lambda-kappa-5} If $\lambda\neq \emptyset$ and $d\neq 0$, then $l\neq 0$.
\item\label{lemma:lambda-kappa-6} If $\lambda\neq \emptyset$, then $2l\integers\subseteq d\integers$.
\item\label{lemma:lambda-kappa-7} If $\lambda\neq \emptyset$, then $d=l$ or $d=2l$.
\end{enumerate}
\end{lemma}
\begin{proof}
	\begin{enumerate}[label=(\alph*),wide]
		\item Of course, $0\in \kappa$ by Assumption~\ref{lemma:gradedsets-condition-4}. And $-\kappa= \kappa$ was established in Lem\-ma~\hyperref[lemma:omegas-2]{\ref*{lemma:omegas}~\ref*{lemma:omegas-2}}. And with the choices  $c_1=\circ$, $c_2=c_3=\bullet$, Assumption~\ref{lemma:gradedsets-condition-5} implies  that
		\begin{align*}
		\kappa+\kappa=\kappa_{\circ\bullet}+\kappa_{\circ\bullet}\overset{\ref{lemma:gradedsets-condition-5}}{\subseteq} \kappa_{\circ\bullet}=\kappa.
		\end{align*}
		Hence, $\kappa$ is indeed a subgroup of $\integers$. The definition of $d$ makes $d$ a generator of $\kappa$, implying $\kappa=d\integers$. 
		\item 
		As $\lambda=-\lambda$ by Lem\-ma~\hyperref[lemma:omegas-1]{\ref*{lemma:omegas}~\ref*{lemma:omegas-1}}, assuming $\lambda\neq \emptyset$ ensures $\lambda\cap (\{0\}\cup\pint)\neq \emptyset$. Hence, under this assumption, $l\in \lambda$ by definition of $l$.
		If we choose $c_1=c_3=\circ$ and $c_2=\bullet$ in As\-sump\-tion~\ref{lemma:gradedsets-condition-5}, it follows that
		\begin{align*}
		\kappa+\lambda=\kappa_{\circ\bullet}+\kappa_{\circ\circ}\overset{\ref{lemma:gradedsets-condition-5}}{\subseteq} \kappa_{\circ\circ}=\lambda.
		\end{align*}
		Since $l\in \lambda$, we can specialize the $\lambda$ on the left hand side of that inclusion to $l$ and then subtract $l$ on both sides. We obtain $\kappa\subseteq \lambda-l$.
		\item If $\lambda= \emptyset$, there is nothing to prove. Hence, let $\lambda\neq \emptyset$, implying $l\in \lambda$ by Part~\ref{lemma:lambda-kappa-2}. Using Assumption~\ref{lemma:gradedsets-condition-5} once more, this time with the choices $c_1=c_2=\circ$ and $c_3=\bullet$, yields
		\begin{align*}
		\lambda-\lambda=\lambda+\lambda=\kappa_{\circ\circ}+\kappa_{\bullet\bullet}\overset{\ref{lemma:gradedsets-condition-5}}{\subseteq} \kappa_{\circ\bullet}=\kappa,
		\end{align*}
		where we have used  $\lambda=-\lambda$ (Lem\-ma~\hyperref[lemma:omegas-1]{\ref*{lemma:omegas}~\ref*{lemma:omegas-1}}) in the first step. Specializing on the left hand side the second instance of $\lambda$ to $l$  yields $\lambda-l\subseteq \kappa$.
		\item Actually, we show the contraposition. Hence, suppose $\lambda\neq \emptyset$ and  $l>d$. Since $\lambda=l+d\integers$ by Parts~\ref{lemma:lambda-kappa-1}--\ref{lemma:lambda-kappa-3}, it then follows that $l-d\in \lambda\cap \pint$. The definition of $l$ consequently requires $l\leq l-d$, i.e.\ $d\leq 0$. As $d\geq 0$ by definition, $d=0$ is the only possibility.
		\item We prove the contraposition indirectly. As $\lambda=l+d\integers$ by Parts~\ref{lemma:lambda-kappa-1}--\ref{lemma:lambda-kappa-3}, supposing $l=0$ entails $\lambda=d\integers$. Thus, if $d\neq 0$ were true, then $\emptyset\neq d\integers\cap \pint = \lambda\cap \pint$ would yield the contradiction $0<\min(\lambda\cap \pint)=l=0$ by definition of $l$.
		\item In the proof of Part~\ref{lemma:lambda-kappa-3} we saw $\lambda+\lambda\subseteq \kappa$. Specializing therein  both instances of $\lambda$ on the left hand side to $l$  (which we can do due to $\lambda\neq \emptyset$ by Part~\ref{lemma:lambda-kappa-2}) yields $2l\in \kappa=d\integers$. It follows $2l\integers\subseteq d\integers$ as asserted. 
		
		\item
		From $2l\integers\subseteq d\integers$, as shown in Part~\ref{lemma:lambda-kappa-6}, it is immediate that, if $d=0$, then $l=0=d$ as claimed. If $d\neq 0$, we know, firstly, $l\leq d$ by Part~\ref{lemma:lambda-kappa-4}, secondly, $l\neq 0$ by Part~\ref{lemma:lambda-kappa-5} and, thirdly, $2l\integers\subseteq d\integers$ by Part~\ref{lemma:lambda-kappa-6}. That is only possible if $d=l$ or $d=2l$: Indeed, if $c\in \integers$ is such that $2l=cd$, then $l> 0$ and $d\geq 0$ ensure $c> 0$. Moreover, $l\leq d$ implies $2l\leq 2d$, i.e., $cd\leq 2d$. We infer $c\leq 2$ by  $d> 0$. Hence, $c\in \{1,2\}$ by $c>0$.\qedhere
	\end{enumerate}
\end{proof}

\begin{lemma}
\label{lemma:lambda-kappa-summary}
	\begin{enumerate}[label=(\alph*)]
		\item\label{lemma:lambda-kappa-summary-1} If $\lambda=\emptyset$, then $(\lambda,\kappa)=(\emptyset,d\integers)$.
		\item\label{lemma:lambda-kappa-summary-2} If $\lambda\neq \emptyset$, then $(\lambda,\kappa)$ is equal to $(l\!+\!2l\integers,2l\integers)$ or $(l\integers,l\integers)$.
	\end{enumerate}
\end{lemma}
\begin{proof}
	In Lem\-ma~\ref{lemma:lambda-kappa} we established that $\kappa=d\integers$ (Part~\ref{lemma:lambda-kappa-1}) and that $\lambda=\emptyset$ or $\lambda=l+d\integers$ (Parts~\ref{lemma:lambda-kappa-2} and~\ref{lemma:lambda-kappa-3}), where $d=l$ or $d=2l$ (Part~\ref{lemma:lambda-kappa-7}). In other words, we have proven that $(\lambda,\kappa)$ is of the asserted form.
      \end{proof}
      We can immediately relate $\sigma$ to $\kappa$.
      \begin{definition}
        \label{definition:k}
	Define
	\begin{align*}
	k\eqpd \begin{cases}
	\min\left(\sigma\cap \pint\right) &\text{if }\sigma\cap \pint\neq \emptyset,\\
	0&\text{otherwise}.
	\end{cases}
	\end{align*}
\end{definition}

\begin{lemma}
	\label{lemma:sigma}
	$\sigma=k\integers\subseteq d\integers=\kappa$.
\end{lemma}
\begin{proof}
Because $\sigma$ is a subgroup of $\integers$, the definition of $k$ implies $\sigma=k\integers$. Moreover, we know $\kappa=\kappa+\sigma$ by Lem\-ma~\hyperref[lemma:omegas-2]{\ref*{lemma:omegas}~\ref*{lemma:omegas-2}}. Hence Assumption~\ref{lemma:gradedsets-condition-4}, namely  $0\in \kappa$, implies $k\integers=\sigma\subseteq \kappa+\sigma\subseteq \kappa=d\integers$. 
\end{proof}

Let us now turn to the description of $\xi$.
\begin{lemma}
\label{lemma:xi}
	\begin{enumerate}[label=(\alph*)]
		\item\label{lemma:xi-1} $\xi=\xi+d\integers$.
		\item\label{lemma:xi-2} If $\lambda\neq \emptyset$, then $\xi=\xi+l\integers$.
	\end{enumerate}
\end{lemma}
\begin{proof}
	\begin{enumerate}[label=(\alph*), wide]
	\item Picking $c_1=\circ$, $c_2=c_3=\bullet$, As\-sump\-tion~\ref{lemma:gradedsets-condition-7} implies the inclusion
	\begin{align*}
	\kappa+\xi=\kappa_{\circ\bullet}+\xi_{\circ\bullet}\overset{\ref{lemma:gradedsets-condition-7}}{\subseteq}\xi_{\circ\bullet}=\xi.
	\end{align*}
	As the reverse inclusion is trivially true by $0\in \kappa$ (Assumption~\ref{lemma:gradedsets-condition-4}), we have thus verified our claim $\xi=\xi+d\integers$ by Lem\-ma~\hyperref[lemma:lambda-kappa-1]{\ref*{lemma:lambda-kappa}~\ref*{lemma:lambda-kappa-1}}.
	\item
	Assumption~\ref{lemma:gradedsets-condition-7}, applied a second time, now with $c_1=c_2=c_3=\circ$, allows us to conclude
	\begin{align*}
	\lambda+\xi=\kappa_{\circ\circ}+\xi_{\bullet\circ}\overset{\ref{lemma:gradedsets-condition-7}}{\subseteq} \xi_{\circ\circ}=\xi.
	\end{align*}
	If $\lambda\neq \emptyset$, then  $l\in \lambda$ by Lem\-ma~\hyperref[lemma:lambda-kappa-2]{\ref*{lemma:lambda-kappa}~\ref*{lemma:lambda-kappa-2}}. Hence, the above inclusion shows in particular $\xi+l\subseteq \xi$. Using this, induction proves $\xi+l\pint\subseteq \xi$. Lem\-ma~\hyperref[lemma:lambda-kappa-7]{\ref*{lemma:lambda-kappa}~\ref*{lemma:lambda-kappa-7}} established that $d=l$ or $d=2l$. Either way,  $\xi=\xi+d\integers$, as seen in Part~\ref{lemma:xi-1}, then ensures $\xi-2l\subseteq \xi$. Combining this conclusion with $\xi+l\subseteq \xi$ lets us infer $\xi-l=(\xi+l)-2l\subseteq \xi$. Again, it follows $\xi-l\pint\subseteq \xi$ by induction. Hence, altogether we have shown $\xi+l\integers=(\xi-l\pint)\cup \xi\cup (\xi+l\pint)\subseteq \xi$. Of course, the converse inclusion is true as well because $0\in \integers$, proving $\xi=\xi+l\integers$ as claimed. \qedhere
	\end{enumerate}
\end{proof}
In order to obtain a refined understanding of $\xi$ we need the following preparatory lemma.
\begin{lemma}
	\label{lemma:chi}
	Let $\chi\subseteq \integers$ and $m\in  \pint$ satisfy $\chi=-\chi=\chi+m\integers$.
	\begin{enumerate}[label=(\alph*)]
		\item\label{lemma:chi-1} $\chi=(\chi\cap (\{0\}\cup \dwi{m\!-\!1}))_m$. 
		\item\label{lemma:chi-2} $\chi\cap \dwi{m\!-\! 1}=m-(\chi\cap \dwi{m\!-\!1})$.
		\item\label{lemma:chi-3} $\chi=(\chi\cap (\{0\}\cup \dwi{\lfloor\frac{m}{2}\rfloor}))_m$.
		\item\label{lemma:chi-4} $\chi=\integers\backslash D_m$ for $D=(\{0\}\cup \dwi{\lfloor\frac{m}{2}\rfloor})\backslash \chi$.
	\end{enumerate}
\end{lemma}
\begin{proof}
	The mapping $S\mapsto S_m\eqpd (S\cup (m-S))+m\integers$ of subsets $S\subseteq \integers$ is a closure operator with respect to $\subseteq$, i.e.,  for all $S,T\subseteq \integers$ with $S\subseteq T$ we have $S\subseteq S_m$ and $S_m\subseteq T_m$ and $(S_m)_m=S_m$. In particular $S=S_m$ if and only if $S=-S=S+m\integers$.
	\begin{enumerate}[label=(\alph*), wide]
		\item The assumption $\chi=-\chi=\chi+m\integers$ implies $\chi=\chi_m$.  Hence,  $\chi=\chi_m\supseteq (\chi\cap (\{0\}\cup \dwi{m\!-\!1}))_m$ is clear by monotonicity of $S\mapsto S_m$. We show the converse: If $x\in \chi$, we find $x'\in \{0\}\cup\dwi{m\!-\!1}$ such that $x'-x\in m\integers$. Consequently, $x'\in x+m\integers\subseteq \chi+m\integers\subseteq \chi$ by assumption. We conclude $x\in x'+m\integers\subseteq (\chi\cap (\{0\}\cup \dwi{m\!-\!1})) +m\integers\subseteq (\chi\cap (\{0\}\cup \dwi{m\!-\!1}))_m$, which is what we needed to show.
		\item We further deduce from $\chi=-\chi=\chi+m\integers$ that $m-\chi\subseteq \chi$. Naturally, $m-\left(\chi\cap \dwi{m\!-\!1}\right)\subseteq m-\dwi{m\!-\!1} =\dwi{m\!-\!1}$. Combining this with $m-\left(\chi\cap \dwi{m\!-\!1}\right)\subseteq m-\chi\subseteq \chi$ yields $m-\left(\chi\cap\dwi{m\!-\!1}\right)\subseteq \chi\cap \dwi{m\!-\!1}$. We conclude $\chi\cap\dwi{m\!-\!1}=m-\left(m-\left(\chi\cap\dwi{m\!-\!1}\right)\right)\subseteq m-\left(\chi\cap\dwi{m\!-\!1}\right)$, which proves one inclusion.
		\par
		Now, the converse. From $\chi=-\chi=\chi+m\integers$ we can infer $m-\chi=-(m-\chi)=(m-\chi)+m\integers$. In consequence we can apply the inclusion we just proved to the set $m-\chi$ in the role of $\chi$. Since $m-\dwi{m\!-\!1}=\dwi{m\!-\!1}$, the resulting inclusion $(m-\chi)\cap \dwi{m\!-\!1}\subseteq m-\left((m-\chi)\cap\dwi{m\!-\!1}\right)$ actually spells $m-(\chi\cap \dwi{m\!-\!1})\subseteq \chi\cap \dwi{m\!-\!1}$. That is just what we had to show.
		\item Due to the monotonicity and idempotency of the mapping $S\mapsto S_m$, it suffices by Part~\ref{lemma:chi-1} to prove $\chi\cap (\{0\}\cup \dwi{m\!-\!1})\subseteq (\chi\cap (\{0\}\cup \dwi{\lfloor\frac{m}{2}\rfloor}))_m$. Let $x\in \chi\cap (\{0\}\cup \dwi{m\!-\!1})$ be arbitrary. If $x\leq \lfloor\frac{m}{2}\rfloor$, then, naturally, $x \in \chi\cup (\{0\}\cup \dwi{\lfloor\frac{m}{2}\rfloor})\subseteq (\chi\cap (\{0\}\cup \dwi{\lfloor\frac{m}{2}\rfloor}))_m$. Hence, we can assume $x>\lfloor\frac{m}{2}\rfloor$. By Part~\ref{lemma:chi-2} we know $m-x\in \chi$. By assumption, $m-x<m-\lfloor\frac{m}{2}\rfloor$. If $m$ is even, then this inequality says $m-x<m-\frac{m}{2}=\frac{m}{2}=\lfloor\frac{m}{2}\rfloor$. Should $m$ be odd instead, it means $m-x<m-\frac{m-1}{2}=\frac{m+1}{2}$, which implies $m-x\leq \frac{m+1}{2}-1=\frac{m-1}{2}=\lfloor\frac{m}{2}\rfloor$. Thus, $m-x\leq \lfloor\frac{m}{2}\rfloor$ in all cases. Hence we have shown $m-x\in \chi\cap (\{0\}\cup \dwi{\lfloor \frac{m}{2}\rfloor})$. It follows $x=m-(m-x)\in m-(\chi\cap (\{0\}\cup \dwi{\lfloor \frac{m}{2}\rfloor}))\subseteq (\chi(\{0\}\cup \dwi{\lfloor \frac{m}{2}\rfloor}))_m$. That is what we needed to see.
		\item The assumption $\chi=-\chi=\chi+m\integers$ implies $\integers\backslash \chi=-(\integers\backslash \chi )=(\integers\backslash \chi)+m\integers$. Hence, we can apply Part~\ref{lemma:chi-3} to the set $\integers\backslash \chi$ in the role of $\chi$ and obtain $\integers\backslash \chi=((\integers\backslash \chi)\cap (\{0\}\cup \dwi{\lfloor\frac{m}{2}\rfloor}))_m$. Since $(\integers\backslash \chi)\cap (\{0\}\cup \dwi{\lfloor\frac{m}{2}\rfloor})=(\{0\}\cup \dwi{\lfloor\frac{m}{2}\rfloor})\backslash \chi=D$ we have shown $\integers\backslash \chi=D_m$. It follows $\chi=\integers\backslash D_m$ as claimed. \qedhere
	\end{enumerate}
\end{proof}
\begin{lemma}
	\label{lemma:xi-summary}
	\begin{enumerate}[label=(\alph*)]
		\item\label{lemma:xi-summary-1} If $d=0$, then $\xi=\integers\backslash E_0$ for $E=(\{0\}\cup\pint)\backslash \xi$.
		\item\label{lemma:xi-summary-2} If $d\geq 1$ and $\lambda\neq \emptyset$, then $\xi=\integers\backslash D_l$ for $D=(\{0\}\cup \dwi{\lfloor\frac{l}{2}\rfloor})\backslash \xi$.
		\item\label{lemma:xi-summary-3} If $d\geq 1$ and $\lambda= \emptyset$, then $\xi=\integers\backslash D_d$ for $D=(\{0\}\cup \dwi{\lfloor\frac{d}{2}\rfloor})\backslash \xi$. 		
	\end{enumerate}
\end{lemma}
\begin{proof}
\begin{enumerate}[label=(\alph*),wide]
\item 
The defining equations $E=(\{0\}\cup \pint)\backslash \xi$ and $E_0=E\cup (-E)$ imply $E_0=((\{0\}\cup \pint) \backslash \xi)\cup ((-(\{0\}\cup \pint))\backslash (-\xi))$.  Hence, $\xi=-\xi$ (by Lem\-ma~\ref{lemma:omegas}) shows $E_0=\integers\backslash \xi$ and thus the claim $\xi=\integers\backslash E_0$. 
\item Because $\lambda \neq \emptyset$,  Lem\-ma~\hyperref[lemma:lambda-kappa-7]{\ref*{lemma:lambda-kappa}~\ref*{lemma:lambda-kappa-7}} guarantees $d=l$ or $d=2l$. Hence, the assumption $d\geq 1$ implies $l\geq 1$. Moreover, Lem\-ma~\hyperref[lemma:xi-2]{\ref*{lemma:xi}~\ref*{lemma:xi-2}} assures us that $\xi=\xi+l\integers$. And, we already know $\xi=-\xi$ by Lem\-ma~\ref{lemma:omegas}. Hence, Lem\-ma~\hyperref[lemma:chi-4]{\ref*{lemma:chi}~\ref*{lemma:chi-4}} yields the claim. 
\item Still, $\xi=-\xi$, of course. And $\xi=\xi+d\integers$ by Lem\-ma~\hyperref[lemma:xi-1]{\ref*{lemma:xi}~\ref*{lemma:xi-1}} as $d\geq 1$. Thus, once more, Lem\-ma~\hyperref[lemma:chi-4]{\ref*{lemma:chi}~\ref*{lemma:chi-4}} proves the claim. \qedhere
\end{enumerate}
\end{proof}
In conclusion we have shown the following auxiliary result.
\begin{lemma}[Arithmetic Lem\-ma]
  \label{lemma:arithmetic}
  If the nine sets of integers $\sigma$ and  $\kappa_{c_1,c_2}$, $\xi_{c_1,c_2}$ for $c_1,c_2\in \colors$  satisfy Axioms~\ref{axioms:arithmetic}, then
  \begin{align*}
     \kappa_{\circ\circ}=\kappa_{\bullet\bullet}=\colon \lambda,\quad   \kappa_{\circ\bullet}=\kappa_{\bullet\circ}=\colon \kappa \quad\text{and}\quad  \xi_{\circ\circ}=\xi_{\bullet\bullet}=\xi_{\circ\bullet}=\xi_{\bullet\circ}=\colon \xi
  \end{align*}
  and there exist  $u\in\{0\}\cup \pint$, $m\in \pint$, $D\subseteq \{0\}\cup\dwi{\lfloor\frac{m}{2}\rfloor}$ and $E\subseteq \{0\}\cup \pint$ such that the tuple $(\sigma,\lambda,\kappa,\xi)$ is given by one of the following:
	\begin{align*}
	\begin{matrix}
	\sigma& \lambda&\kappa& \xi  \\ \hline \\[-0.85em]
	um\integers & m\integers & m\integers & \integers\backslash D_m\\
	2um\integers & m\!+\!2m\integers & 2m\integers & \integers\backslash D_m\\
	um\integers & \emptyset & m\integers & \integers\backslash D_m\\                        
	\{0\} & \{0\} & \{0\} & \integers\backslash E_0 \\
	\{0\} & \emptyset & \{0\} & \integers\backslash E_0
	\end{matrix}
	\end{align*}
      \end{lemma}
      \begin{proof}
        That $\lambda$, $\kappa$ and $\xi$ are well-defined was shown in Lem\-mata~\ref{lemma:omegas} and~\ref{lemma:xi-well-defined}. Hence, we can let $k$, $d$ and $l$ be as in Definitions~\ref{definition:k} and  \ref{definition:d-l}. We distinguish five cases in total. 
        \par
        \textbf{Case~1:} First, suppose that $\lambda=\emptyset$. Then, $\kappa=d\integers$. By Lem\-ma~\hyperref[lemma:lambda-kappa-summary-1]{\ref*{lemma:lambda-kappa-summary}~\ref*{lemma:lambda-kappa-summary-1}}. There are now two possibilities depending on the value of $d\in\{0\}\cup\pint$.
        \par
        \textbf{Case~1.1:} If $d=0$, which is to say $\kappa=\{0\}$, then Lem\-ma~\hyperref[lemma:xi-summary-1]{\ref*{lemma:xi-summary}~\ref*{lemma:xi-summary-1}} yields $\xi=\integers\backslash E_0$ for $E\eqpd(\{0\}\cup \pint)\backslash \xi$. And Lem\-ma~\ref{lemma:sigma} proves $\sigma=k\integers\subseteq  d\integers=\{0\}$, implying $k=0$ and thus $\sigma=\{0\}$. As, naturally, $E\subseteq \{0\}\cup\pint$, the tuple $(\sigma,\lambda,\kappa,\xi)$ is indeed as claimed in the fifth row of the table.
        \par
        \textbf{Case~1.2:} Should $d\geq 1$ on the other hand, then by Lem\-ma~\hyperref[lemma:xi-summary-3]{\ref*{lemma:xi-summary}~\ref*{lemma:xi-summary-3}} we infer $\xi=\integers\backslash D_d$ for $D\eqpd(\{0\}\cup \dwi{\lfloor\frac{d}{2}\rfloor})\backslash \xi$. Since $\sigma=k\integers\subseteq  d\integers$ by Lem\-ma~\ref{lemma:sigma}, if we put $u\eqpd\frac{k}{d}$, then $\sigma=ud\integers$. Recognizing $D\subseteq \{0\}\cup \dwi{\lfloor\frac{d}{2}\rfloor}$ and defining $m\eqpd d$ thus  proves that $(\sigma,\lambda,\kappa,\xi)$ is as asserted by the third row of the table.
        \par
        \textbf{Case~2:} Now, let $\lambda\neq \emptyset$ instead. Then, $(\lambda,\kappa)=(l\!+\!2l\integers,2l\integers)$ or $(\lambda,\kappa)=(l\integers,l\integers)$ by Lem\-ma~\hyperref[lemma:lambda-kappa-summary-1]{\ref*{lemma:lambda-kappa-summary}~\ref*{lemma:lambda-kappa-summary-1}}. Respectively, $d=2l$ or $d=l$. We now distinguish two cases based on the value of $l\in\{0\}\cup \pint$.
        \par
        \textbf{Case~2.1:} Assuming $l=0$ lets us conclude $l\integers=2l\integers=l\!+\!2l\integers=\{0\}$, which implies  $(\lambda,\kappa)=(\{0\},\{0\})$. Lem\-ma~\ref{lemma:sigma} gives $\sigma=k\integers\subseteq \kappa=\{0\}$ and thus $k=0$ and $\sigma=\{0\}$. Because $d=l=2l=0$ we can infer 
        $\xi=\integers\backslash E_0$ for $E\eqpd(\{0\}\cup \pint)\backslash \xi$ by Lem\-ma~\hyperref[lemma:xi-summary-1]{\ref*{lemma:xi-summary}~\ref*{lemma:xi-summary-1}}. As $E\subseteq \{0\}\cup\pint$, the tuple $(\sigma,\lambda,\kappa,\xi)$ is hence given by the fourth row of the table.
        \par
        \textbf{Case~2.2:} Finally, let $l\geq 0$. Then, also $d\geq 0$, no matter whether $d=l$ or $d=2l$. In conclusion,  $\xi=\integers\backslash D_l$ for $D\eqpd(\{0\}\cup \dwi{\lfloor\frac{l}{2}\rfloor})\backslash \xi$ by Lem\-ma~\hyperref[lemma:xi-summary-3]{\ref*{lemma:xi-summary}~\ref*{lemma:xi-summary-3}}. 
        \par
        \textbf{Case~2.2.1:} If $(\lambda,\kappa)=(l\!+\!2l\integers,2l\integers)$, i.e., $d=2l$, then the implication $\sigma=k\integers\subseteq d\integers=2l\integers$ of Lem\-ma~\ref{lemma:sigma} lets us define $u\in\{0\}\cup \pint$ by $u\eqpd \frac{k}{2l}$ and obtain $\sigma=2ul\integers$. Hence, choosing $m\eqpd l$ proves that $(\sigma,\lambda,\kappa,\xi)$ fits the second row of the table.
        \par
        \textbf{Case~2.2.2:} If instead,  $(\lambda,\kappa)=(l\integers,l\integers)$, i.e., $d=l$, then Lem\-ma~\ref{lemma:sigma} yields $\sigma=k\integers\subseteq d\integers=l\integers$, thus permitting us to define $u\in\{0\}\cup \pint$ by $u\eqpd \frac{k}{l}$ and obtain $\sigma=ul\integers$. The choice $m\eqpd l$ hence shows $(\sigma,\lambda,\kappa,\xi)$ to be given by the first row.
      \end{proof}
As mentioned before, our goal will be to show (Sec\-tion~\ref{section:verifying-the-axioms}) that for every non-hy\-per\-oc\-ta\-he\-dral category $\mc C\subseteq \Cp$ the tuple $(\toco,L,K,X)(\mc C)$ is of the form given in the table of the \hyperref[lemma:arithmetic]{Arithmetic Lem\-ma}.
\subsection{Reduction to Singleton and Pair Blocks} Let us return to categories of partitions. To elucidate the ranges of $K$, $L$ and $X$ over $\nhoc$ and central relations between $\toco(\mc C)$, $K(\mc C)$, $L(\mc C)$ and $X(\mc C)$ for non-hy\-per\-octa\-he\-dral categories $\mc C\subseteq \Cp$,  we must consider certain decompositions of $K$, $L$ and $X$ according to leg colors.
\begin{definition}
	Let $\mc S\subseteq \Cp$ and $c_1,c_2\in\colors$ be abitrary. Then, define
	\begin{align}
	K_{c_1,c_2}(\mc S)\eqpd \{\,\delta_p(\alpha_1,\alpha_2)\mid &\;p\in \mc S, \, B\text{ block of }p,\, \alpha_1,\alpha_2\in B,\, \alpha_1\neq \alpha_2,\tag{a}\\
	&\; ]\alpha_1,\alpha_2[_p\cap B=\emptyset,\, \forall i=1,2: \alpha_i\text{ of normalized color }c_i\},\notag\\
	X_{c_1,c_2}(\mc S)\eqpd \{\,\delta_p(\alpha_1,\alpha_2) \mid &\; p\in \mc S,\, B_1,B_2\text{ blocks of }p,\, B_1 \text{ and }B_2 \text{ cross}, \tag{b}\\
	&\; \alpha_1\in B_1,\,\alpha_2\in B_2,\, \forall i=1,2: \alpha_i\text{ of normalized color }c_i\}.\notag
	\end{align}
      \end{definition}
 $L$, $K$ and $X$ can then be written as, where the union occurs pointwise,      
      \begin{IEEEeqnarray*}{rCl}
        L=\bigcup_{\substack{c_1,c_2\in \colors\\c_1= c_2}} K_{c_1,c_2},\quad K=\bigcup_{\substack{c_1,c_2\in \colors\\c_1\neq c_2}} K_{c_1,c_2},\quad\text{and}\quad X=\bigcup_{c_1,c_2\in \colors}X_{c_1,c_2}.
      \end{IEEEeqnarray*}

      \par
 Recall that $\Cp_{\leq 2}$ denotes the set of all partitions with block sizes one or two and that it is a category (see \cite[Lem\-ma~4.4~(a)]{MWNHO1}). By the next lemma we may always restrict to partitions in $\Cp_{\leq 2}$ when studying $K_{c_1,c_2}$ and $X_{c_1,c_2}$. This is trivial in cases~$\mc O$ and~$\mc B$, while for case~$\mc S$ this basically follows from Lem\-ma~\hyperref[lemma:singletons-3]{\ref*{lemma:singletons}~\ref*{lemma:singletons-3}}.
 \begin{lemma}
   \label{lemma:simplification-k-x}
   For all non-hy\-per\-octa\-he\-dral categories $\mc C\subseteq \Cp$ and  $c_1,c_2\in \colors$:
   \begin{enumerate}[label=(\alph*)]
   \item $K_{c_1,c_2}(\mc C)=K_{c_1,c_2}(\mc C\cap\Cp_{\leq 2})$.
   \item $X_{c_1,c_2}(\mc C)=X_{c_1,c_2}(\mc C\cap\Cp_{\leq 2})$.
   \end{enumerate}
 \end{lemma}
 \begin{proof}
   \begin{enumerate}[label=(\alph*),wide]
   \item If $\mc C$ is case~$\mc O$ or case~$\mc B$, i.e., if $\mc C\subseteq \Cp_{\leq 2}$ by Proposition~\ref{proposition:result-F}, there is nothing to show. Hence, suppose that $\mc C$ is case~$\mc S$ and let $c_1,c_2\in \colors$. We only need to prove $K_{c_1,c_2}(\mc C)\subseteq K_{c_1,c_2}(\mc C\cap \Cp_{\leq 2})$. Let $\alpha_1$ and $\alpha_2$ with $\alpha_1\neq \alpha_2$ be points in $p\in \mc C$ such that $\alpha_i$ is of normalized color $c_i$ for every $i\in \{1,2\}$ and such that $\alpha_1,\alpha_2\in B$ and $]\alpha_1,\alpha_2[_p\cap B=\emptyset$ for some block $B$ in $p$. Because $\mc C$ is case~$\mc S$, by Lem\-ma~\hyperref[lemma:singletons-3]{\ref*{lemma:singletons}~\ref*{lemma:singletons-3}} we do not violate the assumption $p\in \mc C$ by assuming that every block other than $B$ is a singleton. In the same way we can assume that $\alpha_1$ and $\alpha_2$ are the only legs of $B$. None of these assumptions affect $\delta_p(\alpha_1,\alpha_2)$ or the normalized colors of $\alpha_1$ or $\alpha_2$. As they ensure $p\in \mc C\cap \Cp_{\leq 2}$ though, we have shown $\delta_p(\alpha_1,\alpha_2)\in K_{c_1,c_2}(\mc C\cap \Cp_{\leq 2})$, which is what we needed to see.
   \item  Again, all that we need to prove is that $X_{c_1,c_2}(\mc C)\subseteq X_{c_1,c_2}(\mc C\cap \Cp_{\leq 2})$ if $\mc C$ is case~$\mc S$ and if $c_1,c_2\in \colors$.  Let the points $\alpha_1$ of normalized color $c_1$ and $\alpha_2$ of normalized color $c_2$ in  $p\in \mc C$ belong to the blocks $B_1$ and $B_2$, respectively, and suppose that $B_1$ and $B_2$ cross. Because $\mc C$ is case~$\mc S$ we can, by Lem\-ma~\hyperref[lemma:singletons-3]{\ref*{lemma:singletons}~\ref*{lemma:singletons-3}}, assume that all other blocks of $p$ besides $B_1$ and $B_2$ are singletons. Now the only thing standing in the way of $p\in \mc C\cap \Cp_{\leq 2}$ is the possibility of at least one of $B_1$ and $B_2$ having more than two legs. We would like to assume that $B_1$ and $B_2$ have only two legs each and still maintain all the other assumptions including $\alpha_1\in B_1$ and $\alpha_2\in B_2$ and, of course, not alter  $\delta_p(\alpha_1,\alpha_2)$. 
      By Lem\-ma~\hyperref[lemma:singletons-3]{\ref*{lemma:singletons}~\ref*{lemma:singletons-3}}, we can always  remove surplus legs of $B_1$ and $B_2$. But it is not immediately clear that we can remove legs without affecting the other assumptions. A priori, the crossing between $B_1$ and $B_2$ only implies that we can find points $\beta_1,\gamma_1\in B_1$ and $\beta_2,\gamma_2\in B_2$ such that $(\beta_1,\beta_2,\gamma_1,\gamma_2)$ is ordered in $p$. If now $\alpha_1\in \{\beta_1,\gamma_1\}$ and $\alpha_2\in \{\beta_2,\gamma_2\}$, then we can certainly remove all legs except $\{\beta_i,\gamma_i\}$ from $B_i$ for all $i\in \{1,2\}$ and still maintain the other assumptions. In fact, we can do so in general as well:
      \par
      Let us only consider the \enquote{worst case} that $\alpha_1\notin \{\beta_1,\gamma_1\}$ and $\alpha_2\notin\{\beta_2,\gamma_2\}$. There are $20$ possible arrangements of the points $\{\alpha_1,\beta_1,\gamma_1,\alpha_2,\beta_2,\gamma_2\}$ relative to each other with respect to the cyclic order respecting that $(\beta_1,\beta_2,\gamma_1,\gamma_2)$ is ordered.
  \begin{align*} 
    \begin{array}{ c | c | c | c }
      \overset{\downarrow}{\alpha_1}\beta_1\beta_2\gamma_1\gamma_2 & \beta_1\overset{\downarrow}{\alpha_1}\beta_2\gamma_1\gamma_2 &
      \beta_1\beta_2\overset{\downarrow}{\alpha_1}\gamma_1\gamma_2 &
      \beta_1\beta_2\gamma_1\overset{\downarrow}{\alpha_1}\gamma_2\\ [0.1em]\hline\Tstrut
            \underline{\alpha_2\alpha_1}\beta_1\underline{\beta_2\gamma_1}\gamma_2 & \underline{\alpha_2}\beta_1\underline{\alpha_1\beta_2\gamma_1}\gamma_2 &
      \underline{\alpha_2\beta_1\beta_2\alpha_1}\gamma_1\gamma_2 &
 \underline{\alpha_2\beta_1\beta_2}\gamma_1\underline{\alpha_1}\gamma_2\\
      \underline{\alpha_1\alpha_2\beta_1\beta_2}\gamma_1\gamma_2 &
 \beta_1\underline{\alpha_2\alpha_1\beta_2\gamma_1}\gamma_2 &
            \underline{\beta_1\alpha_2}\beta_2\underline{\alpha_1}\gamma_1\underline{\gamma_2} &
 \underline{\beta_1\alpha_2}\beta_2\gamma_1\underline{\alpha_1\gamma_2} \\[0.1em]
            \underline{\alpha_1}\beta_1\underline{\alpha_2}\beta_2\underline{\gamma_1\gamma_2} &
 \beta_1\underline{\alpha_1\alpha_2}\beta_2\underline{\gamma_1\gamma_2} &
            \underline{\beta_1}\beta_2\underline{\alpha_2\alpha_1}\gamma_1\underline{\gamma_2} &
 \underline{\beta_1}\beta_2\underline{\alpha_2}\gamma_1\underline{\alpha_1\gamma_2} \\[0.1em]
        \underline{\alpha_1}\beta_1\beta_2\underline{\alpha_2\gamma_1\gamma_2} &
 \beta_1\underline{\alpha_1}\beta_2\underline{\alpha_2\gamma_1\gamma_2} &
            \beta_1\beta_2\underline{\alpha_1\alpha_2\gamma_1\gamma_2} &
 \underline{\beta_1}\beta_2\gamma_1\underline{\alpha_2\alpha_1\gamma_2} \\[0.1em]
              \underline{\alpha_1}\beta_1\underline{\beta_2\gamma_1\alpha_2}\gamma_2 &
 \beta_1\underline{\alpha_1\beta_2\gamma_1\alpha_2}\gamma_2 &
            \underline{\beta_1\beta_2\alpha_1}\gamma_1\underline{\alpha_2}\gamma_2 &
\underline{\beta_1\beta_2}\gamma_1\underline{\alpha_1\alpha_2}\gamma_2
    \end{array}
  \end{align*}
We remove all legs of $B_1$ and $B_2$ except for the underlined ones. Then the above table shows that we can always turn $B_1$ and $B_2$ into crossing pair blocks containing $\alpha_1$ and $\alpha_2$, respectively. That concludes the proof.\qedhere
   \end{enumerate}
 \end{proof}

 \subsection{Verifying the Axioms}
 \label{section:verifying-the-axioms}
We want to apply the \hyperref[lemma:arithmetic]{Arithmetic Lem\-ma~\ref*{lemma:arithmetic}} to the sets $\sigma\eqpd \toco(\mc C)$, $\kappa_{c_1,c_2}\eqpd K_{c_1,c_2}(\mc C)$ and $\xi_{c_1,c_2}\eqpd X_{c_1,c_2}(\mc C)$ for $c_1,c_2\in \colors$ and non-hy\-per\-octa\-he\-dral categories $\mc C\subseteq \Cp$. In order to be able to do so, we, of course, need to show that these sets actually satisfy the prerequisite Axioms~\ref{axioms:arithmetic}. Proving that will crucially utilize the reduction to singleton and pair blocks from Lem\-ma~\ref{lemma:simplification-k-x}.

\begin{lemma}
  \label{lemma:verifying-axioms-1}
  For every non-hy\-per\-octa\-he\-dral category $\mc C\subseteq \Cp$, the set $\sigma\eqpd \toco(\mc C)$ satisfies Axiom~\ref{lemma:gradedsets-condition-0} of \ref{axioms:arithmetic}: $\sigma$ is a subgroup of $\integers$.
\end{lemma}
\begin{proof}
  That was shown in Proposition~\ref{proposition:result-S}.
\end{proof}

\begin{lemma}
  \label{lemma:verifying-axioms-2}
  For every non-hy\-per\-octa\-he\-dral category $\mc C\subseteq \Cp$, the sets  $\sigma\eqpd \toco(\mc C)$ and $\kappa_{c_1,c_2}\eqpd K_{c_1,c_2}(\mc C)$ for $c_1,c_2\in \colors$ satisfy Axioms~\ref{lemma:gradedsets-condition-1}--\ref{lemma:gradedsets-condition-2} of \ref{axioms:arithmetic}:
  \begin{IEEEeqnarray*}{rCCrCCrC}
   \text{(ii)}\hspace{0.5em}& \kappa_{c_1,c_2}+\sigma\subseteq \kappa_{c_1,c_2}, &\hspace{1.5em} &\text{(iii)}\hspace{0.5em}& \kappa_{c_1,c_2}\subseteq - \kappa_{\overline{c_2},\overline{c_1}}, &\hspace{1.5em} &\text{(iv)}\hspace{0.5em}& \kappa_{c_1,c_2}\subseteq - \kappa_{c_2,c_1}+\sigma
 \end{IEEEeqnarray*}
 for all $c_1,c_2\in\colors$.
\end{lemma}
\begin{proof}
  Let $c_1,c_2\in \colors$ be arbitrary and let $\alpha_1$ and $\alpha_2$ be distinct points of the same block $B$ in $p\in \mc C$ such that $]\alpha_1,\alpha_2[_p\cap B=\emptyset$ and such that $\alpha_i$ has normalized color $c_i$ for every $i\in \{1,2\}$. In other words, let  $\delta_{p}(\alpha_1,\alpha_2)$ be a generic element of $K_{c_1,c_2}(\mc C)=\kappa_{c_1,c_2}$.
                    \par
                    \emph{Axiom~\ref{lemma:gradedsets-condition-1}:} Let $q\in \mc C$ be arbitrary. None of the assumptions about $p$, $\alpha_1$, $\alpha_2$ and $\delta_p(\alpha_2,\alpha_2)$ are impacted by assuming that $p$ is rotated in such a way that $\alpha_1$ is the rightmost lower point of $p$. Then, $B$ is a block of $p\otimes q\in \mc C$ as well and $]\alpha_1,\alpha_2[_{p\otimes q}\cap B=\emptyset$.
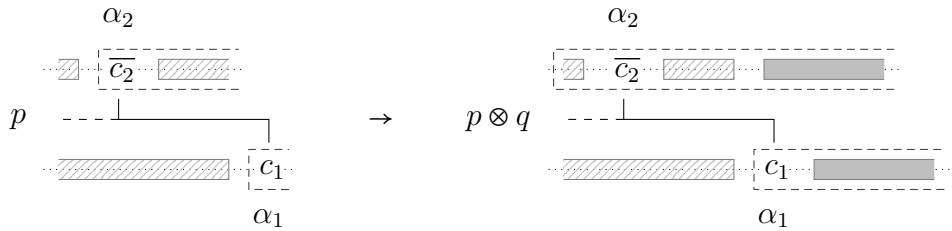
   \begin{mycenter}[0.5em]
\begin{tikzpicture}[baseline=0.666*1cm-0.25em]
    \def\scp{0.666}
    \def\linksize{\scp*0.075cm}
    \def\pointsize{\scp*0.25cm}
    \def\dd{\scp*0.5cm}
    \def\dx{\scp*1cm}
    \def\cx{\scp*0.3cm}
    \def\txu{3*\dx}    
    \def\txl{4*\dx}
    \def\dy{\scp*1cm}
    \def\cy{\scp*0.3cm}
    \def\ty{2*\dy}
    \def\fy{\scp*0.2cm}
    \def\fx{\scp*0.2cm}
    \def\gy{\scp*0.4cm}
    \def\gx{\scp*0.4cm}      
    \tikzset{whp/.style={circle, inner sep=0pt, text width={\pointsize}, draw=black, fill=white}}
    \tikzset{blp/.style={circle, inner sep=0pt, text width={\pointsize}, draw=black, fill=black}}
    \tikzset{lk/.style={regular polygon, regular polygon sides=4, inner sep=0pt, text width={\linksize}, draw=black, fill=black}}
    \tikzset{vp/.style={circle, inner sep=0pt, text width={1.5*\pointsize}, fill=white}}
    \tikzset{sstr/.style={shorten <= 5pt, shorten >= 5pt}}    
    \draw[dotted] ({0-\dd},{0}) -- ({\txl+\dd},{0});
    \draw[dotted] ({0-\dd},{\ty}) -- ({\txu+\dd},{\ty});
    \node[vp] (l1) at ({0+4*\dx},{0+0*\ty}) {$c_1$};
    \node[vp] (u1) at ({0+1*\dx},{0+1*\ty}) {$\overline{c_2}$};    
    \draw[sstr] (l1) --++(0,{0.5*\ty}) -| (u1);
    \draw[dashed] ($(u1)+(0,{-0.5*\ty})$) -- ++ ({-1.25*\dx},0);
    \draw [draw=gray, pattern = north east lines, pattern color = lightgray] ({0+0*\dx-\fx},{0+0*\dy-\fy}) -- ++ ({3*\dx+2*\fx},0) |- ({0+0*\dx-\fx},{0+0*\dy+\fy});
    \draw [draw=gray, pattern = north east lines, pattern color = lightgray] ({0+0*\dx-\fx},{\ty+0*\dy-\fy}) -- ++ ({0*\dx+2*\fx},0) |- ({0+0*\dx-\fx},{\ty+0*\dy+\fy});
    \draw [draw=gray, pattern = north east lines, pattern color = lightgray] ({\txu+0*\dx+\fx},{\ty+0*\dy-\fy}) -- ++ ({-1*\dx-2*\fx},0) |- ({\txu+0*\dx+\fx},{\ty+0*\dy+\fy});
    \node [below = {\cx} of  l1] {$\alpha_1$};
    \node [above ={\cx} of  u1] {$\alpha_2$};
    \draw [draw=darkgray, densely dashed] ({\txu+0*\dx+\gx},{\ty+0*\dy-\gy}) -- ++ ({-2*\dx-2*\gx},0) |- ({\txu+0*\dx+\gx},{\ty+0*\dy+\gy});
    \draw [draw=darkgray, densely dashed] ({\txl+0*\dx+\gx},{0+0*\dy-\gy}) -- ++ ({-0*\dx-2*\gx},0) |- ({\txl+0*\dx+\gx},{0+0*\dy+\gy});
    \node at (-\dx,{0.5*\ty}) {$p$};
  \end{tikzpicture}
\qquad$\to$\qquad
\begin{tikzpicture}[baseline=0.666*1cm-0.25em]
    \def\scp{0.666}
    \def\linksize{\scp*0.075cm}
    \def\pointsize{\scp*0.25cm}
    \def\dd{\scp*0.5cm}
    \def\dx{\scp*1cm}
    \def\cx{\scp*0.3cm}
    \def\txu{6*\dx}    
    \def\txl{7*\dx}
    \def\dy{\scp*1cm}
    \def\cy{\scp*0.3cm}
    \def\ty{2*\dy}
    \def\fy{\scp*0.2cm}
    \def\fx{\scp*0.2cm}
    \def\gy{\scp*0.4cm}
    \def\gx{\scp*0.4cm}      
    \tikzset{whp/.style={circle, inner sep=0pt, text width={\pointsize}, draw=black, fill=white}}
    \tikzset{blp/.style={circle, inner sep=0pt, text width={\pointsize}, draw=black, fill=black}}
    \tikzset{lk/.style={regular polygon, regular polygon sides=4, inner sep=0pt, text width={\linksize}, draw=black, fill=black}}
    \tikzset{vp/.style={circle, inner sep=0pt, text width={1.5*\pointsize}, fill=white}}
    \tikzset{sstr/.style={shorten <= 5pt, shorten >= 5pt}}    
    \draw[dotted] ({0-\dd},{0}) -- ({\txl+\dd},{0});
    \draw[dotted] ({0-\dd},{\ty}) -- ({\txu+\dd},{\ty});
    \node[vp] (l1) at ({0+4*\dx},{0+0*\ty}) {$c_1$};
    \node[vp] (u1) at ({0+1*\dx},{0+1*\ty}) {$\overline{c_2}$};    
    \draw[sstr] (l1) --++(0,{0.5*\ty}) -| (u1);
    \draw[dashed] ($(u1)+(0,{-0.5*\ty})$) -- ++ ({-1.25*\dx},0);
    \draw [draw=gray, pattern = north east lines, pattern color = lightgray] ({0+0*\dx-\fx},{0+0*\dy-\fy}) -- ++ ({3*\dx+2*\fx},0) |- ({0+0*\dx-\fx},{0+0*\dy+\fy});
    \draw [draw=gray, pattern = north east lines, pattern color = lightgray] ({0+0*\dx-\fx},{\ty+0*\dy-\fy}) -- ++ ({0*\dx+2*\fx},0) |- ({0+0*\dx-\fx},{\ty+0*\dy+\fy});
    \draw [draw=gray, pattern = north east lines, pattern color = lightgray] ({3*\dx+0*\dx+\fx},{\ty+0*\dy-\fy}) -- ++ ({-1*\dx-2*\fx},0) |- ({3*\dx+0*\dx+\fx},{\ty+0*\dy+\fy}) -- cycle;
    \node [below = {\cx} of  l1] {$\alpha_1$};
    \node [above ={\cx} of  u1] {$\alpha_2$};
    \draw [draw=darkgray, densely dashed] ({\txu+0*\dx+\gx},{\ty+0*\dy-\gy}) -- ++ ({-6*\dx-2*\gx},0) |- ({\txu+0*\dx+\gx},{\ty+0*\dy+\gy});
    \draw [draw=darkgray, densely dashed] ({\txl+0*\dx+\gx},{0+0*\dy-\gy}) -- ++ ({-3*\dx-2*\gx},0) |- ({\txl+0*\dx+\gx},{0+0*\dy+\gy});
    \node at ({-1.5*\dx},{0.5*\ty}) {$p\otimes q$};
    %
    % \draw  [draw=gray, fill = lightgray] ({\txl+\fx},{-\fy}) --++({-2*\dx-2*\fx},0) |-({\txu+\fx},{\ty+\fy});
    \draw [draw=gray, fill = lightgray] ({\txl+0*\dx+\fx},{0*\dy-\fy}) -- ++ ({-2*\dx-2*\fx},0) |- ({\txl+0*\dx+\fx},{0*\dy+\fy});
    \draw [draw=gray, fill = lightgray] ({\txu+0*\dx+\fx},{\ty+0*\dy-\fy}) -- ++ ({-2*\dx-2*\fx},0) |- ({\txu+0*\dx+\fx},{\ty+0*\dy+\fy});
  \end{tikzpicture}
\end{mycenter}
\noindent
Now, because all points stemming from $q$ lie within $]\alpha_1,\alpha_2[_{p\otimes q}$, \[\delta_{p\otimes q}(\alpha_1,\alpha_2)=\delta_p(\alpha_1,\alpha_2)+\toco(q).\]
                     That proves $\delta_{p}(\alpha_1,\alpha_2)+\toco(q)\in K_{c_1,c_2}(\mc C)=\kappa_{c_1,c_2}$, which is what we needed to see.
                    \par
                    \emph{Axiom~\ref{lemma:gradedsets-condition-3}:} The verticolor reflection $\tilde p$ of $p$ belongs to $\mc C$. The set $]\alpha_1,\alpha_2]_p$ in $p$ is mapped by the reflection $\rho$  to the set $[\rho(\alpha_2),\rho(\alpha_1)[_{\tilde p}$ in $\tilde p$. As the operation of ver\-ti\-col\-or re\-flec\-tion inverts normalized colors, $\sigma_{p}(S)=-\sigma_{\tilde p}(\rho(S))$ for any set $S$ of points in $p$.
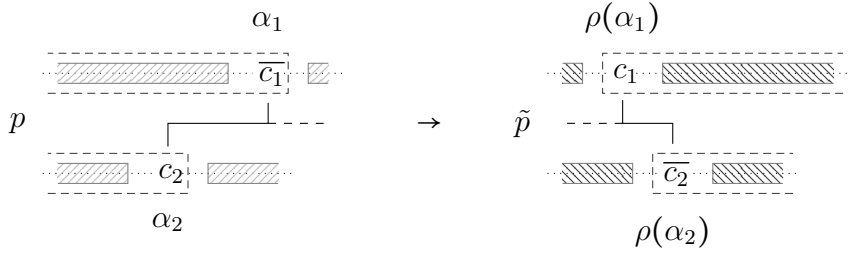
   \begin{mycenter}[0.5em]
\begin{tikzpicture}[baseline=0.666*1cm-0.25em]
    \def\scp{0.666}
    \def\linksize{\scp*0.075cm}
    \def\pointsize{\scp*0.25cm}
    \def\dd{\scp*0.5cm}
    \def\dx{\scp*1cm}
    \def\cx{\scp*0.3cm}
    \def\txu{5*\dx}    
    \def\txl{4*\dx}
    \def\dy{\scp*1cm}
    \def\cy{\scp*0.3cm}
    \def\ty{2*\dy}
    \def\fy{\scp*0.2cm}
    \def\fx{\scp*0.2cm}
    \def\gy{\scp*0.4cm}
    \def\gx{\scp*0.4cm}      
    \tikzset{whp/.style={circle, inner sep=0pt, text width={\pointsize}, draw=black, fill=white}}
    \tikzset{blp/.style={circle, inner sep=0pt, text width={\pointsize}, draw=black, fill=black}}
    \tikzset{lk/.style={regular polygon, regular polygon sides=4, inner sep=0pt, text width={\linksize}, draw=black, fill=black}}
    \tikzset{vp/.style={circle, inner sep=0pt, text width={1.5*\pointsize}, fill=white}}
    \tikzset{sstr/.style={shorten <= 5pt, shorten >= 5pt}}    
    \draw[dotted] ({0-\dd},{0}) -- ({\txl+\dd},{0});
    \draw[dotted] ({0-\dd},{\ty}) -- ({\txu+\dd},{\ty});
    \node[vp] (l1) at ({0+2*\dx},{0+0*\ty}) {$c_2$};    
    \node[vp] (u1) at ({0+4*\dx},{0+1*\ty}) {$\overline {c_1}$};
    \draw[sstr] (u1) --++(0,{-0.5*\ty}) -| (l1);
    \draw[dashed] ($(u1)+(0,{-0.5*\ty})$) -- ++ ({1.25*\dx},0);
    \draw [draw=gray, pattern = north east lines, pattern color = lightgray] ({0+0*\dx-\fx},{0+0*\dy-\fy}) -- ++ ({1*\dx+2*\fx},0) |- ({0+0*\dx-\fx},{0+0*\dy+\fy});
    \draw [draw=gray, pattern = north east lines, pattern color = lightgray]
({\txl+0*\dx+\fx},{0*\dy-\fy}) -- ++ ({-1*\dx-2*\fx},0) |- ({\txl+0*\dx+\fx},{0*\dy+\fy});
    \draw [draw=gray, pattern = north east lines, pattern color = lightgray] ({0+0*\dx-\fx},{\ty+0*\dy-\fy}) -- ++ ({3*\dx+2*\fx},0) |- ({0+0*\dx-\fx},{\ty+0*\dy+\fy});
    \draw [draw=gray, pattern = north east lines, pattern color = lightgray] ({\txu+0*\dx+\fx},{\ty+0*\dy-\fy}) -- ++ ({-0*\dx-2*\fx},0) |- ({\txu+0*\dx+\fx},{\ty+0*\dy+\fy});
    \node [above = {\cx} of  u1] {$\alpha_1$};
    \node [below ={\cx} of  l1] {$\alpha_2$};
    %
    %% dashed rectangles
    %%%% lower row
    \draw [draw=darkgray, densely dashed] ({0+0*\dx-\gx},{0+0*\dy-\gy}) -- ++ ({2*\dx+2*\gx},0) |- ({0+0*\dx-\gx},{0+0*\dy+\gy}); %left
%    \draw [draw=darkgray, densely dashed]
({\txl+0*\dx+\gx},{0*\dy-\gy}) -- ++ ({-1*\dx-2*\gx},0) |- ({\txl+0*\dx+\gx},{0*\dy+\gy});  %right
    %%%% upper row
    \draw [draw=darkgray, densely dashed] ({0+0*\dx-\gx},{\ty+0*\dy-\gy}) -- ++ ({4*\dx+2*\gx},0) |- ({0+0*\dx-\gx},{\ty+0*\dy+\gy}); %left
%    \draw [draw=darkgray, densely dashed] ({\txu+0*\dx+\gx},{\ty+0*\dy-\gy}) -- ++ ({-0*\dx-2*\gx},0) |- ({\txu+0*\dx+\gx},{\ty+0*\dy+\gy}); %right
    %
    \node at (-\dx,{0.5*\ty}) {$p$};
  \end{tikzpicture}
\qquad$\to$\qquad
\begin{tikzpicture}[baseline=0.666*1cm-0.25em]
    \def\scp{0.666}
    \def\linksize{\scp*0.075cm}
    \def\pointsize{\scp*0.25cm}
    \def\dd{\scp*0.5cm}
    \def\dx{\scp*1cm}
    \def\cx{\scp*0.3cm}
    \def\txu{5*\dx}    
    \def\txl{4*\dx}
    \def\dy{\scp*1cm}
    \def\cy{\scp*0.3cm}
    \def\ty{2*\dy}
    \def\fy{\scp*0.2cm}
    \def\fx{\scp*0.2cm}
    \def\gy{\scp*0.4cm}
    \def\gx{\scp*0.4cm}      
    \tikzset{whp/.style={circle, inner sep=0pt, text width={\pointsize}, draw=black, fill=white}}
    \tikzset{blp/.style={circle, inner sep=0pt, text width={\pointsize}, draw=black, fill=black}}
    \tikzset{lk/.style={regular polygon, regular polygon sides=4, inner sep=0pt, text width={\linksize}, draw=black, fill=black}}
    \tikzset{vp/.style={circle, inner sep=0pt, text width={1.5*\pointsize}, fill=white}}
    \tikzset{sstr/.style={shorten <= 5pt, shorten >= 5pt}}    
    \draw[dotted] ({0-\dd},{0}) -- ({\txl+\dd},{0});
    \draw[dotted] ({0-\dd},{\ty}) -- ({\txu+\dd},{\ty});
    \node[vp] (l1) at ({0+2*\dx},{0+0*\ty}) {$\overline{c_2}$};    
    \node[vp] (u1) at ({0+1*\dx},{0+1*\ty}) {$c_1$};
    \draw[sstr] (u1) --++(0,{-0.5*\ty}) -| (l1);
    \draw[dashed] ($(u1)+(0,{-0.5*\ty})$) -- ++ ({-1.25*\dx},0);
    \draw [draw=gray, pattern = north west lines, pattern color = darkgray] ({0+0*\dx-\fx},{0+0*\dy-\fy}) -- ++ ({1*\dx+2*\fx},0) |- ({0+0*\dx-\fx},{0+0*\dy+\fy});
    \draw [draw=gray, pattern = north west lines, pattern color = darkgray]
({\txl+0*\dx+\fx},{0*\dy-\fy}) -- ++ ({-1*\dx-2*\fx},0) |- ({\txl+0*\dx+\fx},{0*\dy+\fy});
    \draw [draw=gray, pattern = north west lines, pattern color = darkgray] ({0+0*\dx-\fx},{\ty+0*\dy-\fy}) -- ++ ({0*\dx+2*\fx},0) |- ({0+0*\dx-\fx},{\ty+0*\dy+\fy});
    \draw [draw=gray, pattern = north west lines, pattern color = darkgray] ({\txu+0*\dx+\fx},{\ty+0*\dy-\fy}) -- ++ ({-3*\dx-2*\fx},0) |- ({\txu+0*\dx+\fx},{\ty+0*\dy+\fy});
    \node [above = {\cx} of  u1] {$\rho(\alpha_1)$};
    \node [below ={\cx} of  l1] {$\rho(\alpha_2)$};
    %
%% dashed rectangles
    %%%% lower row
%    \draw [draw=darkgray, densely dashed] ({0+0*\dx-\gx},{0+0*\dy-\gy}) -- ++ ({2*\dx+2*\gx},0) |- ({0+0*\dx-\gx},{0+0*\dy+\gy}); %left
    \draw [draw=darkgray, densely dashed]
({\txl+0*\dx+\gx},{0*\dy-\gy}) -- ++ ({-2*\dx-2*\gx},0) |- ({\txl+0*\dx+\gx},{0*\dy+\gy});  %right
    %%%% upper row
%    \draw [draw=darkgray, densely dashed] ({0+0*\dx-\gx},{\ty+0*\dy-\gy}) -- ++ ({4*\dx+2*\gx},0) |- ({0+0*\dx-\gx},{\ty+0*\dy+\gy}); %left
    \draw [draw=darkgray, densely dashed] ({\txu+0*\dx+\gx},{\ty+0*\dy-\gy}) -- ++ ({-4*\dx-2*\gx},0) |- ({\txu+0*\dx+\gx},{\ty+0*\dy+\gy}); %right
    \node at (-\dx,{0.5*\ty}) {$\tilde p$};
  \end{tikzpicture}
\end{mycenter}  

                   Using the case distinction free formula for $\delta_p(\alpha_1,\alpha_2)$ given in the proof of \cite[Lem\-ma~3.1~(b)]{MWNHO1}, we thus compute
        \begin{align*}
      \delta_{p}(\alpha_1,\alpha_2)&=\sigma_{p}(]\alpha_1,\alpha_2]_{ p})+{\textstyle\frac{1}{2}}(\sigma_{ p}(\alpha_1)-\sigma_{p}(\alpha_2))  \\
                                                      &=-\sigma_{\tilde p}([\rho(\alpha_2),\rho(\alpha_1)[_{\tilde p})-{\textstyle\frac{1}{2}}(\sigma_{\tilde  p}(\rho(\alpha_1))-\sigma_{\tilde p}(\rho(\alpha_2)))  \\
                                   &=-\sigma_{\tilde p}(]\rho(\alpha_2),\rho(\alpha_1)]_{\tilde p})-\sigma_{\tilde p}(\rho(\alpha_2))+\sigma_{\tilde p}(\rho (\alpha_1)) \\
          &\phantom{{}={}}-{\textstyle\frac{1}{2}}(\sigma_{\tilde  p}(\rho(\alpha_1))-\sigma_{\tilde p}(\rho(\alpha_2)))  \\
      &=-\sigma_{\tilde p}(]\rho(\alpha_2),\rho(\alpha_1)]_{\tilde p})-{\textstyle\frac{1}{2}}(\sigma_{\tilde  p}(\rho (\alpha_2))-\sigma_{\tilde p}(\rho(\alpha_1)))  \\
                                                             &=-\delta_{\tilde p}(\rho(\alpha_2),\rho(\alpha_1)).
        \end{align*}
        Because, for every $i\in \{1,2\}$, the point $\rho(\alpha_i)$ has normalized color $\overline{c_i}$ in $\tilde p$ and because $\rho(B)$ is a block of $\tilde p$ with $]\rho(\alpha_2),\rho(\alpha_1)[_{\tilde p}\cap \rho(B)=\emptyset$, we conclude $\delta_p(\alpha_1,\alpha_2)\in -K_{\overline{c_2},\overline{c_1}}(\mc C)=-\kappa_{\overline{c_2},\overline{c_1}}$. And that is what we had to show.
                    \par
                    \emph{Axiom~\ref{lemma:gradedsets-condition-2}:} So far, we have not made use of Lem\-ma~\ref{lemma:simplification-k-x}. Now, though, we employ it to additionally assume $p\in \mc C\cap \Cp_{\leq 2}$. In particular, then, $B= \{\alpha_1,\alpha_2\}$ is a pair block. Consequently, not only $]\alpha_1,\alpha_2[_p\cap B=\emptyset$ but also $]\alpha_2,\alpha_1[_p\cap B=\emptyset$.
   \begin{mycenter}[0.5em]
\begin{tikzpicture}[baseline=0.666*1cm-0.25em]
    \def\scp{0.666}
    \def\linksize{\scp*0.075cm}
    \def\pointsize{\scp*0.25cm}
    \def\dd{\scp*0.5cm}
    \def\dx{\scp*1cm}
    \def\cx{\scp*0.3cm}
    \def\txu{5*\dx}    
    \def\txl{4*\dx}
    \def\dy{\scp*1cm}
    \def\cy{\scp*0.3cm}
    \def\ty{2*\dy}
    \def\fy{\scp*0.2cm}
    \def\fx{\scp*0.2cm}
    \def\gy{\scp*0.4cm}
    \def\gx{\scp*0.4cm}      
    \tikzset{whp/.style={circle, inner sep=0pt, text width={\pointsize}, draw=black, fill=white}}
    \tikzset{blp/.style={circle, inner sep=0pt, text width={\pointsize}, draw=black, fill=black}}
    \tikzset{lk/.style={regular polygon, regular polygon sides=4, inner sep=0pt, text width={\linksize}, draw=black, fill=black}}
    \tikzset{vp/.style={circle, inner sep=0pt, text width={1.5*\pointsize}, fill=white}}
    \tikzset{sstr/.style={shorten <= 5pt, shorten >= 5pt}}    
    \draw[dotted] ({0-\dd},{0}) -- ({\txl+\dd},{0});
    \draw[dotted] ({0-\dd},{\ty}) -- ({\txu+\dd},{\ty});
    \node[vp] (l1) at ({0+2*\dx},{0+0*\ty}) {$c_2$};    
    \node[vp] (u1) at ({0+4*\dx},{0+1*\ty}) {$\overline {c_1}$};
    \draw[sstr] (u1) --++(0,{-0.5*\ty}) -| (l1);
    \draw [draw=gray, pattern = north east lines, pattern color = lightgray] ({0+0*\dx-\fx},{0+0*\dy-\fy}) -- ++ ({1*\dx+2*\fx},0) |- ({0+0*\dx-\fx},{0+0*\dy+\fy});
    \draw [draw=gray, pattern = north east lines, pattern color = lightgray]
({\txl+0*\dx+\fx},{0*\dy-\fy}) -- ++ ({-1*\dx-2*\fx},0) |- ({\txl+0*\dx+\fx},{0*\dy+\fy});
    \draw [draw=gray, pattern = north east lines, pattern color = lightgray] ({0+0*\dx-\fx},{\ty+0*\dy-\fy}) -- ++ ({3*\dx+2*\fx},0) |- ({0+0*\dx-\fx},{\ty+0*\dy+\fy});
    \draw [draw=gray, pattern = north east lines, pattern color = lightgray] ({\txu+0*\dx+\fx},{\ty+0*\dy-\fy}) -- ++ ({-0*\dx-2*\fx},0) |- ({\txu+0*\dx+\fx},{\ty+0*\dy+\fy});
    \node [above = {\cx} of  u1] {$\alpha_1$};
    \node [below ={\cx} of  l1] {$\alpha_2$};
    %
    %% dashed rectangles
    %%%% lower row
    \draw [draw=darkgray, densely dashed] ({0+0*\dx-\gx},{0+0*\dy-\gy}) -- ++ ({2*\dx+2*\gx},0) |- ({0+0*\dx-\gx},{0+0*\dy+\gy}); %left
%    \draw [draw=darkgray, densely dashed]
({\txl+0*\dx+\gx},{0*\dy-\gy}) -- ++ ({-1*\dx-2*\gx},0) |- ({\txl+0*\dx+\gx},{0*\dy+\gy});  %right
    %%%% upper row
    \draw [draw=darkgray, densely dashed] ({0+0*\dx-\gx},{\ty+0*\dy-\gy}) -- ++ ({4*\dx+2*\gx},0) |- ({0+0*\dx-\gx},{\ty+0*\dy+\gy}); %left
%    \draw [draw=darkgray, densely dashed] ({\txu+0*\dx+\gx},{\ty+0*\dy-\gy}) -- ++ ({-0*\dx-2*\gx},0) |- ({\txu+0*\dx+\gx},{\ty+0*\dy+\gy}); %right
    %
    \node at (-\dx,{0.5*\ty}) {$p$};
  \end{tikzpicture}
  \qquad$\to$\qquad
\begin{tikzpicture}[baseline=0.666*1cm-0.25em]
    \def\scp{0.666}
    \def\linksize{\scp*0.075cm}
    \def\pointsize{\scp*0.25cm}
    \def\dd{\scp*0.5cm}
    \def\dx{\scp*1cm}
    \def\cx{\scp*0.3cm}
    \def\txu{5*\dx}    
    \def\txl{4*\dx}
    \def\dy{\scp*1cm}
    \def\cy{\scp*0.3cm}
    \def\ty{2*\dy}
    \def\fy{\scp*0.2cm}
    \def\fx{\scp*0.2cm}
    \def\gy{\scp*0.4cm}
    \def\gx{\scp*0.4cm}      
    \tikzset{whp/.style={circle, inner sep=0pt, text width={\pointsize}, draw=black, fill=white}}
    \tikzset{blp/.style={circle, inner sep=0pt, text width={\pointsize}, draw=black, fill=black}}
    \tikzset{lk/.style={regular polygon, regular polygon sides=4, inner sep=0pt, text width={\linksize}, draw=black, fill=black}}
    \tikzset{vp/.style={circle, inner sep=0pt, text width={1.5*\pointsize}, fill=white}}
    \tikzset{sstr/.style={shorten <= 5pt, shorten >= 5pt}}    
    \draw[dotted] ({0-\dd},{0}) -- ({\txl+\dd},{0});
    \draw[dotted] ({0-\dd},{\ty}) -- ({\txu+\dd},{\ty});
    \node[vp] (l1) at ({0+2*\dx},{0+0*\ty}) {$c_2$};    
    \node[vp] (u1) at ({0+4*\dx},{0+1*\ty}) {$\overline {c_1}$};
    \draw[sstr] (u1) --++(0,{-0.5*\ty}) -| (l1);
    \draw [draw=gray, pattern = north east lines, pattern color = lightgray] ({0+0*\dx-\fx},{0+0*\dy-\fy}) -- ++ ({1*\dx+2*\fx},0) |- ({0+0*\dx-\fx},{0+0*\dy+\fy});
    \draw [draw=gray, pattern = north east lines, pattern color = lightgray]
({\txl+0*\dx+\fx},{0*\dy-\fy}) -- ++ ({-1*\dx-2*\fx},0) |- ({\txl+0*\dx+\fx},{0*\dy+\fy});
    \draw [draw=gray, pattern = north east lines, pattern color = lightgray] ({0+0*\dx-\fx},{\ty+0*\dy-\fy}) -- ++ ({3*\dx+2*\fx},0) |- ({0+0*\dx-\fx},{\ty+0*\dy+\fy});
    \draw [draw=gray, pattern = north east lines, pattern color = lightgray] ({\txu+0*\dx+\fx},{\ty+0*\dy-\fy}) -- ++ ({-0*\dx-2*\fx},0) |- ({\txu+0*\dx+\fx},{\ty+0*\dy+\fy});
    \node [above = {\cx} of  u1] {$\alpha_1$};
    \node [below ={\cx} of  l1] {$\alpha_2$};
    %
    %% dashed rectangles
    %%%% lower row
%    \draw [draw=darkgray, densely dashed] ({0+0*\dx-\gx},{0+0*\dy-\gy}) -- ++ ({2*\dx+2*\gx},0) |- ({0+0*\dx-\gx},{0+0*\dy+\gy}); %left
    \draw [draw=darkgray, densely dashed]
({\txl+0*\dx+\gx},{0*\dy-\gy}) -- ++ ({-2*\dx-2*\gx},0) |- ({\txl+0*\dx+\gx},{0*\dy+\gy});  %right
    %%%% upper row
%    \draw [draw=darkgray, densely dashed] ({0+0*\dx-\gx},{\ty+0*\dy-\gy}) -- ++ ({4*\dx+2*\gx},0) |- ({0+0*\dx-\gx},{\ty+0*\dy+\gy}); %left
    \draw [draw=darkgray, densely dashed] ({\txu+0*\dx+\gx},{\ty+0*\dy-\gy}) -- ++ ({-1*\dx-2*\gx},0) |- ({\txu+0*\dx+\gx},{\ty+0*\dy+\gy}); %right
    \node at (-\dx,{0.5*\ty}) {$p$};
  \end{tikzpicture}  
\end{mycenter}
\noindent
                    By \cite[Lem\-ma~2.1~(b)]{MWNHO1} we infer \[\delta_p(\alpha_1,\alpha_2)= -\delta_{p}(\alpha_2,\alpha_1)\mod \toco(p).\] As $\toco(p)\in \toco(\mc C)$, it follows $\delta_p(\alpha_1,\alpha_2)\in  -K_{c_2,c_1}(\mc C)+\toco(\mc C)=-\kappa_{c_2,c_1}+\sigma$, which is what we wanted to see.                  
\end{proof}

\begin{lemma}
  \label{lemma:verifying-axioms-3}
  For every non-hy\-per\-octa\-he\-dral category $\mc C\subseteq \Cp$, the sets $\sigma\eqpd \toco(\mc C)$ and $\xi_{c_1,c_2}\eqpd X_{c_1,c_2}(\mc C)$ for $c_1,c_2\in \colors$ satisfy Axioms~\ref{lemma:gradedsets-condition-1}--\ref{lemma:gradedsets-condition-2} of \ref{axioms:arithmetic}:
  \begin{IEEEeqnarray*}{rCCrCCrC}
   \text{(ii)}\hspace{0.5em}& \xi_{c_1,c_2}+\sigma\subseteq \xi_{c_1,c_2}, &\hspace{1.5em} &\text{(iii)}\hspace{0.5em}& \xi_{c_1,c_2}\subseteq - \xi_{\overline{c_2},\overline{c_1}}, &\hspace{1.5em} &\text{(iv)}\hspace{0.5em}& \xi_{c_1,c_2}\subseteq - \xi_{c_2,c_1}+\sigma
 \end{IEEEeqnarray*}
 for all $c_1,c_2\in\colors$.
\end{lemma}
\begin{proof}
The proof is similar to that of Lem\-ma~\ref{lemma:verifying-axioms-2}.
  Let $c_1,c_2\in \colors$, let $B_1$ and $B_2$ be crossing blocks of $p\in \mc C$ and let $\alpha_1\in B_1$ and $\alpha_2\in B_2$ have normalized colors $c_1$ and $c_2$, respectively. That makes $\delta_{p}(\alpha_1,\alpha_2)$ a generic element of $X_{c_1,c_2}(\mc C)=\xi_{c_1,c_2}$.
                  \par
                  \emph{Axiom~\ref{lemma:gradedsets-condition-1}:} Just like in the proof of Lem\-ma~\ref{lemma:verifying-axioms-2}, we can assume that $\alpha_1$ is the rightmost lower point. Given arbitrary $q\in \mc C$, the sets $B_1$ and $B_2$ are crossing blocks of $p\otimes q$ as well, 
                     
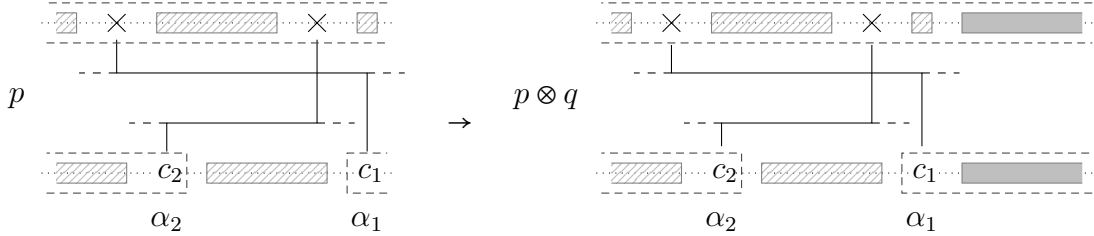
\begin{mycenter}[0.5em]
\begin{tikzpicture}[baseline=0.666*1cm-0.25em]
    \def\scp{0.666}
    \def\linksize{\scp*0.075cm}
    \def\pointsize{\scp*0.25cm}
    \def\dd{\scp*0.5cm}
    \def\dx{\scp*1cm}
    \def\cx{\scp*0.3cm}
    \def\txu{6*\dx}    
    \def\txl{6*\dx}
    \def\dy{\scp*1cm}
    \def\cy{\scp*0.3cm}
    \def\ty{3*\dy}
    \def\fy{\scp*0.2cm}
    \def\fx{\scp*0.2cm}
    \def\gy{\scp*0.4cm}
    \def\gx{\scp*0.4cm}      
    \tikzset{whp/.style={circle, inner sep=0pt, text width={\pointsize}, draw=black, fill=white}}
    \tikzset{blp/.style={circle, inner sep=0pt, text width={\pointsize}, draw=black, fill=black}}
    \tikzset{lk/.style={regular polygon, regular polygon sides=4, inner sep=0pt, text width={\linksize}, draw=black, fill=black}}
    \tikzset{cc/.style={cross out, minimum size={1.5*\pointsize-\pgflinewidth}, inner sep=0pt, outer sep=0pt,  draw=black}}
    \tikzset{vp/.style={circle, inner sep=0pt, text width={1.5*\pointsize}, fill=white}}
    \tikzset{sstr/.style={shorten <= 3pt, shorten >= 3pt}}    
    \draw[dotted] ({0-\dd},{0}) -- ({\txl+\dd},{0});
    \draw[dotted] ({0-\dd},{\ty}) -- ({\txu+\dd},{\ty});
    \node[vp] (l1) at ({0+6*\dx},{0+0*\ty}) {$c_1$};
    \node[vp] (l2) at ({0+2*\dx},{0+0*\ty}) {$c_2$};    
    \node[cc] (u1) at ({0+1*\dx},{0+1*\ty}) {};
    %\node[rotate=45, cc] at ({0+1*\dx},{0+1*\ty}) {};
    \node[cc] (u2) at ({0+5*\dx},{0+1*\ty}) {};    
    %\node[rotate=45, cc] at ({0+5*\dx},{0+1*\ty}) {};
    %
    \draw[sstr] (u1) -- ++ (0,{-1/3*\ty}) -| (l1);
    \draw[dashed] ($(u1)+ (0,{-1/3*\ty})$) --++ ({-0.75*\dx},0);
    \draw[dashed] ($(u2)+ (0,{-2/3*\ty})$) --++ ({0.75*\dx},0);
    \draw[sstr] (u2) -- ++ (0,{-2/3*\ty}) -| (l2);
    \draw[dashed] ($(l2)+ (0,{\ty / 3})$) --++ ({-0.75*\dx},0);
    \draw[dashed] ($(l1)+ (0,{\ty *2/3})$) --++ ({0.75*\dx},0);
    %
    %%% lower row, left-open, filled
    \draw [draw=gray, pattern = north east lines, pattern color = lightgray] ({0+0*\dx-\fx},{0+0*\dy-\fy}) -- ++ ({1*\dx+2*\fx},0) |- ({0+0*\dx-\fx},{0+0*\dy+\fy});
    %%% lower row, right-open, filled    
    % \draw [draw=gray, pattern = north east lines, pattern color = lightgray]    ({\txl+0*\dx+\fx},{0*\dy-\fy}) -- ++ ({-1*\dx-2*\fx},0) |- ({\txl+0*\dx+\fx},{0*\dy+\fy});
    %%% lower row, closed, filled
    \draw [draw=gray, pattern = north east lines, pattern color = lightgray] ({3*\dx-\fx},{-\fy}) rectangle ({5*\dx+\fx},{\fy});
    %%% upper row, left-open, filled    
    \draw [draw=gray, pattern = north east lines, pattern color = lightgray] ({0+0*\dx-\fx},{\ty+0*\dy-\fy}) -- ++ ({0*\dx+2*\fx},0) |- ({0+0*\dx-\fx},{\ty+0*\dy+\fy});
    %%% upper row, right-open, filled        
    % \draw [draw=gray, pattern = north east lines, pattern color = lightgray] ({\txu+0*\dx+\fx},{\ty+0*\dy-\fy}) -- ++ ({-0*\dx-2*\fx},0) |- ({\txu+0*\dx+\fx},{\ty+0*\dy+\fy});
     %%% upper row, closed filled
    \draw [draw=gray, pattern = north east lines, pattern color = lightgray] ({2*\dx-\fx},{\ty-\fy}) rectangle ({4*\dx+\fx},{\ty+\fy});
    \draw [draw=gray, pattern = north east lines, pattern color = lightgray] ({6*\dx-\fx},{\ty-\fy}) rectangle ({6*\dx+\fx},{\ty+\fy});         
    \node [below = {\cx} of  l1] {$\alpha_1$};
    \node [below ={\cx} of  l2] {$\alpha_2$};
     %
     %% lower row, left-open
     \draw [draw=darkgray, densely dashed] ({0+0*\dx-\gx},{0+0*\dy-\gy}) --++ ({2*\dx+2*\gx},0) |- ({0+0*\dx-\gx},{0+0*\dy+\gy}); 
     %% lower row, right-open    
      \draw [draw=darkgray, densely dashed] ({\txl+0*\dx+\gx},{0*\dy-\gy}) --++ ({-0*\dx-2*\gx},0) |- ({\txl+0*\dx+\gx},{0*\dy+\gy});
     %% lower row, open
     %\draw [draw=darkgray, densely dashed] ({-\gx},{-\gy}) -- ({\txl+\gx},{-\gy})  ({\txl+\gx},{\gy}) -- ({-\gx},{\gy});
     %% upper row, left-open
     %\draw [draw=darkgray, densely dashed] ({0+0*\dx-\gx},{\ty+0*\dy-\gy}) --++ ({4*\dx+2*\gx},0) |- ({0+0*\dx-\gx},{\ty+0*\dy+\gy}); 
     %% upper row, right-open
     % \draw [draw=darkgray, densely dashed] ({\txu+0*\dx+\gx},{\ty+0*\dy-\gy}) --++ ({-1*\dx-2*\gx},0) |- ({\txu+0*\dx+\gx},{\ty+0*\dy+\gy});
     %% upper row, open
     \draw [draw=darkgray, densely dashed] ({-\gx},{\ty-\gy}) -- ({\txu+\gx},{\ty-\gy})  ({\txu+\gx},{\ty+\gy}) -- ({-\gx},{\ty+\gy});
    \node at (-\dx,{0.5*\ty}) {$p$};
  \end{tikzpicture}
\quad$\to$\quad
\begin{tikzpicture}[baseline=0.666*1cm-0.25em]
    \def\scp{0.666}
    \def\linksize{\scp*0.075cm}
    \def\pointsize{\scp*0.25cm}
    \def\dd{\scp*0.5cm}
    \def\dx{\scp*1cm}
    \def\cx{\scp*0.3cm}
    \def\txu{9*\dx}    
    \def\txl{9*\dx}
    \def\dy{\scp*1cm}
    \def\cy{\scp*0.3cm}
    \def\ty{3*\dy}
    \def\fy{\scp*0.2cm}
    \def\fx{\scp*0.2cm}
    \def\gy{\scp*0.4cm}
    \def\gx{\scp*0.4cm}      
    \tikzset{whp/.style={circle, inner sep=0pt, text width={\pointsize}, draw=black, fill=white}}
    \tikzset{blp/.style={circle, inner sep=0pt, text width={\pointsize}, draw=black, fill=black}}
    \tikzset{lk/.style={regular polygon, regular polygon sides=4, inner sep=0pt, text width={\linksize}, draw=black, fill=black}}
    \tikzset{cc/.style={cross out, minimum size={1.5*\pointsize-\pgflinewidth}, inner sep=0pt, outer sep=0pt,  draw=black}}    
    \tikzset{vp/.style={circle, inner sep=0pt, text width={1.5*\pointsize}, fill=white}}
    \tikzset{sstr/.style={shorten <= 5pt, shorten >= 5pt}}    
    \draw[dotted] ({0-\dd},{0}) -- ({\txl+\dd},{0});
    \draw[dotted] ({0-\dd},{\ty}) -- ({\txu+\dd},{\ty});
        \node[vp] (l1) at ({0+6*\dx},{0+0*\ty}) {$c_1$};
    \node[vp] (l2) at ({0+2*\dx},{0+0*\ty}) {$c_2$};    
    \node[cc] (u1) at ({0+1*\dx},{0+1*\ty}) {};
    \node[cc] (u2) at ({0+5*\dx},{0+1*\ty}) {};    
    \draw[sstr] (u1) -- ++ (0,{-1/3*\ty}) -| (l1);
    \draw[dashed] ($(u1)+ (0,{-1/3*\ty})$) --++ ({-0.75*\dx},0);
    \draw[dashed] ($(u2)+ (0,{-2/3*\ty})$) --++ ({0.75*\dx},0);
    \draw[sstr] (u2) -- ++ (0,{-2/3*\ty}) -| (l2);
    \draw[dashed] ($(l2)+ (0,{\ty / 3})$) --++ ({-0.75*\dx},0);
    \draw[dashed] ($(l1)+ (0,{\ty *2/3})$) --++ ({0.75*\dx},0);
    %
    %%% lower row, left-open, filled
    \draw [draw=gray, pattern = north east lines, pattern color = lightgray] ({0+0*\dx-\fx},{0+0*\dy-\fy}) -- ++ ({1*\dx+2*\fx},0) |- ({0+0*\dx-\fx},{0+0*\dy+\fy});
    %%% lower row, right-open, filled    
    % \draw [draw=gray, pattern = north east lines, pattern color = lightgray]    ({\txl+0*\dx+\fx},{0*\dy-\fy}) -- ++ ({-1*\dx-2*\fx},0) |- ({\txl+0*\dx+\fx},{0*\dy+\fy});
    %%% lower row, closed, filled
    \draw [draw=gray, pattern = north east lines, pattern color = lightgray] ({3*\dx-\fx},{-\fy}) rectangle ({5*\dx+\fx},{\fy});
    %%% upper row, left-open, filled    
    \draw [draw=gray, pattern = north east lines, pattern color = lightgray] ({0+0*\dx-\fx},{\ty+0*\dy-\fy}) -- ++ ({0*\dx+2*\fx},0) |- ({0+0*\dx-\fx},{\ty+0*\dy+\fy});
    %%% upper row, right-open, filled        
%    \draw [draw=gray, pattern = north east lines, pattern color = lightgray] ({\txu+0*\dx+\fx},{\ty+0*\dy-\fy}) -- ++ ({-0*\dx-2*\fx},0) |- ({\txu+0*\dx+\fx},{\ty+0*\dy+\fy});
     %%% upper row, closed, filled
    \draw [draw=gray, pattern = north east lines, pattern color = lightgray] ({2*\dx-\fx},{\ty-\fy}) rectangle ({4*\dx+\fx},{\ty+\fy});
    \draw [draw=gray, pattern = north east lines, pattern color = lightgray] ({6*\dx-\fx},{\ty-\fy}) rectangle ({6*\dx+\fx},{\ty+\fy});    
    \node [below = {\cx} of  l1] {$\alpha_1$};
    \node [below ={\cx} of  l2] {$\alpha_2$};
    %
     %% lower row, left-open
     \draw [draw=darkgray, densely dashed] ({0+0*\dx-\gx},{0+0*\dy-\gy}) --++ ({2*\dx+2*\gx},0) |- ({0+0*\dx-\gx},{0+0*\dy+\gy}); 
     %% lower row, right-open    
     \draw [draw=darkgray, densely dashed] ({\txl+0*\dx+\gx},{0*\dy-\gy}) --++ ({-3*\dx-2*\gx},0) |- ({\txl+0*\dx+\gx},{0*\dy+\gy});
     %% lower row, open
     %\draw [draw=darkgray, densely dashed] ({-\gx},{-\gy}) -- ({\txl+\gx},{-\gy})  ({\txl+\gx},{\gy}) -- ({-\gx},{\gy});
     %% upper row, left-open
     %\draw [draw=darkgray, densely dashed] ({0+0*\dx-\gx},{\ty+0*\dy-\gy}) --++ ({4*\dx+2*\gx},0) |- ({0+0*\dx-\gx},{\ty+0*\dy+\gy}); 
     %% upper row, right-open
     % \draw [draw=darkgray, densely dashed] ({\txu+0*\dx+\gx},{\ty+0*\dy-\gy}) --++ ({-1*\dx-2*\gx},0) |- ({\txu+0*\dx+\gx},{\ty+0*\dy+\gy});
     %% upper row, open
     \draw [draw=darkgray, densely dashed] ({-\gx},{\ty-\gy}) -- ({\txu+\gx},{\ty-\gy})  ({\txu+\gx},{\ty+\gy}) -- ({-\gx},{\ty+\gy});
    \node at ({-1.5*\dx},{0.5*\ty}) {$p\otimes q$};
    %
    % \draw  [draw=gray, fill = lightgray] ({\txl+\fx},{-\fy}) --++({-2*\dx-2*\fx},0) |-({\txu+\fx},{\ty+\fy});
    \draw [draw=gray, fill = lightgray] ({\txl+0*\dx+\fx},{0*\dy-\fy}) -- ++ ({-2*\dx-2*\fx},0) |- ({\txl+0*\dx+\fx},{0*\dy+\fy});
    \draw [draw=gray, fill = lightgray] ({\txu+0*\dx+\fx},{\ty+0*\dy-\fy}) -- ++ ({-2*\dx-2*\fx},0) |- ({\txu+0*\dx+\fx},{\ty+0*\dy+\fy});
  \end{tikzpicture}
\end{mycenter}
\noindent
which proves \[\delta_{p}(\alpha_1,\alpha_2)+\toco(q)=\delta_{p\otimes q}(\alpha_1,\alpha_2)\in X_{c_1,c_2}(\mc C)=\xi_{c_1,c_2}.\]
Thus, $\xi_{c_1,c_2}+\sigma\subseteq \xi_{c_1,c_2}$ as claimed.
\par
                    \emph{Axiom~\ref{lemma:gradedsets-condition-3}:} Likewise, the sets $\rho(B_1)$ and $\rho(B_2)$ are still crossing blocks in $\tilde p\in \mc C$.  There, $\alpha_i$ has normalized color $\overline{c_i}$ for every $i\in \{1,2\}$.
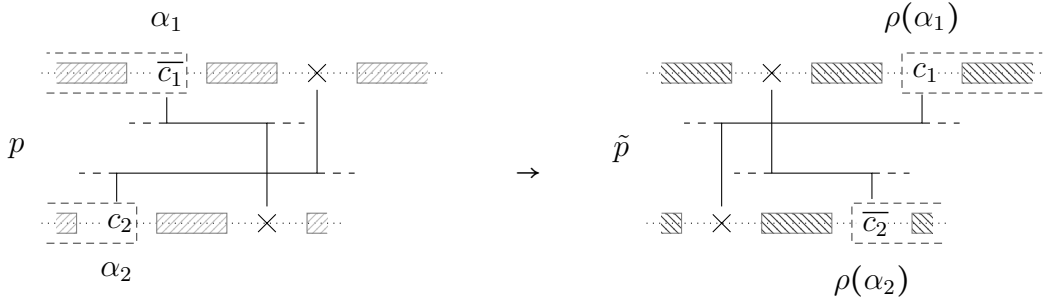
                       \begin{mycenter}[0.5em]
\begin{tikzpicture}[baseline=0.666*1cm-0.25em]
    \def\scp{0.666}
    \def\linksize{\scp*0.075cm}
    \def\pointsize{\scp*0.25cm}
    \def\dd{\scp*0.5cm}
    \def\dx{\scp*1cm}
    \def\cx{\scp*0.3cm}
    \def\txu{7*\dx}    
    \def\txl{5*\dx}
    \def\dy{\scp*1cm}
    \def\cy{\scp*0.3cm}
    \def\ty{3*\dy}
    \def\fy{\scp*0.2cm}
    \def\fx{\scp*0.2cm}
    \def\gy{\scp*0.4cm}
    \def\gx{\scp*0.4cm}      
    \tikzset{whp/.style={circle, inner sep=0pt, text width={\pointsize}, draw=black, fill=white}}
    \tikzset{blp/.style={circle, inner sep=0pt, text width={\pointsize}, draw=black, fill=black}}
    \tikzset{lk/.style={regular polygon, regular polygon sides=4, inner sep=0pt, text width={\linksize}, draw=black, fill=black}}
    \tikzset{cc/.style={cross out, minimum size={1.5*\pointsize-\pgflinewidth}, inner sep=0pt, outer sep=0pt,  draw=black}}    
    \tikzset{vp/.style={circle, inner sep=0pt, text width={1.5*\pointsize}, fill=white}}
    \tikzset{sstr/.style={shorten <= 3pt, shorten >= 3pt}}    
    \draw[dotted] ({0-\dd},{0}) -- ({\txl+\dd},{0});
    \draw[dotted] ({0-\dd},{\ty}) -- ({\txu+\dd},{\ty});
    \node[vp] (l1) at ({0+1*\dx},{0+0*\ty}) {$c_2$};
    \node[cc] (l2) at ({0+4*\dx},{0+0*\ty}) {};
    \node[vp] (u1) at ({0+2*\dx},{0+1*\ty}) {$\overline {c_1}$};
    \node[cc] (u2) at ({0+5*\dx},{0+1*\ty}) {};    
    \draw[sstr] (u1) -- ++ (0,{-1/3*\ty}) -| (l2);
    \draw[dashed] ($(u1)+ (0,{-1/3*\ty})$) --++ ({-0.75*\dx},0);
    \draw[dashed] ($(u2)+ (0,{-2/3*\ty})$) --++ ({0.75*\dx},0);
    \draw[sstr] (u2) -- ++ (0,{-2/3*\ty}) -| (l1);
    \draw[dashed] ($(l1)+ (0,{\ty / 3})$) --++ ({-0.75*\dx},0);
    \draw[dashed] ($(l2)+ (0,{\ty *2/3})$) --++ ({0.75*\dx},0);
    %
    %%% lower row, left-open, filled
    \draw [draw=gray, pattern = north east lines, pattern color = lightgray] ({0+0*\dx-\fx},{0+0*\dy-\fy}) -- ++ ({0*\dx+2*\fx},0) |- ({0+0*\dx-\fx},{0+0*\dy+\fy});
    %%% lower row, right-open, filled    
     \draw [draw=gray, pattern = north east lines, pattern color = lightgray]    ({\txl+0*\dx+\fx},{0*\dy-\fy}) -- ++ ({-0*\dx-2*\fx},0) |- ({\txl+0*\dx+\fx},{0*\dy+\fy});
    %%% lower row, closed, filled
    \draw [draw=gray, pattern = north east lines, pattern color = lightgray] ({2*\dx-\fx},{-\fy}) rectangle ({3*\dx+\fx},{\fy});
    %%% upper row, left-open, filled    
    \draw [draw=gray, pattern = north east lines, pattern color = lightgray] ({0+0*\dx-\fx},{\ty+0*\dy-\fy}) -- ++ ({1*\dx+2*\fx},0) |- ({0+0*\dx-\fx},{\ty+0*\dy+\fy});
    %%% upper row, right-open, filled        
    \draw [draw=gray, pattern = north east lines, pattern color = lightgray] ({\txu+0*\dx+\fx},{\ty+0*\dy-\fy}) -- ++ ({-1*\dx-2*\fx},0) |- ({\txu+0*\dx+\fx},{\ty+0*\dy+\fy});
     %%% upper row, closed, filled
    \draw [draw=gray, pattern = north east lines, pattern color = lightgray] ({3*\dx-\fx},{\ty-\fy}) rectangle ({4*\dx+\fx},{\ty+\fy});
    \node [above = {\cx} of  u1] {$\alpha_1$};
    \node [below ={\cx} of  l1] {$\alpha_2$};
    %
     %% lower row, left-open
     \draw [draw=darkgray, densely dashed] ({0+0*\dx-\gx},{0+0*\dy-\gy}) --++ ({1*\dx+2*\gx},0) |- ({0+0*\dx-\gx},{0+0*\dy+\gy}); 
     %% lower row, right-open    
     %\draw [draw=darkgray, densely dashed] ({\txl+0*\dx+\gx},{0*\dy-\gy}) --++ ({-3*\dx-2*\gx},0) |- ({\txl+0*\dx+\gx},{0*\dy+\gy});
     %% lower row, open
     %\draw [draw=darkgray, densely dashed] ({-\gx},{-\gy}) -- ({\txl+\gx},{-\gy})  ({\txl+\gx},{\gy}) -- ({-\gx},{\gy});
     %% upper row, left-open
     \draw [draw=darkgray, densely dashed] ({0+0*\dx-\gx},{\ty+0*\dy-\gy}) --++ ({2*\dx+2*\gx},0) |- ({0+0*\dx-\gx},{\ty+0*\dy+\gy}); 
     %% upper row, right-open
     % \draw [draw=darkgray, densely dashed] ({\txu+0*\dx+\gx},{\ty+0*\dy-\gy}) --++ ({-1*\dx-2*\gx},0) |- ({\txu+0*\dx+\gx},{\ty+0*\dy+\gy});
     %% upper row, open
     %\draw [draw=darkgray, densely dashed] ({-\gx},{\ty-\gy}) -- ({\txu+\gx},{\ty-\gy})  ({\txu+\gx},{\ty+\gy}) -- ({-\gx},{\ty+\gy});
    %
    \node at (-\dx,{0.5*\ty}) {$p$};
  \end{tikzpicture}
  \qquad$\to$\qquad
\begin{tikzpicture}[baseline=0.666*1cm-0.25em]
    \def\scp{0.666}
    \def\linksize{\scp*0.075cm}
    \def\pointsize{\scp*0.25cm}
    \def\dd{\scp*0.5cm}
    \def\dx{\scp*1cm}
    \def\cx{\scp*0.3cm}
    \def\txu{7*\dx}    
    \def\txl{5*\dx}
    \def\dy{\scp*1cm}
    \def\cy{\scp*0.3cm}
    \def\ty{3*\dy}
    \def\fy{\scp*0.2cm}
    \def\fx{\scp*0.2cm}
    \def\gy{\scp*0.4cm}
    \def\gx{\scp*0.4cm}      
    \tikzset{whp/.style={circle, inner sep=0pt, text width={\pointsize}, draw=black, fill=white}}
    \tikzset{blp/.style={circle, inner sep=0pt, text width={\pointsize}, draw=black, fill=black}}
    \tikzset{lk/.style={regular polygon, regular polygon sides=4, inner sep=0pt, text width={\linksize}, draw=black, fill=black}}
    \tikzset{cc/.style={cross out, minimum size={1.5*\pointsize-\pgflinewidth}, inner sep=0pt, outer sep=0pt,  draw=black}}    
    \tikzset{vp/.style={circle, inner sep=0pt, text width={1.5*\pointsize}, fill=white}}
    \tikzset{sstr/.style={shorten <= 3pt, shorten >= 3pt}}    
    \draw[dotted] ({0-\dd},{0}) -- ({\txl+\dd},{0});
    \draw[dotted] ({0-\dd},{\ty}) -- ({\txu+\dd},{\ty});
    \node[vp] (l1) at ({0+4*\dx},{0+0*\ty}) {$\overline{c_2}$};
    \node[cc] (l2) at ({0+1*\dx},{0+0*\ty}) {};
    \node[vp] (u1) at ({0+5*\dx},{0+1*\ty}) {$c_1$};
    \node[cc] (u2) at ({0+2*\dx},{0+1*\ty}) {};    
    \draw[sstr] (u1) -- ++ (0,{-1/3*\ty}) -| (l2);
    \draw[dashed] ($(u1)+ (0,{-1/3*\ty})$) --++ ({0.75*\dx},0);
    \draw[dashed] ($(u2)+ (0,{-2/3*\ty})$) --++ ({-0.75*\dx},0);
    \draw[sstr] (u2) -- ++ (0,{-2/3*\ty}) -| (l1);
    \draw[dashed] ($(l2)+ (0,{\ty *2 / 3})$) --++ ({-0.75*\dx},0);
    \draw[dashed] ($(l1)+ (0,{\ty /3})$) --++ ({0.75*\dx},0);
    %
    %%% lower row, left-open, filled
    \draw [draw=gray, pattern = north west lines, pattern color = darkgray] ({0+0*\dx-\fx},{0+0*\dy-\fy}) -- ++ ({0*\dx+2*\fx},0) |- ({0+0*\dx-\fx},{0+0*\dy+\fy});
    %%% lower row, right-open, filled    
     \draw [draw=gray, pattern = north west lines, pattern color = darkgray]    ({\txl+0*\dx+\fx},{0*\dy-\fy}) -- ++ ({0*\dx-2*\fx},0) |- ({\txl+0*\dx+\fx},{0*\dy+\fy});
    %%% lower row, closed, filled
    \draw [draw=gray, pattern = north west lines, pattern color = darkgray] ({2*\dx-\fx},{-\fy}) rectangle ({3*\dx+\fx},{\fy});
    %%% upper row, left-open, filled    
    \draw [draw=gray, pattern = north west lines, pattern color = darkgray] ({0+0*\dx-\fx},{\ty+0*\dy-\fy}) -- ++ ({1*\dx+2*\fx},0) |- ({0+0*\dx-\fx},{\ty+0*\dy+\fy});
    %%% upper row, right-open, filled        
    \draw [draw=gray, pattern = north west lines, pattern color = darkgray] ({\txu+0*\dx+\fx},{\ty+0*\dy-\fy}) -- ++ ({-1*\dx-2*\fx},0) |- ({\txu+0*\dx+\fx},{\ty+0*\dy+\fy});
     %%% upper row, closed, filled
    \draw [draw=gray, pattern = north west lines, pattern color = darkgray] ({3*\dx-\fx},{\ty-\fy}) rectangle ({4*\dx+\fx},{\ty+\fy});
    \node [above = {\cx} of  u1] {$\rho(\alpha_1)$};
    \node [below ={\cx} of  l1] {$\rho(\alpha_2)$};
    %
     %% lower row, left-open
     %\draw [draw=darkgray, densely dashed] ({0+0*\dx-\gx},{0+0*\dy-\gy}) --++ ({1*\dx+2*\gx},0) |- ({0+0*\dx-\gx},{0+0*\dy+\gy}); 
     %% lower row, right-open    
     \draw [draw=darkgray, densely dashed] ({\txl+0*\dx+\gx},{0*\dy-\gy}) --++ ({-1*\dx-2*\gx},0) |- ({\txl+0*\dx+\gx},{0*\dy+\gy});
     %% lower row, open
     %\draw [draw=darkgray, densely dashed] ({-\gx},{-\gy}) -- ({\txl+\gx},{-\gy})  ({\txl+\gx},{\gy}) -- ({-\gx},{\gy});
     %% upper row, left-open
     %\draw [draw=darkgray, densely dashed] ({0+0*\dx-\gx},{\ty+0*\dy-\gy}) --++ ({2*\dx+2*\gx},0) |- ({0+0*\dx-\gx},{\ty+0*\dy+\gy}); 
     %% upper row, right-open
      \draw [draw=darkgray, densely dashed] ({\txu+0*\dx+\gx},{\ty+0*\dy-\gy}) --++ ({-2*\dx-2*\gx},0) |- ({\txu+0*\dx+\gx},{\ty+0*\dy+\gy});
     %% upper row, open
     %\draw [draw=darkgray, densely dashed] ({-\gx},{\ty-\gy}) -- ({\txu+\gx},{\ty-\gy})  ({\txu+\gx},{\ty+\gy}) -- ({-\gx},{\ty+\gy});
    %
    \node at (-\dx,{0.5*\ty}) {$\tilde p$};
  \end{tikzpicture}  
  
\end{mycenter}
\noindent
By the same calculation as in the proof of Lem\-ma~\ref{lemma:verifying-axioms-2}, we obtain $\delta_p(\alpha_1,\alpha_2)=-\delta_{\tilde p}(\rho(\alpha_2),\rho(\alpha_1))$. Hence, $\xi_{c_1,c_2}=X_{c_1,c_2}(\mc C)\subseteq  -X_{\overline{c_2},\overline{c_1}}(\mc C)=-\xi_{\overline{c_2},\overline{c_1}}$ as claimed.
                  \par
                  \emph{Axiom~\ref{lemma:gradedsets-condition-2}:} Lastly, as crossing each other is a symmetric $2$-relation on blocks,
                  
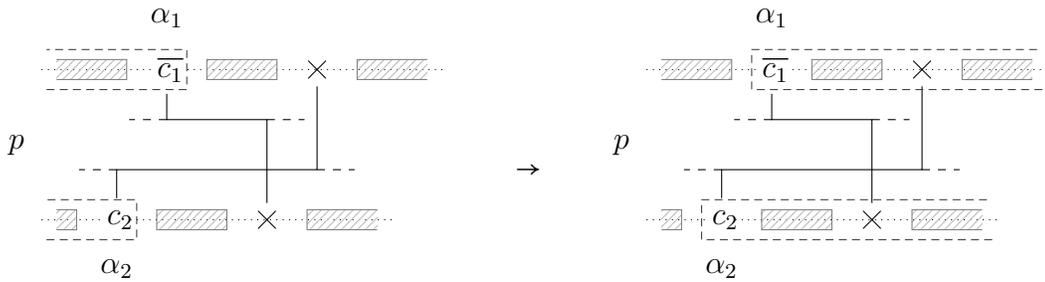
\begin{mycenter}[0.5em]
                    \begin{tikzpicture}[baseline=0.666*1cm-0.25em]
    \def\scp{0.666}
    \def\linksize{\scp*0.075cm}
    \def\pointsize{\scp*0.25cm}
    \def\dd{\scp*0.5cm}
    \def\dx{\scp*1cm}
    \def\cx{\scp*0.3cm}
    \def\txu{7*\dx}    
    \def\txl{6*\dx}
    \def\dy{\scp*1cm}
    \def\cy{\scp*0.3cm}
    \def\ty{3*\dy}
    \def\fy{\scp*0.2cm}
    \def\fx{\scp*0.2cm}
    \def\gy{\scp*0.4cm}
    \def\gx{\scp*0.4cm}      
    \tikzset{whp/.style={circle, inner sep=0pt, text width={\pointsize}, draw=black, fill=white}}
    \tikzset{blp/.style={circle, inner sep=0pt, text width={\pointsize}, draw=black, fill=black}}
    \tikzset{lk/.style={regular polygon, regular polygon sides=4, inner sep=0pt, text width={\linksize}, draw=black, fill=black}}
    \tikzset{cc/.style={cross out, minimum size={1.5*\pointsize-\pgflinewidth}, inner sep=0pt, outer sep=0pt,  draw=black}}    
    \tikzset{vp/.style={circle, inner sep=0pt, text width={1.5*\pointsize}, fill=white}}
    \tikzset{sstr/.style={shorten <= 3pt, shorten >= 3pt}}    
    \draw[dotted] ({0-\dd},{0}) -- ({\txl+\dd},{0});
    \draw[dotted] ({0-\dd},{\ty}) -- ({\txu+\dd},{\ty});
    \node[vp] (l1) at ({0+1*\dx},{0+0*\ty}) {$c_2$};
    \node[cc] (l2) at ({0+4*\dx},{0+0*\ty}) {};
    \node[vp] (u1) at ({0+2*\dx},{0+1*\ty}) {$\overline {c_1}$};
    \node[cc] (u2) at ({0+5*\dx},{0+1*\ty}) {};    
    \draw[sstr] (u1) -- ++ (0,{-1/3*\ty}) -| (l2);
    \draw[dashed] ($(u1)+ (0,{-1/3*\ty})$) --++ ({-0.75*\dx},0);
    \draw[dashed] ($(u2)+ (0,{-2/3*\ty})$) --++ ({0.75*\dx},0);
    \draw[sstr] (u2) -- ++ (0,{-2/3*\ty}) -| (l1);
    \draw[dashed] ($(l2)+ (0,{\ty *2 / 3})$) --++ ({0.75*\dx},0);
    \draw[dashed] ($(l1)+ (0,{\ty /3})$) --++ ({-0.75*\dx},0);
    %
    %%% lower row, left-open, filled
    \draw [draw=gray, pattern = north east lines, pattern color = lightgray] ({0+0*\dx-\fx},{0+0*\dy-\fy}) -- ++ ({0*\dx+2*\fx},0) |- ({0+0*\dx-\fx},{0+0*\dy+\fy});
    %%% lower row, right-open, filled    
     \draw [draw=gray, pattern = north east lines, pattern color = lightgray]    ({\txl+0*\dx+\fx},{0*\dy-\fy}) -- ++ ({-1*\dx-2*\fx},0) |- ({\txl+0*\dx+\fx},{0*\dy+\fy});
    %%% lower row, closed, filled
    \draw [draw=gray, pattern = north east lines, pattern color = lightgray] ({2*\dx-\fx},{-\fy}) rectangle ({3*\dx+\fx},{\fy});
    %%% upper row, left-open, filled    
    \draw [draw=gray, pattern = north east lines, pattern color = lightgray] ({0+0*\dx-\fx},{\ty+0*\dy-\fy}) -- ++ ({1*\dx+2*\fx},0) |- ({0+0*\dx-\fx},{\ty+0*\dy+\fy});
    %%% upper row, right-open, filled        
    \draw [draw=gray, pattern = north east lines, pattern color = lightgray] ({\txu+0*\dx+\fx},{\ty+0*\dy-\fy}) -- ++ ({-1*\dx-2*\fx},0) |- ({\txu+0*\dx+\fx},{\ty+0*\dy+\fy});
     %%% upper row, closed, filled
    \draw [draw=gray, pattern = north east lines, pattern color = lightgray] ({3*\dx-\fx},{\ty-\fy}) rectangle ({4*\dx+\fx},{\ty+\fy});
    \node [above = {\cx} of  u1] {$\alpha_1$};
    \node [below ={\cx} of  l1] {$\alpha_2$};
    %
     %% lower row, left-open
     \draw [draw=darkgray, densely dashed] ({0+0*\dx-\gx},{0+0*\dy-\gy}) --++ ({1*\dx+2*\gx},0) |- ({0+0*\dx-\gx},{0+0*\dy+\gy}); 
     %% lower row, right-open    
     %\draw [draw=darkgray, densely dashed] ({\txl+0*\dx+\gx},{0*\dy-\gy}) --++ ({-3*\dx-2*\gx},0) |- ({\txl+0*\dx+\gx},{0*\dy+\gy});
     %% lower row, open
     %\draw [draw=darkgray, densely dashed] ({-\gx},{-\gy}) -- ({\txl+\gx},{-\gy})  ({\txl+\gx},{\gy}) -- ({-\gx},{\gy});
     %% upper row, left-open
     \draw [draw=darkgray, densely dashed] ({0+0*\dx-\gx},{\ty+0*\dy-\gy}) --++ ({2*\dx+2*\gx},0) |- ({0+0*\dx-\gx},{\ty+0*\dy+\gy}); 
     %% upper row, right-open
     % \draw [draw=darkgray, densely dashed] ({\txu+0*\dx+\gx},{\ty+0*\dy-\gy}) --++ ({-1*\dx-2*\gx},0) |- ({\txu+0*\dx+\gx},{\ty+0*\dy+\gy});
     %% upper row, open
     %\draw [draw=darkgray, densely dashed] ({-\gx},{\ty-\gy}) -- ({\txu+\gx},{\ty-\gy})  ({\txu+\gx},{\ty+\gy}) -- ({-\gx},{\ty+\gy});
    %
    \node at (-\dx,{0.5*\ty}) {$p$};
  \end{tikzpicture}
  \qquad$\to$\qquad
  \begin{tikzpicture}[baseline=0.666*1cm-0.25em]
    \def\scp{0.666}
    \def\linksize{\scp*0.075cm}
    \def\pointsize{\scp*0.25cm}
    \def\dd{\scp*0.5cm}
    \def\dx{\scp*1cm}
    \def\cx{\scp*0.3cm}
    \def\txu{7*\dx}    
    \def\txl{6*\dx}
    \def\dy{\scp*1cm}
    \def\cy{\scp*0.3cm}
    \def\ty{3*\dy}
    \def\fy{\scp*0.2cm}
    \def\fx{\scp*0.2cm}
    \def\gy{\scp*0.4cm}
    \def\gx{\scp*0.4cm}      
    \tikzset{whp/.style={circle, inner sep=0pt, text width={\pointsize}, draw=black, fill=white}}
    \tikzset{blp/.style={circle, inner sep=0pt, text width={\pointsize}, draw=black, fill=black}}
    \tikzset{lk/.style={regular polygon, regular polygon sides=4, inner sep=0pt, text width={\linksize}, draw=black, fill=black}}
    \tikzset{cc/.style={cross out, minimum size={1.5*\pointsize-\pgflinewidth}, inner sep=0pt, outer sep=0pt,  draw=black}}    
    \tikzset{vp/.style={circle, inner sep=0pt, text width={1.5*\pointsize}, fill=white}}
    \tikzset{sstr/.style={shorten <= 3pt, shorten >= 3pt}}    
    \draw[dotted] ({0-\dd},{0}) -- ({\txl+\dd},{0});
    \draw[dotted] ({0-\dd},{\ty}) -- ({\txu+\dd},{\ty});
    \node[vp] (l1) at ({0+1*\dx},{0+0*\ty}) {$c_2$};
    \node[cc] (l2) at ({0+4*\dx},{0+0*\ty}) {};
    \node[vp] (u1) at ({0+2*\dx},{0+1*\ty}) {$\overline {c_1}$};
    \node[cc] (u2) at ({0+5*\dx},{0+1*\ty}) {};    
    \draw[sstr] (u1) -- ++ (0,{-1/3*\ty}) -| (l2);
    \draw[dashed] ($(u1)+ (0,{-1/3*\ty})$) --++ ({-0.75*\dx},0);
    \draw[dashed] ($(u2)+ (0,{-2/3*\ty})$) --++ ({0.75*\dx},0);
    \draw[sstr] (u2) -- ++ (0,{-2/3*\ty}) -| (l1);
    \draw[dashed] ($(l2)+ (0,{\ty *2 / 3})$) --++ ({0.75*\dx},0);
    \draw[dashed] ($(l1)+ (0,{\ty /3})$) --++ ({-0.75*\dx},0);
    %
    %%% lower row, left-open, filled
    \draw [draw=gray, pattern = north east lines, pattern color = lightgray] ({0+0*\dx-\fx},{0+0*\dy-\fy}) -- ++ ({0*\dx+2*\fx},0) |- ({0+0*\dx-\fx},{0+0*\dy+\fy});
    %%% lower row, right-open, filled    
     \draw [draw=gray, pattern = north east lines, pattern color = lightgray]    ({\txl+0*\dx+\fx},{0*\dy-\fy}) -- ++ ({-1*\dx-2*\fx},0) |- ({\txl+0*\dx+\fx},{0*\dy+\fy});
    %%% lower row, closed, filled
    \draw [draw=gray, pattern = north east lines, pattern color = lightgray] ({2*\dx-\fx},{-\fy}) rectangle ({3*\dx+\fx},{\fy});
    %%% upper row, left-open, filled    
    \draw [draw=gray, pattern = north east lines, pattern color = lightgray] ({0+0*\dx-\fx},{\ty+0*\dy-\fy}) -- ++ ({1*\dx+2*\fx},0) |- ({0+0*\dx-\fx},{\ty+0*\dy+\fy});
    %%% upper row, right-open, filled        
    \draw [draw=gray, pattern = north east lines, pattern color = lightgray] ({\txu+0*\dx+\fx},{\ty+0*\dy-\fy}) -- ++ ({-1*\dx-2*\fx},0) |- ({\txu+0*\dx+\fx},{\ty+0*\dy+\fy});
     %%% upper row, closed, filled
    \draw [draw=gray, pattern = north east lines, pattern color = lightgray] ({3*\dx-\fx},{\ty-\fy}) rectangle ({4*\dx+\fx},{\ty+\fy});
    \node [above = {\cx} of  u1] {$\alpha_1$};
    \node [below ={\cx} of  l1] {$\alpha_2$};
    %
     %% lower row, left-open
     %\draw [draw=darkgray, densely dashed] ({0+0*\dx-\gx},{0+0*\dy-\gy}) --++ ({1*\dx+2*\gx},0) |- ({0+0*\dx-\gx},{0+0*\dy+\gy}); 
     %% lower row, right-open    
     \draw [draw=darkgray, densely dashed] ({\txl+0*\dx+\gx},{0*\dy-\gy}) --++ ({-5*\dx-2*\gx},0) |- ({\txl+0*\dx+\gx},{0*\dy+\gy});
     %% lower row, open
     %\draw [draw=darkgray, densely dashed] ({-\gx},{-\gy}) -- ({\txl+\gx},{-\gy})  ({\txl+\gx},{\gy}) -- ({-\gx},{\gy});
     %% upper row, left-open
     %\draw [draw=darkgray, densely dashed] ({0+0*\dx-\gx},{\ty+0*\dy-\gy}) --++ ({2*\dx+2*\gx},0) |- ({0+0*\dx-\gx},{\ty+0*\dy+\gy}); 
     %% upper row, right-open
      \draw [draw=darkgray, densely dashed] ({\txu+0*\dx+\gx},{\ty+0*\dy-\gy}) --++ ({-5*\dx-2*\gx},0) |- ({\txu+0*\dx+\gx},{\ty+0*\dy+\gy});
     %% upper row, open
     %\draw [draw=darkgray, densely dashed] ({-\gx},{\ty-\gy}) -- ({\txu+\gx},{\ty-\gy})  ({\txu+\gx},{\ty+\gy}) -- ({-\gx},{\ty+\gy});
    %
    \node at (-\dx,{0.5*\ty}) {$p$};
  \end{tikzpicture}
\end{mycenter}
\noindent
and because, \(\delta_p(\alpha_1,\alpha_2)\equiv -\delta_{p}(\alpha_2,\alpha_1)\mod \toco(p)\) (by \cite[Lem\-ma~2.1~(b)]{MWNHO1}), we can immediately conclude $\delta_p(\alpha_1,\alpha_2)\in -X_{c_2,c_1}(\mc C)+\toco(\mc C)$. Thus, $\xi_{c_1,c_2}\subseteq -\xi_{c_2,c_1}+\sigma$. Differently from 
Lem\-ma~\ref{lemma:verifying-axioms-2}, we did not need Lem\-ma~\ref{lemma:simplification-k-x} to see this.
\end{proof}

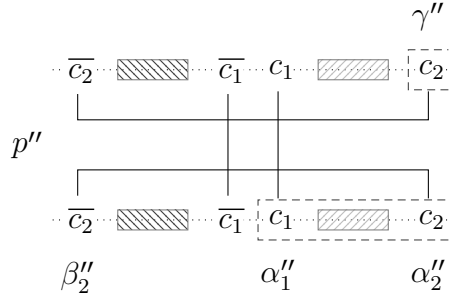
\begin{lemma}
  \label{lemma:verifying-axioms-4}
  For every non-hy\-per\-octa\-he\-dral category $\mc C\subseteq \Cp$, the sets $\sigma\eqpd \toco(\mc C)$ and $\xi_{c_1,c_2}\eqpd X_{c_1,c_2}(\mc C)$ for $c_1,c_2\in \colors$ satisfy Axiom~\ref{lemma:gradedsets-condition-6} of~\ref{axioms:arithmetic}:  For all $c_1,c_2\in\colors$,         
  \begin{IEEEeqnarray*}{rCl}
    \xi_{c_1,c_2}\subseteq \xi_{c_1,\overline{c_2}}\cup \left(-\xi_{c_2,\overline{c_1}}+\sigma\right).
  \end{IEEEeqnarray*}
\end{lemma}
\begin{proof}
  Let $B_1$ and $B_2$ be crossing blocks in $p\in \mc C\cap \Cp_{\leq 2}$ and let $\alpha_1\in B_1$ and $\alpha_2\in B_2$ have normalized colors $c_1\in \colors$ and $c_2\in \colors$, respectively. According to Lem\-ma~\ref{lemma:simplification-k-x} then, every element  of $\xi_{c_1,c_2}=X_{c_1,c_2}(\mc C)=X_{c_1,c_2}(\mc C\cap \Cp_{\leq 2})$ is of the form  $\delta_p(\alpha_1,\alpha_2)$. Because $p\in \Cp_{\leq 2}$, the blocks $B_1$ and $B_2$ are pairs. Hence, the crossing between these blocks means that we find points $\beta_1\in B_1$ and $\beta_2\in B_2$ with $\alpha_1\neq \beta_1$ and $\alpha_2\neq \beta_2$ such that either $(\alpha_1,\alpha_2,\beta_1,\beta_2)$ or $(\alpha_2,\alpha_1,\beta_2,\beta_1)$ is ordered.
                    \par
                    \emph{Case~1:} First, we suppose that $(\alpha_1,\alpha_2,\beta_1,\beta_2)$ is ordered and show $\delta_p(\alpha_1,\alpha_2)\in X_{c_1,\overline{c_2}}(\mc C)$. We can assume that $\alpha_1$ is the leftmost and $\beta_1$ the rightmost lower  point.
                       \begin{mycenter}[0.5em]
\begin{tikzpicture}[baseline=0.666*1cm-0.25em]
    \def\scp{0.666}
    \def\linksize{\scp*0.075cm}
    \def\pointsize{\scp*0.25cm}
    \def\dd{\scp*0.5cm}
    \def\dx{\scp*1cm}
    \def\cx{\scp*0.3cm}
    \def\txu{8*\dx}    
    \def\txl{7*\dx}
    \def\dy{\scp*1cm}
    \def\cy{\scp*0.3cm}
    \def\ty{3*\dy}
    \def\fy{\scp*0.2cm}
    \def\fx{\scp*0.2cm}
    \def\gy{\scp*0.4cm}
    \def\gx{\scp*0.4cm}      
    \tikzset{whp/.style={circle, inner sep=0pt, text width={\pointsize}, draw=black, fill=white}}
    \tikzset{blp/.style={circle, inner sep=0pt, text width={\pointsize}, draw=black, fill=black}}
    \tikzset{lk/.style={regular polygon, regular polygon sides=4, inner sep=0pt, text width={\linksize}, draw=black, fill=black}}
    \tikzset{cc/.style={cross out, minimum size={1.5*\pointsize-\pgflinewidth}, inner sep=0pt, outer sep=0pt,  draw=black}}    
    \tikzset{vp/.style={circle, inner sep=0pt, text width={1.5*\pointsize}, fill=white}}
    \tikzset{sstr/.style={shorten <= 3pt, shorten >= 3pt}}    
    \draw[dotted] ({0-\dd},{0}) -- ({\txl+\dd},{0});
    \draw[dotted] ({0-\dd},{\ty}) -- ({\txu+\dd},{\ty});
    \node[vp] (l1) at ({0+0*\dx},{0+0*\ty}) {$c_1$};
    \node[vp] (l2) at ({0+3*\dx},{0+0*\ty}) {$c_2$};
    \node[cc] (l3) at ({0+7*\dx},{0+0*\ty}) {};    
    \node[cc] (u1) at ({0+4*\dx},{0+1*\ty}) {};    
    \draw[sstr] (l1) --++(0,{\ty /3}) -| (l3);
    \draw[sstr] (l2) --++(0,{\ty *2/3}) -| (u1);    
    %
    %%% lower row, left-open, filled
    %\draw [draw=gray, pattern = north east lines, pattern color = lightgray] ({0+0*\dx-\fx},{0+0*\dy-\fy}) -- ++ ({0*\dx+2*\fx},0) |- ({0+0*\dx-\fx},{0+0*\dy+\fy});
    %%% lower row, right-open, filled    
    %\draw [draw=gray, pattern = north east lines, pattern color = lightgray]    ({\txl+0*\dx+\fx},{0*\dy-\fy}) -- ++ ({-0*\dx-2*\fx},0) |- ({\txl+0*\dx+\fx},{0*\dy+\fy});
    %%% lower row, closed, filled
    \draw [draw=gray, pattern = north east lines, pattern color = lightgray] ({1*\dx-\fx},{-\fy}) rectangle ({2*\dx+\fx},{\fy});
    \draw [draw=gray, pattern = north east lines, pattern color = lightgray] ({4*\dx-\fx},{-\fy}) rectangle ({6*\dx+\fx},{\fy});    
    %%% upper row, left-open, filled    
    %\draw [draw=gray, pattern = north east lines, pattern color = lightgray] ({0+0*\dx-\fx},{\ty+0*\dy-\fy}) -- ++ ({3*\dx+2*\fx},0) |- ({0+0*\dx-\fx},{\ty+0*\dy+\fy});
    %%% upper row, right-open, filled        
    %\draw [draw=gray, pattern = north east lines, pattern color = lightgray] ({\txu+0*\dx+\fx},{\ty+0*\dy-\fy}) -- ++ ({-3*\dx-2*\fx},0) |- ({\txu+0*\dx+\fx},{\ty+0*\dy+\fy});
    %%% upper row, closed, filled
    \draw [draw=gray, pattern = north east lines, pattern color = lightgray] ({0*\dx-\fx},{\ty-\fy}) rectangle ({3*\dx+\fx},{\ty+\fy});
    \draw [draw=gray, pattern = north east lines, pattern color = lightgray] ({5*\dx-\fx},{\ty-\fy}) rectangle ({8*\dx+\fx},{\ty+\fy});    
    \node [below = {\cx} of  l1] {$\alpha_1$};
    \node [below ={\cx} of  l2] {$\alpha_2$};
    \node [below ={\cx} of  l3] {$\beta_1$};
    \node [above ={\cx} of  u1] {$\beta_2$};        
    %
     %% lower row, left-open
     %\draw [draw=darkgray, densely dashed] ({0+0*\dx-\gx},{0+0*\dy-\gy}) --++ ({1*\dx+2*\gx},0) |- ({0+0*\dx-\gx},{0+0*\dy+\gy}); 
     %% lower row, right-open    
     %\draw [draw=darkgray, densely dashed] ({\txl+0*\dx+\gx},{0*\dy-\gy}) --++ ({-3*\dx-2*\gx},0) |- ({\txl+0*\dx+\gx},{0*\dy+\gy});
     %% lower row, open
     % \draw [draw=darkgray, densely dashed] ({-\gx},{-\gy}) -- ({\txl+\gx},{-\gy})  ({\txl+\gx},{\gy}) -- ({-\gx},{\gy});
     % lower row, closed    
    \draw [draw=darkgray, densely dashed] ({0*\dx-\gx},{-\gy}) rectangle ({3*\dx+\gx},{\gy});            
     %% upper row, left-open
     %\draw [draw=darkgray, densely dashed] ({0+0*\dx-\gx},{\ty+0*\dy-\gy}) --++ ({3*\dx+2*\gx},0) |- ({0+0*\dx-\gx},{\ty+0*\dy+\gy}); 
     %% upper row, right-open
      %\draw [draw=darkgray, densely dashed] ({\txu+0*\dx+\gx},{\ty+0*\dy-\gy}) --++ ({-3*\dx-2*\gx},0) |- ({\txu+0*\dx+\gx},{\ty+0*\dy+\gy});
     %% upper row, open
     %\draw [draw=darkgray, densely dashed] ({-\gx},{\ty-\gy}) -- ({\txu+\gx},{\ty-\gy})  ({\txu+\gx},{\ty+\gy}) -- ({-\gx},{\ty+\gy});
    %
    \node at (-\dx,{0.5*\ty}) {$p$};
  \end{tikzpicture}
  \end{mycenter}
  By Lem\-ma~\ref{lemma:projection}, the partition $p'\eqpd P(p,[\alpha_1,\beta_1]_p)$ belongs to $\mc C$. The definition of the projection operation has the following consequences: The three lower points $\alpha_1,\alpha_2$ and $\beta_1$ of $p$, also points of $p'$, all retain their normalized colors in $p'$; the set $B_1=\{\alpha_1,\beta_1\}$ is still a block of $p'$; the point $\alpha_2$ is now connected to its counterpart $\beta_2'$ on the upper row of $p'$, implying in particular that the blocks of $\alpha_1$ and $\alpha_2$ still cross in $p'$; and it holds \[\delta_{p'}(\alpha_1,\alpha_2)=\delta_p(\alpha_1,\alpha_2).\]
                       \begin{mycenter}[0.5em]
\begin{tikzpicture}[baseline=0.666*1cm-0.25em]
    \def\scp{0.666}
    \def\linksize{\scp*0.075cm}
    \def\pointsize{\scp*0.25cm}
    \def\dd{\scp*0.5cm}
    \def\dx{\scp*1cm}
    \def\cx{\scp*0.3cm}
    \def\txu{7*\dx}    
    \def\txl{7*\dx}
    \def\dy{\scp*1cm}
    \def\cy{\scp*0.3cm}
    \def\ty{3*\dy}
    \def\fy{\scp*0.2cm}
    \def\fx{\scp*0.2cm}
    \def\gy{\scp*0.4cm}
    \def\gx{\scp*0.4cm}      
    \tikzset{whp/.style={circle, inner sep=0pt, text width={\pointsize}, draw=black, fill=white}}
    \tikzset{blp/.style={circle, inner sep=0pt, text width={\pointsize}, draw=black, fill=black}}
    \tikzset{lk/.style={regular polygon, regular polygon sides=4, inner sep=0pt, text width={\linksize}, draw=black, fill=black}}
    \tikzset{cc/.style={cross out, minimum size={1.5*\pointsize-\pgflinewidth}, inner sep=0pt, outer sep=0pt,  draw=black}}    
    \tikzset{vp/.style={circle, inner sep=0pt, text width={1.5*\pointsize}, fill=white}}
    \tikzset{sstr/.style={shorten <= 3pt, shorten >= 3pt}}    
    \draw[dotted] ({0-\dd},{0}) -- ({\txl+\dd},{0});
    \draw[dotted] ({0-\dd},{\ty}) -- ({\txu+\dd},{\ty});
    \node[vp] (l1) at ({0+0*\dx},{0+0*\ty}) {$c_1$};
    \node[vp] (l2) at ({0+3*\dx},{0+0*\ty}) {$c_2$};
    \node[cc] (l3) at ({0+7*\dx},{0+0*\ty}) {};    
    \node[vp] (u1) at ({0+0*\dx},{0+1*\ty}) {$c_1$};
    \node[vp] (u2) at ({0+3*\dx},{0+1*\ty}) {$c_2$};
    \node[cc] (u3) at ({0+7*\dx},{0+1*\ty}) {};    
    \draw[sstr] (l1) --++(0,{\dy}) -| (l3);
    \draw[sstr] (l2) to (u2);
    \draw[sstr] (u1) --++(0,{-\dy}) -| (u3);    
    %
    %%% lower row, left-open, filled
    %\draw [draw=gray, pattern = north east lines, pattern color = lightgray] ({0+0*\dx-\fx},{0+0*\dy-\fy}) -- ++ ({0*\dx+2*\fx},0) |- ({0+0*\dx-\fx},{0+0*\dy+\fy});
    %%% lower row, right-open, filled    
    %\draw [draw=gray, pattern = north east lines, pattern color = lightgray]    ({\txl+0*\dx+\fx},{0*\dy-\fy}) -- ++ ({-0*\dx-2*\fx},0) |- ({\txl+0*\dx+\fx},{0*\dy+\fy});
    %%% lower row, closed, filled
    \draw [draw=gray, pattern = north east lines, pattern color = lightgray] ({1*\dx-\fx},{-\fy}) rectangle ({2*\dx+\fx},{\fy});
    \draw [draw=gray, pattern = north east lines, pattern color = lightgray] ({4*\dx-\fx},{-\fy}) rectangle ({6*\dx+\fx},{\fy});    
    %%% upper row, left-open, filled    
    %\draw [draw=gray, pattern = north east lines, pattern color = lightgray] ({0+0*\dx-\fx},{\ty+0*\dy-\fy}) -- ++ ({3*\dx+2*\fx},0) |- ({0+0*\dx-\fx},{\ty+0*\dy+\fy});
    %%% upper row, right-open, filled        
    %\draw [draw=gray, pattern = north east lines, pattern color = lightgray] ({\txu+0*\dx+\fx},{\ty+0*\dy-\fy}) -- ++ ({-3*\dx-2*\fx},0) |- ({\txu+0*\dx+\fx},{\ty+0*\dy+\fy});
    %%% upper row, closed, filled
    \draw [draw=gray, pattern = north east lines, pattern color = lightgray] ({1*\dx-\fx},{\ty-\fy}) rectangle ({2*\dx+\fx},{\ty+\fy});
    \draw [draw=gray, pattern = north east lines, pattern color = lightgray] ({4*\dx-\fx},{\ty-\fy}) rectangle ({6*\dx+\fx},{\ty+\fy});    
    \node [below = {\cx} of  l1] {$\alpha_1$};
    \node [below ={\cx} of  l2] {$\alpha_2$};
    \node [below ={\cx} of  l3] {$\beta_1$};
    \node [above ={\cx} of  u2] {$\beta_2'$};        
    %
     %% lower row, left-open
     \draw [draw=darkgray, densely dotted] ({0+0*\dx-\gx},{0+0*\dy-\gy}) --++ ({3*\dx+2*\gx},0) |- ({0+0*\dx-\gx},{0+0*\dy+\gy}); 
     %% lower row, right-open    
     %\draw [draw=darkgray, densely dashed] ({\txl+0*\dx+\gx},{0*\dy-\gy}) --++ ({-3*\dx-2*\gx},0) |- ({\txl+0*\dx+\gx},{0*\dy+\gy});
     %% lower row, open
     %\draw [draw=darkgray, densely dashed] ({-\gx},{-\gy}) -- ({\txl+\gx},{-\gy})  ({\txl+\gx},{\gy}) -- ({-\gx},{\gy});
     %% upper row, left-open
     \draw [draw=darkgray, densely dotted] ({0+0*\dx-\gx},{\ty+0*\dy-\gy}) --++ ({3*\dx+2*\gx},0) |- ({0+0*\dx-\gx},{\ty+0*\dy+\gy}); 
     %% upper row, right-open
      %\draw [draw=darkgray, densely dashed] ({\txu+0*\dx+\gx},{\ty+0*\dy-\gy}) --++ ({-3*\dx-2*\gx},0) |- ({\txu+0*\dx+\gx},{\ty+0*\dy+\gy});
     %% upper row, open
     %\draw [draw=darkgray, densely dashed] ({-\gx},{\ty-\gy}) -- ({\txu+\gx},{\ty-\gy})  ({\txu+\gx},{\ty+\gy}) -- ({-\gx},{\ty+\gy});
    %
     \node at ({\txl+\dx},{0.5*\ty}) {$p'$};
     \node (lab1) at ({-2.5*\dx},{0.5*\ty}) {$[\beta_2',\alpha]_{p'}$};
     \draw[dotted] (lab1) -- ({-\dd-\cx},0);
     \draw[dotted] (lab1) -- ({-\dd-\cx},{\ty});     
  \end{tikzpicture}
  \end{mycenter}
  We  apply Lem\-ma~\ref{lemma:projection} a second time to infer $p''\eqpd P(p',[\beta_2',\alpha_2])\in \mc C$. Denote the images of the points $\beta_2'$, $\alpha_1$ and $\alpha_2$ of $p'$ in $p''$ by $\beta_2''$, $\alpha_1''$ and $\alpha_2''$, respectively. Now, $\beta_2''$ is the leftmost lower point and $\alpha_2''$ the rightmost lower point of $p''$ and the two form a block; the point $\alpha_1''\in [\beta_2'',\alpha_2'']_{p''}$ is connected to its counterpart on the upper row; and \[\delta_{p''}(\alpha_1'',\alpha_2'')=\delta_{p'}(\alpha_1,\alpha_2)=\delta_p(\alpha_1,\alpha_2).\]

                       \begin{mycenter}[0.5em]
\begin{tikzpicture}[baseline=0.666*1cm-0.25em]
    \def\scp{0.666}
    \def\linksize{\scp*0.075cm}
    \def\pointsize{\scp*0.25cm}
    \def\dd{\scp*0.5cm}
    \def\dx{\scp*1cm}
    \def\cx{\scp*0.3cm}
    \def\txu{7*\dx}    
    \def\txl{7*\dx}
    \def\dy{\scp*1cm}
    \def\cy{\scp*0.3cm}
    \def\ty{3*\dy}
    \def\fy{\scp*0.2cm}
    \def\fx{\scp*0.2cm}
    \def\gy{\scp*0.4cm}
    \def\gx{\scp*0.4cm}      
    \tikzset{whp/.style={circle, inner sep=0pt, text width={\pointsize}, draw=black, fill=white}}
    \tikzset{blp/.style={circle, inner sep=0pt, text width={\pointsize}, draw=black, fill=black}}
    \tikzset{lk/.style={regular polygon, regular polygon sides=4, inner sep=0pt, text width={\linksize}, draw=black, fill=black}}
    \tikzset{cc/.style={cross out, minimum size={1.5*\pointsize-\pgflinewidth}, inner sep=0pt, outer sep=0pt,  draw=black}}    
    \tikzset{vp/.style={circle, inner sep=0pt, text width={1.5*\pointsize}, fill=white}}
    \tikzset{sstr/.style={shorten <= 3pt, shorten >= 3pt}}    
    \draw[dotted] ({0-\dd},{0}) -- ({\txl+\dd},{0});
    \draw[dotted] ({0-\dd},{\ty}) -- ({\txu+\dd},{\ty});
    \node[vp] (l1) at ({0+0*\dx},{0+0*\ty}) {$\overline{c_2}$};
    \node[vp] (l2) at ({0+3*\dx},{0+0*\ty}) {$\overline{c_1}$};
    \node[vp] (l3) at ({0+4*\dx},{0+0*\ty}) {$c_1$};
    \node[vp] (l4) at ({0+7*\dx},{0+0*\ty}) {$c_2$};        
    \node[vp] (u1) at ({0+0*\dx},{0+1*\ty}) {$\overline{c_2}$};
    \node[vp] (u2) at ({0+3*\dx},{0+1*\ty}) {$\overline{c_1}$};
    \node[vp] (u3) at ({0+4*\dx},{0+1*\ty}) {$c_1$};
    \node[vp] (u4) at ({0+7*\dx},{0+1*\ty}) {$c_2$};        
    \draw[sstr] (l1) --++(0,{\dy}) -| (l4);
    \draw[sstr] (l2) to (u2);
    \draw[sstr] (l3) to (u3);    
    \draw[sstr] (u1) --++(0,{-\dy}) -| (u4);    
    %
    %%% lower row, left-open, filled
    %\draw [draw=gray, pattern = north east lines, pattern color = lightgray] ({0+0*\dx-\fx},{0+0*\dy-\fy}) -- ++ ({0*\dx+2*\fx},0) |- ({0+0*\dx-\fx},{0+0*\dy+\fy});
    %%% lower row, right-open, filled    
    %\draw [draw=gray, pattern = north east lines, pattern color = lightgray]    ({\txl+0*\dx+\fx},{0*\dy-\fy}) -- ++ ({-0*\dx-2*\fx},0) |- ({\txl+0*\dx+\fx},{0*\dy+\fy});
    %%% lower row, closed, filled
    \draw [draw=gray, pattern = north west lines, pattern color = darkgray] ({1*\dx-\fx},{-\fy}) rectangle ({2*\dx+\fx},{\fy});
    \draw [draw=gray, pattern = north east lines, pattern color = lightgray] ({5*\dx-\fx},{-\fy}) rectangle ({6*\dx+\fx},{\fy});    
    %%% upper row, left-open, filled    
    %\draw [draw=gray, pattern = north east lines, pattern color = lightgray] ({0+0*\dx-\fx},{\ty+0*\dy-\fy}) -- ++ ({3*\dx+2*\fx},0) |- ({0+0*\dx-\fx},{\ty+0*\dy+\fy});
    %%% upper row, right-open, filled        
    %\draw [draw=gray, pattern = north east lines, pattern color = lightgray] ({\txu+0*\dx+\fx},{\ty+0*\dy-\fy}) -- ++ ({-3*\dx-2*\fx},0) |- ({\txu+0*\dx+\fx},{\ty+0*\dy+\fy});
    %%% upper row, closed, filled
    \draw [draw=gray, pattern = north west lines, pattern color = darkgray] ({1*\dx-\fx},{\ty-\fy}) rectangle ({2*\dx+\fx},{\ty+\fy});    
    \draw [draw=gray, pattern = north east lines, pattern color = lightgray] ({5*\dx-\fx},{\ty-\fy}) rectangle ({6*\dx+\fx},{\ty+\fy});
 %   \draw [draw=gray, pattern = north east lines, pattern color = lightgray] ({4*\dx-\fx},{\ty-\fy}) rectangle ({6*\dx+\fx},{\ty+\fy});    
    %
    \node [below = {\cx} of  l1] {$\beta_2''$};
    \node [below ={\cx} of  l3] {$\alpha_1''$};
    \node [below ={\cx} of  l4] {$\alpha_2''$};
    \node [above ={\cx} of  u4] {$\gamma''$};        
    %
     %% lower row, left-open
%     \draw [draw=darkgray, densely dotted] ({0+0*\dx-\gx},{0+0*\dy-\gy}) --++ ({3*\dx+2*\gx},0) |- ({0+0*\dx-\gx},{0+0*\dy+\gy}); 
     %% lower row, right-open    
     \draw [draw=darkgray, densely dashed] ({\txl+0*\dx+\gx},{0*\dy-\gy}) --++ ({-3*\dx-2*\gx},0) |- ({\txl+0*\dx+\gx},{0*\dy+\gy});
     %% lower row, open
     % \draw [draw=darkgray, densely dashed] ({-\gx},{-\gy}) -- ({\txl+\gx},{-\gy})  ({\txl+\gx},{\gy}) -- ({-\gx},{\gy});
     %% lower row, closed    
    %\draw [draw=darkgray, densely dashed] ({4*\dx-\gx},{-\gy}) rectangle ({7*\dx+\gx},{\gy});        
     %% upper row, left-open
%     \draw [draw=darkgray, densely dotted] ({0+0*\dx-\gx},{\ty+0*\dy-\gy}) --++ ({3*\dx+2*\gx},0) |- ({0+0*\dx-\gx},{\ty+0*\dy+\gy}); 
     %% upper row, right-open
     \draw [draw=darkgray, densely dashed] ({\txu+0*\dx+\gx},{\ty+0*\dy-\gy}) --++ ({-0*\dx-2*\gx},0) |- ({\txu+0*\dx+\gx},{\ty+0*\dy+\gy});
     %% upper row, open
     % \draw [draw=darkgray, densely dashed] ({-\gx},{\ty-\gy}) -- ({\txu+\gx},{\ty-\gy})  ({\txu+\gx},{\ty+\gy}) -- ({-\gx},{\ty+\gy});

    %
     \node at ({-\dx},{0.5*\ty}) {$p''$};
  \end{tikzpicture}
  \end{mycenter}
  There are two crucial observations to make about the successor $\gamma''$ of $\alpha_2''$ in $p''$, the rightmost upper point of $p''$. Firstly, $\gamma''$ has the inverse normalized color $\overline{c_2}$ of $\alpha_2''$, which in particular implies that $\delta_{p''}(\alpha_2'',\gamma'')=0$. Secondly, $\gamma''$ forms a block of $p''$ together with the leftmost upper point of $p''$, which entails that its block crosses the block of $\alpha_1''$ in $p''$.  Hence, $\delta_{p''}(\alpha_1'',\gamma'')\in X_{c_1,\overline{c_2}}(\mc C)$ and
                    \begin{align*}
                      \delta_{p''}(\alpha_1'',\gamma'')=\delta_{p''}(\alpha_1'',\alpha_2'')+\delta_{p''}(\alpha_2'',\gamma'')=\delta_p(\alpha_1,\alpha_2)
                    \end{align*}
                    together show $\delta_p(\alpha_1,\alpha_2)\in X_{c_1,\overline{c_2}}(\mc C)= \xi_{c_1,\overline{c_2}}$, which is what we set out to prove.
                    \par
                    \emph{Case~2:} Now, let $(\alpha_2,\alpha_1,\beta_2,\beta_1)$ be ordered instead. By Case~1 then, $\delta_p(\alpha_2,\alpha_1)\in X_{c_2,\overline{c_1}}(\mc C)$.   \cite[Lem\-ma~2.1~(b)]{MWNHO1} shows $\delta_p(\alpha_2,\alpha_1)\equiv -\delta_p(\alpha_1,\alpha_2)\mod \toco(p)$. That implies $\delta_p(\alpha_1,\alpha_2)\in -X_{c_2,\overline{c_1}}(\mc C)+\toco(\mc C)=-\xi_{c_2,\overline{c_1}}+ \sigma$, which is what we needed to see.
                  \end{proof}

                  \begin{lemma}
                    \label{lemma:verifying-axioms-5}
                    For every non-hy\-per\-octa\-he\-dral category $\mc C\subseteq \Cp$, the sets $\sigma\eqpd \toco(\mc C)$ and $\kappa_{c_1,c_2}\eqpd K_{c_1,c_2}(\mc C)$ for $c_1,c_2\in \colors$ satisfy Axiom~\ref{lemma:gradedsets-condition-4} of~\ref{axioms:arithmetic}: $0\in\kappa_{\circ\bullet}\cap \kappa_{\bullet\circ}$.
                  \end{lemma}
                  \begin{proof}
                    Since $\PartIdenLoWB\in \mc C$ and $K_{\circ\bullet}(\{\PartIdenLoWB\})=K_{\bullet\circ}(\{\PartIdenLoWB\})=\{0\}$, this is clear.
                  \end{proof}

                  \begin{lemma}
                    \label{lemma:verifying-axioms-6}
                    For every non-hy\-per\-octa\-he\-dral category $\mc C\subseteq \Cp$, the sets $\sigma\eqpd \toco(\mc C)$ and $\kappa_{c_1,c_2}\eqpd K_{c_1,c_2}(\mc C)$ for $c_1,c_2\in \colors$ satisfy Axiom~\ref{lemma:gradedsets-condition-5} of~\ref{axioms:arithmetic}:
                    \begin{IEEEeqnarray*}{rCl}
                      \kappa_{c_1,c_2}+\kappa_{\overline{c_2},c_3}\subseteq \kappa_{c_1,c_3}
                    \end{IEEEeqnarray*}
                    for all $c_1,c_2,c_3\in\colors$.
                  \end{lemma}
                  \begin{proof}
                    Let $c_1,c_2,c_3\in \colors$ be arbitrary and let $\eta_1$ and $\eta_2$ be distinct points of the same block $B$ of $p\in \mc C$ such that $]\eta_1,\eta_2[_p\cap B=\emptyset$ and such that $\eta_i$ has normalized color $c_i$ in $p$ for every $i\in \{1,2\}$. Furthermore, let $\theta_1$ and $\theta_2$ be distinct points of the same block $C$ of $q\in \mc C$ with $]\theta_1,\theta_2[_q\cap C=\emptyset$ such that $\theta_1$ has normalized color $\overline{c_2}$ in $q$ and $\theta_2$ normalized color $c_3$. None of these assumptions are impacted and neither $\delta_p(\eta_1,\eta_2)$ nor $\delta_q(\theta_1,\theta_2)$ altered by assuming that $\eta_2$ is the rightmost lower point of $p$ and $\theta_1$ the leftmost lower point of $q$. 
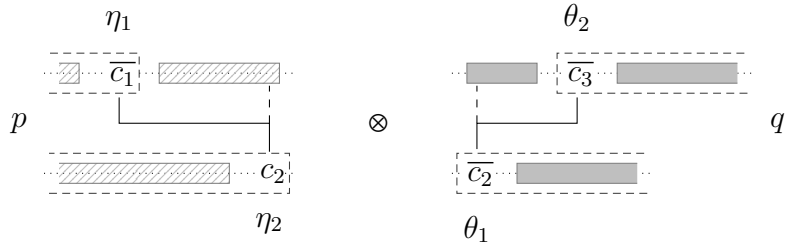
                       \begin{mycenter}[0.5em]
\begin{tikzpicture}[baseline=0.666*1cm-0.25em]
    \def\scp{0.666}
    \def\linksize{\scp*0.075cm}
    \def\pointsize{\scp*0.25cm}
    \def\dd{\scp*0.5cm}
    \def\dx{\scp*1cm}
    \def\cx{\scp*0.3cm}
    \def\txu{4*\dx}    
    \def\txl{4*\dx}
    \def\dy{\scp*1cm}
    \def\cy{\scp*0.3cm}
    \def\ty{2*\dy}
    \def\fy{\scp*0.2cm}
    \def\fx{\scp*0.2cm}
    \def\gy{\scp*0.4cm}
    \def\gx{\scp*0.4cm}      
    \tikzset{whp/.style={circle, inner sep=0pt, text width={\pointsize}, draw=black, fill=white}}
    \tikzset{blp/.style={circle, inner sep=0pt, text width={\pointsize}, draw=black, fill=black}}
    \tikzset{lk/.style={regular polygon, regular polygon sides=4, inner sep=0pt, text width={\linksize}, draw=black, fill=black}}
    \tikzset{cc/.style={cross out, minimum size={1.5*\pointsize-\pgflinewidth}, inner sep=0pt, outer sep=0pt,  draw=black}}    
    \tikzset{vp/.style={circle, inner sep=0pt, text width={1.5*\pointsize}, fill=white}}
    \tikzset{sstr/.style={shorten <= 3pt, shorten >= 3pt}}    
    \draw[dotted] ({0-\dd},{0}) -- ({\txl+\dd},{0});
    \draw[dotted] ({0-\dd},{\ty}) -- ({\txu+\dd},{\ty});
    \node[vp] (l1) at ({0+4*\dx},{0+0*\ty}) {$c_2$};
    \node[vp] (u1) at ({0+1*\dx},{0+1*\ty}) {$\overline{c_1}$};
    \draw[sstr] (l1) --++(0,{\dy})-| (u1);
    \draw[dashed] ($(l1) +(0,{\dy})$) --++(0,{0.75*\dy});
    %
    %%% lower row, left-open, filled
    \draw [draw=gray, pattern = north east lines, pattern color = lightgray] ({0+0*\dx-\fx},{0+0*\dy-\fy}) -- ++ ({3*\dx+2*\fx},0) |- ({0+0*\dx-\fx},{0+0*\dy+\fy});
    %%% lower row, right-open, filled    
    %\draw [draw=gray, pattern = north east lines, pattern color = lightgray]    ({\txl+0*\dx+\fx},{0*\dy-\fy}) -- ++ ({-0*\dx-2*\fx},0) |- ({\txl+0*\dx+\fx},{0*\dy+\fy});
    %%% lower row, closed, filled
%    \draw [draw=gray, pattern = north west lines, pattern color = darkgray] ({1*\dx-\fx},{-\fy}) rectangle ({2*\dx+\fx},{\fy});
%    \draw [draw=gray, pattern = north east lines, pattern color = lightgray] ({5*\dx-\fx},{-\fy}) rectangle ({6*\dx+\fx},{\fy});    
    %%% upper row, left-open, filled    
    \draw [draw=gray, pattern = north east lines, pattern color = lightgray] ({0+0*\dx-\fx},{\ty+0*\dy-\fy}) -- ++ ({0*\dx+2*\fx},0) |- ({0+0*\dx-\fx},{\ty+0*\dy+\fy});
    %%% upper row, right-open, filled        
    %\draw [draw=gray, pattern = north east lines, pattern color = lightgray] ({\txu+0*\dx+\fx},{\ty+0*\dy-\fy}) -- ++ ({-3*\dx-2*\fx},0) |- ({\txu+0*\dx+\fx},{\ty+0*\dy+\fy});
    %%% upper row, closed, filled
   \draw [draw=gray, pattern = north east lines, pattern color = lightgray] ({2*\dx-\fx},{\ty-\fy}) rectangle ({4*\dx+\fx},{\ty+\fy});    
    \node [below = {\cx} of  l1] {$\eta_2$};
    \node [above ={\cx} of  u1] {$\eta_1$};        
    %
     %% lower row, left-open
     \draw [draw=darkgray, densely dashed] ({0+0*\dx-\gx},{0+0*\dy-\gy}) --++ ({4*\dx+2*\gx},0) |- ({0+0*\dx-\gx},{0+0*\dy+\gy}); 
     %% lower row, right-open    
     %\draw [draw=darkgray, densely dashed] ({\txl+0*\dx+\gx},{0*\dy-\gy}) --++ ({-3*\dx-2*\gx},0) |- ({\txl+0*\dx+\gx},{0*\dy+\gy});
     %% lower row, open
     % \draw [draw=darkgray, densely dashed] ({-\gx},{-\gy}) -- ({\txl+\gx},{-\gy})  ({\txl+\gx},{\gy}) -- ({-\gx},{\gy});
     %% lower row, closed    
    %\draw [draw=darkgray, densely dashed] ({4*\dx-\gx},{-\gy}) rectangle ({7*\dx+\gx},{\gy});        
     %% upper row, left-open
     \draw [draw=darkgray, densely dashed] ({0+0*\dx-\gx},{\ty+0*\dy-\gy}) --++ ({1*\dx+2*\gx},0) |- ({0+0*\dx-\gx},{\ty+0*\dy+\gy}); 
     %% upper row, right-open
     %\draw [draw=darkgray, densely dashed] ({\txu+0*\dx+\gx},{\ty+0*\dy-\gy}) --++ ({-0*\dx-2*\gx},0) |- ({\txu+0*\dx+\gx},{\ty+0*\dy+\gy});
     %% upper row, open
     % \draw [draw=darkgray, densely dashed] ({-\gx},{\ty-\gy}) -- ({\txu+\gx},{\ty-\gy})  ({\txu+\gx},{\ty+\gy}) -- ({-\gx},{\ty+\gy});

    %
     \node at ({-\dx},{0.5*\ty}) {$p$};
   \end{tikzpicture}
   \qquad$\otimes$\qquad
\begin{tikzpicture}[baseline=0.666*1cm-0.25em]
    \def\scp{0.666}
    \def\linksize{\scp*0.075cm}
    \def\pointsize{\scp*0.25cm}
    \def\dd{\scp*0.5cm}
    \def\dx{\scp*1cm}
    \def\cx{\scp*0.3cm}
    \def\txu{5*\dx}    
    \def\txl{3*\dx}
    \def\dy{\scp*1cm}
    \def\cy{\scp*0.3cm}
    \def\ty{2*\dy}
    \def\fy{\scp*0.2cm}
    \def\fx{\scp*0.2cm}
    \def\gy{\scp*0.4cm}
    \def\gx{\scp*0.4cm}      
    \tikzset{whp/.style={circle, inner sep=0pt, text width={\pointsize}, draw=black, fill=white}}
    \tikzset{blp/.style={circle, inner sep=0pt, text width={\pointsize}, draw=black, fill=black}}
    \tikzset{lk/.style={regular polygon, regular polygon sides=4, inner sep=0pt, text width={\linksize}, draw=black, fill=black}}
    \tikzset{cc/.style={cross out, minimum size={1.5*\pointsize-\pgflinewidth}, inner sep=0pt, outer sep=0pt,  draw=black}}    
    \tikzset{vp/.style={circle, inner sep=0pt, text width={1.5*\pointsize}, fill=white}}
    \tikzset{sstr/.style={shorten <= 3pt, shorten >= 3pt}}    
    \draw[dotted] ({0-\dd},{0}) -- ({\txl+\dd},{0});
    \draw[dotted] ({0-\dd},{\ty}) -- ({\txu+\dd},{\ty});
    \node[vp] (l1) at ({0+0*\dx},{0+0*\ty}) {$\overline{c_2}$};
    \node[vp] (u1) at ({0+2*\dx},{0+1*\ty}) {$\overline{c_3}$};
    \draw[sstr] (l1) --++(0,{\dy})-| (u1);
    \draw[dashed] ($(l1) +(0,{\dy})$) --++(0,{0.75*\dy});
    %
    %%% lower row, left-open, filled
    %\draw [draw=gray, pattern = north east lines, pattern color = lightgray] ({0+0*\dx-\fx},{0+0*\dy-\fy}) -- ++ ({3*\dx+2*\fx},0) |- ({0+0*\dx-\fx},{0+0*\dy+\fy});
    %%% lower row, right-open, filled    
    \draw [draw=gray, fill = lightgray]    ({\txl+0*\dx+\fx},{0*\dy-\fy}) -- ++ ({-2*\dx-2*\fx},0) |- ({\txl+0*\dx+\fx},{0*\dy+\fy});
    %%% lower row, closed, filled
%    \draw [draw=gray, pattern = north west lines, pattern color = darkgray] ({1*\dx-\fx},{-\fy}) rectangle ({2*\dx+\fx},{\fy});
%    \draw [draw=gray, pattern = north east lines, pattern color = lightgray] ({5*\dx-\fx},{-\fy}) rectangle ({6*\dx+\fx},{\fy});    
    %%% upper row, left-open, filled    
    %\draw [draw=gray, pattern = north east lines, pattern color = lightgray] ({0+0*\dx-\fx},{\ty+0*\dy-\fy}) -- ++ ({0*\dx+2*\fx},0) |- ({0+0*\dx-\fx},{\ty+0*\dy+\fy});
    %%% upper row, right-open, filled        
    \draw [draw=gray, fill = lightgray] ({\txu+0*\dx+\fx},{\ty+0*\dy-\fy}) -- ++ ({-2*\dx-2*\fx},0) |- ({\txu+0*\dx+\fx},{\ty+0*\dy+\fy});
    %%% upper row, closed, filled
   \draw [draw=gray, fill  = lightgray] ({0*\dx-\fx},{\ty-\fy}) rectangle ({1*\dx+\fx},{\ty+\fy});    
    \node [below = {\cx} of  l1] {$\theta_1$};
    \node [above ={\cx} of  u1] {$\theta_2$};        
    %
     %% lower row, left-open
     %\draw [draw=darkgray, densely dashed] ({0+0*\dx-\gx},{0+0*\dy-\gy}) --++ ({4*\dx+2*\gx},0) |- ({0+0*\dx-\gx},{0+0*\dy+\gy}); 
     %% lower row, right-open    
     \draw [draw=darkgray, densely dashed] ({\txl+0*\dx+\gx},{0*\dy-\gy}) --++ ({-3*\dx-2*\gx},0) |- ({\txl+0*\dx+\gx},{0*\dy+\gy});
     %% lower row, open
     % \draw [draw=darkgray, densely dashed] ({-\gx},{-\gy}) -- ({\txl+\gx},{-\gy})  ({\txl+\gx},{\gy}) -- ({-\gx},{\gy});
     %% lower row, closed    
    %\draw [draw=darkgray, densely dashed] ({4*\dx-\gx},{-\gy}) rectangle ({7*\dx+\gx},{\gy});        
     %% upper row, left-open
     %\draw [draw=darkgray, densely dashed] ({0+0*\dx-\gx},{\ty+0*\dy-\gy}) --++ ({1*\dx+2*\gx},0) |- ({0+0*\dx-\gx},{\ty+0*\dy+\gy}); 
     %% upper row, right-open
     \draw [draw=darkgray, densely dashed] ({\txu+0*\dx+\gx},{\ty+0*\dy-\gy}) --++ ({-3*\dx-2*\gx},0) |- ({\txu+0*\dx+\gx},{\ty+0*\dy+\gy});
     %% upper row, open
     % \draw [draw=darkgray, densely dashed] ({-\gx},{\ty-\gy}) -- ({\txu+\gx},{\ty-\gy})  ({\txu+\gx},{\ty+\gy}) -- ({-\gx},{\ty+\gy});

    %
     \node at ({\txu+\dx},{0.5*\ty}) {$q$};
   \end{tikzpicture}   
\end{mycenter}
Denote the images of the points $\theta_1$ and $\theta_2$ of $q$ in $p\otimes q\in \mc C$ by $\theta_1'$ and $\theta_2'$, respectively.
The assumptions about the normalized colors of $\eta_2$ and $\theta_1$ imply that $T\eqpd \{\eta_2,\theta_1'\}$ is a turn in $p\otimes q$, meaning in particular $\delta_{p\otimes q}(\eta_2,\theta_1')=0$.
                       
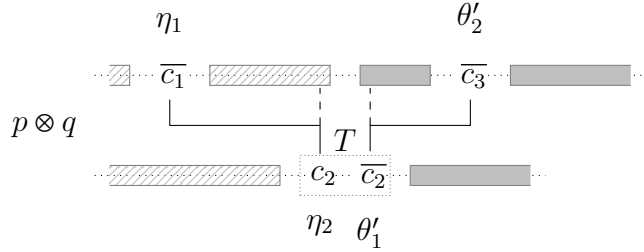
\begin{mycenter}[0.5em]
\begin{tikzpicture}[baseline=0.666*1cm-0.25em]
    \def\scp{0.666}
    \def\linksize{\scp*0.075cm}
    \def\pointsize{\scp*0.25cm}
    \def\dd{\scp*0.5cm}
    \def\dx{\scp*1cm}
    \def\cx{\scp*0.3cm}
    \def\txu{10*\dx}    
    \def\txl{8*\dx}
    \def\dy{\scp*1cm}
    \def\cy{\scp*0.3cm}
    \def\ty{2*\dy}
    \def\fy{\scp*0.2cm}
    \def\fx{\scp*0.2cm}
    \def\gy{\scp*0.4cm}
    \def\gx{\scp*0.4cm}      
    \tikzset{whp/.style={circle, inner sep=0pt, text width={\pointsize}, draw=black, fill=white}}
    \tikzset{blp/.style={circle, inner sep=0pt, text width={\pointsize}, draw=black, fill=black}}
    \tikzset{lk/.style={regular polygon, regular polygon sides=4, inner sep=0pt, text width={\linksize}, draw=black, fill=black}}
    \tikzset{cc/.style={cross out, minimum size={1.5*\pointsize-\pgflinewidth}, inner sep=0pt, outer sep=0pt,  draw=black}}    
    \tikzset{vp/.style={circle, inner sep=0pt, text width={1.5*\pointsize}, fill=white}}
    \tikzset{sstr/.style={shorten <= 3pt, shorten >= 3pt}}    
    \draw[dotted] ({0-\dd},{0}) -- ({\txl+\dd},{0});
    \draw[dotted] ({0-\dd},{\ty}) -- ({\txu+\dd},{\ty});
    \node[vp] (l1) at ({0+4*\dx},{0+0*\ty}) {$c_2$};
    \node[vp] (u1) at ({0+1*\dx},{0+1*\ty}) {$\overline{c_1}$};
    \node[vp] (l2) at ({0+5*\dx},{0+0*\ty}) {$\overline{c_2}$};
    \node[vp] (u2) at ({0+7*\dx},{0+1*\ty}) {$\overline{c_3}$};
    \draw[sstr] (l1) --++(0,{\dy})-| (u1);
    \draw[dashed] ($(l1) +(0,{\dy})$) --++(0,{0.75*\dy});
    \draw[sstr] (l2) --++(0,{\dy})-| (u2);
    \draw[dashed] ($(l2) +(0,{\dy})$) --++(0,{0.75*\dy});
    %
    %%% lower row, left-open, filled
    \draw [draw=gray, pattern = north east lines, pattern color = lightgray] ({0+0*\dx-\fx},{0+0*\dy-\fy}) -- ++ ({3*\dx+2*\fx},0) |- ({0+0*\dx-\fx},{0+0*\dy+\fy});
    %%% lower row, right-open, filled    
    % \draw [draw=gray, pattern = north east lines, pattern color = lightgray]    ({\txl+0*\dx+\fx},{0*\dy-\fy}) -- ++ ({-0*\dx-2*\fx},0) |- ({\txl+0*\dx+\fx},{0*\dy+\fy});
 \draw [draw=gray, fill = lightgray]    ({\txl+0*\dx+\fx},{0*\dy-\fy}) -- ++ ({-2*\dx-2*\fx},0) |- ({\txl+0*\dx+\fx},{0*\dy+\fy});    
    %%% lower row, closed, filled
%    \draw [draw=gray, pattern = north west lines, pattern color = darkgray] ({1*\dx-\fx},{-\fy}) rectangle ({2*\dx+\fx},{\fy});
%    \draw [draw=gray, pattern = north east lines, pattern color = lightgray] ({5*\dx-\fx},{-\fy}) rectangle ({6*\dx+\fx},{\fy});    
    %%% upper row, left-open, filled    
    \draw [draw=gray, pattern = north east lines, pattern color = lightgray] ({0+0*\dx-\fx},{\ty+0*\dy-\fy}) -- ++ ({0*\dx+2*\fx},0) |- ({0+0*\dx-\fx},{\ty+0*\dy+\fy});
    %%% upper row, right-open, filled        
    % \draw [draw=gray, pattern = north east lines, pattern color = lightgray] ({\txu+0*\dx+\fx},{\ty+0*\dy-\fy}) -- ++ ({-3*\dx-2*\fx},0) |- ({\txu+0*\dx+\fx},{\ty+0*\dy+\fy});
    \draw [draw=gray, fill = lightgray] ({\txu+0*\dx+\fx},{\ty+0*\dy-\fy}) -- ++ ({-2*\dx-2*\fx},0) |- ({\txu+0*\dx+\fx},{\ty+0*\dy+\fy});    
    %%% upper row, closed, filled
    \draw [draw=gray, pattern = north east lines, pattern color = lightgray] ({2*\dx-\fx},{\ty-\fy}) rectangle ({4*\dx+\fx},{\ty+\fy});
   \draw [draw=gray, fill  = lightgray] ({5*\dx-\fx},{\ty-\fy}) rectangle ({6*\dx+\fx},{\ty+\fy});            
    \node [below = {\cx} of  l1] {$\eta_2$};
    \node [above ={\cx} of  u1] {$\eta_1$};
    \node [below = {\cx} of  l2] {$\theta_1'$};
    \node [above ={\cx} of  u2] {$\theta_2'$};            
    %
     %% lower row, left-open
     %\draw [draw=darkgray, densely dashed] ({0+0*\dx-\gx},{0+0*\dy-\gy}) --++ ({4*\dx+2*\gx},0) |- ({0+0*\dx-\gx},{0+0*\dy+\gy}); 
     %% lower row, right-open    
     %\draw [draw=darkgray, densely dashed] ({\txl+0*\dx+\gx},{0*\dy-\gy}) --++ ({-3*\dx-2*\gx},0) |- ({\txl+0*\dx+\gx},{0*\dy+\gy});
     %% lower row, open
     % \draw [draw=darkgray, densely dashed] ({-\gx},{-\gy}) -- ({\txl+\gx},{-\gy})  ({\txl+\gx},{\gy}) -- ({-\gx},{\gy});
     %% lower row, closed    
    %\draw [draw=darkgray, densely dashed] ({4*\dx-\gx},{-\gy}) rectangle ({7*\dx+\gx},{\gy});        
     %% upper row, left-open
     %\draw [draw=darkgray, densely dashed] ({0+0*\dx-\gx},{\ty+0*\dy-\gy}) --++ ({1*\dx+2*\gx},0) |- ({0+0*\dx-\gx},{\ty+0*\dy+\gy}); 
     %% upper row, right-open
     %\draw [draw=darkgray, densely dashed] ({\txu+0*\dx+\gx},{\ty+0*\dy-\gy}) --++ ({-0*\dx-2*\gx},0) |- ({\txu+0*\dx+\gx},{\ty+0*\dy+\gy});
     %% upper row, open
     % \draw [draw=darkgray, densely dashed] ({-\gx},{\ty-\gy}) -- ({\txu+\gx},{\ty-\gy})  ({\txu+\gx},{\ty+\gy}) -- ({-\gx},{\ty+\gy});
%
    \draw [densely dotted, draw=gray] ({4*\dx-\gx},{-\gy}) rectangle ({5*\dx+\gx},{\gy});
     \node at ({4.5*\dx},{0.35*\ty}) {$T$};    
     \node at ({-1.5*\dx},{0.5*\ty}) {$p\otimes q$};
   \end{tikzpicture}
 \end{mycenter}
 \noindent
 Moreover, $\delta_{p\otimes q}(\eta_1,\eta_2)=\delta_p(\eta_1,\eta_2)$ and $\delta_{p\otimes q}(\theta_1',\theta_2')=\delta_{q}(\theta_1,\theta_2)$ by nature of the tensor product.
 \par
 Let $\theta_2''$ denote the image of $\theta_2'$ in $r\eqpd E(p\otimes q,T)\in \mc C$, the partition obtained from $p\otimes q$ by erasing the turn $T$ (see \cite[Sec\-tion~4.3]{MWNHO1}). By definition of the erasing operation, $\eta_1$ and $\theta_2''$ belong to the same block $D$ in $r$ with $]\eta_1,\theta_2''[_{r}\cap D=\emptyset$.
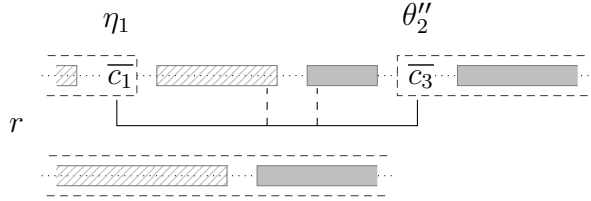
                       \begin{mycenter}[0.5em]
\begin{tikzpicture}[baseline=0.666*1cm-0.25em]
    \def\scp{0.666}
    \def\linksize{\scp*0.075cm}
    \def\pointsize{\scp*0.25cm}
    \def\dd{\scp*0.5cm}
    \def\dx{\scp*1cm}
    \def\cx{\scp*0.3cm}
    \def\txu{10*\dx}    
    \def\txl{6*\dx}
    \def\dy{\scp*1cm}
    \def\cy{\scp*0.3cm}
    \def\ty{2*\dy}
    \def\fy{\scp*0.2cm}
    \def\fx{\scp*0.2cm}
    \def\gy{\scp*0.4cm}
    \def\gx{\scp*0.4cm}      
    \tikzset{whp/.style={circle, inner sep=0pt, text width={\pointsize}, draw=black, fill=white}}
    \tikzset{blp/.style={circle, inner sep=0pt, text width={\pointsize}, draw=black, fill=black}}
    \tikzset{lk/.style={regular polygon, regular polygon sides=4, inner sep=0pt, text width={\linksize}, draw=black, fill=black}}
    \tikzset{cc/.style={cross out, minimum size={1.5*\pointsize-\pgflinewidth}, inner sep=0pt, outer sep=0pt,  draw=black}}    
    \tikzset{vp/.style={circle, inner sep=0pt, text width={1.5*\pointsize}, fill=white}}
    \tikzset{sstr/.style={shorten <= 3pt, shorten >= 3pt}}    
    \draw[dotted] ({0-\dd},{0}) -- ({\txl+\dd},{0});
    \draw[dotted] ({0-\dd},{\ty}) -- ({\txu+\dd},{\ty});
    \node[vp] (u1) at ({0+1*\dx},{0+1*\ty}) {$\overline{c_1}$};
    \node[vp] (u2) at ({0+7*\dx},{0+1*\ty}) {$\overline{c_3}$};
    \coordinate (l1) at ({0+4*\dx},{0+0*\ty});
    \coordinate (l2) at ({0+5*\dx},{0+0*\ty});
    \draw[sstr] (u1) --++(0,{-\dy})-| (u2);
    \draw[dashed] ($(l1) +(0,{\dy})$) --++(0,{0.75*\dy});
    \draw[dashed] ($(l2) +(0,{\dy})$) --++(0,{0.75*\dy});
    %
    %%% lower row, left-open, filled
    \draw [draw=gray, pattern = north east lines, pattern color = lightgray] ({0+0*\dx-\fx},{0+0*\dy-\fy}) -- ++ ({3*\dx+2*\fx},0) |- ({0+0*\dx-\fx},{0+0*\dy+\fy});
    %%% lower row, right-open, filled    
    % \draw [draw=gray, pattern = north east lines, pattern color = lightgray]    ({\txl+0*\dx+\fx},{0*\dy-\fy}) -- ++ ({-0*\dx-2*\fx},0) |- ({\txl+0*\dx+\fx},{0*\dy+\fy});
 \draw [draw=gray, fill = lightgray]    ({\txl+0*\dx+\fx},{0*\dy-\fy}) -- ++ ({-2*\dx-2*\fx},0) |- ({\txl+0*\dx+\fx},{0*\dy+\fy});    
    %%% lower row, closed, filled
%    \draw [draw=gray, pattern = north west lines, pattern color = darkgray] ({1*\dx-\fx},{-\fy}) rectangle ({2*\dx+\fx},{\fy});
%    \draw [draw=gray, pattern = north east lines, pattern color = lightgray] ({5*\dx-\fx},{-\fy}) rectangle ({6*\dx+\fx},{\fy});    
    %%% upper row, left-open, filled    
    \draw [draw=gray, pattern = north east lines, pattern color = lightgray] ({0+0*\dx-\fx},{\ty+0*\dy-\fy}) -- ++ ({0*\dx+2*\fx},0) |- ({0+0*\dx-\fx},{\ty+0*\dy+\fy});
    %%% upper row, right-open, filled        
    % \draw [draw=gray, pattern = north east lines, pattern color = lightgray] ({\txu+0*\dx+\fx},{\ty+0*\dy-\fy}) -- ++ ({-3*\dx-2*\fx},0) |- ({\txu+0*\dx+\fx},{\ty+0*\dy+\fy});
    \draw [draw=gray, fill = lightgray] ({\txu+0*\dx+\fx},{\ty+0*\dy-\fy}) -- ++ ({-2*\dx-2*\fx},0) |- ({\txu+0*\dx+\fx},{\ty+0*\dy+\fy});    
    %%% upper row, closed, filled
    \draw [draw=gray, pattern = north east lines, pattern color = lightgray] ({2*\dx-\fx},{\ty-\fy}) rectangle ({4*\dx+\fx},{\ty+\fy});
   \draw [draw=gray, fill  = lightgray] ({5*\dx-\fx},{\ty-\fy}) rectangle ({6*\dx+\fx},{\ty+\fy});            
    \node [above= {\cx} of  u1] {$\eta_1$};
    \node [above ={\cx} of  u2] {$\theta_2''$};            
    %
     %% lower row, left-open
     %\draw [draw=darkgray, densely dashed] ({0+0*\dx-\gx},{0+0*\dy-\gy}) --++ ({4*\dx+2*\gx},0) |- ({0+0*\dx-\gx},{0+0*\dy+\gy}); 
     %% lower row, right-open    
     %\draw [draw=darkgray, densely dashed] ({\txl+0*\dx+\gx},{0*\dy-\gy}) --++ ({-3*\dx-2*\gx},0) |- ({\txl+0*\dx+\gx},{0*\dy+\gy});
     %% lower row, open
      \draw [draw=darkgray, densely dashed] ({-\gx},{-\gy}) -- ({\txl+\gx},{-\gy})  ({\txl+\gx},{\gy}) -- ({-\gx},{\gy});
     %% lower row, closed    
    %\draw [draw=darkgray, densely dashed] ({4*\dx-\gx},{-\gy}) rectangle ({7*\dx+\gx},{\gy});        
     %% upper row, left-open
     \draw [draw=darkgray, densely dashed] ({0+0*\dx-\gx},{\ty+0*\dy-\gy}) --++ ({1*\dx+2*\gx},0) |- ({0+0*\dx-\gx},{\ty+0*\dy+\gy}); 
     %% upper row, right-open
     \draw [draw=darkgray, densely dashed] ({\txu+0*\dx+\gx},{\ty+0*\dy-\gy}) --++ ({-3*\dx-2*\gx},0) |- ({\txu+0*\dx+\gx},{\ty+0*\dy+\gy});
     %% upper row, open
     % \draw [draw=darkgray, densely dashed] ({-\gx},{\ty-\gy}) -- ({\txu+\gx},{\ty-\gy})  ({\txu+\gx},{\ty+\gy}) -- ({-\gx},{\ty+\gy});
%
     \node at ({-\dx},{0.5*\ty}) {$r$};
   \end{tikzpicture}
 \end{mycenter}
 \noindent
   Hence, from $\delta_{r}(\eta_1,\theta_2'')\in K_{c_1,c_3}(\mc C)=\kappa_{c_1,c_3}$ and from
                    \begin{align*}
                      \delta_{r}(\eta_1,\theta_2'')&=\delta_{p\otimes q}(\eta_1,\theta_2)-\sigma_{p\otimes q}(T)\\
                                                   &=\delta_{p\otimes q}(\eta_1,\theta_2)\\
                                                   &=\delta_{p\otimes q}(\eta_1,\eta_2)+\delta_{p\otimes q}(\eta_2,\theta_1')+\delta_{p\otimes q}(\theta_1',\theta_2')\\
                      &=\delta_{p\otimes q}(\eta_1,\eta_2)+\delta_{p\otimes q}(\theta_1',\theta_2')\\
                      &=\delta_p(\eta_1,\eta_2)+\delta_q(\theta_1,\theta_2)
                    \end{align*}
it                    follows $\delta_p(\eta_1,\eta_2)+\delta_q(\theta_1,\theta_2)\in\kappa_{c_1,c_3}$. And that is what we needed to show.
                  \end{proof}

                  \begin{lemma}
                    \label{lemma:verifying-axioms-7}
                    For every non-hy\-per\-octa\-he\-dral category $\mc C\subseteq \Cp$, the sets $\sigma\eqpd \toco(\mc C)$ and $\kappa_{c_1,c_2}\eqpd K_{c_1,c_2}(\mc C)$ and $\xi_{c_1,c_2}\eqpd X_{c_1,c_2}(\mc C)$ for $c_1,c_2\in \colors$ satisfy Axiom~\ref{lemma:gradedsets-condition-7} of~\ref{axioms:arithmetic}: For all $c_1,c_2,c_3\in\colors$,
                    \begin{IEEEeqnarray*}{rCl}
                      \kappa_{c_1,c_2}+\xi_{\overline{c_2},c_3}\subseteq \xi_{c_1,c_3}.
                    \end{IEEEeqnarray*}
                  \end{lemma}
                  \begin{proof}
                    We adapt the proof of Lem\-ma~\ref{lemma:verifying-axioms-6}.
Let $c_1,c_2\in \colors$ be arbitrary. Let $p,q\in \mc C$, let $B$ be a block in $p$, and let $C$ and $D$ be two crossing blocks in $q$. Let  $\gamma_1$ and $\gamma_2$ be two distinct points of $B$ of normalized colors $c_1$ respectively $c_2$ in $p$ with $]\gamma_1,\gamma_2[_p\cap B=\emptyset$. In $q$, let $\eta_1\in C$ have normalized color $\overline{c_2}$ and $\theta_1\in D$ normalized color $c_3$. Then, $\delta_p(\gamma_1,\gamma_2)$ is a generic element of $K_{c_1,c_2}(\mc C)=\kappa_{c_1,c_2}$ and $\delta_q(\eta_1,\theta_1)$ one of $X_{\overline{c_2},c_3}(\mc C)=\xi_{\overline{c_2},c_3}$.  No generality is lost assuming that $\gamma_2$ is the rightmost lower point of $p$ and $\eta_1$ the leftmost lower point of $q$. We find $\eta_2\in C$ and $\theta_2\in D$ such that $\eta_1\neq \eta_2$ and $\theta_1\neq \theta_2$ and such that $(\eta_1,\theta_1,\eta_2,\theta_2)$ or $(\eta_1,\theta_2,\eta_2,\theta_1)$ is ordered in $q$.
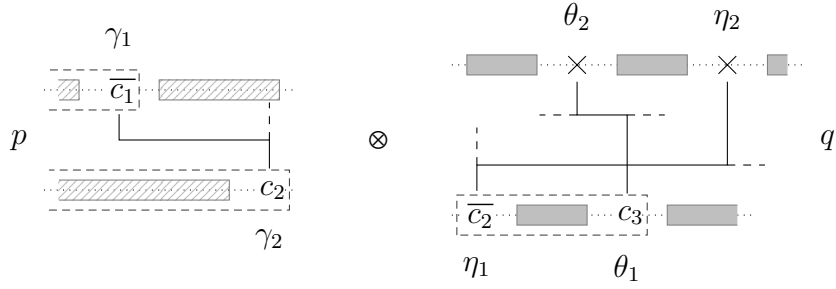
                       \begin{mycenter}[0.5em]
\begin{tikzpicture}[baseline=0.666*1cm-0.25em]
    \def\scp{0.666}
    \def\linksize{\scp*0.075cm}
    \def\pointsize{\scp*0.25cm}
    \def\dd{\scp*0.5cm}
    \def\dx{\scp*1cm}
    \def\cx{\scp*0.3cm}
    \def\txu{4*\dx}    
    \def\txl{4*\dx}
    \def\dy{\scp*1cm}
    \def\cy{\scp*0.3cm}
    \def\ty{2*\dy}
    \def\fy{\scp*0.2cm}
    \def\fx{\scp*0.2cm}
    \def\gy{\scp*0.4cm}
    \def\gx{\scp*0.4cm}      
    \tikzset{whp/.style={circle, inner sep=0pt, text width={\pointsize}, draw=black, fill=white}}
    \tikzset{blp/.style={circle, inner sep=0pt, text width={\pointsize}, draw=black, fill=black}}
    \tikzset{lk/.style={regular polygon, regular polygon sides=4, inner sep=0pt, text width={\linksize}, draw=black, fill=black}}
    \tikzset{cc/.style={cross out, minimum size={1.5*\pointsize-\pgflinewidth}, inner sep=0pt, outer sep=0pt,  draw=black}}    
    \tikzset{vp/.style={circle, inner sep=0pt, text width={1.5*\pointsize}, fill=white}}
    \tikzset{sstr/.style={shorten <= 3pt, shorten >= 3pt}}    
    \draw[dotted] ({0-\dd},{0}) -- ({\txl+\dd},{0});
    \draw[dotted] ({0-\dd},{\ty}) -- ({\txu+\dd},{\ty});

    \node[vp] (l1) at ({0+4*\dx},{0+0*\ty}) {$c_2$};
    \node[vp] (u1) at ({0+1*\dx},{0+1*\ty}) {$\overline{c_1}$};
    \draw[sstr] (l1) --++(0,{\dy}) -| (u1);
    \draw[dashed] ($(l1) +(0,{\dy})$) --++(0,{0.75*\dy});    
    %
    %%% lower row, left-open, filled
    \draw [draw=gray, pattern = north east lines, pattern color = lightgray] ({0+0*\dx-\fx},{0+0*\dy-\fy}) -- ++ ({3*\dx+2*\fx},0) |- ({0+0*\dx-\fx},{0+0*\dy+\fy});
    %%% lower row, right-open, filled    
    %\draw [draw=gray, pattern = north east lines, pattern color = lightgray]    ({\txl+0*\dx+\fx},{0*\dy-\fy}) -- ++ ({-0*\dx-2*\fx},0) |- ({\txl+0*\dx+\fx},{0*\dy+\fy});
    %%% lower row, closed, filled
%    \draw [draw=gray, pattern = north west lines, pattern color = darkgray] ({1*\dx-\fx},{-\fy}) rectangle ({2*\dx+\fx},{\fy});
%    \draw [draw=gray, pattern = north east lines, pattern color = lightgray] ({5*\dx-\fx},{-\fy}) rectangle ({6*\dx+\fx},{\fy});    
    %%% upper row, left-open, filled    
    \draw [draw=gray, pattern = north east lines, pattern color = lightgray] ({0+0*\dx-\fx},{\ty+0*\dy-\fy}) -- ++ ({0*\dx+2*\fx},0) |- ({0+0*\dx-\fx},{\ty+0*\dy+\fy});
    %%% upper row, right-open, filled        
    %\draw [draw=gray, pattern = north east lines, pattern color = lightgray] ({\txu+0*\dx+\fx},{\ty+0*\dy-\fy}) -- ++ ({-3*\dx-2*\fx},0) |- ({\txu+0*\dx+\fx},{\ty+0*\dy+\fy});
    %%% upper row, closed, filled
   \draw [draw=gray, pattern = north east lines, pattern color = lightgray] ({2*\dx-\fx},{\ty-\fy}) rectangle ({4*\dx+\fx},{\ty+\fy});    
    \node [below = {\cx} of  l1] {$\gamma_2$};
    \node [above ={\cx} of  u1] {$\gamma_1$};        
    %
     %% lower row, left-open
     \draw [draw=darkgray, densely dashed] ({0+0*\dx-\gx},{0+0*\dy-\gy}) --++ ({4*\dx+2*\gx},0) |- ({0+0*\dx-\gx},{0+0*\dy+\gy}); 
     %% lower row, right-open    
     %\draw [draw=darkgray, densely dashed] ({\txl+0*\dx+\gx},{0*\dy-\gy}) --++ ({-3*\dx-2*\gx},0) |- ({\txl+0*\dx+\gx},{0*\dy+\gy});
     %% lower row, open
     % \draw [draw=darkgray, densely dashed] ({-\gx},{-\gy}) -- ({\txl+\gx},{-\gy})  ({\txl+\gx},{\gy}) -- ({-\gx},{\gy});
     %% lower row, closed    
    %\draw [draw=darkgray, densely dashed] ({4*\dx-\gx},{-\gy}) rectangle ({7*\dx+\gx},{\gy});        
     %% upper row, left-open
     \draw [draw=darkgray, densely dashed] ({0+0*\dx-\gx},{\ty+0*\dy-\gy}) --++ ({1*\dx+2*\gx},0) |- ({0+0*\dx-\gx},{\ty+0*\dy+\gy}); 
     %% upper row, right-open
     %\draw [draw=darkgray, densely dashed] ({\txu+0*\dx+\gx},{\ty+0*\dy-\gy}) --++ ({-0*\dx-2*\gx},0) |- ({\txu+0*\dx+\gx},{\ty+0*\dy+\gy});
     %% upper row, open
     % \draw [draw=darkgray, densely dashed] ({-\gx},{\ty-\gy}) -- ({\txu+\gx},{\ty-\gy})  ({\txu+\gx},{\ty+\gy}) -- ({-\gx},{\ty+\gy});

    %
     \node at ({-\dx},{0.5*\ty}) {$p$};
   \end{tikzpicture}
   \qquad$\otimes$\qquad
\begin{tikzpicture}[baseline=0.666*1.5cm-0.25em]
    \def\scp{0.666}
    \def\linksize{\scp*0.075cm}
    \def\pointsize{\scp*0.25cm}
    \def\dd{\scp*0.5cm}
    \def\dx{\scp*1cm}
    \def\cx{\scp*0.3cm}
    \def\txu{6*\dx}    
    \def\txl{5*\dx}
    \def\dy{\scp*1cm}
    \def\cy{\scp*0.3cm}
    \def\ty{3*\dy}
    \def\fy{\scp*0.2cm}
    \def\fx{\scp*0.2cm}
    \def\gy{\scp*0.4cm}
    \def\gx{\scp*0.4cm}      
    \tikzset{whp/.style={circle, inner sep=0pt, text width={\pointsize}, draw=black, fill=white}}
    \tikzset{blp/.style={circle, inner sep=0pt, text width={\pointsize}, draw=black, fill=black}}
    \tikzset{lk/.style={regular polygon, regular polygon sides=4, inner sep=0pt, text width={\linksize}, draw=black, fill=black}}
    \tikzset{cc/.style={cross out, minimum size={1.5*\pointsize-\pgflinewidth}, inner sep=0pt, outer sep=0pt,  draw=black}}    
    \tikzset{vp/.style={circle, inner sep=0pt, text width={1.5*\pointsize}, fill=white}}
    \tikzset{sstr/.style={shorten <= 3pt, shorten >= 3pt}}    
    \draw[dotted] ({0-\dd},{0}) -- ({\txl+\dd},{0});
    \draw[dotted] ({0-\dd},{\ty}) -- ({\txu+\dd},{\ty});
    \node[vp] (l1) at ({0+0*\dx},{0+0*\ty}) {$\overline{c_2}$};
    \node[vp] (l2) at ({0+3*\dx},{0+0*\ty}) {$c_3$};
    \node[cc] (u1) at ({0+2*\dx},{0+1*\ty}) {};
    \node[cc] (u2) at ({0+5*\dx},{0+1*\ty}) {};    
    \draw[sstr] (l2) --++(0,{2*\dy}) -| (u1);    
    \draw[sstr] (l1) --++(0,{1*\dy}) -| (u2);
    \draw[dashed] ($(l1) +(0,{\dy})$) --++(0,{0.75*\dy});
    \draw[dashed] ($(u2) -(0,{2*\dy})$) --++({0.75*\dx},0);
    \draw[dashed] ($(l2) +(0,{2*\dy})$) --++({0.75*\dx},0);
    \draw[dashed] ($(u1) -(0,{1*\dy})$) --++({-0.75*\dx},0);      
    %
    %%% lower row, left-open, filled
    %\draw [draw=gray, pattern = north east lines, pattern color = lightgray] ({0+0*\dx-\fx},{0+0*\dy-\fy}) -- ++ ({3*\dx+2*\fx},0) |- ({0+0*\dx-\fx},{0+0*\dy+\fy});
    %%% lower row, right-open, filled    
    \draw [draw=gray, fill = lightgray]    ({\txl+0*\dx+\fx},{0*\dy-\fy}) -- ++ ({-1*\dx-2*\fx},0) |- ({\txl+0*\dx+\fx},{0*\dy+\fy});
    %%% lower row, closed, filled
    \draw [draw=gray, fill  = lightgray] ({1*\dx-\fx},{0*\ty-\fy}) rectangle ({2*\dx+\fx},{0*\ty+\fy});            
    %%% upper row, left-open, filled    
    %\draw [draw=gray, pattern = north east lines, pattern color = lightgray] ({0+0*\dx-\fx},{\ty+0*\dy-\fy}) -- ++ ({0*\dx+2*\fx},0) |- ({0+0*\dx-\fx},{\ty+0*\dy+\fy});
    %%% upper row, right-open, filled        
    \draw [draw=gray, fill = lightgray] ({\txu+0*\dx+\fx},{\ty+0*\dy-\fy}) -- ++ ({-0*\dx-2*\fx},0) |- ({\txu+0*\dx+\fx},{\ty+0*\dy+\fy});
   %%% upper row, closed, filled
    \draw [draw=gray, fill  = lightgray] ({3*\dx-\fx},{\ty-\fy}) rectangle ({4*\dx+\fx},{\ty+\fy});
\draw [draw=gray, fill  = lightgray] ({0*\dx-\fx},{\ty-\fy}) rectangle ({1*\dx+\fx},{\ty+\fy});        
    \node [below = {\cx} of  l1] {$\eta_1$};
    \node [below = {\cx} of  l2] {$\theta_1$};
    \node [above ={\cx} of  u1] {$\theta_2$};
    \node [above ={\cx} of  u2] {$\eta_2$};            
    %
     %% lower row, left-open
     %\draw [draw=darkgray, densely dashed] ({0+0*\dx-\gx},{0+0*\dy-\gy}) --++ ({4*\dx+2*\gx},0) |- ({0+0*\dx-\gx},{0+0*\dy+\gy}); 
     %% lower row, right-open    
     %\draw [draw=darkgray, densely dashed] ({\txl+0*\dx+\gx},{0*\dy-\gy}) --++ ({-3*\dx-2*\gx},0) |- ({\txl+0*\dx+\gx},{0*\dy+\gy});
     %% lower row, open
     % \draw [draw=darkgray, densely dashed] ({-\gx},{-\gy}) -- ({\txl+\gx},{-\gy})  ({\txl+\gx},{\gy}) -- ({-\gx},{\gy});
     %% lower row, closed    
    \draw [draw=darkgray, densely dashed] ({0*\dx-\gx},{-\gy}) rectangle ({3*\dx+\gx},{\gy});        
     %% upper row, left-open
     %\draw [draw=darkgray, densely dashed] ({0+0*\dx-\gx},{\ty+0*\dy-\gy}) --++ ({1*\dx+2*\gx},0) |- ({0+0*\dx-\gx},{\ty+0*\dy+\gy}); 
     %% upper row, right-open
     %\draw [draw=darkgray, densely dashed] ({\txu+0*\dx+\gx},{\ty+0*\dy-\gy}) --++ ({-3*\dx-2*\gx},0) |- ({\txu+0*\dx+\gx},{\ty+0*\dy+\gy});
     %% upper row, open
     % \draw [draw=darkgray, densely dashed] ({-\gx},{\ty-\gy}) -- ({\txu+\gx},{\ty-\gy})  ({\txu+\gx},{\ty+\gy}) -- ({-\gx},{\ty+\gy});
    %
     \node at ({\txu+\dx},{0.5*\ty}) {$q$};
   \end{tikzpicture}   
\end{mycenter}

Let $\eta_1'$, $\eta_2'$, $\theta_1'$ and $\theta_2'$ denote the images of, respectively, $\eta_1$, $\eta_2$, $\theta_1$ and $\theta_2$ in $p\otimes q\in \mc C$.  By nature of the tensor product, $B$ is a block of $p\otimes q$. Likewise, $\eta_1'$ and $\eta_2'$ belong to the same block in $p\otimes q$ and so do $\theta_1'$ and $\theta_2'$. And each involved point has the same normalized color in $p\otimes q$ as the corresponding preimage in $p$ or $q$. The set $T\eqpd \{\gamma_2,\eta_1'\}$ is a turn in $p\otimes q$.
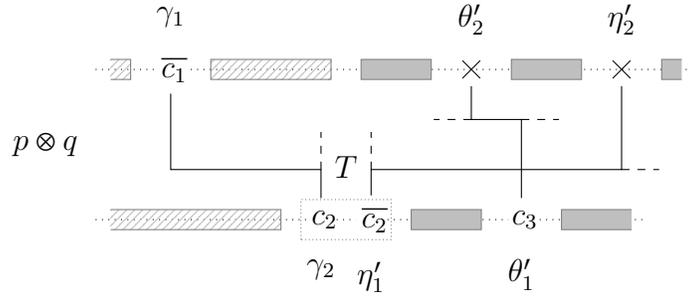
                       \begin{mycenter}[0.5em]
\begin{tikzpicture}[baseline=0.666*1cm-0.25em]
    \def\scp{0.666}
    \def\linksize{\scp*0.075cm}
    \def\pointsize{\scp*0.25cm}
    \def\dd{\scp*0.5cm}
    \def\dx{\scp*1cm}
    \def\cx{\scp*0.3cm}
    \def\txu{11*\dx}    
    \def\txl{10*\dx}
    \def\dy{\scp*1cm}
    \def\cy{\scp*0.3cm}
    \def\ty{3*\dy}
    \def\fy{\scp*0.2cm}
    \def\fx{\scp*0.2cm}
    \def\gy{\scp*0.4cm}
    \def\gx{\scp*0.4cm}      
    \tikzset{whp/.style={circle, inner sep=0pt, text width={\pointsize}, draw=black, fill=white}}
    \tikzset{blp/.style={circle, inner sep=0pt, text width={\pointsize}, draw=black, fill=black}}
    \tikzset{lk/.style={regular polygon, regular polygon sides=4, inner sep=0pt, text width={\linksize}, draw=black, fill=black}}
    \tikzset{cc/.style={cross out, minimum size={1.5*\pointsize-\pgflinewidth}, inner sep=0pt, outer sep=0pt,  draw=black}}    
    \tikzset{vp/.style={circle, inner sep=0pt, text width={1.5*\pointsize}, fill=white}}
    \tikzset{sstr/.style={shorten <= 3pt, shorten >= 3pt}}    
    \draw[dotted] ({0-\dd},{0}) -- ({\txl+\dd},{0});
    \draw[dotted] ({0-\dd},{\ty}) -- ({\txu+\dd},{\ty});
    \node[vp] (l1) at ({0+4*\dx},{0+0*\ty}) {$c_2$};
    \node[vp] (u1) at ({0+1*\dx},{0+1*\ty}) {$\overline{c_1}$};
    \node[vp] (l2) at ({0+5*\dx},{0+0*\ty}) {$\overline{c_2}$};
    \node[vp] (l3) at ({0+8*\dx},{0+0*\ty}) {$c_3$};
    \node[cc] (u2) at ({0+7*\dx},{0+1*\ty}) {};
    \node[cc] (u3) at ({0+10*\dx},{0+1*\ty}) {};    
    \draw[sstr] (l3) --++(0,{2*\dy}) -| (u2);    
    \draw[sstr] (l2) --++(0,{1*\dy}) -| (u3);    
    \draw[sstr] (l1) --++(0,{1*\dy}) -| (u1);
    \draw[dashed] ($(l1)+(0,{\dy})$) --++ (0,{0.75*\dy});
    \draw[dashed] ($(l2)+(0,{1*\dy})$) --++ (0,{0.75*\dy});
    \draw[dashed] ($(u3)+(0,{-2*\dy})$) --++ ({0.75*\dy},0);
    \draw[dashed] ($(l3)+(0,{2*\dy})$) --++ ({0.75*\dy},0);
    \draw[dashed] ($(u2)+(0,{-1*\dy})$) --++ ({-0.75*\dy},0);    
    %
    %%% lower row, left-open, filled
    \draw [draw=gray, pattern = north east lines, pattern color = lightgray] ({0+0*\dx-\fx},{0+0*\dy-\fy}) -- ++ ({3*\dx+2*\fx},0) |- ({0+0*\dx-\fx},{0+0*\dy+\fy});
    %%% lower row, right-open, filled    
    % \draw [draw=gray, pattern = north east lines, pattern color = lightgray]    ({\txl+0*\dx+\fx},{0*\dy-\fy}) -- ++ ({-0*\dx-2*\fx},0) |- ({\txl+0*\dx+\fx},{0*\dy+\fy});
\draw [draw=gray, fill = lightgray]    ({\txl+0*\dx+\fx},{0*\dy-\fy}) -- ++ ({-1*\dx-2*\fx},0) |- ({\txl+0*\dx+\fx},{0*\dy+\fy});    
    %%% lower row, closed, filled
%    \draw [draw=gray, pattern = north west lines, pattern color = darkgray] ({1*\dx-\fx},{-\fy}) rectangle ({2*\dx+\fx},{\fy});
%    \draw [draw=gray, pattern = north east lines, pattern color = lightgray] ({5*\dx-\fx},{-\fy}) rectangle ({6*\dx+\fx},{\fy});
\draw [draw=gray, fill  = lightgray] ({6*\dx-\fx},{0*\ty-\fy}) rectangle ({7*\dx+\fx},{0*\ty+\fy});            
    %%% upper row, left-open, filled    
    \draw [draw=gray, pattern = north east lines, pattern color = lightgray] ({0+0*\dx-\fx},{\ty+0*\dy-\fy}) -- ++ ({0*\dx+2*\fx},0) |- ({0+0*\dx-\fx},{\ty+0*\dy+\fy});
    %%% upper row, right-open, filled        
    % \draw [draw=gray, pattern = north east lines, pattern color = lightgray] ({\txu+0*\dx+\fx},{\ty+0*\dy-\fy}) -- ++ ({-3*\dx-2*\fx},0) |- ({\txu+0*\dx+\fx},{\ty+0*\dy+\fy});
    \draw [draw=gray, fill = lightgray] ({\txu+0*\dx+\fx},{\ty+0*\dy-\fy}) -- ++ ({-0*\dx-2*\fx},0) |- ({\txu+0*\dx+\fx},{\ty+0*\dy+\fy});    
    %%% upper row, closed, filled
    \draw [draw=gray, pattern = north east lines, pattern color = lightgray] ({2*\dx-\fx},{\ty-\fy}) rectangle ({4*\dx+\fx},{\ty+\fy});
    \draw [draw=gray, fill  = lightgray] ({8*\dx-\fx},{\ty-\fy}) rectangle ({9*\dx+\fx},{\ty+\fy});
\draw [draw=gray, fill  = lightgray] ({5*\dx-\fx},{\ty-\fy}) rectangle ({6*\dx+\fx},{\ty+\fy});            
    \node [below = {\cx} of  l1] {$\gamma_2$};
    \node [above ={\cx} of  u1] {$\gamma_1$};
    \node [below = {\cx} of  l2] {$\eta_1'$};
    \node [below = {\cx} of  l3] {$\theta_1'$};
    \node [above ={\cx} of  u2] {$\theta_2'$};
    \node [above ={\cx} of  u3] {$\eta_2'$};      
    %
     %% lower row, left-open
     %\draw [draw=darkgray, densely dashed] ({0+0*\dx-\gx},{0+0*\dy-\gy}) --++ ({4*\dx+2*\gx},0) |- ({0+0*\dx-\gx},{0+0*\dy+\gy}); 
     %% lower row, right-open    
     %\draw [draw=darkgray, densely dashed] ({\txl+0*\dx+\gx},{0*\dy-\gy}) --++ ({-3*\dx-2*\gx},0) |- ({\txl+0*\dx+\gx},{0*\dy+\gy});
     %% lower row, open
     % \draw [draw=darkgray, densely dashed] ({-\gx},{-\gy}) -- ({\txl+\gx},{-\gy})  ({\txl+\gx},{\gy}) -- ({-\gx},{\gy});
     %% lower row, closed    
    %\draw [draw=darkgray, densely dashed] ({4*\dx-\gx},{-\gy}) rectangle ({7*\dx+\gx},{\gy});        
     %% upper row, left-open
     %\draw [draw=darkgray, densely dashed] ({0+0*\dx-\gx},{\ty+0*\dy-\gy}) --++ ({1*\dx+2*\gx},0) |- ({0+0*\dx-\gx},{\ty+0*\dy+\gy}); 
     %% upper row, right-open
     %\draw [draw=darkgray, densely dashed] ({\txu+0*\dx+\gx},{\ty+0*\dy-\gy}) --++ ({-0*\dx-2*\gx},0) |- ({\txu+0*\dx+\gx},{\ty+0*\dy+\gy});
     %% upper row, open
     % \draw [draw=darkgray, densely dashed] ({-\gx},{\ty-\gy}) -- ({\txu+\gx},{\ty-\gy})  ({\txu+\gx},{\ty+\gy}) -- ({-\gx},{\ty+\gy});
    %
        \draw [densely dotted, draw=gray] ({4*\dx-\gx},{-\gy}) rectangle ({5*\dx+\gx},{\gy});
     \node at ({4.5*\dx},{0.35*\ty}) {$T$};    
     \node at ({-1.5*\dx},{0.5*\ty}) {$p\otimes q$};
   \end{tikzpicture}
\end{mycenter}

If we denote by $\eta_2''$, $\theta_1''$ and $\theta_2''$ the images of $\eta_1'$, $\theta_1'$ and $\theta_2'$ in $r\eqpd E(p\otimes q,T)\in \mc C$, then $\gamma_1$ and $\eta_2''$ belong to the same block in $r$ and so do $\theta_1'$ and $\theta_2'$. Because $(\gamma_1,\gamma_2,\eta_1',\theta_i',\eta_2',\theta_{\neg i}')$ is ordered in $p\otimes q$ for some $i,\neg i\in \{1,2\}$ with $\{i,\neg i\}=\{1,2\}$, the tuple $(\gamma_1,\theta_i'',\eta_2'',\theta_{\neg i}'')$ is then ordered in $r$. Thus, the blocks of $\gamma_1$ and $\eta_2''$ and of  $\theta_1''$ and $\theta_2''$ cross in $r$.
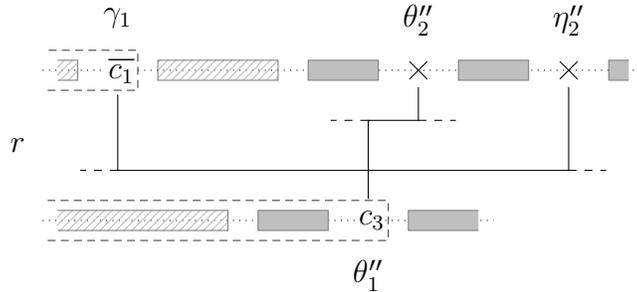
                       \begin{mycenter}[0.5em]
\begin{tikzpicture}[baseline=0.666*1cm-0.25em]
    \def\scp{0.666}
    \def\linksize{\scp*0.075cm}
    \def\pointsize{\scp*0.25cm}
    \def\dd{\scp*0.5cm}
    \def\dx{\scp*1cm}
    \def\cx{\scp*0.3cm}
    \def\txu{11*\dx}    
    \def\txl{8*\dx}
    \def\dy{\scp*1cm}
    \def\cy{\scp*0.3cm}
    \def\ty{3*\dy}
    \def\fy{\scp*0.2cm}
    \def\fx{\scp*0.2cm}
    \def\gy{\scp*0.4cm}
    \def\gx{\scp*0.4cm}      
    \tikzset{whp/.style={circle, inner sep=0pt, text width={\pointsize}, draw=black, fill=white}}
    \tikzset{blp/.style={circle, inner sep=0pt, text width={\pointsize}, draw=black, fill=black}}
    \tikzset{lk/.style={regular polygon, regular polygon sides=4, inner sep=0pt, text width={\linksize}, draw=black, fill=black}}
    \tikzset{cc/.style={cross out, minimum size={1.5*\pointsize-\pgflinewidth}, inner sep=0pt, outer sep=0pt,  draw=black}}    
    \tikzset{vp/.style={circle, inner sep=0pt, text width={1.5*\pointsize}, fill=white}}
    \tikzset{sstr/.style={shorten <= 3pt, shorten >= 3pt}}    
    \draw[dotted] ({0-\dd},{0}) -- ({\txl+\dd},{0});
    \draw[dotted] ({0-\dd},{\ty}) -- ({\txu+\dd},{\ty});
    \node[vp] (u1) at ({0+1*\dx},{0+1*\ty}) {$\overline{c_1}$};
    \node[vp] (l3) at ({0+6*\dx},{0+0*\ty}) {$c_3$};
    \node[cc] (u2) at ({0+7*\dx},{0+1*\ty}) {};
    \node[cc] (u3) at ({0+10*\dx},{0+1*\ty}) {};    
    \draw[sstr] (l3) --++(0,{2*\dy}) -| (u2);    
    \draw[sstr] (u1) --++(0,{-2*\dy}) -| (u3);
    \draw[dashed] ($(u1)+(0,{-2*\dy})$) --++({-0.75*\dx},0);
    \draw[dashed] ($(u3)+(0,{-2*\dy})$) --++({0.75*\dx},0);
    \draw[dashed] ($(l3)+(0,{2*\dy})$) --++({-0.75*\dx},0);
    \draw[dashed] ($(u2)+(0,{-1*\dy})$) --++({0.75*\dx},0);    
    %
    %%% lower row, left-open, filled
    \draw [draw=gray, pattern = north east lines, pattern color = lightgray] ({0+0*\dx-\fx},{0+0*\dy-\fy}) -- ++ ({3*\dx+2*\fx},0) |- ({0+0*\dx-\fx},{0+0*\dy+\fy});
    %%% lower row, right-open, filled    
    % \draw [draw=gray, pattern = north east lines, pattern color = lightgray]    ({\txl+0*\dx+\fx},{0*\dy-\fy}) -- ++ ({-0*\dx-2*\fx},0) |- ({\txl+0*\dx+\fx},{0*\dy+\fy});
\draw [draw=gray, fill = lightgray]    ({\txl+0*\dx+\fx},{0*\dy-\fy}) -- ++ ({-1*\dx-2*\fx},0) |- ({\txl+0*\dx+\fx},{0*\dy+\fy});    
    %%% lower row, closed, filled
%    \draw [draw=gray, pattern = north west lines, pattern color = darkgray] ({1*\dx-\fx},{-\fy}) rectangle ({2*\dx+\fx},{\fy});
%    \draw [draw=gray, pattern = north east lines, pattern color = lightgray] ({5*\dx-\fx},{-\fy}) rectangle ({6*\dx+\fx},{\fy});
\draw [draw=gray, fill  = lightgray] ({4*\dx-\fx},{0*\ty-\fy}) rectangle ({5*\dx+\fx},{0*\ty+\fy});            
    %%% upper row, left-open, filled    
    \draw [draw=gray, pattern = north east lines, pattern color = lightgray] ({0+0*\dx-\fx},{\ty+0*\dy-\fy}) -- ++ ({0*\dx+2*\fx},0) |- ({0+0*\dx-\fx},{\ty+0*\dy+\fy});
    %%% upper row, right-open, filled        
    % \draw [draw=gray, pattern = north east lines, pattern color = lightgray] ({\txu+0*\dx+\fx},{\ty+0*\dy-\fy}) -- ++ ({-3*\dx-2*\fx},0) |- ({\txu+0*\dx+\fx},{\ty+0*\dy+\fy});
    \draw [draw=gray, fill = lightgray] ({\txu+0*\dx+\fx},{\ty+0*\dy-\fy}) -- ++ ({-0*\dx-2*\fx},0) |- ({\txu+0*\dx+\fx},{\ty+0*\dy+\fy});    
    %%% upper row, closed, filled
    \draw [draw=gray, pattern = north east lines, pattern color = lightgray] ({2*\dx-\fx},{\ty-\fy}) rectangle ({4*\dx+\fx},{\ty+\fy});
    \draw [draw=gray, fill  = lightgray] ({8*\dx-\fx},{\ty-\fy}) rectangle ({9*\dx+\fx},{\ty+\fy});
\draw [draw=gray, fill  = lightgray] ({5*\dx-\fx},{\ty-\fy}) rectangle ({6*\dx+\fx},{\ty+\fy});            
    \node [above ={\cx} of  u1] {$\gamma_1$};
    \node [below = {\cx} of  l3] {$\theta_1''$};
    \node [above ={\cx} of  u2] {$\theta_2''$};
    \node [above ={\cx} of  u3] {$\eta_2''$};      
    %
     %% lower row, left-open
     \draw [draw=darkgray, densely dashed] ({0+0*\dx-\gx},{0+0*\dy-\gy}) --++ ({6*\dx+2*\gx},0) |- ({0+0*\dx-\gx},{0+0*\dy+\gy}); 
     %% lower row, right-open    
     %\draw [draw=darkgray, densely dashed] ({\txl+0*\dx+\gx},{0*\dy-\gy}) --++ ({-3*\dx-2*\gx},0) |- ({\txl+0*\dx+\gx},{0*\dy+\gy});
     %% lower row, open
     % \draw [draw=darkgray, densely dashed] ({-\gx},{-\gy}) -- ({\txl+\gx},{-\gy})  ({\txl+\gx},{\gy}) -- ({-\gx},{\gy});
     %% lower row, closed    
    %\draw [draw=darkgray, densely dashed] ({4*\dx-\gx},{-\gy}) rectangle ({7*\dx+\gx},{\gy});        
     %% upper row, left-open
     \draw [draw=darkgray, densely dashed] ({0+0*\dx-\gx},{\ty+0*\dy-\gy}) --++ ({1*\dx+2*\gx},0) |- ({0+0*\dx-\gx},{\ty+0*\dy+\gy}); 
     %% upper row, right-open
     %\draw [draw=darkgray, densely dashed] ({\txu+0*\dx+\gx},{\ty+0*\dy-\gy}) --++ ({-0*\dx-2*\gx},0) |- ({\txu+0*\dx+\gx},{\ty+0*\dy+\gy});
     %% upper row, open
     % \draw [draw=darkgray, densely dashed] ({-\gx},{\ty-\gy}) -- ({\txu+\gx},{\ty-\gy})  ({\txu+\gx},{\ty+\gy}) -- ({-\gx},{\ty+\gy});
    %
     \node at ({-1*\dx},{0.5*\ty}) {$r$};
   \end{tikzpicture}
 \end{mycenter}
 \noindent
Consequently, from $\delta_{r}(\gamma_1,\theta_1'')\in X_{c_1,c_3}(\mc C)$ and from
                                        \begin{align*}
                      \delta_{r}(\gamma_1,\theta_1'')&=\delta_{p\otimes q}(\gamma_1,\theta_1')-\sigma_{p\otimes q}(T)\\
                                                   &=\delta_{p\otimes q}(\gamma_1,\theta_1')\\
                                                   &=\delta_{p\otimes q}(\gamma_1,\gamma_2)+\delta_{p\otimes q}(\gamma_2,\eta_1')+\delta_{p\otimes q}(\eta_1',\theta_1')\\
                      &=\delta_{p\otimes q}(\gamma_1,\gamma_2)+\delta_{p\otimes q}(\eta_1',\theta_1')\\
                      &=\delta_p(\gamma_1,\gamma_2)+\delta_q(\eta_1,\theta_1)
                    \end{align*}
it                    follows $\delta_p(\gamma_1,\gamma_2)+\delta_q(\eta_1,\theta_1)\in X_{c_1,c_3}(\mc C)= \xi_{c_1,c_3}$. And that is what we needed to see.
                  \end{proof}
                  Finally, we can give the final result of this sec\-tion.
                  
 \begin{proposition}
   \label{lemma:result-s-l-k-x}
   Let $\mc C\subseteq \Cp$ be a non-hy\-per\-octa\-he\-dral category. Then,
   \begin{gather*}
     L(\mc C)=K_{\circ\circ}(\mc C)=K_{\bullet\bullet}(\mc C),\qquad K(\mc C)=K_{\circ\bullet}(\mc C)=K_{\bullet\circ}(\mc C)
        \end{gather*}     
        and
        \begin{gather*}
X(\mc C)=X_{\circ\circ}(\mc C)=X_{\bullet\bullet}(\mc C)=X_{\circ\bullet}(\mc C)=X_{\bullet\circ}(\mc C)
   \end{gather*}      

   and there exist $u\in \{0\}\cup \pint$, $m\in \pint$, $D\subseteq \{0\}\cup\dwi{\lfloor\frac{m}{2}\rfloor}$ and $E\subseteq \{0\}\cup\pint$ such that  the tuple $(\toco,L,K,X)(\mc C)$ is one of the following:
                      \begin{align*}
                      \begin{matrix}
                        \toco(\mc C)& L(\mc C)&K(\mc C)& X(\mc C)  \\ \hline \\[-0.85em]
                        um\integers & m\integers & m\integers & \integers\backslash D_m\\
                        2um\integers & m\!+\!2m\integers & 2m\integers & \integers\backslash D_m\\
                        um\integers & \emptyset & m\integers & \integers\backslash D_m\\                        
 \{0\} & \{0\} & \{0\} & \integers\backslash E_0 \\
 \{0\} & \emptyset & \{0\} & \integers\backslash E_0
                    \end{matrix}
                    \end{align*}   
                  \end{proposition}
                  \begin{proof}
                    Follows from Lem\-mata~\ref{lemma:verifying-axioms-1} -- \ref{lemma:verifying-axioms-7} and the \hyperref[lemma:arithmetic]{Arithmetic Lem\-ma~\ref*{lemma:arithmetic}}.
                  \end{proof}
\section{\texorpdfstring{Step~5: Special Relations between $\toco$, $L$, $K$ and $X$\\depending on $F$ and $V$}{Step~5: Special Relations between Sigma, L, K and X depending on F and V}}
\label{section:special-restrictions}
Our objective remains proving $Z(\mc C)\in\mathsf Q$ for any non-hy\-per\-oc\-ta\-he\-dral category $\mc C\subseteq \Cp$. After studying components $F$  (Sec\-tion~\ref{section:block-sizes}) and $\toco$ (Sec\-tion~\ref{section:total-color-sums}) in isolation and after in\-ves\-ti\-ga\-ting the images of the mappings  $(F,V,L)$  (Sec\-tion~\ref{section:block-color-sums}) and $(\toco,L,K,X)$ (Sec\-tion~\ref{section:color-lengths}), we have arrived at the point where we must take all six components of $Z=(F,V,\toco,L,K,X)$ into account simultaneously. Fortunately, we can capitalize on the results of Sec\-tions~\ref{section:block-sizes}--\ref{section:color-lengths} in this endeavor. In consequence, it largely suffices to understand better the behavior of $(\toco,L,K,X)$ as dependent on $(F,V)$ or, roughly, on $F$.
\par
Recall from \cite[Definition~4.1]{MWNHO1} that a category is non-hy\-per\-oc\-ta\-he\-dral if and only if it is case~$\mc O$, $\mc B$ or~$\mc S$ and that these cases are mutually exclusive.

\subsection{Special Relations in Case~$\mc S$}
\label{section:special-restrictions-s}
For case~$\mc S$ categories $\mc C\subseteq \Cp$, i.e., by Proposition~\ref{proposition:result-F} assuming $F(\mc C)=\pint$, there is just a single fact about $(\toco,L,K, X)(\mc C)$ we have to note, one about $L(\mc C)$.
\begin{proposition}
  \label{lemma:restriction-case-s}
  $0\in L(\mc C)$ for every case~$\mc S$ category $\mc C\subseteq \Cp$.
\end{proposition}
\begin{proof}
  As $\PartSinglesWBTensor\in \mc C$, we can, by Lem\-ma~\hyperref[lemma:singletons-3]{\ref*{lemma:singletons}~\ref*{lemma:singletons-3}}, disconnect the black points in $\PartFourWBWB\in \mc C$ and obtain $\PartCoSingleWBWB\in \mc C$. It follows $\{0\}=L(\{\PartCoSingleWBWB\})\subseteq L(\mc C)$.
\end{proof}
\subsection{Special Relations in Case~$\mc O$}
\label{section:special-restrictions-o}
For case~$\mc O$ categories $\mc C\subseteq \Cp$, i.e., assuming $F(\mc C)=\{2\}$, more than what Proposition~\ref{lemma:result-s-l-k-x} is able to discern can  be said about $\toco(\mc C)$ and $X(\mc C)$. 
\subsubsection{\texorpdfstring{Relation of $\Sigma$ to $L$ and $K$ in Case~$\mc O$}{Relation of Sigma to L and K in Case~O}} First, we treat the total color sums of case~$\mc O$ categories. 
\begin{proposition}
  \label{lemma:k-v-l}  
  Let $\mc C\subseteq \Cp$ be a case~$\mc O$ category and let $m\in \pint$.
  \begin{enumerate}[label=(\alph*)]
  \item\label{lemma:k-v-l-1}   If $(L,K)(\mc C)=(\emptyset,m\integers)$, then $\toco(\mc C)=\{0\}$.
    \item\label{lemma:k-v-l-2} If $(L,K)(\mc C)=(m\integers,m\integers)$ or $(L,K)(\mc C)=(m\!+\!2m\integers,2m\integers)$, then \[\toco(\mc C)=2um\integers\] for some $u\in \{0\}\cup\pint$.
  \end{enumerate}
\end{proposition}

\begin{proof}
  \begin{enumerate}[wide,label=(\alph*)]
  \item By Proposition~\ref{lemma:result-s-l-k-x} there exists $\tilde u\in \{0\}\cup \pint$ such that $\toco(\mc C)=\tilde um\integers$. We suppose $\tilde u\neq 0$ and derive a contradiction. As $\mc C$ is closed under erasing turns and as erasing turns does not affect total color sum, we find $p\in\mc C$ with no turns such that $\toco(p)=\tilde um$. Because $\tilde um>0$, the partition $p$ has at least one block. As all blocks of $p$ are pairs by Proposition~\ref{proposition:result-F}, there is a block $B$ of $p$ with (necessarily subsequent) legs $\alpha,\beta\in B$ and $\alpha\neq \beta$. Since $p$ has no turns, all points of $p$ have normalized color $\circ$. In particular, $\alpha$ and $\beta$ do. That proves $L(\mc C)\neq \emptyset$, contradicting the assumption. 
  \item  Proposition~\ref{lemma:result-s-l-k-x} guarantees that $\toco(\mc C)=\tilde um\integers$ for some $\tilde u\in \{0\}\cup\pint$ and that $\tilde u$ is even if $(L,K)(\mc C)=(m\!+\!2m\integers,2m\integers)$. We want to show that $\tilde u$ is even also if $(L,K)(\mc C)=(m\integers,m\integers)$. If $\tilde u=0$, this claim is true. Hence, suppose $\tilde u>0$. As in Part~\ref{lemma:k-v-l-1}, we utilize $p\in\mc C$ with no turns such that $\toco(p)=\tilde um>0$ and, this time, also with no upper points.
    \par
    For every $i\in \pint$ with $i\leq m$ consider the set
    \begin{align*}
      S_i=\{\lop{j}\mid j\in (i+m\nnint),\, j\leq \tilde um\}
    \end{align*}
    comprising the $i$-th lower point and all its $m$-th neighbors to the right. Then, $\bigcup_{i=1}^m S_i$ comprises all points of $p$ and $|S_i|=\tilde u$ for every $i\in \pint$ with $i\leq m$.
    \par
    The sets $S_1,\ldots,S_{m}$ must all be subpartitions of $p$: Otherwise, we find  $j,j'\in \pint$ with $j<j'\leq \tilde u m$ and $j'-j\notin m\integers$ such that $\lop{j}$ and $\lop{j'}$ belong to the same block. As all of $]\lop{j},\lop{j'}]_p$ has  normalized color $\circ$,  \[\delta_p(\lop{j},\lop{j'})=\sigma_p(]\lop{j},\lop{j'}]_p)=|]\lop{j},\lop{j'}]_p|=j'-j\notin m\integers.\] That contradicts the assumption $L(\mc C)= m\integers$.
    \par
    Because all blocks of $p$ are pairs by Proposition~\ref{proposition:result-F}, subpartitions of $p$ have even cardinality. We conclude $\tilde u=|S_1|\in 2\integers$, which then proves the claim.
    \qedhere
  \end{enumerate}
\end{proof}
\subsubsection{Relation of $X$ to $L$ and $K$ in Case~$\mc O$}
When studying $X(\mc C)$ further for case~$\mc O$ categories $\mc C\subseteq \Cp$, it is best to distinguish whether $(L\cup K)(\mc C)$ contains non-zero elements or not.
\begin{proposition}
  \label{lemma:result-x-case-o-positive-w}
  Let $\mc C\subseteq \Cp$ be a case~$\mc O$ category and let $m\in \pint$.
  \begin{enumerate}[label=(\alph*)]
  \item\label{lemma:result-x-case-o-positive-w-1} If $(L,K)(\mc C)=(m\!+\!2m\integers,2m\integers)$, then $X(\mc C)=\integers$ or $X(\mc C)=\integers\backslash m\integers$.    
  \item\label{lemma:result-x-case-o-positive-w-2} If $(L,K)(\mc C)=(m\integers,m\integers)$ or $(L,K)(\mc C)=(\emptyset,m\integers)$, then $X(\mc C)=\integers$.
  \end{enumerate}
\end{proposition}
\begin{proof} No matter which of the three values $(L,K)(\mc C)$ takes, by Proposition~\ref{lemma:result-s-l-k-x} the set $X(\mc C)$ is $m$-periodic. Therefore, showing $\dwi{m}\subseteq X(\mc C)$ already implies $X(\mc C)=\integers$. Likewise, provided $m\geq 2$,  establishing $\dwi{m\!-\!1}\subseteq X(\mc C)$ forces the conclusion that $X(\mc C)=\integers\backslash m\integers$ or $X(\mc C)=\integers$. 
  \begin{enumerate}[wide, label=(\alph*)]
  \item  First, let $(L,K)(\mc C)=(m\!+\!2m\integers,2m\integers)$. If $m=1$, the $1$-periodicity of $X(\mc C)$ immediately implies $X(\mc C)=\emptyset$ or $X(\mc C)=\integers$. Hence, we can suppose $m\geq 2$ and only need to prove $\dwi{m\!-\!1}\subseteq X(\mc C)$ by the initial remark. 
  \par
  Proposition~\ref{lemma:result-s-l-k-x} lets us infer $K_{\circ\circ}(\mc C)=m\!+\!2m\integers$. Hence, we find a partition $p\in \mc C\subseteq \Cpp$, therein a block $\{\alpha,\beta\}$ with $\alpha$ and $\beta$ both of normalized color $\circ$, with $\alpha\neq \beta$ and with $\delta_p(\alpha,\beta)=m$. Without infringing on any of these assumptions we can additionally suppose that there are no turns $T$ in $p$ such that $T\subseteq ]\alpha,\beta[_p$ (otherwise we erase them). Then, all of $]\alpha,\beta[_p$ has the same normalized color $c\in \colors$.
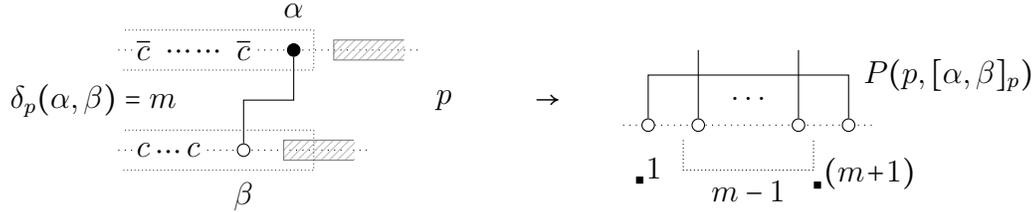
   \begin{mycenter}[0.5em]
\begin{tikzpicture}[baseline=0.666*1cm-0.25em]
    \def\scp{0.666}
    \def\linksize{\scp*0.075cm}
    \def\pointsize{\scp*0.25cm}
    \def\dd{\scp*0.5cm}
    \def\dx{\scp*1cm}
    \def\cx{\scp*0.3cm}
    \def\txu{5*\dx}    
    \def\txl{4*\dx}
    \def\dy{\scp*1cm}
    \def\cy{\scp*0.3cm}
    \def\ty{2*\dy}
    \def\fy{\scp*0.2cm}
    \def\fx{\scp*0.2cm}
    \def\gy{\scp*0.4cm}
    \def\gx{\scp*0.4cm}      
    \tikzset{whp/.style={circle, inner sep=0pt, text width={\pointsize}, draw=black, fill=white}}
    \tikzset{blp/.style={circle, inner sep=0pt, text width={\pointsize}, draw=black, fill=black}}
    \tikzset{lk/.style={regular polygon, regular polygon sides=4, inner sep=0pt, text width={\linksize}, draw=black, fill=black}}
    \tikzset{vp/.style={circle, inner sep=0pt, text width={1.5*\pointsize}, fill=white}}
    \tikzset{sstr/.style={shorten <= 5pt, shorten >= 5pt}}    
    \draw[dotted] ({0-\dd},{0}) -- ({\txl+\dd},{0});
    \draw[dotted] ({0-\dd},{\ty}) -- ({\txu+\dd},{\ty});
    \node[whp] (l1) at ({0+2*\dx},{0+0*\ty}) {};    
    \node[blp] (u1) at ({0+3*\dx},{0+1*\ty}) {};
    \draw (l1) -- ++(0,{\dy}) -| (u1);
    %
    %%% lower row, left-open, filled
    \draw [draw=white, fill=white] ({0+0*\dx-\fx},{0+0*\dy-\fy}) -- ++ ({1*\dx+2*\fx},0) |- ({0+0*\dx-\fx},{0+0*\dy+\fy});
    %%% lower row, right-open, filled    
    \draw [draw=gray, pattern = north east lines, pattern color = lightgray]    ({\txl+0*\dx+\fx},{0*\dy-\fy}) -- ++ ({-1*\dx-2*\fx},0) |- ({\txl+0*\dx+\fx},{0*\dy+\fy});
    %%% lower row, closed, filled
    %\draw [draw=gray, pattern = north east lines, pattern color = lightgray] ({1*\dx-\fx},{-\fy}) rectangle ({2*\dx+\fx},{\fy});
    %\draw [draw=gray, pattern = north east lines, pattern color = lightgray] ({4*\dx-\fx},{-\fy}) rectangle ({6*\dx+\fx},{\fy});    
    %%% upper row, left-open, filled    
    \draw [draw=white, fill=white] ({0+0*\dx-\fx},{\ty+0*\dy-\fy}) -- ++ ({2*\dx+2*\fx},0) |- ({0+0*\dx-\fx},{\ty+0*\dy+\fy});
    %%% upper row, right-open, filled        
    \draw [draw=gray, pattern = north east lines, pattern color = lightgray] ({\txu+0*\dx+\fx},{\ty+0*\dy-\fy}) -- ++ ({-1*\dx-2*\fx},0) |- ({\txu+0*\dx+\fx},{\ty+0*\dy+\fy});
    %%% upper row, closed, filled
    %\draw [draw=gray, pattern = north east lines, pattern color = lightgray] ({1*\dx-\fx},{\ty-\fy}) rectangle ({2*\dx+\fx},{\ty+\fy});
    %\draw [draw=gray, pattern = north east lines, pattern color = lightgray] ({4*\dx-\fx},{\ty-\fy}) rectangle ({6*\dx+\fx},{\ty+\fy});    
    %
    \node (y1) at ({0*\dx},{\ty}) {$\overline c$};
    \node (y2) at ({2*\dx},{\ty}) {$\overline c$};
    \path (y1) -- node[pos=0.5] {\raisebox{-0.25cm}{$\hdots\,\hdots$}} (y2);
    \node (x1) at ({0*\dx},{0}) {$c$};
    \node (x2) at ({1*\dx},{0}) {$c$};
    \path (x1) -- node[pos=0.5] {\raisebox{-0.25cm}{$\hdots$}} (x2);
    \node [below ={\cx} of  l1] {$\beta$};
    \node [above ={\cx} of  u1] {$\alpha$};        
    %
     %% lower row, left-open
     \draw [draw=darkgray, densely dotted] ({0+0*\dx-\gx},{0+0*\dy-\gy}) --++ ({3*\dx+2*\gx},0) |- ({0+0*\dx-\gx},{0+0*\dy+\gy}); 
     %% lower row, right-open    
     %\draw [draw=darkgray, densely dashed] ({\txl+0*\dx+\gx},{0*\dy-\gy}) --++ ({-3*\dx-2*\gx},0) |- ({\txl+0*\dx+\gx},{0*\dy+\gy});
     %% lower row, open
     %\draw [draw=darkgray, densely dashed] ({-\gx},{-\gy}) -- ({\txl+\gx},{-\gy})  ({\txl+\gx},{\gy}) -- ({-\gx},{\gy});
     %% upper row, left-open
     \draw [draw=darkgray, densely dotted] ({0+0*\dx-\gx},{\ty+0*\dy-\gy}) --++ ({3*\dx+2*\gx},0) |- ({0+0*\dx-\gx},{\ty+0*\dy+\gy}); 
     %% upper row, right-open
     %\draw [draw=darkgray, densely dashed] ({\txu+0*\dx+\gx},{\ty+0*\dy-\gy}) --++ ({-3*\dx-2*\gx},0) |- ({\txu+0*\dx+\gx},{\ty+0*\dy+\gy});
     %% upper row, open
     %\draw [draw=darkgray, densely dashed] ({-\gx},{\ty-\gy}) -- ({\txu+\gx},{\ty-\gy})  ({\txu+\gx},{\ty+\gy}) -- ({-\gx},{\ty+\gy});
    %\draw [draw=gray, pattern = north east lines, pattern color = lightgray] ({0+0*\dx-\fx},{0+0*\dy-\fy}) -- ++ ({1*\dx+2*\fx},0) |- ({0+0*\dx-\fx},{0+0*\dy+\fy});
    %\draw [draw=gray, pattern = north east lines, pattern color = lightgray]
({\txl+0*\dx+\fx},{0*\dy-\fy}) -- ++ ({-1*\dx-2*\fx},0) |- ({\txl+0*\dx+\fx},{0*\dy+\fy});
    %\draw [draw=gray, pattern = north east lines, pattern color = lightgray] ({0+0*\dx-\fx},{\ty+0*\dy-\fy}) -- ++ ({3*\dx+2*\fx},0) |- ({0+0*\dx-\fx},{\ty+0*\dy+\fy});
    %\draw [draw=gray, pattern = north east lines, pattern color = lightgray] ({\txu+0*\dx+\fx},{\ty+0*\dy-\fy}) -- ++ ({-0*\dx-2*\fx},0) |- ({\txu+0*\dx+\fx},{\ty+0*\dy+\fy});
    %
    \node at ({\txu+\dx},{0.5*\ty}) {$p$};
    \node at (-1*\dx,{0.5*\ty}) {$\delta_p(\alpha,\beta)=m$};
  \end{tikzpicture}
  \qquad$\to$\qquad
  \begin{tikzpicture}[baseline=0.666*0.5cm-0.25em]
    \def\scp{0.666}
    \def\linksize{\scp*0.075cm}
    \def\pointsize{\scp*0.25cm}
    \def\dd{\scp*0.5cm}
    \def\dx{\scp*1cm}
    \def\cx{\scp*0.3cm}
    \def\txu{5*\dx}    
    \def\txl{4*\dx}
    \def\dy{\scp*1cm}
    \def\cy{\scp*0.3cm}
    \def\ty{2*\dy}
    \def\fy{\scp*0.2cm}
    \def\fx{\scp*0.2cm}
    \def\gy{\scp*0.4cm}
    \def\gx{\scp*0.4cm}      
    \tikzset{whp/.style={circle, inner sep=0pt, text width={\pointsize}, draw=black, fill=white}}
    \tikzset{blp/.style={circle, inner sep=0pt, text width={\pointsize}, draw=black, fill=black}}
    \tikzset{lk/.style={regular polygon, regular polygon sides=4, inner sep=0pt, text width={\linksize}, draw=black, fill=black}}
    \tikzset{vp/.style={circle, inner sep=0pt, text width={1.5*\pointsize}, fill=white}}
    \tikzset{sstr/.style={shorten <= 5pt, shorten >= 5pt}}    
    \draw[dotted] ({0-\dd},{0}) -- ({\txl+\dd},{0});
    %\draw[dotted] ({0-\dd},{\ty}) -- ({\txu+\dd},{\ty});
    %
    \node[whp] (l1) at ({0+0*\dx},{0+0*\ty}) {};    
    \node[whp] (l2) at ({0+1*\dx},{0+0*\ty}) {};
    \node[whp] (l3) at ({0+3*\dx},{0+0*\ty}) {};
    \node[whp] (l4) at ({0+4*\dx},{0+0*\ty}) {};
    \draw (l1) -- ++ (0,{\dy}) -| (l4);
    \draw (l2) -- ++ (0, {1.5*\dy});
    \draw (l3) -- ++ (0, {1.5*\dy});
    %
    %%% lower row, left-open, filled
    %\draw [draw=white, fill=white] ({0+0*\dx-\fx},{0+0*\dy-\fy}) -- ++ ({1*\dx+2*\fx},0) |- ({0+0*\dx-\fx},{0+0*\dy+\fy});
    %%% lower row, right-open, filled    
    %\draw [draw=gray, pattern = north east lines, pattern color = lightgray]    ({\txl+0*\dx+\fx},{0*\dy-\fy}) -- ++ ({-1*\dx-2*\fx},0) |- ({\txl+0*\dx+\fx},{0*\dy+\fy});
    %%% lower row, closed, filled
    %\draw [draw=gray, pattern = north east lines, pattern color = lightgray] ({1*\dx-\fx},{-\fy}) rectangle ({2*\dx+\fx},{\fy});
    %\draw [draw=gray, pattern = north east lines, pattern color = lightgray] ({4*\dx-\fx},{-\fy}) rectangle ({6*\dx+\fx},{\fy});    
    %%% upper row, left-open, filled    
    %\draw [draw=white, fill=white] ({0+0*\dx-\fx},{\ty+0*\dy-\fy}) -- ++ ({2*\dx+2*\fx},0) |- ({0+0*\dx-\fx},{\ty+0*\dy+\fy});
    %%% upper row, right-open, filled        
    %\draw [draw=gray, pattern = north east lines, pattern color = lightgray] ({\txu+0*\dx+\fx},{\ty+0*\dy-\fy}) -- ++ ({-1*\dx-2*\fx},0) |- ({\txu+0*\dx+\fx},{\ty+0*\dy+\fy});
    %%% upper row, closed, filled
    %\draw [draw=gray, pattern = north east lines, pattern color = lightgray] ({1*\dx-\fx},{\ty-\fy}) rectangle ({2*\dx+\fx},{\ty+\fy});
    %\draw [draw=gray, pattern = north east lines, pattern color = lightgray] ({4*\dx-\fx},{\ty-\fy}) rectangle ({6*\dx+\fx},{\ty+\fy});    
    %
    \path (l2) -- node[pos=0.5, yshift=0.35cm] {$\ldots$} (l3);
    \node [below ={\cx} of  l1] {$\lop{1}$};
    \node [xshift=0.5em, below ={\cx} of  l4] {$\lop{(m\!+\!1)}$};        
    %
     %% lower row, left-open
     %\draw [draw=darkgray, densely dotted] ({0+0*\dx-\gx},{0+0*\dy-\gy}) --++ ({3*\dx+2*\gx},0) |- ({0+0*\dx-\gx},{0+0*\dy+\gy}); 
     %% lower row, right-open    
     %\draw [draw=darkgray, densely dashed] ({\txl+0*\dx+\gx},{0*\dy-\gy}) --++ ({-3*\dx-2*\gx},0) |- ({\txl+0*\dx+\gx},{0*\dy+\gy});
     %% lower row, open
     %\draw [draw=darkgray, densely dashed] ({-\gx},{-\gy}) -- ({\txl+\gx},{-\gy})  ({\txl+\gx},{\gy}) -- ({-\gx},{\gy});
     %% upper row, left-open
     %\draw [draw=darkgray, densely dotted] ({0+0*\dx-\gx},{\ty+0*\dy-\gy}) --++ ({3*\dx+2*\gx},0) |- ({0+0*\dx-\gx},{\ty+0*\dy+\gy}); 
     %% upper row, right-open
     %\draw [draw=darkgray, densely dashed] ({\txu+0*\dx+\gx},{\ty+0*\dy-\gy}) --++ ({-3*\dx-2*\gx},0) |- ({\txu+0*\dx+\gx},{\ty+0*\dy+\gy});
     %% upper row, open
     %\draw [draw=darkgray, densely dashed] ({-\gx},{\ty-\gy}) -- ({\txu+\gx},{\ty-\gy})  ({\txu+\gx},{\ty+\gy}) -- ({-\gx},{\ty+\gy});
    %\draw [draw=gray, pattern = north east lines, pattern color = lightgray] ({0+0*\dx-\fx},{0+0*\dy-\fy}) -- ++ ({1*\dx+2*\fx},0) |- ({0+0*\dx-\fx},{0+0*\dy+\fy});
    %\draw [draw=gray, pattern = north east lines, pattern color = lightgray]
({\txl+0*\dx+\fx},{0*\dy-\fy}) -- ++ ({-1*\dx-2*\fx},0) |- ({\txl+0*\dx+\fx},{0*\dy+\fy});
    %\draw [draw=gray, pattern = north east lines, pattern color = lightgray] ({0+0*\dx-\fx},{\ty+0*\dy-\fy}) -- ++ ({3*\dx+2*\fx},0) |- ({0+0*\dx-\fx},{\ty+0*\dy+\fy});
    %\draw [draw=gray, pattern = north east lines, pattern color = lightgray] ({\txu+0*\dx+\fx},{\ty+0*\dy-\fy}) -- ++ ({-0*\dx-2*\fx},0) |- ({\txu+0*\dx+\fx},{\ty+0*\dy+\fy});
    %
    \node at ({\txu+\dx},{0.5*\ty}) {$P(p,[\alpha,\beta]_p)$};
    \draw[densely dotted] ($(l2)+({-\cx},{-\cy})$) --++ (0,{-2*\cy}) -| node[pos=0.25, below] {$m-1$}  ($(l3)+({\cx},{-\cy})$);
  \end{tikzpicture}
    \end{mycenter}
  Because $\alpha$ and $\beta$ also identically have normalized color $\circ$,
  \begin{align*}
    m=\delta_p(\alpha,\beta)=\sigma_p(]\alpha,\beta]_p)=
    \begin{cases}
      |]\alpha,\beta]_p| & \text{if }c=\circ,\\
      -|]\alpha,\beta]_p| &\text{otherwise}.
    \end{cases}
  \end{align*}
  As $m>0$, the only option is $c=\circ$. That means $[\alpha,\beta]_p$ consists of $m+1$ points of normalized color $\circ$.
  \par
  By definition of the projection operation and by Lem\-ma~\ref{lemma:projection}, it is possible to further add the premise $p=P(p,[\alpha,\beta]_p)$ without impacting any of the previous assumptions. Now, $p$ is also projective and $[\alpha,\beta]_p=[\lop{1},\lop{(m+1)}]_p$ is its lower row.
  \par
  For every $j\in \pint$ with $1<j<m+1$ the point $\lop{j}$ belongs to a through block: Assuming otherwise, forces us to accept the existence of $j,j'\in \pint$ with $1<j<j'<m+1$ such that $\lop{j}$ and $\lop{j'}$ belong to the same block. But then, the uniform color $\circ$ of $[\alpha,\beta]_p$ implies \[1\leq \delta_p(\lop{j},\lop{j'})=j'-j\leq m-2\leq m-1\] and thus $L(\mc C)\cap \{1,\ldots,m-1\}\neq \emptyset$, contradicting $L(\mc C)\subseteq m\integers$.
  \par
  Thus we have shown that $\alpha=\lop{1}$ and $\lop{j}$ belong to crossing blocks for every $j\in \pint$ with $1<j<m+1$. Because $\delta_p(\alpha,\lop{j})=j-1$ for every such $j$, this proves $\dwi{m\!-\!1}\subseteq X(\mc C)$. And that is what we needed to show.
\item Let $(L,K)(\mc C)$ be given by $(m\integers,m\integers)$ or $(\emptyset,m\integers)$.  We adapt the proof of Part~\ref{lemma:result-x-case-o-positive-w-1}. However, this time, we do \emph{not} yet impose any restriction on $m$.
  \par
  Proposition~\ref{lemma:result-s-l-k-x} assures us that $K_{\circ\bullet}(\mc C)=K(\mc C)=m\integers$. Hence, we again find $p\in \mc C$, a block $B$ of $p$ and legs $\alpha,\beta\in B$ with $\alpha\neq \beta$, with $]\alpha,\beta[_p\cap B=\emptyset$ and with $\delta_p(\alpha,\beta)=m$, but this time, such that $\alpha$ is of normalized color $\circ$ and $\beta$ of normalized color $\bullet$. By the same argument as before we can assume that all points of $]\alpha,\beta[_p$ share the same  normalized color. Then, the deviating assumption on the colors of $\alpha$ and $\beta$ implies $m=\delta_p(\alpha,\beta)=\sigma_p(]\alpha,\beta[_p)=|]\alpha,\beta[_p|$, which forces $[\alpha,\beta]_p$ to consist of exactly $m+2$ points (rather than $m+1$ as in Part~\ref{lemma:result-x-case-o-positive-w-1}), the first $m+1$ of which have normalized color $\circ$.  Once more, we can assume $p=P(p,[\alpha,\beta]_p)$.
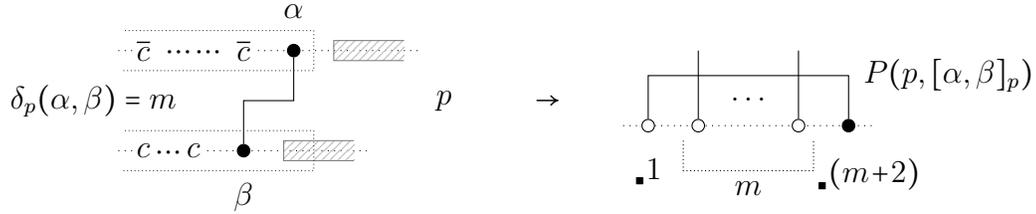
     \begin{mycenter}[0.5em]
\begin{tikzpicture}[baseline=0.666*1cm-0.25em]
    \def\scp{0.666}
    \def\linksize{\scp*0.075cm}
    \def\pointsize{\scp*0.25cm}
    \def\dd{\scp*0.5cm}
    \def\dx{\scp*1cm}
    \def\cx{\scp*0.3cm}
    \def\txu{5*\dx}    
    \def\txl{4*\dx}
    \def\dy{\scp*1cm}
    \def\cy{\scp*0.3cm}
    \def\ty{2*\dy}
    \def\fy{\scp*0.2cm}
    \def\fx{\scp*0.2cm}
    \def\gy{\scp*0.4cm}
    \def\gx{\scp*0.4cm}      
    \tikzset{whp/.style={circle, inner sep=0pt, text width={\pointsize}, draw=black, fill=white}}
    \tikzset{blp/.style={circle, inner sep=0pt, text width={\pointsize}, draw=black, fill=black}}
    \tikzset{lk/.style={regular polygon, regular polygon sides=4, inner sep=0pt, text width={\linksize}, draw=black, fill=black}}
    \tikzset{vp/.style={circle, inner sep=0pt, text width={1.5*\pointsize}, fill=white}}
    \tikzset{sstr/.style={shorten <= 5pt, shorten >= 5pt}}    
    \draw[dotted] ({0-\dd},{0}) -- ({\txl+\dd},{0});
    \draw[dotted] ({0-\dd},{\ty}) -- ({\txu+\dd},{\ty});
    \node[blp] (l1) at ({0+2*\dx},{0+0*\ty}) {};    
    \node[blp] (u1) at ({0+3*\dx},{0+1*\ty}) {};
    \draw (l1) -- ++(0,{1*\dy}) -| (u1);
    %
    %%% lower row, left-open, filled
    \draw [draw=white, fill=white] ({0+0*\dx-\fx},{0+0*\dy-\fy}) -- ++ ({1*\dx+2*\fx},0) |- ({0+0*\dx-\fx},{0+0*\dy+\fy});
    %%% lower row, right-open, filled    
    \draw [draw=gray, pattern = north east lines, pattern color = lightgray]    ({\txl+0*\dx+\fx},{0*\dy-\fy}) -- ++ ({-1*\dx-2*\fx},0) |- ({\txl+0*\dx+\fx},{0*\dy+\fy});
    %%% lower row, closed, filled
    %\draw [draw=gray, pattern = north east lines, pattern color = lightgray] ({1*\dx-\fx},{-\fy}) rectangle ({2*\dx+\fx},{\fy});
    %\draw [draw=gray, pattern = north east lines, pattern color = lightgray] ({4*\dx-\fx},{-\fy}) rectangle ({6*\dx+\fx},{\fy});    
    %%% upper row, left-open, filled    
    \draw [draw=white, fill=white] ({0+0*\dx-\fx},{\ty+0*\dy-\fy}) -- ++ ({2*\dx+2*\fx},0) |- ({0+0*\dx-\fx},{\ty+0*\dy+\fy});
    %%% upper row, right-open, filled        
    \draw [draw=gray, pattern = north east lines, pattern color = lightgray] ({\txu+0*\dx+\fx},{\ty+0*\dy-\fy}) -- ++ ({-1*\dx-2*\fx},0) |- ({\txu+0*\dx+\fx},{\ty+0*\dy+\fy});
    %%% upper row, closed, filled
    %\draw [draw=gray, pattern = north east lines, pattern color = lightgray] ({1*\dx-\fx},{\ty-\fy}) rectangle ({2*\dx+\fx},{\ty+\fy});
    %\draw [draw=gray, pattern = north east lines, pattern color = lightgray] ({4*\dx-\fx},{\ty-\fy}) rectangle ({6*\dx+\fx},{\ty+\fy});    
    %
    \node (y1) at ({0*\dx},{\ty}) {$\overline c$};
    \node (y2) at ({2*\dx},{\ty}) {$\overline c$};
    \path (y1) -- node[pos=0.5] {\raisebox{-0.25cm}{$\hdots\,\hdots$}} (y2);
    \node (x1) at ({0*\dx},{0}) {$c$};
    \node (x2) at ({1*\dx},{0}) {$c$};
    \path (x1) -- node[pos=0.5] {\raisebox{-0.25cm}{$\hdots$}} (x2);
    \node [below ={\cx} of  l1] {$\beta$};
    \node [above ={\cx} of  u1] {$\alpha$};        
    %
     %% lower row, left-open
     \draw [draw=darkgray, densely dotted] ({0+0*\dx-\gx},{0+0*\dy-\gy}) --++ ({3*\dx+2*\gx},0) |- ({0+0*\dx-\gx},{0+0*\dy+\gy}); 
     %% lower row, right-open    
     %\draw [draw=darkgray, densely dashed] ({\txl+0*\dx+\gx},{0*\dy-\gy}) --++ ({-3*\dx-2*\gx},0) |- ({\txl+0*\dx+\gx},{0*\dy+\gy});
     %% lower row, open
     %\draw [draw=darkgray, densely dashed] ({-\gx},{-\gy}) -- ({\txl+\gx},{-\gy})  ({\txl+\gx},{\gy}) -- ({-\gx},{\gy});
     %% upper row, left-open
     \draw [draw=darkgray, densely dotted] ({0+0*\dx-\gx},{\ty+0*\dy-\gy}) --++ ({3*\dx+2*\gx},0) |- ({0+0*\dx-\gx},{\ty+0*\dy+\gy}); 
     %% upper row, right-open
     %\draw [draw=darkgray, densely dashed] ({\txu+0*\dx+\gx},{\ty+0*\dy-\gy}) --++ ({-3*\dx-2*\gx},0) |- ({\txu+0*\dx+\gx},{\ty+0*\dy+\gy});
     %% upper row, open
     %\draw [draw=darkgray, densely dashed] ({-\gx},{\ty-\gy}) -- ({\txu+\gx},{\ty-\gy})  ({\txu+\gx},{\ty+\gy}) -- ({-\gx},{\ty+\gy});
    %\draw [draw=gray, pattern = north east lines, pattern color = lightgray] ({0+0*\dx-\fx},{0+0*\dy-\fy}) -- ++ ({1*\dx+2*\fx},0) |- ({0+0*\dx-\fx},{0+0*\dy+\fy});
    %\draw [draw=gray, pattern = north east lines, pattern color = lightgray]
({\txl+0*\dx+\fx},{0*\dy-\fy}) -- ++ ({-1*\dx-2*\fx},0) |- ({\txl+0*\dx+\fx},{0*\dy+\fy});
    %\draw [draw=gray, pattern = north east lines, pattern color = lightgray] ({0+0*\dx-\fx},{\ty+0*\dy-\fy}) -- ++ ({3*\dx+2*\fx},0) |- ({0+0*\dx-\fx},{\ty+0*\dy+\fy});
    %\draw [draw=gray, pattern = north east lines, pattern color = lightgray] ({\txu+0*\dx+\fx},{\ty+0*\dy-\fy}) -- ++ ({-0*\dx-2*\fx},0) |- ({\txu+0*\dx+\fx},{\ty+0*\dy+\fy});
    %
    \node at ({\txu+\dx},{0.5*\ty}) {$p$};
    \node at (-1*\dx,{0.5*\ty}) {$\delta_p(\alpha,\beta)=m$};
  \end{tikzpicture}
  \qquad$\to$\qquad
  \begin{tikzpicture}[baseline=0.666*0.5cm-0.25em]
    \def\scp{0.666}
    \def\linksize{\scp*0.075cm}
    \def\pointsize{\scp*0.25cm}
    \def\dd{\scp*0.5cm}
    \def\dx{\scp*1cm}
    \def\cx{\scp*0.3cm}
    \def\txu{5*\dx}    
    \def\txl{4*\dx}
    \def\dy{\scp*1cm}
    \def\cy{\scp*0.3cm}
    \def\ty{2*\dy}
    \def\fy{\scp*0.2cm}
    \def\fx{\scp*0.2cm}
    \def\gy{\scp*0.4cm}
    \def\gx{\scp*0.4cm}      
    \tikzset{whp/.style={circle, inner sep=0pt, text width={\pointsize}, draw=black, fill=white}}
    \tikzset{blp/.style={circle, inner sep=0pt, text width={\pointsize}, draw=black, fill=black}}
    \tikzset{lk/.style={regular polygon, regular polygon sides=4, inner sep=0pt, text width={\linksize}, draw=black, fill=black}}
    \tikzset{vp/.style={circle, inner sep=0pt, text width={1.5*\pointsize}, fill=white}}
    \tikzset{sstr/.style={shorten <= 5pt, shorten >= 5pt}}    
    \draw[dotted] ({0-\dd},{0}) -- ({\txl+\dd},{0});
    %\draw[dotted] ({0-\dd},{\ty}) -- ({\txu+\dd},{\ty});
    %
    \node[whp] (l1) at ({0+0*\dx},{0+0*\ty}) {};    
    \node[whp] (l2) at ({0+1*\dx},{0+0*\ty}) {};
    \node[whp] (l3) at ({0+3*\dx},{0+0*\ty}) {};
    \node[blp] (l4) at ({0+4*\dx},{0+0*\ty}) {};
    \draw (l1) -- ++ (0,{\dy}) -| (l4);
    \draw (l2) -- ++ (0, {1.5*\dy});
    \draw (l3) -- ++ (0, {1.5*\dy});
    %
    %%% lower row, left-open, filled
    %\draw [draw=white, fill=white] ({0+0*\dx-\fx},{0+0*\dy-\fy}) -- ++ ({1*\dx+2*\fx},0) |- ({0+0*\dx-\fx},{0+0*\dy+\fy});
    %%% lower row, right-open, filled    
    %\draw [draw=gray, pattern = north east lines, pattern color = lightgray]    ({\txl+0*\dx+\fx},{0*\dy-\fy}) -- ++ ({-1*\dx-2*\fx},0) |- ({\txl+0*\dx+\fx},{0*\dy+\fy});
    %%% lower row, closed, filled
    %\draw [draw=gray, pattern = north east lines, pattern color = lightgray] ({1*\dx-\fx},{-\fy}) rectangle ({2*\dx+\fx},{\fy});
    %\draw [draw=gray, pattern = north east lines, pattern color = lightgray] ({4*\dx-\fx},{-\fy}) rectangle ({6*\dx+\fx},{\fy});    
    %%% upper row, left-open, filled    
    %\draw [draw=white, fill=white] ({0+0*\dx-\fx},{\ty+0*\dy-\fy}) -- ++ ({2*\dx+2*\fx},0) |- ({0+0*\dx-\fx},{\ty+0*\dy+\fy});
    %%% upper row, right-open, filled        
    %\draw [draw=gray, pattern = north east lines, pattern color = lightgray] ({\txu+0*\dx+\fx},{\ty+0*\dy-\fy}) -- ++ ({-1*\dx-2*\fx},0) |- ({\txu+0*\dx+\fx},{\ty+0*\dy+\fy});
    %%% upper row, closed, filled
    %\draw [draw=gray, pattern = north east lines, pattern color = lightgray] ({1*\dx-\fx},{\ty-\fy}) rectangle ({2*\dx+\fx},{\ty+\fy});
    %\draw [draw=gray, pattern = north east lines, pattern color = lightgray] ({4*\dx-\fx},{\ty-\fy}) rectangle ({6*\dx+\fx},{\ty+\fy});    
    %
    \path (l2) -- node[pos=0.5, yshift=0.35cm] {$\ldots$} (l3);
    \node [below ={\cx} of  l1] {$\lop{1}$};
    \node [xshift=0.65em, below ={\cx} of  l4] {$\lop{(m\!+\!2)}$};        
    %
     %% lower row, left-open
     %\draw [draw=darkgray, densely dotted] ({0+0*\dx-\gx},{0+0*\dy-\gy}) --++ ({3*\dx+2*\gx},0) |- ({0+0*\dx-\gx},{0+0*\dy+\gy}); 
     %% lower row, right-open    
     %\draw [draw=darkgray, densely dashed] ({\txl+0*\dx+\gx},{0*\dy-\gy}) --++ ({-3*\dx-2*\gx},0) |- ({\txl+0*\dx+\gx},{0*\dy+\gy});
     %% lower row, open
     %\draw [draw=darkgray, densely dashed] ({-\gx},{-\gy}) -- ({\txl+\gx},{-\gy})  ({\txl+\gx},{\gy}) -- ({-\gx},{\gy});
     %% upper row, left-open
     %\draw [draw=darkgray, densely dotted] ({0+0*\dx-\gx},{\ty+0*\dy-\gy}) --++ ({3*\dx+2*\gx},0) |- ({0+0*\dx-\gx},{\ty+0*\dy+\gy}); 
     %% upper row, right-open
     %\draw [draw=darkgray, densely dashed] ({\txu+0*\dx+\gx},{\ty+0*\dy-\gy}) --++ ({-3*\dx-2*\gx},0) |- ({\txu+0*\dx+\gx},{\ty+0*\dy+\gy});
     %% upper row, open
     %\draw [draw=darkgray, densely dashed] ({-\gx},{\ty-\gy}) -- ({\txu+\gx},{\ty-\gy})  ({\txu+\gx},{\ty+\gy}) -- ({-\gx},{\ty+\gy});
    %\draw [draw=gray, pattern = north east lines, pattern color = lightgray] ({0+0*\dx-\fx},{0+0*\dy-\fy}) -- ++ ({1*\dx+2*\fx},0) |- ({0+0*\dx-\fx},{0+0*\dy+\fy});
    %\draw [draw=gray, pattern = north east lines, pattern color = lightgray]
({\txl+0*\dx+\fx},{0*\dy-\fy}) -- ++ ({-1*\dx-2*\fx},0) |- ({\txl+0*\dx+\fx},{0*\dy+\fy});
    %\draw [draw=gray, pattern = north east lines, pattern color = lightgray] ({0+0*\dx-\fx},{\ty+0*\dy-\fy}) -- ++ ({3*\dx+2*\fx},0) |- ({0+0*\dx-\fx},{\ty+0*\dy+\fy});
    %\draw [draw=gray, pattern = north east lines, pattern color = lightgray] ({\txu+0*\dx+\fx},{\ty+0*\dy-\fy}) -- ++ ({-0*\dx-2*\fx},0) |- ({\txu+0*\dx+\fx},{\ty+0*\dy+\fy});
    %
    \node at ({\txu+\dx},{0.5*\ty}) {$P(p,[\alpha,\beta]_p)$};
    \draw[densely dotted] ($(l2)+({-\cx},{-\cy})$) --++ (0,{-2*\cy}) -| node[pos=0.25, below] {$m$}  ($(l3)+({\cx},{-\cy})$);
  \end{tikzpicture}
    \end{mycenter}
  \par
  If $m=1$, then $F(\{p\})=\{2\}$ requires the unique point $\lop{2}\in ]\lop{1},\lop{3}[_p$ to belong to a through block, proving $1\in X(\mc C)$ and thus $X(\mc C)=\integers$ as claimed. Hence, suppose $m\geq 2$ in the following.
  \par
  We prove that only through blocks intersect $]\lop{1},\lop{(m\!+\!2)}[_p$: Supposing that $\lop{j}$ and $\lop{j'}$, where $j,j'\in \pint$ and $1<j<j'<m+2$, belong to the same block requires us to believe, as both $\lop{j}$ and $\lop{j'}$ are $\circ$-colored, that \[1\leq \delta_p(\lop j,\lop j')=j'-j\leq (m+1)-2=m-1\] and thus $L(\mc C)\cap \{1,\ldots,m-1\}\neq \emptyset$. As this would contradict the assumption $L(\mc C)\subseteq m\integers$, this cannot be the case.
  \par
  Now, the conclusion that the blocks of $\lop{1}$ and of $\lop{j}$ cross for every $j\in \pint$ with $1<j<m+2$ and the fact $\delta_p(\lop{1},\lop{j})=j-1$ let us deduce $\dwi{m}\subseteq X(\mc C)$, which is what needed to see.\qedhere
    \end{enumerate}
\end{proof}
                  
                  \begin{proposition}
                    \label{lemma:result-x-subsemigroup}
                    Let  $\mc C\subseteq \Cp$ be a case~$\mc O$ category.
                    \begin{enumerate}[label=(\alph*)]
                    \item\label{lemma:result-x-subsemigroup-1} If $(L,K)(\mc C)=(\{0\},\{0\})$, then  $X(\mc C)=\integers\backslash N_0$ for a sub\-se\-mi\-group $N$ of $(\pint,+)$.
                    \item\label{lemma:result-x-subsemigroup-2}  If $(L,K)(\mc C)=(\emptyset,\{0\})$, then  there exists a sub\-se\-mi\-group $N$ of $(\pint,+)$ such that $X(\mc C)=\integers\backslash N_0$ or $X(\mc C)=\integers\backslash N_0'$.
                      \end{enumerate}
                  \end{proposition}
                  \begin{proof} Let $(L,K)(\mc C)$ be given by $(\{0\},\{0\})$ or $(\emptyset,\{0\})$. We show the two claims jointly in two  steps:
                    \par
                    \textbf{Step~1:} First, we prove that there exists a sub\-se\-mi\-group $N$ of $(\pint,+)$ such that $X(\mc C)=\integers\backslash N_0$ or  $X(\mc C)=\integers\backslash N_0'$. That in itself requires two steps as well.
                    \par
                    \textbf{Step~1.1:} Recall from \cite[Definition~4.1]{MaWe18a} that by $\mc S_0$ we denote the set of all  $p\in \Cpp$ with $\sigma_p(B)=0$ and  $\delta_p(\alpha,\beta)=0$ for all blocks $B$ of $p$ and all $\alpha,\beta\in B$. We justify that it suffices to prove
                    \begin{align}
                      \label{eq:result-x-subsemigroup}
                      \{|z|\mid z\in X(\mc C)\}\backslash\{0\}\overset{!}{\subseteq}\{|z|\mid z\in X(\mc C\cap \mc S_0)\}\tag{$\ast$}
                    \end{align}
                    in order to verify the assertion of Step~1.
                    \par
                    Indeed, in \cite[Theorem~8.3, Lem\-mata~8.1~(b) and~7.16~(c)]{MaWe18b} it was shown that for every category $\mc I\subseteq \mc S_0$ there exists a sub\-se\-mi\-group $N$ of $(\pint,+)$ such that
                    \begin{align*}
                      \{|z|\mid z\in X(\mc I)\}\backslash \{0\}=\pint\backslash N.
                    \end{align*}
                    The set $\mc S_0$ is a category by \cite[Proposition~5.3]{MaWe18a}, which means that so is $\mc C\cap \mc S_0$. Thus, we find a corresponding sub\-se\-mi\-group $N$ for the special case $\mc I= \mc C\cap \mc S_0$. If we now suppose \eqref{eq:result-x-subsemigroup}, which can immediately be sharpened to
                    \begin{align*}
                      \{|z|\mid z\in X(\mc C)\}\backslash\{0\}=\{|z|\mid z\in X(\mc C\cap \mc S_0)\}\backslash \{0\},
                    \end{align*}
                    that implies
                    \begin{align*}
                      \{|z|\mid z\in X(\mc C)\}\backslash \{0\}=\pint\backslash N.
                    \end{align*}
                    As we know $X(\mc C)=-X(\mc C)$ by Proposition~\ref{lemma:result-s-l-k-x}, this is equivalent to
                    \begin{align*}
                      X(\mc C)\backslash \{0\}=\integers\backslash N_0'
                    \end{align*}
and thus the claim of Step~1. Hence, it is indeed sufficient to show \eqref{eq:result-x-subsemigroup}.
\par
                      \textbf{Step~1.2:} We prove \eqref{eq:result-x-subsemigroup}.
                    As $\mc C\subseteq \Cpp$ by Proposition~\ref{proposition:result-F}, we are assured by Lem\-ma~\ref{lemma:simplification-k-x} and Proposition~\ref{lemma:result-s-l-k-x} that $X(\mc C)=X_{c_1,c_2}(\mc C\cap \Cpp)$ for all $c_1,c_2\in \colors$. Now, let $z\in X(\mc C)\backslash \{0\}$ be arbitrary. By definition we find $p\in \mc C\cap \Cpp$ and therein crossing blocks $B_1$ and $B_2$ as well as points $\alpha_1\in B_1$ and $\alpha_2\in B_2$ such that $\delta_p(\alpha_1,\alpha_2)=z$. Then, there exist points $\beta_1\in B_1$ and $\beta_2\in B_2$ such that $\alpha_1\neq \beta_1$ and $\alpha_2\neq \beta_2$ and such that either $(\alpha_1,\alpha_2,\beta_1,\beta_2)$ or $(\alpha_2,\alpha_1,\beta_2,\beta_1)$ is ordered in $p$. As $\toco(\mc C)= \{0\}$ by Proposition~\ref{lemma:result-s-l-k-x} and thus $\toco(p)=0$, we know $\delta_p(\alpha_2,\alpha_1)=-\delta_p(\alpha_1,\alpha_2)$ by \cite[Lem\-ma~2.1]{MWNHO1}. Hence, by renaming $B_1$ and $B_2$ if necessary we can, at the cost of weakening $\delta_p(\alpha_1,\alpha_2)=z$ to $|\delta_p(\alpha_1,\alpha_2)|=|z|$, assume that $(\alpha_1,\alpha_2,\beta_1,\beta_2)$ is ordered.
                    \par As $\mc C\cap \Cpp$ is closed under erasing turns and as $(B_1\cup B_2)\cap ]\alpha_1,\alpha_2[_p=\emptyset$ we can further suppose that no turns $T$ exist in $p$ with $T\subseteq ]\alpha_1,\alpha_2[_p$. In other words, there is $c\in \colors$ such that every point in $]\alpha_1,\alpha_2[_p$ has normalized color $c$.
                  \begin{mycenter}[0.5em]
                    \begin{tikzpicture}[baseline=0.666*1cm-0.25em]
    \def\scp{0.666}
    \def\linksize{\scp*0.075cm}
    \def\pointsize{\scp*0.25cm}
    \def\dd{\scp*0.5cm}
    \def\dx{\scp*1cm}
    \def\cx{\scp*0.3cm}
    \def\txu{8*\dx}    
    \def\txl{9*\dx}
    \def\dy{\scp*1cm}
    \def\cy{\scp*0.3cm}
    \def\ty{3*\dy}
    \def\fy{\scp*0.2cm}
    \def\fx{\scp*0.2cm}
    \def\gy{\scp*0.4cm}
    \def\gx{\scp*0.4cm}      
    \tikzset{whp/.style={circle, inner sep=0pt, text width={\pointsize}, draw=black, fill=white}}
    \tikzset{blp/.style={circle, inner sep=0pt, text width={\pointsize}, draw=black, fill=black}}
    \tikzset{lk/.style={regular polygon, regular polygon sides=4, inner sep=0pt, text width={\linksize}, draw=black, fill=black}}
    \tikzset{cc/.style={cross out, minimum size={1.5*\pointsize-\pgflinewidth}, inner sep=0pt, outer sep=0pt,  draw=black}}    
    \tikzset{vp/.style={circle, inner sep=0pt, text width={1.5*\pointsize}, fill=white}}
    \tikzset{sstr/.style={shorten <= 3pt, shorten >= 3pt}}    
    \draw[dotted] ({0-\dd},{0}) -- ({\txl+\dd},{0});
    \draw[dotted] ({0-\dd},{\ty}) -- ({\txu+\dd},{\ty});
    \draw [draw=white, fill=white] ({7*\dx-\fx},{-\fy}) rectangle ({9*\dx+\fx},{\fy});
    \draw [draw=white, fill=white] ({6*\dx-\fx},{\ty-\fy}) rectangle ({8*\dx+\fx},{\ty+\fy});
    \node[cc] (l1) at ({0+2*\dx},{0+0*\ty}) {};
    \node[cc] (l2) at ({0+6*\dx},{0+0*\ty}) {};
    \node[cc] (u1) at ({0+2*\dx},{0+1*\ty}) {};
    \node[cc] (u2) at ({0+5*\dx},{0+1*\ty}) {};
    \node[vp] (d1) at ({0+7*\dx},{0+0*\ty}) {$c$};
    \node[vp] (d2) at ({0+9*\dx},{0+0*\ty}) {$c$};
    \node[vp] (e1) at ({0+6*\dx},{0+1*\ty}) {$\overline c$};
    \node[vp] (e2) at ({0+8*\dx},{0+1*\ty}) {$\overline c$};                
    \draw[sstr] (l2) -- ++ (0,{2*\dy}) -| (u1);
    \draw[sstr] (l1) -- ++ (0,{1*\dy}) -| (u2);
    %
    %%% lower row, left-open, filled
    \draw [draw=gray, pattern = north east lines, pattern color = lightgray] ({0+0*\dx-\fx},{0+0*\dy-\fy}) -- ++ ({1*\dx+2*\fx},0) |- ({0+0*\dx-\fx},{0+0*\dy+\fy});
    %%% lower row, right-open, filled    
    %\draw [draw=gray, pattern = north east lines, pattern color = lightgray]    ({\txl+0*\dx+\fx},{0*\dy-\fy}) -- ++ ({-1*\dx-2*\fx},0) |- ({\txl+0*\dx+\fx},{0*\dy+\fy});
    %%% lower row, closed, filled
    \draw [draw=gray, pattern = north east lines, pattern color = lightgray] ({3*\dx-\fx},{-\fy}) rectangle ({5*\dx+\fx},{\fy});
    %%% upper row, left-open, filled    
    \draw [draw=gray, pattern = north east lines, pattern color = lightgray] ({0+0*\dx-\fx},{\ty+0*\dy-\fy}) -- ++ ({1*\dx+2*\fx},0) |- ({0+0*\dx-\fx},{\ty+0*\dy+\fy});
    %%% upper row, right-open, filled        
    %\draw [draw=gray, pattern = north east lines, pattern color = lightgray] ({\txu+0*\dx+\fx},{\ty+0*\dy-\fy}) -- ++ ({-1*\dx-2*\fx},0) |- ({\txu+0*\dx+\fx},{\ty+0*\dy+\fy});
     %%% upper row, closed, filled
    \draw [draw=gray, pattern = north east lines, pattern color = lightgray] ({3*\dx-\fx},{\ty-\fy}) rectangle ({4*\dx+\fx},{\ty+\fy});
    \node [above = {\cx} of  u1] {$\beta_1$};
    \node [above = {\cx} of  u2] {$\alpha_2$};
    \node [below ={\cx} of  l1] {$\beta_2$};
    \node [below ={\cx} of  l2] {$\alpha_1$};
    %
     %% lower row, left-open
     %\draw [draw=darkgray, densely dashed] ({0+0*\dx-\gx},{0+0*\dy-\gy}) --++ ({1*\dx+2*\gx},0) |- ({0+0*\dx-\gx},{0+0*\dy+\gy}); 
     %% lower row, right-open    
     \draw [draw=darkgray, densely dashed] ({\txl+0*\dx+\gx},{0*\dy-\gy}) --++ ({-3*\dx-2*\gx},0) |- ({\txl+0*\dx+\gx},{0*\dy+\gy});
     %% lower row, open
     %\draw [draw=darkgray, densely dashed] ({-\gx},{-\gy}) -- ({\txl+\gx},{-\gy})  ({\txl+\gx},{\gy}) -- ({-\gx},{\gy});
     %% upper row, left-open
     %\draw [draw=darkgray, densely dashed] ({0+0*\dx-\gx},{\ty+0*\dy-\gy}) --++ ({2*\dx+2*\gx},0) |- ({0+0*\dx-\gx},{\ty+0*\dy+\gy}); 
     %% upper row, right-open
      \draw [draw=darkgray, densely dashed] ({\txu+0*\dx+\gx},{\ty+0*\dy-\gy}) --++ ({-3*\dx-2*\gx},0) |- ({\txu+0*\dx+\gx},{\ty+0*\dy+\gy});
     %% upper row, open
     %\draw [draw=darkgray, densely dashed] ({-\gx},{\ty-\gy}) -- ({\txu+\gx},{\ty-\gy})  ({\txu+\gx},{\ty+\gy}) -- ({-\gx},{\ty+\gy});
     %
     \path (d1) -- node[pos=0.5] {\raisebox{-0.25cm}{$\hdots\,\hdots$}} (d2);
     \path (e1) -- node[pos=0.5] {\raisebox{-0.25cm}{$\hdots\,\hdots$}} (e2);
    \node at (-\dx,{0.5*\ty}) {$p$};
    \node at ({\txl+2*\dx},{0.5*\ty}) {$|\delta_p(\alpha_1,\alpha_2)|=|z|$};
  \end{tikzpicture}
\end{mycenter}
                    Even further, by   Lem\-ma~\ref{lemma:projection} none of the previous assumptions are violated by assuming that $p=P(p,[\alpha_1,\beta_1]_p)$. Then, $\beta_2$ is the counterpart of $\alpha_2$ on the upper row, $\alpha_1\in [\beta_2,\alpha_2]_p$ and $\beta_1\notin [\beta_2,\alpha_2]_p$. If we let $\epsilon$ be the predecessor of $\alpha_1$, i.e., if $\epsilon$ is the leftmost upper point of $p$, then $(\beta_2,\epsilon,\alpha_1,\alpha_2,\beta_1)$ is ordered.

                  
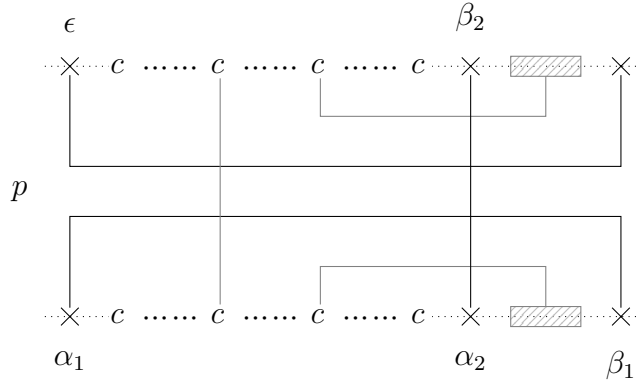
\begin{mycenter}[0.5em]
                    \begin{tikzpicture}[baseline=0.666*1cm-0.25em]
    \def\scp{0.666}
    \def\linksize{\scp*0.075cm}
    \def\pointsize{\scp*0.25cm}
    \def\dd{\scp*0.5cm}
    \def\dx{\scp*1cm}
    \def\cx{\scp*0.3cm}
    \def\txu{11*\dx}    
    \def\txl{11*\dx}
    \def\dy{\scp*1cm}
    \def\cy{\scp*0.3cm}
    \def\ty{5*\dy}
    \def\fy{\scp*0.2cm}
    \def\fx{\scp*0.2cm}
    \def\gy{\scp*0.4cm}
    \def\gx{\scp*0.4cm}      
    \tikzset{whp/.style={circle, inner sep=0pt, text width={\pointsize}, draw=black, fill=white}}
    \tikzset{blp/.style={circle, inner sep=0pt, text width={\pointsize}, draw=black, fill=black}}
    \tikzset{lk/.style={regular polygon, regular polygon sides=4, inner sep=0pt, text width={\linksize}, draw=black, fill=black}}
    \tikzset{cc/.style={cross out, minimum size={1.5*\pointsize-\pgflinewidth}, inner sep=0pt, outer sep=0pt,  draw=black}}    
    \tikzset{vp/.style={circle, inner sep=0pt, text width={1.5*\pointsize}, fill=white}}
    \tikzset{sstr/.style={shorten <= 3pt, shorten >= 3pt}}    
    \draw[dotted] ({0-\dd},{0}) -- ({\txl+\dd},{0});
    \draw[dotted] ({0-\dd},{\ty}) -- ({\txu+\dd},{\ty});
    \draw [draw=white, fill=white] ({1*\dx-\fx},{-\fy}) rectangle ({7*\dx+\fx},{\fy});
    \draw [draw=white, fill=white] ({1*\dx-\fx},{\ty-\fy}) rectangle ({7*\dx+\fx},{\ty+\fy});
    \node[cc] (l1) at ({0+0*\dx},{0+0*\ty}) {};
    \node[cc] (l2) at ({0+8*\dx},{0+0*\ty}) {};
    \node[cc] (l3) at ({0+11*\dx},{0+0*\ty}) {};
    \node[cc] (u1) at ({0+0*\dx},{0+1*\ty}) {};
    \node[cc] (u2) at ({0+8*\dx},{0+1*\ty}) {};
    \node[cc] (u3) at ({0+11*\dx},{0+1*\ty}) {};    
    \node[vp] (d1) at ({0+1*\dx},{0+0*\ty}) {$c$};
    \node[vp] (d2) at ({0+3*\dx},{0+0*\ty}) {$c$};
    \node[vp] (d3) at ({0+5*\dx},{0+0*\ty}) {$c$};
    \node[vp] (d4) at ({0+7*\dx},{0+0*\ty}) {$c$};
    \node[vp] (e1) at ({0+1*\dx},{0+1*\ty}) {$c$};
    \node[vp] (e2) at ({0+3*\dx},{0+1*\ty}) {$c$};
    \node[vp] (e3) at ({0+5*\dx},{0+1*\ty}) {$c$};
    \node[vp] (e4) at ({0+7*\dx},{0+1*\ty}) {$c$};    
    \draw (l1) --++(0,{2*\dy}) -| (l3);
    \draw (u1) --++(0,{-2*\dy}) -| (u3);
    \draw[gray] (d2) -- (e2);
    \draw[gray, shorten >= {\fy}] (d3) --++(0,{\dy}) -| ({9.5*\dx},0);
    \draw[gray, shorten >= {\fy}] (e3) --++(0,{-\dy}) -| ({9.5*\dx},{\ty});     
    \draw (l2) to (u2);    
    %
    %%% lower row, left-open, filled
    %\draw [draw=gray, pattern = north east lines, pattern color = lightgray] ({0+0*\dx-\fx},{0+0*\dy-\fy}) -- ++ ({1*\dx+2*\fx},0) |- ({0+0*\dx-\fx},{0+0*\dy+\fy});
    %%% lower row, right-open, filled    
    %\draw [draw=gray, pattern = north east lines, pattern color = lightgray]    ({\txl+0*\dx+\fx},{0*\dy-\fy}) -- ++ ({-1*\dx-2*\fx},0) |- ({\txl+0*\dx+\fx},{0*\dy+\fy});
    %%% lower row, closed, filled
    \draw [draw=gray, pattern = north east lines, pattern color = lightgray] ({9*\dx-\fx},{-\fy}) rectangle ({10*\dx+\fx},{\fy});
    %%% upper row, left-open, filled    
    %\draw [draw=gray, pattern = north east lines, pattern color = lightgray] ({0+0*\dx-\fx},{\ty+0*\dy-\fy}) -- ++ ({1*\dx+2*\fx},0) |- ({0+0*\dx-\fx},{\ty+0*\dy+\fy});
    %%% upper row, right-open, filled        
    %\draw [draw=gray, pattern = north east lines, pattern color = lightgray] ({\txu+0*\dx+\fx},{\ty+0*\dy-\fy}) -- ++ ({-1*\dx-2*\fx},0) |- ({\txu+0*\dx+\fx},{\ty+0*\dy+\fy});
     %%% upper row, closed, filled
     \draw [draw=gray, pattern = north east lines, pattern color = lightgray] ({9*\dx-\fx},{\ty-\fy}) rectangle ({10*\dx+\fx},{\ty+\fy});
    \node [below = {\cx} of  l1] {$\alpha_1$};
    \node [below = {\cx} of  l2] {$\alpha_2$};
    \node [below ={\cx} of  l3] {$\beta_1$};
    \node [above ={\cx} of  u2] {$\beta_2$};
    \node [above ={\cx} of  u1] {$\epsilon$};    
    %
     %% lower row, left-open
     %\draw [draw=darkgray, densely dashed] ({0+0*\dx-\gx},{0+0*\dy-\gy}) --++ ({1*\dx+2*\gx},0) |- ({0+0*\dx-\gx},{0+0*\dy+\gy}); 
     %% lower row, right-open    
     %\draw [draw=darkgray, densely dashed] ({\txl+0*\dx+\gx},{0*\dy-\gy}) --++ ({-3*\dx-2*\gx},0) |- ({\txl+0*\dx+\gx},{0*\dy+\gy});
     %% lower row, open
     %\draw [draw=darkgray, densely dashed] ({-\gx},{-\gy}) -- ({\txl+\gx},{-\gy})  ({\txl+\gx},{\gy}) -- ({-\gx},{\gy});
     %% upper row, left-open
     %\draw [draw=darkgray, densely dashed] ({0+0*\dx-\gx},{\ty+0*\dy-\gy}) --++ ({2*\dx+2*\gx},0) |- ({0+0*\dx-\gx},{\ty+0*\dy+\gy}); 
     %% upper row, right-open
      %\draw [draw=darkgray, densely dashed] ({\txu+0*\dx+\gx},{\ty+0*\dy-\gy}) --++ ({-3*\dx-2*\gx},0) |- ({\txu+0*\dx+\gx},{\ty+0*\dy+\gy});
     %% upper row, open
     %\draw [draw=darkgray, densely dashed] ({-\gx},{\ty-\gy}) -- ({\txu+\gx},{\ty-\gy})  ({\txu+\gx},{\ty+\gy}) -- ({-\gx},{\ty+\gy});
     %
     \path (d1) -- node[pos=0.5] {\raisebox{-0.25cm}{$\hdots\,\hdots$}} (d2)-- node[pos=0.5] {\raisebox{-0.25cm}{$\hdots\,\hdots$}} (d3)-- node[pos=0.5] {\raisebox{-0.25cm}{$\hdots\,\hdots$}} (d4);
     \path (e1) -- node[pos=0.5] {\raisebox{-0.25cm}{$\hdots\,\hdots$}} (e2) -- node[pos=0.5] {\raisebox{-0.25cm}{$\hdots\,\hdots$}} (e3) -- node[pos=0.5] {\raisebox{-0.25cm}{$\hdots\,\hdots$}} (e4);
    \node at ({-\dx},{0.5*\ty}) {$p$};
    %\node at ({\txl+2*\dx},{0.5*\ty}) {$|\delta_p(\alpha_1,\alpha_2)|=|z|$};
  \end{tikzpicture}
\end{mycenter}                    
Recall that there are no turns $T$ in $p$ with $T\subseteq ]\alpha_1,\alpha_2[_p$. As $p=p^\ast$, there are none with $T\subseteq ]\beta_2,\epsilon[_p$ either. That means every point in $]\alpha_1,\alpha_2[_p$ has  normalized color $c$ and every point in $]\beta_2,\epsilon[_p$  normalized color $\overline{c}$. We can also say a lot about the blocks of $p$ which intersect $[\beta_2,\alpha_2]_p$: If a point $\theta_1\in ]\alpha_1,\alpha_2[_p$ belongs to a through block it must be connected to its counterpart on the upper row because $p\in\Cpp$ is projective. If $\theta_1$ belongs to a non-through block instead, then the partner $\theta_2$ of $\theta_1$ must lie  outside $[\beta_2,\alpha_2]_p$: Supposing otherwise, i.e., $\theta_2\in ]\alpha_1,\alpha_2[_p$, produces a contradiction: If $(\alpha_1,\theta_i,\theta_{\neg i},\beta_2)$ with $i,\neg i\in \{1,2\}$ and $\{i,\neg i\}=\{1,2\}$ is ordered, then, as all points in $[\theta_i,\theta_{\neg i}]_p$ are $c$-colored, the consequence $|\delta_p(\theta_i,\theta_{\neg i})|=|]\theta_i,\theta_{\neg i}]_p|>0$ violates $L(\mc C)\subseteq \{0\}$, which follows from $K(\mc C)=\{0\}$ by Proposition~\ref{lemma:result-s-l-k-x}.
\par
Define $p'\eqpd P(p,[\beta_2,\alpha_2]_p)\in \mc C\cap \Cpp$ and denote by $\beta_2'$, $\epsilon'$, $\alpha_1'$ and $\alpha_2'$ the images in $p'$ of $\beta_2$, $\epsilon$, $\alpha_1$ and $\alpha_2$, respectively. In $p'$ the leftmost lower point $\beta_2'$ and the rightmost lower point $\alpha_2'$ form a block. The points $\epsilon',\alpha_1'\in [\beta_2',\alpha_2']$ are each paired with their respective counterpart on the upper row.   In particular the blocks of $\alpha_1'$ and $\alpha_2'$ cross in $p'$.
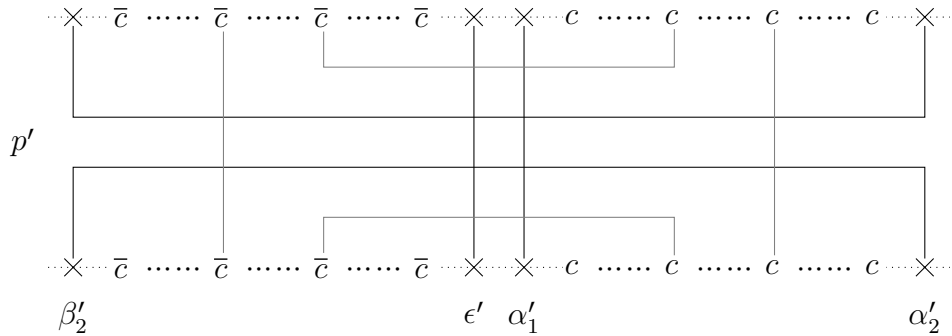
                  \begin{mycenter}[0.5em]
                    \begin{tikzpicture}[baseline=0.666*1cm-0.25em]
    \def\scp{0.666}
    \def\linksize{\scp*0.075cm}
    \def\pointsize{\scp*0.25cm}
    \def\dd{\scp*0.5cm}
    \def\dx{\scp*1cm}
    \def\cx{\scp*0.3cm}
    \def\txu{17*\dx}    
    \def\txl{17*\dx}
    \def\dy{\scp*1cm}
    \def\cy{\scp*0.3cm}
    \def\ty{5*\dy}
    \def\fy{\scp*0.2cm}
    \def\fx{\scp*0.2cm}
    \def\gy{\scp*0.4cm}
    \def\gx{\scp*0.4cm}      
    \tikzset{whp/.style={circle, inner sep=0pt, text width={\pointsize}, draw=black, fill=white}}
    \tikzset{blp/.style={circle, inner sep=0pt, text width={\pointsize}, draw=black, fill=black}}
    \tikzset{lk/.style={regular polygon, regular polygon sides=4, inner sep=0pt, text width={\linksize}, draw=black, fill=black}}
    \tikzset{cc/.style={cross out, minimum size={1.5*\pointsize-\pgflinewidth}, inner sep=0pt, outer sep=0pt,  draw=black}}    
    \tikzset{vp/.style={circle, inner sep=0pt, text width={1.5*\pointsize}, fill=white}}
    \tikzset{sstr/.style={shorten <= 3pt, shorten >= 3pt}}    
    \draw[dotted] ({0-\dd},{0}) -- ({\txl+\dd},{0});
    \draw[dotted] ({0-\dd},{\ty}) -- ({\txu+\dd},{\ty});
    \draw [draw=white, fill=white] ({1*\dx-\fx},{-\fy}) rectangle ({7*\dx+\fx},{\fy});
    \draw [draw=white, fill=white] ({1*\dx-\fx},{\ty-\fy}) rectangle ({7*\dx+\fx},{\ty+\fy});
    \draw [draw=white, fill=white] ({10*\dx-\fx},{-\fy}) rectangle ({16*\dx+\fx},{\fy});
    \draw [draw=white, fill=white] ({10*\dx-\fx},{\ty-\fy}) rectangle ({16*\dx+\fx},{\ty+\fy});
    \node[cc] (l1) at ({0+0*\dx},{0+0*\ty}) {};
    \node[cc] (l2) at ({0+8*\dx},{0+0*\ty}) {};
    \node[cc] (l3) at ({0+9*\dx},{0+0*\ty}) {};
    \node[cc] (l4) at ({0+17*\dx},{0+0*\ty}) {};    
    \node[cc] (u1) at ({0+0*\dx},{0+1*\ty}) {};
    \node[cc] (u2) at ({0+8*\dx},{0+1*\ty}) {};
    \node[cc] (u3) at ({0+9*\dx},{0+1*\ty}) {};
    \node[cc] (u4) at ({0+17*\dx},{0+1*\ty}) {};        
    \node[vp] (d1) at ({0+1*\dx},{0+0*\ty}) {$\overline c$};
    \node[vp] (d2) at ({0+3*\dx},{0+0*\ty}) {$\overline c$};
    \node[vp] (d3) at ({0+5*\dx},{0+0*\ty}) {$\overline c$};
    \node[vp] (d4) at ({0+7*\dx},{0+0*\ty}) {$\overline c$};
    \node[vp] (e1) at ({0+1*\dx},{0+1*\ty}) {$\overline c$};
    \node[vp] (e2) at ({0+3*\dx},{0+1*\ty}) {$\overline c$};
    \node[vp] (e3) at ({0+5*\dx},{0+1*\ty}) {$\overline c$};
    \node[vp] (e4) at ({0+7*\dx},{0+1*\ty}) {$\overline c$};
    \node[vp] (d5) at ({0+10*\dx},{0+0*\ty}) {$c$};
    \node[vp] (d6) at ({0+12*\dx},{0+0*\ty}) {$c$};
    \node[vp] (d7) at ({0+14*\dx},{0+0*\ty}) {$c$};
    \node[vp] (d8) at ({0+16*\dx},{0+0*\ty}) {$c$};
    \node[vp] (e5) at ({0+10*\dx},{0+1*\ty}) {$c$};
    \node[vp] (e6) at ({0+12*\dx},{0+1*\ty}) {$c$};
    \node[vp] (e7) at ({0+14*\dx},{0+1*\ty}) {$c$};
    \node[vp] (e8) at ({0+16*\dx},{0+1*\ty}) {$c$};        
    \draw (l1) --++(0,{2*\dy}) -| (l4);
    \draw (u1) --++(0,{-2*\dy}) -| (u4);
    \draw (l2) to (u2);
    \draw (l3) to (u3);            
    \draw[gray] (d2) -- (e2);
    \draw[gray] (d7) -- (e7);    
    \draw[gray] (d3) --++(0,{\dy}) -| (d6);
    \draw[gray] (e3) --++(0,{-\dy}) -| (e6);

    %
    %%% lower row, left-open, filled
    %\draw [draw=gray, pattern = north east lines, pattern color = lightgray] ({0+0*\dx-\fx},{0+0*\dy-\fy}) -- ++ ({1*\dx+2*\fx},0) |- ({0+0*\dx-\fx},{0+0*\dy+\fy});
    %%% lower row, right-open, filled    
    %\draw [draw=gray, pattern = north east lines, pattern color = lightgray]    ({\txl+0*\dx+\fx},{0*\dy-\fy}) -- ++ ({-1*\dx-2*\fx},0) |- ({\txl+0*\dx+\fx},{0*\dy+\fy});
    %%% lower row, closed, filled
    %\draw [draw=gray, pattern = north east lines, pattern color = lightgray] ({9*\dx-\fx},{-\fy}) rectangle ({10*\dx+\fx},{\fy});
    %%% upper row, left-open, filled    
    %\draw [draw=gray, pattern = north east lines, pattern color = lightgray] ({0+0*\dx-\fx},{\ty+0*\dy-\fy}) -- ++ ({1*\dx+2*\fx},0) |- ({0+0*\dx-\fx},{\ty+0*\dy+\fy});
    %%% upper row, right-open, filled        
    %\draw [draw=gray, pattern = north east lines, pattern color = lightgray] ({\txu+0*\dx+\fx},{\ty+0*\dy-\fy}) -- ++ ({-1*\dx-2*\fx},0) |- ({\txu+0*\dx+\fx},{\ty+0*\dy+\fy});
     %%% upper row, closed, filled
     %\draw [draw=gray, pattern = north east lines, pattern color = lightgray] ({9*\dx-\fx},{\ty-\fy}) rectangle ({10*\dx+\fx},{\ty+\fy});
    %
    \node [below = {\cx} of  l1] {$\beta_2'$};
    \node [below = {\cx} of  l2] {$\epsilon'$};
    \node [below ={\cx} of  l3] {$\alpha_1'$};
    \node [below ={\cx} of  l4] {$\alpha_2'$}; 
    %
     %% lower row, left-open
     %\draw [draw=darkgray, densely dashed] ({0+0*\dx-\gx},{0+0*\dy-\gy}) --++ ({1*\dx+2*\gx},0) |- ({0+0*\dx-\gx},{0+0*\dy+\gy}); 
     %% lower row, right-open    
     %\draw [draw=darkgray, densely dashed] ({\txl+0*\dx+\gx},{0*\dy-\gy}) --++ ({-3*\dx-2*\gx},0) |- ({\txl+0*\dx+\gx},{0*\dy+\gy});
     %% lower row, open
     %\draw [draw=darkgray, densely dashed] ({-\gx},{-\gy}) -- ({\txl+\gx},{-\gy})  ({\txl+\gx},{\gy}) -- ({-\gx},{\gy});
     %% upper row, left-open
     %\draw [draw=darkgray, densely dashed] ({0+0*\dx-\gx},{\ty+0*\dy-\gy}) --++ ({2*\dx+2*\gx},0) |- ({0+0*\dx-\gx},{\ty+0*\dy+\gy}); 
     %% upper row, right-open
      %\draw [draw=darkgray, densely dashed] ({\txu+0*\dx+\gx},{\ty+0*\dy-\gy}) --++ ({-3*\dx-2*\gx},0) |- ({\txu+0*\dx+\gx},{\ty+0*\dy+\gy});
     %% upper row, open
     %\draw [draw=darkgray, densely dashed] ({-\gx},{\ty-\gy}) -- ({\txu+\gx},{\ty-\gy})  ({\txu+\gx},{\ty+\gy}) -- ({-\gx},{\ty+\gy});
     %
     \path (d1) -- node[pos=0.5] {\raisebox{-0.25cm}{$\hdots\,\hdots$}} (d2)-- node[pos=0.5] {\raisebox{-0.25cm}{$\hdots\,\hdots$}} (d3)-- node[pos=0.5] {\raisebox{-0.25cm}{$\hdots\,\hdots$}} (d4);
     \path (e1) -- node[pos=0.5] {\raisebox{-0.25cm}{$\hdots\,\hdots$}} (e2) -- node[pos=0.5] {\raisebox{-0.25cm}{$\hdots\,\hdots$}} (e3) -- node[pos=0.5] {\raisebox{-0.25cm}{$\hdots\,\hdots$}} (e4);
     \path (d5) -- node[pos=0.5] {\raisebox{-0.25cm}{$\hdots\,\hdots$}} (d6)-- node[pos=0.5] {\raisebox{-0.25cm}{$\hdots\,\hdots$}} (d7)-- node[pos=0.5] {\raisebox{-0.25cm}{$\hdots\,\hdots$}} (d8);
     \path (e5) -- node[pos=0.5] {\raisebox{-0.25cm}{$\hdots\,\hdots$}} (e6) -- node[pos=0.5] {\raisebox{-0.25cm}{$\hdots\,\hdots$}} (e7) -- node[pos=0.5] {\raisebox{-0.25cm}{$\hdots\,\hdots$}} (e8);     
    \node at ({-1*\dx},{0.5*\ty}) {$p'$};
    %\node at ({\txl+2*\dx},{0.5*\ty}) {$|\delta_p(\alpha_1,\alpha_2)|=|z|$};
  \end{tikzpicture}
\end{mycenter}                    
Our knowledge about the blocks of $p$ intersecting $[\beta_2,\alpha_2]_p$ lets us draw the following conclusions about the blocks of $p'$: A point in $]\alpha_1',\alpha_2'[_p$ is either partnered with its reflection at the center $[\epsilon',\alpha_1']_{p'}$ of the lower row of $p'$ in $]\beta_2',\epsilon'[_{p'}$ or, as $p'$ is projective, it is partnered with its counterpart on the opposite row. As  $]\beta_2',\epsilon'[_{p'}$ is uniformly $\overline c$-colored and $]\alpha_1',\alpha_2'[_{p'}$ uniformly $c$-colored, that means that all blocks emanating from $]\beta_2',\epsilon'[_{p'}\cup ]\alpha_1',\alpha_2'[_p$ are neutral. But then, all blocks of $p'$ are neutral. Due to $L(\mc C)\subseteq \{0\}$ and $K(\mc C)=\{0\}$, this is already enough to know $p'\in \mc S_0$. Because $\delta_{p'}(\alpha_1',\alpha_2')=\delta_p(\alpha_1,\alpha_2)$, that proves $|z|=|\delta_p(\alpha_1,\alpha_2)|=|\delta_{p'}(\alpha_1',\alpha_2)|\in \{|z|\mid z\in X(\mc C\cap \mc S_0)\}$. As $z$ was arbitrary, \eqref{eq:result-x-subsemigroup} holds true and Part~\ref{lemma:result-x-subsemigroup-2} has been proven.  
                    \par
                    \textbf{Step~2:} In order to prove Part~\ref{lemma:result-x-subsemigroup-1} it remains to show $0\in X(\mc C)$ provided $L(\mc C)=\{0\}$. Under this latter assumption, by Proposition~\ref{lemma:result-s-l-k-x} we infer $K_{\circ\circ}(\mc C)=\{0\}$. Hence, we find $p\in \mc C$, therein a block $B$ and legs $\alpha,\beta\in B$ of normalized color $\circ$ with $\alpha\neq \beta$, with $]\alpha,\beta[_p\cap B=\emptyset$ and with $\delta_p(\alpha,\beta)=0$. As in the proof of Proposition~\ref{lemma:result-x-case-o-positive-w} we can assume that there are no turns $T$ in $p$ such that $T\subseteq]\alpha,\beta[_p$, i.e.\ that all points in $]\alpha,\beta[_p$ have the same normalized color $c\in \colors$. From
                    \begin{align*}
                      0=\delta_p(\alpha,\beta)=\sigma_p(]\alpha,\beta]_p)=\sigma_p(]\alpha,\beta[_p)+\sigma_p(\{\beta\})=
                      \begin{cases}
                        |]\alpha,\beta[_p|+1&\text{if }c=\circ,\\
                        -|]\alpha,\beta[_p|+1&\text{otherwise}
                      \end{cases}
                    \end{align*}
                    and from $|]\alpha,\beta[_p|\geq 0$ it follows that $c=\bullet$ and that $]\alpha,\beta[_p$ is a singleton set. Emulating the proof of Proposition~\ref{lemma:result-x-case-o-positive-w} further, we can assume $p=P(p,[\alpha,\beta]_p)$.
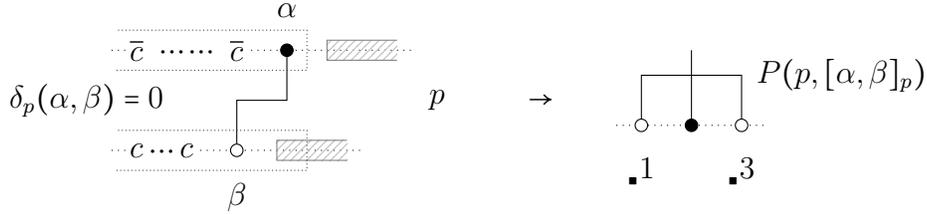
                       \begin{mycenter}[0.5em]
\begin{tikzpicture}[baseline=0.666*1cm-0.25em]
    \def\scp{0.666}
    \def\linksize{\scp*0.075cm}
    \def\pointsize{\scp*0.25cm}
    \def\dd{\scp*0.5cm}
    \def\dx{\scp*1cm}
    \def\cx{\scp*0.3cm}
    \def\txu{5*\dx}    
    \def\txl{4*\dx}
    \def\dy{\scp*1cm}
    \def\cy{\scp*0.3cm}
    \def\ty{2*\dy}
    \def\fy{\scp*0.2cm}
    \def\fx{\scp*0.2cm}
    \def\gy{\scp*0.4cm}
    \def\gx{\scp*0.4cm}      
    \tikzset{whp/.style={circle, inner sep=0pt, text width={\pointsize}, draw=black, fill=white}}
    \tikzset{blp/.style={circle, inner sep=0pt, text width={\pointsize}, draw=black, fill=black}}
    \tikzset{lk/.style={regular polygon, regular polygon sides=4, inner sep=0pt, text width={\linksize}, draw=black, fill=black}}
    \tikzset{vp/.style={circle, inner sep=0pt, text width={1.5*\pointsize}, fill=white}}
    \tikzset{sstr/.style={shorten <= 5pt, shorten >= 5pt}}    
    \draw[dotted] ({0-\dd},{0}) -- ({\txl+\dd},{0});
    \draw[dotted] ({0-\dd},{\ty}) -- ({\txu+\dd},{\ty});
    \node[whp] (l1) at ({0+2*\dx},{0+0*\ty}) {};    
    \node[blp] (u1) at ({0+3*\dx},{0+1*\ty}) {};
    \draw (l1) --++(0,{1*\dy}) -| (u1);
    %
    %%% lower row, left-open, filled
    \draw [draw=white, fill=white] ({0+0*\dx-\fx},{0+0*\dy-\fy}) -- ++ ({1*\dx+2*\fx},0) |- ({0+0*\dx-\fx},{0+0*\dy+\fy});
    %%% lower row, right-open, filled    
    \draw [draw=gray, pattern = north east lines, pattern color = lightgray]    ({\txl+0*\dx+\fx},{0*\dy-\fy}) -- ++ ({-1*\dx-2*\fx},0) |- ({\txl+0*\dx+\fx},{0*\dy+\fy});
    %%% lower row, closed, filled
    %\draw [draw=gray, pattern = north east lines, pattern color = lightgray] ({1*\dx-\fx},{-\fy}) rectangle ({2*\dx+\fx},{\fy});
    %\draw [draw=gray, pattern = north east lines, pattern color = lightgray] ({4*\dx-\fx},{-\fy}) rectangle ({6*\dx+\fx},{\fy});    
    %%% upper row, left-open, filled    
    \draw [draw=white, fill=white] ({0+0*\dx-\fx},{\ty+0*\dy-\fy}) -- ++ ({2*\dx+2*\fx},0) |- ({0+0*\dx-\fx},{\ty+0*\dy+\fy});
    %%% upper row, right-open, filled        
    \draw [draw=gray, pattern = north east lines, pattern color = lightgray] ({\txu+0*\dx+\fx},{\ty+0*\dy-\fy}) -- ++ ({-1*\dx-2*\fx},0) |- ({\txu+0*\dx+\fx},{\ty+0*\dy+\fy});
    %%% upper row, closed, filled
    %\draw [draw=gray, pattern = north east lines, pattern color = lightgray] ({1*\dx-\fx},{\ty-\fy}) rectangle ({2*\dx+\fx},{\ty+\fy});
    %\draw [draw=gray, pattern = north east lines, pattern color = lightgray] ({4*\dx-\fx},{\ty-\fy}) rectangle ({6*\dx+\fx},{\ty+\fy});    
    %
    \node (y1) at ({0*\dx},{\ty}) {$\overline c$};
    \node (y2) at ({2*\dx},{\ty}) {$\overline c$};
    \path (y1) -- node[pos=0.5] {\raisebox{-0.25cm}{$\hdots\,\hdots$}} (y2);
    \node (x1) at ({0*\dx},{0}) {$c$};
    \node (x2) at ({1*\dx},{0}) {$c$};
    \path (x1) -- node[pos=0.5] {\raisebox{-0.25cm}{$\hdots$}} (x2);
    \node [below ={\cx} of  l1] {$\beta$};
    \node [above ={\cx} of  u1] {$\alpha$};        
    %
     %% lower row, left-open
     \draw [draw=darkgray, densely dotted] ({0+0*\dx-\gx},{0+0*\dy-\gy}) --++ ({3*\dx+2*\gx},0) |- ({0+0*\dx-\gx},{0+0*\dy+\gy}); 
     %% lower row, right-open    
     %\draw [draw=darkgray, densely dashed] ({\txl+0*\dx+\gx},{0*\dy-\gy}) --++ ({-3*\dx-2*\gx},0) |- ({\txl+0*\dx+\gx},{0*\dy+\gy});
     %% lower row, open
     %\draw [draw=darkgray, densely dashed] ({-\gx},{-\gy}) -- ({\txl+\gx},{-\gy})  ({\txl+\gx},{\gy}) -- ({-\gx},{\gy});
     %% upper row, left-open
     \draw [draw=darkgray, densely dotted] ({0+0*\dx-\gx},{\ty+0*\dy-\gy}) --++ ({3*\dx+2*\gx},0) |- ({0+0*\dx-\gx},{\ty+0*\dy+\gy}); 
     %% upper row, right-open
     %\draw [draw=darkgray, densely dashed] ({\txu+0*\dx+\gx},{\ty+0*\dy-\gy}) --++ ({-3*\dx-2*\gx},0) |- ({\txu+0*\dx+\gx},{\ty+0*\dy+\gy});
     %% upper row, open
     %\draw [draw=darkgray, densely dashed] ({-\gx},{\ty-\gy}) -- ({\txu+\gx},{\ty-\gy})  ({\txu+\gx},{\ty+\gy}) -- ({-\gx},{\ty+\gy});
    %\draw [draw=gray, pattern = north east lines, pattern color = lightgray] ({0+0*\dx-\fx},{0+0*\dy-\fy}) -- ++ ({1*\dx+2*\fx},0) |- ({0+0*\dx-\fx},{0+0*\dy+\fy});
    %\draw [draw=gray, pattern = north east lines, pattern color = lightgray]
({\txl+0*\dx+\fx},{0*\dy-\fy}) -- ++ ({-1*\dx-2*\fx},0) |- ({\txl+0*\dx+\fx},{0*\dy+\fy});
    %\draw [draw=gray, pattern = north east lines, pattern color = lightgray] ({0+0*\dx-\fx},{\ty+0*\dy-\fy}) -- ++ ({3*\dx+2*\fx},0) |- ({0+0*\dx-\fx},{\ty+0*\dy+\fy});
    %\draw [draw=gray, pattern = north east lines, pattern color = lightgray] ({\txu+0*\dx+\fx},{\ty+0*\dy-\fy}) -- ++ ({-0*\dx-2*\fx},0) |- ({\txu+0*\dx+\fx},{\ty+0*\dy+\fy});
    %
    \node at ({\txu+\dx},{0.5*\ty}) {$p$};
    \node at (-1*\dx,{0.5*\ty}) {$\delta_p(\alpha,\beta)=0$};
  \end{tikzpicture}
  \qquad$\to$\qquad
  \begin{tikzpicture}[baseline=0.666*0.5cm-0.25em]
    \def\scp{0.666}
    \def\linksize{\scp*0.075cm}
    \def\pointsize{\scp*0.25cm}
    \def\dd{\scp*0.5cm}
    \def\dx{\scp*1cm}
    \def\cx{\scp*0.3cm}
    \def\txu{2*\dx}    
    \def\txl{2*\dx}
    \def\dy{\scp*1cm}
    \def\cy{\scp*0.3cm}
    \def\ty{2*\dy}
    \def\fy{\scp*0.2cm}
    \def\fx{\scp*0.2cm}
    \def\gy{\scp*0.4cm}
    \def\gx{\scp*0.4cm}      
    \tikzset{whp/.style={circle, inner sep=0pt, text width={\pointsize}, draw=black, fill=white}}
    \tikzset{blp/.style={circle, inner sep=0pt, text width={\pointsize}, draw=black, fill=black}}
    \tikzset{lk/.style={regular polygon, regular polygon sides=4, inner sep=0pt, text width={\linksize}, draw=black, fill=black}}
    \tikzset{vp/.style={circle, inner sep=0pt, text width={1.5*\pointsize}, fill=white}}
    \tikzset{sstr/.style={shorten <= 5pt, shorten >= 5pt}}    
    \draw[dotted] ({0-\dd},{0}) -- ({\txl+\dd},{0});
    %\draw[dotted] ({0-\dd},{\ty}) -- ({\txu+\dd},{\ty});
    %
    \node[whp] (l1) at ({0+0*\dx},{0+0*\ty}) {};    
    \node[blp] (l2) at ({0+1*\dx},{0+0*\ty}) {};
    \node[whp] (l3) at ({0+2*\dx},{0+0*\ty}) {};
    \draw (l1) -- ++ (0,{\dy}) -| (l3);
    \draw (l2) -- ++ (0, {1.5*\dy});
    %
    %%% lower row, left-open, filled
    %\draw [draw=white, fill=white] ({0+0*\dx-\fx},{0+0*\dy-\fy}) -- ++ ({1*\dx+2*\fx},0) |- ({0+0*\dx-\fx},{0+0*\dy+\fy});
    %%% lower row, right-open, filled    
    %\draw [draw=gray, pattern = north east lines, pattern color = lightgray]    ({\txl+0*\dx+\fx},{0*\dy-\fy}) -- ++ ({-1*\dx-2*\fx},0) |- ({\txl+0*\dx+\fx},{0*\dy+\fy});
    %%% lower row, closed, filled
    %\draw [draw=gray, pattern = north east lines, pattern color = lightgray] ({1*\dx-\fx},{-\fy}) rectangle ({2*\dx+\fx},{\fy});
    %\draw [draw=gray, pattern = north east lines, pattern color = lightgray] ({4*\dx-\fx},{-\fy}) rectangle ({6*\dx+\fx},{\fy});    
    %%% upper row, left-open, filled    
    %\draw [draw=white, fill=white] ({0+0*\dx-\fx},{\ty+0*\dy-\fy}) -- ++ ({2*\dx+2*\fx},0) |- ({0+0*\dx-\fx},{\ty+0*\dy+\fy});
    %%% upper row, right-open, filled        
    %\draw [draw=gray, pattern = north east lines, pattern color = lightgray] ({\txu+0*\dx+\fx},{\ty+0*\dy-\fy}) -- ++ ({-1*\dx-2*\fx},0) |- ({\txu+0*\dx+\fx},{\ty+0*\dy+\fy});
    %%% upper row, closed, filled
    %\draw [draw=gray, pattern = north east lines, pattern color = lightgray] ({1*\dx-\fx},{\ty-\fy}) rectangle ({2*\dx+\fx},{\ty+\fy});
    %\draw [draw=gray, pattern = north east lines, pattern color = lightgray] ({4*\dx-\fx},{\ty-\fy}) rectangle ({6*\dx+\fx},{\ty+\fy});    
    %
%    \path (l2) -- node[pos=0.5, yshift=0.35cm] {$\ldots$} (l3);
    %
    \node [below ={\cx} of  l1] {$\lop{1}$};
    \node [below ={\cx} of  l3] {$\lop{3}$};        
    %
     %% lower row, left-open
     %\draw [draw=darkgray, densely dotted] ({0+0*\dx-\gx},{0+0*\dy-\gy}) --++ ({3*\dx+2*\gx},0) |- ({0+0*\dx-\gx},{0+0*\dy+\gy}); 
     %% lower row, right-open    
     %\draw [draw=darkgray, densely dashed] ({\txl+0*\dx+\gx},{0*\dy-\gy}) --++ ({-3*\dx-2*\gx},0) |- ({\txl+0*\dx+\gx},{0*\dy+\gy});
     %% lower row, open
     %\draw [draw=darkgray, densely dashed] ({-\gx},{-\gy}) -- ({\txl+\gx},{-\gy})  ({\txl+\gx},{\gy}) -- ({-\gx},{\gy});
     %% upper row, left-open
     %\draw [draw=darkgray, densely dotted] ({0+0*\dx-\gx},{\ty+0*\dy-\gy}) --++ ({3*\dx+2*\gx},0) |- ({0+0*\dx-\gx},{\ty+0*\dy+\gy}); 
     %% upper row, right-open
     %\draw [draw=darkgray, densely dashed] ({\txu+0*\dx+\gx},{\ty+0*\dy-\gy}) --++ ({-3*\dx-2*\gx},0) |- ({\txu+0*\dx+\gx},{\ty+0*\dy+\gy});
     %% upper row, open
     %\draw [draw=darkgray, densely dashed] ({-\gx},{\ty-\gy}) -- ({\txu+\gx},{\ty-\gy})  ({\txu+\gx},{\ty+\gy}) -- ({-\gx},{\ty+\gy});
    %\draw [draw=gray, pattern = north east lines, pattern color = lightgray] ({0+0*\dx-\fx},{0+0*\dy-\fy}) -- ++ ({1*\dx+2*\fx},0) |- ({0+0*\dx-\fx},{0+0*\dy+\fy});
    %\draw [draw=gray, pattern = north east lines, pattern color = lightgray]({\txl+0*\dx+\fx},{0*\dy-\fy}) -- ++ ({-1*\dx-2*\fx},0) |- ({\txl+0*\dx+\fx},{0*\dy+\fy});
    %\draw [draw=gray, pattern = north east lines, pattern color = lightgray] ({0+0*\dx-\fx},{\ty+0*\dy-\fy}) -- ++ ({3*\dx+2*\fx},0) |- ({0+0*\dx-\fx},{\ty+0*\dy+\fy});
    %\draw [draw=gray, pattern = north east lines, pattern color = lightgray] ({\txu+0*\dx+\fx},{\ty+0*\dy-\fy}) -- ++ ({-0*\dx-2*\fx},0) |- ({\txu+0*\dx+\fx},{\ty+0*\dy+\fy});
    %
    \node at ({\txu+2*\dx},{0.5*\ty}) {$P(p,[\alpha,\beta]_p)$};
    %\draw[densely dotted] ($(l2)+({-\cx},{-\cy})$) --++ (0,{-2*\cy}) -| node[pos=0.25, below] {$m-1$}  ($(l3)+({\cx},{-\cy})$);
  \end{tikzpicture}
    \end{mycenter}
\noindent
                    Then, the lower row $[\alpha,\beta]_p=[\lop{1},\lop{3}]_p$ of $p$ has coloration $\circ\bullet\circ$. As $p\in \Cpp$ and as $p$ is projective, the block of $\lop{2}$ is the pair $\{\lop 2, \upp2\}$. That means the blocks of $\alpha=\lop{1}$ and $\lop 2$ cross, implying $0=\delta_p(\lop{1},\lop 2)\in X(\mc C)$. That concludes the proof.
                  \end{proof}
                  \section{Step~6: Synthesis}
                  Combining the results from Sec\-tions~\ref{section:block-sizes}--\ref{section:special-restrictions}, we are able to show the main theorem.
                  \begin{theorem}
\label{theorem:main}
                    $Z(\mc C)\in \mathsf Q$ for every non-hy\-per\-octa\-he\-dral category $\mc C\subseteq \Cp$.
                  \end{theorem}
                  \begin{proof}
                    By Lem\-ma~\ref{lemma:result-s-l-k-x} there exist  $u\in \{0\}\cup \pint$, $m\in \pint$, $D\subseteq \{0\}\cup\dwi{\lfloor\frac{m}{2}\rfloor}$ and $E\subseteq \{0\}\cup\pint$ such that the tuple $(\toco,L,K,X)(\mc C)$ is given by one of the following:
                    \begin{align}
                      \label{eq:main-table-1}
                      \begin{matrix}
                        \toco& L&K& X  \\ \hline \\[-0.85em]
                        um\integers & m\integers & m\integers & \integers\backslash D_m\\
                        2um\integers & m\!+\!2m\integers & 2m\integers & \integers\backslash D_m\\
                        um\integers & \emptyset & m\integers & \integers\backslash D_m\\                        
 \{0\} & \{0\} & \{0\} & \integers\backslash E_0 \\
 \{0\} & \emptyset & \{0\} & \integers\backslash E_0
                    \end{matrix}\tag{$\ast$}
                    \end{align}   
                    We treat the three cases $\mc O$, $\mc B$ and $\mc S$ individually. The formulaic presentation will mirror that of Definition~\ref{definition:Q} exactly, to facilitate cross-checking.
                    \par
                    \textbf{Case~$\mc B$:} First, let $\mc C$ be case~$\mc B$. Proposition~\hyperref[proposition:result-F-2]{\ref*{proposition:result-F}~\ref*{proposition:result-F-2}} implies $F(\mc C)=\{1,2\}$. So,  we can immediately add the column for $F(\mc C)$ to table \eqref{eq:main-table-1}. Further, Proposition~\hyperref[proposition:result-V-2]{\ref*{proposition:result-V}~\ref*{proposition:result-V-2}} shows $V(\mc C)=\pm\{0,1,2\}$ if and only if $L(\mc C)\neq \emptyset$ and $V(\mc C)=\pm\{0,1\}$ otherwise. That allows us to fill in the column for $V(\mc C)$ as well. The result is that $Z(\mc C)$ concurs with a row of the  table
                                        \begin{align*}
                      \begin{matrix}
                        F&V&\toco&L&K& X  \\ \hline \\[-0.85em]
                         \{1,2\}&\pm\{0, 1, 2\} & um\integers & m\integers & m\integers & \integers\backslash D_m\\
                                                 \{1,2\}&\pm\{0, 1, 2\} & 2um\integers & m\hspace{-2.5pt}+\hspace{-2.5pt}2m\integers & 2m\integers & \integers\backslash D_m\\
                         \{1,2\}&\pm \{0, 1\} & um\integers & \emptyset & m\integers & \integers\backslash D_m \\
                        \{1,2\}&\pm\{0, 1, 2\} & \{0\} & \{0\} & \{0\} & \integers\backslash E_0 \\
                      \{1,2\}&\pm\{0,  1\} & \{0\} & \emptyset & \{0\} & \integers\backslash E_0\\
                    \end{matrix}
                                        \end{align*}
                                        for some $u\in \{0\}\cup \pint$, $m\in \pint$,  $D\subseteq \{0\}\cup\dwi{\lfloor\frac{m}{2}\rfloor}$ and $E\subseteq \{0\}\cup\pint$.
                                        Hence, by Definition~\ref{definition:Q}, we have shown $Z(\mc C)\in \mathsf Q$ if $\mc C$ is case~$\mc B$.
                    \par
                    \textbf{Case~$\mc S$:} Next, let $\mc C$ be case~$\mc S$. Propositions~\hyperref[proposition:result-F-3]{\ref*{proposition:result-F}~\ref*{proposition:result-F-3}} and~\hyperref[proposition:result-V-3]{\ref*{proposition:result-V}~\ref*{proposition:result-V-3}} guarantee $F(\mc C)=\pint$ and $V(\mc C)=\integers$. Hence, we can fill in the columns for $F$ and $V$ in \eqref{eq:main-table-1} once more. Moreover, $0\in L(\mc C)$ by Proposition~\ref{lemma:restriction-case-s}. Thus, we can exclude that $(\toco,L,K,X)(\mc C)$ is given by the second, third or fifth rows of \eqref{eq:main-table-1}.
In other words, there are  $u\in \{0\}\cup \pint$, $m\in \pint$,  $D\subseteq \{0\}\cup\dwi{\lfloor\frac{m}{2}\rfloor}$ and $E\subseteq \{0\}\cup\pint$ such that $Z(\mc C)$ is given by one of the rows of the following table:
                    \begin{align*}
                      \begin{matrix}
                        F&V&\toco&L&K& X  \\ \hline \\[-0.85em]
                        \pint & \integers & um\integers & m\integers & m\integers & \integers\backslash D_m\\
                        \pint & \integers & \{0\} & \{0\} & \{0\} & \integers\backslash E_0\\
                    \end{matrix}
                    \end{align*}
                    And, by Definition~\ref{definition:Q}, this means $Z(\mc C)\in\mathsf Q$ for $\mc C$ in case~$\mc S$.
                    \par
                    \textbf{Case~$\mc O$:} Lastly, let $\mc C$ be case~$\mc O$. Once more, Propositions~\hyperref[proposition:result-F-1]{\ref*{proposition:result-F}~\ref*{proposition:result-F-1}} and~\hyperref[proposition:result-V-1]{\ref*{proposition:result-V}~\ref*{proposition:result-V-1}} give, on the one hand, $F(\mc C)=\{2\}$ and, on the other hand, $V(\mc C)=\pm\{0,2\}$ if $L(\mc C)\neq \emptyset$ and $V(\mc C)=\{0\}$ otherwise. That enables us to fill in the columns for $F(\mc C)$ and $V(\mc C)$ in \eqref{eq:main-table-1}:
                    \begin{align}
                      \label{eq:main-table-2}
                      \begin{matrix}
                                                F&V&\toco&L&K& X  \\ \hline \\[-0.85em]
                        \{2\} & \pm\{0, 2\} & um\integers & m\integers & m\integers & \integers\backslash D_m\\
                      \{2\} & \pm\{0, 2\} & 2um\integers & m\hspace{-2.5pt}+\hspace{-2.5pt}2m\integers & 2m\integers & \integers\backslash D_m \\
                      \{2\} & \{0\} & um\integers & \emptyset & m\integers & \integers\backslash D_m\\
                      \{2\} & \pm\{0, 2\} & \{0\} & \{0\} & \{0\} &  \integers\backslash E_0 \\
                      \{2\} & \{0\} & \{0\} & \emptyset & \{0\} & \integers\backslash E_0 \\
                    \end{matrix}\tag{$\ast\ast$}
                    \end{align}
                    This is not yet what we claim as this range is not contained in $\mathsf Q$. We need to exclude certain values for $u$, $D$ and $E$ by taking into account the results of Sec\-tion~\ref{section:special-restrictions-o}. This we shall do on a row-by-row basis.
                    \par
\textbf{Case~$\mc O$.1:} First, suppose $(L,K)(\mc C)=(m\integers,m\integers)$ for some $m\in \pint$, as in the first row of Table~\eqref{eq:main-table-2}. Then $\toco(\mc C)\subseteq 2m\integers$ (corresponding to parameters $u\in 2\integers$) according to Pro\-po\-si\-tion~\hyperref[lemma:k-v-l-2]{\ref*{lemma:k-v-l}~\ref*{lemma:k-v-l-2}}. Moreover, $X(\mc C)=\integers$ (corresponding to $D=\emptyset$) as seen in 
Pro\-po\-si\-tion~\hyperref[lemma:result-x-case-o-positive-w-2]{\ref*{lemma:result-x-case-o-positive-w}~\ref*{lemma:result-x-case-o-positive-w-2}}. Hence, we can replace the first row of Table~\eqref{eq:main-table-2} by 
                    \begin{align*}
                      \begin{matrix}
                                                F&V&\toco&L&K& X  \\ \hline \\[-0.85em]
                        \{2\} & \pm\{0, 2\} & 2um\integers & m\integers & m\integers & \integers
                    \end{matrix}
                    \end{align*}
still for parameters $u\in \{0\}\cup  \pint$ and $m\in \pint$ exactly as before.
\par
\textbf{Case~$\mc O$.2:} Now, proceeding to the second row of Table~\eqref{eq:main-table-2}, let $(L,K)(\mc C)=(m\!+\!2m\integers,2m\integers)$ for some $m\in \pint$. By Proposition~\hyperref[lemma:result-x-case-o-positive-w-1]{\ref*{lemma:result-x-case-o-positive-w}~\ref*{lemma:result-x-case-o-positive-w-1}} the only two values $X(\mc C)$ can possibly take are $\integers$ and $\integers\backslash m\integers$ (corresponding to $D=\emptyset$ and $D=\{0\}$, respectively). Thus, we can delete the second row of Table~\eqref{eq:main-table-2} and insert the two new rows
                    \begin{align*}
                      \begin{matrix}
                                                F&V&\toco&L&K& X  \\ \hline \\[-0.85em]
                      \{2\} & \pm\{0, 2\} & 2um\integers & m\hspace{-2.5pt}+\hspace{-2.5pt}2m\integers & 2m\integers & \integers \\
                      \{2\} & \pm\{0, 2\} & 2um\integers & m\hspace{-2.5pt}+\hspace{-2.5pt}2m\integers & 2m\integers & \integers\backslash m\integers \\
                    \end{matrix}
                    \end{align*}
                    in its stead, still for parameters $m\in \pint$ and $u\in \{0\}\cup \pint$.
                    \par
                    \textbf{Case~$\mc O$.3:} Next, assume $(L,K)(\mc C)=(\emptyset,m\integers)$ for some $m\in\pint$ as in row three of Table~\eqref{eq:main-table-2}. Then, in fact, $\toco(\mc C)=\{0\}$ as seen in Proposition~\hyperref[lemma:k-v-l-1]{\ref*{lemma:k-v-l}~\ref*{lemma:k-v-l-1}}. Furthermore, $X(\mc C)=\integers$ by Proposition~\hyperref[lemma:result-x-case-o-positive-w-2]{\ref*{lemma:result-x-case-o-positive-w}~\ref*{lemma:result-x-case-o-positive-w-2}}. Hence, we rewrite the third row of \eqref{eq:main-table-2} as
                    \begin{align*}
                      \begin{matrix}
                                                F&V&\toco&L&K& X  \\ \hline \\[-0.85em]
                      \{2\} & \{0\} & \{0\} & \emptyset & m\integers & \integers
                    \end{matrix}
                    \end{align*}
                    depending only on the parameter $m\in \pint$.
                    \par
                    \textbf{Case~$\mc O$.4:} Let $(L,K)(\mc C)=(\{0\},\{0\})$, i.e., consider the fourth row of Table~\eqref{eq:main-table-2}. Then, $X(\mc C)=\integers\backslash N_0$ for some sub\-se\-mi\-group of $(\pint,+)$ by Proposition~\hyperref[lemma:result-x-subsemigroup-1]{\ref*{lemma:result-x-subsemigroup}~\ref*{lemma:result-x-subsemigroup-1}} (corresponding to $E=N$ being a sub\-se\-mi\-group). Accordingly, we can replace the fourth row of Table~\eqref{eq:main-table-2} by
                                        \begin{align*}
                      \begin{matrix}
                                                F&V&\toco&L&K& X  \\ \hline \\[-0.85em]
                      \{2\} & \pm\{0,2\} & \{0\} & \emptyset & \{0\} & \integers\backslash N_0
                    \end{matrix}                      
                    \end{align*}
                    for a new table parameter $N$, running through all sub\-se\-mi\-groups of $(\pint,+)$.
                    \par
                    \textbf{Case~$\mc O$.5:} Lastly, suppose $(L,K)(\mc C)=(\emptyset,\{0\})$ as in the fifth row of Table~\eqref{eq:main-table-2}. In Proposition~\hyperref[lemma:result-x-subsemigroup-1]{\ref*{lemma:result-x-subsemigroup}~\ref*{lemma:result-x-subsemigroup-1}} we showed $X(\mc C)$ is of the form $\integers\backslash N_0$ or $\integers\backslash N_0'$ for some sub\-se\-mi\-group $N$ of $(\pint,+)$ (corresponding to $E=N$ and $E=\{0\}\cup N$, respectively). Thus, strike the last row of Table~\eqref{eq:main-table-2} and append the two rows
                    \begin{align*}
                      \begin{matrix}
                                                F&V&\toco&L&K& X  \\ \hline \\[-0.85em]
                                                \{2\} & \{0\} & \{0\} & \emptyset & \{0\} & \integers\backslash N_0\\
                      \{2\} & \{0\} & \{0\} & \emptyset & \{0\} & \integers\backslash N_0'                                                
                    \end{matrix}                      
                    \end{align*}
                    to the table, with $N$ being a sub\-se\-mi\-group of $(\pint,+)$.
                    \par
                    \textbf{Synthesis in case~$\mc O$:} If we combine the results of Cases~1--5, then we can say that there exist $m\in \pint$, $u\in \{0\}\cup \pint$ and a sub\-se\-mi\-group $N$ of $(\pint,+)$ such that $Z(\mc C)$ is given by one of the rows of the following table:
                  \begin{align*}
                      \begin{matrix}                    
                        F&V&S&L&K& X  \\ \hline \\[-0.85em]
                        \{2\} & \pm\{0, 2\} & 2um\integers & m\integers & m\integers & \integers\\
                      \{2\} & \pm\{0, 2\} & 2um\integers & m\hspace{-2.5pt}+\hspace{-2.5pt}2m\integers & 2m\integers & \integers \\
                        \{2\} & \pm \{0, 2\} & 2um\integers & m\hspace{-2.5pt}+\hspace{-2.5pt}2m\integers & 2m\integers & \integers\backslash m\integers \\
                      \{2\} & \{0\} & \{0\} & \emptyset & m\integers & \integers\\
                      \{2\} & \pm\{0, 2\} & \{0\} & \{0\} & \{0\} &  \integers\backslash N_0 \\
                      \{2\} & \{0\} & \{0\} & \emptyset & \{0\} & \integers\backslash N_0 \\
                      \{2\} & \{0\} & \{0\} & \emptyset & \{0\} &\integers\backslash N_0' \\
                    \end{matrix}                          
                  \end{align*}
                  Definition~\ref{definition:Q} thus yields $Z(\mc C)\in \mathsf Q$ if $\mc C$ is case~$\mc O$. Hence, the overall claim is true.
                  \end{proof}
   
\printbibliography%[heading=bibintoc]

\end{document}